\documentclass[12pt, oneside]{amsart}

\usepackage{amsmath,amssymb,cite,mathrsfs,tikz-cd}
\usepackage[all]{xy}
\usepackage{typearea} %%Makes better use of the page, IMHO
\usepackage{hyperref} %%Clickable reference in PDF
\usepackage{bm}

\usepackage{etoolbox}

\AtBeginEnvironment{displaymath}{\small}
\AtBeginEnvironment{equation}{\small}
\AtBeginEnvironment{equation*}{\small}

\AtBeginEnvironment{align}{\small}
\AtBeginEnvironment{align*}{\small}
\AtBeginEnvironment{gather}{\small}
\AtBeginEnvironment{gather*}{\small}
\AtBeginEnvironment{flalign}{\small}
\AtBeginEnvironment{flalign*}{\small}

\usepackage{color} %% For colorful text

\usepackage{epstopdf} 
\usepackage{booktabs}

\newtheorem{teo}{Theorem}[section]
\newtheorem{thm}[teo]{Theorem}
\newtheorem{prop}[teo]{Proposition}
\newtheorem{lemma}[teo]{Lemma}
\newtheorem{cor}[teo]{Corollary}
\newtheorem{conj}[teo]{Conjecture}

\newtheorem{defn}[teo]{Definition}

\newtheorem{rmk}[teo]{Remark}

\newtheorem{def-prop}[teo]{Definition-Proposition}

\newtheoremstyle{named}{}{}{\itshape}{}{\bfseries}{.}{.5em}{\thmnote{#3 }#1}
\theoremstyle{named}

\makeatletter
\newcommand{\neutralize}[1]{\expandafter\let\csname c@#1\endcsname\count@}
\makeatother

\newenvironment{thmbis}[1]
  {\renewcommand{\theteo}{\ref*{#1}$'$}%
   \neutralize{teo}\phantomsection
   \begin{thm}}
  {\end{thm}}

\newenvironment{propbis}[1]
  {\renewcommand{\theteo}{\ref*{#1}$'$}%
   \neutralize{teo}\phantomsection
   \begin{prop}}
  {\end{prop}}

\numberwithin{equation}{section}

  \newcommand{\C}{\mathbb{C}}

  \renewcommand{\P}{\mathbb{P}}
  \newcommand{\Q}{\mathbb{Q}}
  \newcommand{\R}{\mathbb{R}}

  \newcommand{\Z}{\mathbb{Z}}

\newcommand{\CC}{\mathbb{C}}

\newcommand{\PP}{\mathbb{P}}

\newcommand{\QQ}{\mathbb{Q}}
\newcommand{\RR}{\mathbb{R}}

\newcommand{\ZZ}{\mathbb{Z}}

\newcommand{\calA}{\mathcal{A}}

\newcommand{\calC}{\mathcal{C}}

\newcommand{\calE}{\mathcal{E}}

\newcommand{\calL}{\mathcal{L}}

\newcommand{\calN}{\mathcal{N}}
\newcommand{\calO}{\mathcal{O}}
\newcommand{\calP}{\mathcal{P}}

\newcommand{\fm}{\mathfrak{m}}

  \newcommand{\idp}{\mathfrak{p}} %prime ideal p
  
  \renewcommand{\cong}{\simeq}
  \renewcommand{\bar}{\overline}
  \renewcommand{\tilde}{\widetilde}
  \renewcommand{\hat}{\widehat}

  \providecommand{\frac}[1]{\operatorname{Frac}(#1)}

  \renewcommand{\hom}{\operatorname{Hom}}

  \newcommand{\Spec}{\operatorname{\textbf{Spec}}}

  \newcommand{\rank}{\operatorname{rank}}

\DeclareMathOperator{\an}{an}

\DeclareMathOperator{\diag}{diag}

\DeclareMathOperator{\Gal}{Gal}

\DeclareMathOperator{\Lie}{Lie}

\DeclareMathOperator{\vol}{vol}

  \renewcommand{\ker}{\operatorname{ker}}
  
  \newcommand{\codim}{\operatorname{codim}}

  \renewcommand{\deg}{\operatorname{deg}}
  
  \newcommand{\ord}{\operatorname{ord}}

  \newcommand{\lie}{\operatorname{Lie}}

\newcommand{\cA}{\mathcal{A}}
\newcommand{\cB}{\mathcal{B}}

\newcommand{\cE}{\mathcal{E}}
\newcommand{\cF}{\mathcal{F}}

\newcommand{\cI}{\mathcal{I}}
\newcommand{\cL}{\mathcal{L}}
\newcommand{\cM}{\mathcal{M}}
\newcommand{\cO}{\mathcal{O}}

\newcommand{\cS}{\mathcal{S}}

\newcommand{\cX}{\mathcal{X}}

\newcommand\supervisor[1]{\def\@supervisor{#1}}

\newcounter{elno}

\renewcommand{\cong}{\simeq}

\title{Uniformity in rational torsion and small points on abelian varieties}
\author{Ziyang Gao, Kaiyuan Gu}

\address{Department of Mathematics, UCLA, Los Angeles, CA 90095, USA}
\email{ziyang.gao@math.ucla.edu}

\address{BICMR, Peking University, Haidian District, Beijing 100871, China}
\email{gky@stu.pku.edu.cn}

\begin{document}

\begin{abstract}
In this paper, we propose a method to study the {\it Uniform Boundedness Conjecture} and the {\it Lang--Silverman Conjecture} for abelian varieties $A$ defined over a global field $K$; the latter is a uniform lower bound on the heights of non-torsion rational points. Our method is inspired by Vojta's proof of the Mordell Conjecture (Faltings's~Theorem).

Over function fields  of characteristic $0$, a recent breakthrough of Looper--Yap \cite{LooperYap} proves both conjectures with inexplicit bounds. In our paper, we give a new proof of both conjectures with explicit bounds, which  depend polynomially on the field $K$ unless $A/K$ admits a factor of good reduction everywhere. We also prove explicit bounds for elliptic curves over function fields of characteristic $p>0$.

Over number fields, we prove both conjectures under a suitable high dimensional Szpiro conjecture (weaker than Hindry's version \cite[Conj.~3.4]{HindrySzpiro}) that we propose.
\end{abstract}

\maketitle

\tableofcontents

\section{Introduction}
The goal of this paper is to investigate and prove uniformity results in the number of rational torsion points  and in the heights of the non-torsion rational points on abelian varieties. Our first motivation is the following {\it Uniform Boundedness Conjecture}:
\begin{conj}[Uniform Boundedness Conjecture]\label{ConjUBC}
Let $A$ be an abelian variety of dimension $g \ge 1$ defined over a number field $K$. Then  $\#A(K)_{\mathrm{tor}} \le c(g,K)$.
\end{conj}
 
Conjecture~\ref{ConjUBC}, if known for all $g \ge 1$, self-improves to $\#A(K)_{\mathrm{tor}} \le c(g,[K \!: \!\Q])$ by  Weil restriction. For $g=1$, Frey \cite{Frey} deduced this conjecture from the $abc$ Conjecture, while the  degree-uniform bound was proved by Mazur \cite{MazurModular-curves-} for $K=\Q$, Kamienny and Mazur \cite{kamienny1986torsion, kamiennymazur1995rational} for $K$ of small degrees, and Merel \cite{MerelBornes-pour-la-} in general, with effective bounds by Parent \cite{ParentBornes-effectiv}. For $g \ge 2$, Conjecture~\ref{ConjUBC} is known for CM abelian varieties by Silverberg \cite{SilverbergTorsion-points-} (see also the first-named author's \cite[Thm.~1.3]{GaoTowards-the-And}) and for $\ell$-primary torsions over $1$-parameter families by Cadoret–Tamagawa \cite{CadoretUniform-bounded} and  Ellenberg–Hall–Kowalski \cite{EllenbergHallKowalski2012} (see also Ji--Song--Xie \cite{JiSongXie2026}).  In general, Conjecture~\ref{ConjUBC} remains widely open when $g\ge 2$.

\vskip 0.3em
The analogue of Conjecture~\ref{ConjUBC} over function fields $k(B)$, \textit{i.e.} the {\it Geometric Uniform Boundedness Conjecture}, also remained open; here $B$ is a curve  over $k$. In characteristic $0$, it was proved for $g=1$ by Levin \cite{Levin1968}, Abramovich and Nguyen--Saito  \cite{Abramovich1996,NguyenSaito1996} and for $A$ with real multiplication by Bakker--Tsimerman \cite{BakkerTsimerman2018}. In characteristic $p>0$, it was proved for $g=1$ by Cojocaru--Hall \cite{CojocaruHall2005} and Poonen \cite{Poonen2007}.\footnote{\cite{Levin1968, CojocaruHall2005} depend on the genus of $B$, \cite{Abramovich1996,NguyenSaito1996, BakkerTsimerman2018,Poonen2007} depend on the gonality of $B$.} 

Remarkably, the Geometric Uniform Boundedness Conjecture is recently proved by Looper--Yap \cite[Thm.~1.1]{LooperYap} in characteristic $0$ in its full generality. Here is the statement. Let $k$ be an algebraically closed field. A {\it traceless} condition needs to be imposed.

\begin{thm}[Looper--Yap]\label{ThmLooperYap}
Assume $\mathrm{char} (k) = 0$. 
Let $A$ be an abelian variety of dimension $g\ge 1$ defined over $K:=k(B)$, where $K$ is the  function field of a smooth projective irreducible curve $B$ defined over $k$. Assume $\mathrm{Tr}_{K/k}(A)=0$. Then $\#A(K)_{\mathrm{tor}} \le c(g,K)$.
\end{thm}
 
Theorem~\ref{ThmLooperYap} also self-improves to $\#A(K)_{\mathrm{tor}}\le c(g,[K \!: \! k(\mathbb{P}^1)])$ by Weil restriction. Looper--Yap's proof follows the strategy of Hindry--Silverman \cite{HS} on the Lang--Silverman Conjecture  
\cite{lang1978elliptic, silverman1984lower} for elliptic curves, and moreover uses Looper's dynamical basis \cite{Looper24}, equidistribution results \textit{\`a la Fekete}, and ultrafilters. A key ingredient is Deligne's  Arakelov Inequality  \cite{Deligne87}. 
Their bound is inexplicit. They also prove the Lang--Silverman Conjecture  in this setting \cite[Thm.~1.2]{LooperYap} with an inexplicit bound, generalizing \cite[Thm.~0.2]{HS} for elliptic curves whose bound is explicit.

\vskip 0.5em
%\subsection{Brief account of our work} 

In this paper, we {\it propose a different method} (see §\ref{SubsectionStrategy} for more details) to study both the Uniform Boundedness Conjecture and the Lang--Silverman Conjecture. Our method is inspired by Vojta's proof of the Mordell Conjecture \cite{Vojta:siegelcompact} and -- already in the case of elliptic curves -- is different from \cite{HS}. Also we do not use dynamical basis, equidistribution or ultrafilters. We achieve the following goals:
\begin{enumerate}
\item Over function fields of characteristic $0$, we give a new proof of \cite[Thm.~1.1 and Thm.~1.2]{LooperYap}  with  {\it explicit} bounds, which furthermore depends {\it polynomially on the field} $K$ unless $A$ has a $K$-factor having   good reduction everywhere; see Theorem~\ref{ThmUBCFF} and Theorem~\ref{ThmLangSilvermanFF}. We also use Deligne's Arakelov Inequality.
\item Over function fields of characteristic $p>0$, we prove an explicit version of 
both conjectures 
%the Geometric Uniform Boundedness Conjecture and the Lang--Silverman Conjecture 
when $g=1$; see Theorem~\ref{ThmBoundECFFCharP}.
\item Over number fields, in Theorem~\ref{ThmMainNF} we prove both conjectures 
%Conjecture~\ref{ConjUBC} and the Lang--Silverman Conjecture
 for semi-stable $A/K$ assuming a suitable generalization of Szpiro's  Conjecture that we propose (Conjecture~\ref{ConjSzpiroWeakIntro}). 
 We also have some unconditional results; see  Theorem~\ref{ThmMainNFUnconditional}. 
%Unconditionally, we prove explicit bounds on $\#A(K)_{\mathrm{tor}}$  and of Lang--Silverman type (Theorem~\ref{ThmMainNFUnconditional}), depending polynomially on $[K:\Q]$, the Szpiro ratio, and $\frac{h_{\mathrm{Fal}}(A)}{h_{\partial}(A)}$ (see \eqref{DefnBoundaryHeightIntro} for $h_{\partial}(A)$; this last ratio is $\le 6/\pi$ when $g=1$).
\end{enumerate}
 %When $g=1$, our approach gives a new proof of  results of  \cite{HS}.% In each case, we also prove results in the direction of the Lang--Silverman Conjecture.% in the case of elliptic curves.

\subsection{Over function fields}
%It was proved by Hindry--Silverman \cite{} when $g=1$ and by Bakker--Tsimerman \cite{} over characteristic $0$ for abelian varieties with real multiplications.

%{\it Our main result in this setting} is to prove an explicit version of  GUBC in a different method, which is inspired by Vojta's proof of the Mordell Conjecture. %, which is inspired by Vojta's proof of the Mordell Conjecture. 
% with an explicit bound which is polynomial in $g(B)$ with an optimal degree, under the mild assumption that $A$ has semi-stable reduction. 
%Our approach is inspired by Vojta's proof of the Mordell Conjecture and is hence different from Looper--Yap's proof. 
%We also use Deligne's Arakelov Inequality as a core. 
% However, we do not use decompositions of the N\'{e}ron--Tate height, and we handle degeneration by studying a suitable height function on the moduli space of abelian varieties -- more precisely via a decomposition of the Faltings height of $A$. 
%For elliptic curves this gives a new proof of the main results of \cite{}. %with more explicit bounds. 

We start with our results over function fields. {\it From now on, $k$ is an algebraically closed field.} Let $K = k(B)$ be the  function field of a smooth projective  irreducible curve $B$ defined over $k$. Denote by $\mathrm{gon}(B) = [K:k(\mathbb{P}^1)]$ the {\it gonality} of $B$ and  by $g(B)$ the genus of $B$.\footnote{Compared to  number fields, $g(B)$ is the analogue of the logarithm of the discriminant of the field.} Let $\gamma_g < e^{g^2/2}$ be the Minkowski constant from Lemma~\ref{LemmaReductionMinkowski}. 
\begin{thm}[Main Theorem on rational torsion, function field]\label{ThmUBCFF}
Assume $\mathrm{char}(k) = 0$. 
For the explicit constant $c_0(g) = \left( 512\cdot 3^{12(8g)^2+2}(8g)^4(8g+2)\gamma_{8g}\right)^{8g}  > 0$ we have  the following property. Any abelian variety $A/K$ of dimension $g$, with $\mathrm{Tr}_{K/k}(A)=0$, satisfies
\begin{equation}\label{EqUBCFFFFF}
\#A(K)_{\mathrm{tor}} \le c_0(g) \left( g(B)+1 \right)^{16g} \cdot 2^{2^{2^{32g+2}(g(B)+3)^2}}.
\end{equation}
Moreover, if no $K$-factor of $A$ has good reduction everywhere, then we have a better bound
\begin{equation}\label{EqUBCFF}
\#A(K)_{\mathrm{tor}} \le c_0(g) \left( g(B)+1 \right)^{16g}.
\end{equation}
\end{thm}
For $g=2$ Looper \cite{Looper26} independently proves such an explicit uniform bound in~$g(B)$. 

For $A/K$  as in Theorem~\ref{ThmUBCFF}, the Weil restriction $\mathrm{Res}_{K/k(\mathbb{P}^1)}A$ is still traceless but has dimension $g \!\cdot\! \mathrm{gon}(B)$, and $ A(K) = (\mathrm{Res}_{K/k(\mathbb{P}^1)}A)(k(\mathbb{P}^1))$. Hence 
\eqref{EqUBCFFFFF}, applied to $\mathrm{Res}_{K/k(\mathbb{P}^1)}A$ and the base curve $\mathbb{P}_k^1$ (genus $0$), implies an explicit bound in terms of $g$ and $\mathrm{gon}(B)$:
\begin{equation}\label{EqTorsionFFGon}
\#A(K)_{\mathrm{tor}} \le c_0 \!\big(g \!\cdot\! \mathrm{gon}(B)\big) \cdot 2^{2^{2^{32g\cdot \mathrm{gon}(B) +6}}} <  e^{7200 g^3 \mathrm{gon}(B)^3} \cdot 2^{2^{2^{32g\cdot \mathrm{gon}(B) +6}}}.
\end{equation}
The second term disappears if no $K$-factor of $A$ has good reduction everywhere.
%This is an explicit bound is in terms of $g$ and $\mathrm{gon}(B)$.%; a safe $\log c_0(g\cdot \mathrm{gon}(B)) \le 7200 g^3 \mathrm{gon}(B)^3$.

A standard application of Lefschetz Principle reduces Theorem~\ref{ThmUBCFF} to $k=\C$. 
The core case of Theorem~\ref{ThmUBCFF} is when {\it no} $K$-factor of $A$ 
%non-trivial abelian subvariety of $A$ defined over $K$ 
has good reduction everywhere; from this, it suffices to prove separately for $A/K$ having good reduction everywhere with a rather short proof §\ref{SubsubsectionCaseGoodEverywhere}. 
Compared to Looper--Yap \cite[Thm.~1.1]{LooperYap}, our bound is explicit and {\it has polynomial growth in $g(B)$ in the core case}. 
Moreover in the core case, the essential part is to treat the case where $A/K$ furthermore has semi-stable reduction; this is done in §\ref{SectionSS}, see Theorem~\ref{ThmUBCFFOrig} and below. 
%In this case, our bound can be further improved.
%\begin{thmbis}{ThmUBCFF}\label{ThmUBCFFSS}
%Assume $\mathrm{char}(k) = 0$. The following property holds true 
%for the explicit constant $c_0'(g) = \left( 24\cdot 24000 (8g)^4(8g+2)\gamma_{8g}\right)^{8g} >0$. If $A/K$ has semi-stable reduction and no abelian subvariety of $A$ defined over $K$ has good reduction everywhere, then we have
%\[
%\#A(K)_{\mathrm{tor}} \le  c_0'(g) \left( g(B)+1 \right)^{8g}.
%\]
%\end{thmbis}

\begin{rmk}
When $A/K$ is traceless and has  good reduction everywhere, in an ongoing work \cite{Gu_NSTor}, the second-named author improves the bound to $\#A(K)_{\mathrm{tor}}\leq \bigl(768\mathfrak q^5(\mathfrak q+2)\bigr)^{\mathfrak q}$,  
where $\mathfrak q=\frac {(8g+3)!3^{8g}}{2}(g(B)+1)$. The proof uses methods and techniques of the current paper as well as Lefschetz pencils. Together with \eqref{EqUBCFF}, this will improve the bound \eqref{EqUBCFFFFF} to 
\[
\#A(K)_{\mathrm{tor}} \le c_0(g) \left( g(B)+1 \right)^{16g} \cdot \bigl(768\mathfrak q^5(\mathfrak q+2)\bigr)^{\mathfrak q},
\]
and \eqref{EqTorsionFFGon} accordingly (with the second term replaced by $e^{e^{17 g\cdot \mathrm{gon}(B) \log (10g \cdot \mathrm{gon}(B) )}}$). %get another bound in $g$ and $\mathrm{gon}(B)$ accordingly (but unfortunately is worse than \eqref{EqTorsionFFGon}). 
\end{rmk}

We also give a proof of the  {\it Lang--Silverman Conjecture}, which is a uniform lower bound on the heights of non-torsion {\it rational} points, in this setting. 
Compared to \cite[Thm.~1.2]{LooperYap}, our bound is explicit and the denominator depends polynomially on $K$. 
\begin{thm}[Main Theorem on small rational points, function field]\label{ThmLangSilvermanFF}
Assume $\mathrm{char}(k)=0$. 
Set $c'_0(g) := 2^{-1}\left( 1024\cdot 3^{12(8g)^2+2}(8g)^4(8g+2)\gamma_{8g}\right)^{-(2g+1)} > 0$. Then for (i) any abelian variety $A/K$ of dimension $g$ such that $A\not=\mathrm{Tr}_{K/k}(A)\otimes_k K$, 
(ii) any symmetric ample line bundle $L$ on $A$ defined over $K$, 
(iii) any $P \in A(K)$ with $\bar{\Z \cdot P}^{\mathrm{Zar}} = A$, we have
\begin{equation}\label{EqLSFF1Intro}
\hat{h}_L(P) \ge c'_0(g) \frac{\max\{h_{\mathrm{Fal}}(A),1\}}{\mathrm{gon}(B) \cdot \left(g(B)+1\right)^{4g+2}}
\end{equation}
and, for the explicit constant $c_0(g)$ from Theorem~\ref{ThmUBCFF}, 
\begin{equation}\label{EqLSFF2Intro}
\hat{h}_L(P) \ge    \frac{c'_0\big(g \!\cdot \!\mathrm{gon}(B)\big)}{\left(c_0\big(g \!\cdot \!\mathrm{gon}(B)^2\big) \cdot 2^{2^{2^{32g\cdot \mathrm{gon}(B)^2 +6}}}\right)^2 \mathrm{gon}(B) } \max\{h_{\mathrm{Fal}}(A),1\}.
\end{equation}
\end{thm}
Both bounds are explicit. 
The denominator of the right hand side of \eqref{EqLSFF1Intro} depends polynomially on $g(B)$, while the right hand side of \eqref{EqLSFF2Intro} depends only on $g$ and $\mathrm{gon}(B)$. 
A standard application of Lefschetz Principle reduces Theorem~\ref{ThmLangSilvermanFF} to $k=\C$. Lang--Silverman type bounds are of a rather different -- in some sense orthogonal -- nature from the {\it Lehmer-type bounds} of \cite{SilvermanLehmer, LooperSilverman2024, Looper24}\!\!\cite{Masser86, GR25}, which concern algebraic points of bounded degree but involve $h_{\mathrm{Fal}}(A)$ in the denominator of the lower bound.

\vskip 0.3em
For elliptic curves (\textit{i.e.} $g=1$), we can also prove uniform results when $\mathrm{char}(k) = p > 0$.  A traceless semi-stable elliptic curve $E$ defined over $K=k(B)$ induces a surjective modular map $j \colon B \to \mathbb{P}^1_k$. Denote by $p^{\nu}$ the inseparable degree of $j$. % ($p^{\nu}=1$ if $\mathrm{char}(k)=0$).
 Let $L=O(0)$.
\begin{thm}\label{ThmBoundECFFCharP}
%We have
\begin{enumerate}
    \item[(i)] $\# E(K)_{\mathrm {tor}}\leq 108\cdot 144 \left(2g(B)-1\right)^4$.
%    \item[(i')] If $\mathrm{char}(k) = 0$, then the bound above can be improved to $\# E(K)_{\mathrm {tor}}\leq 108 \left(2g(B)-1\right)^2$.
    \item[(ii)]  For any non-torsion $P \in E(K)$, we have
        \[
    \hat{h}_L(P) \ge \frac{\max\{h_{\mathrm{Fal}}(E),1\}}{2^4 \cdot 3^{11} (p^{\nu})^2 \mathrm{gon}(B) \left(2g(B)-1\right)^6} \max\left\{\frac{1}{(p^{\nu})^4}, \frac{1}{2^{10} \cdot 3^5 |2g(B)-1|^5}\right\}.
    \]
    \end{enumerate}
\end{thm}

\subsection{Over number fields}
Let $K$ be a number field. Let $A$ be an abelian variety of dimension $g\ge 1$ defined over $K$. For convenience, by a {\it polarization on $A$}, we mean a {\it symmetric ample line bundle $L$ on $A$ defined over $K$} -- this is a genuine abuse of notation.% but more convenient for us.

Mazur, Kamienny, and Merel's proof of Conjecture~\ref{ConjUBC} for $g=1$ relies heavily on the theory of modular curves and their Jacobians. It is hence hard to be generalized to $g\ge 2$. Their bound is double exponential in $[K:\Q]$.

On the other hand, Petsche \cite{Petsche2006} -- using Hindry--Silverman's method \cite{HS} -- proved a weaker bound for elliptic curves $E/K$, which involves the {\it Szpiro ratio} but is polynomial in $[K:\Q]$. The Szpiro ratio is conjectured to be bounded above solely in terms of $K$, according to the {\it Szpiro Conjecture} which is equivalent to some weak version of the famous {\it $abc$ Conjecture}. A stronger conjecture is the {\it Height Conjecture} \cite[pp.~39, Conj.~(H)]{Frey} which claims that $h_{\mathrm{Fal}}(E)$ is bounded above solely in terms of $K$ and the conductor of $E$; it immediately implies  Szpiro's Conjecture by the Faltings--Silverman formula.

When $g\ge 2$, Hindry \cite[Conj.~3.4]{HindrySzpiro} proposed a generalized version of this Height Conjecture (he calls it the generalized Szpiro Conjecture), which claims that $h_{\mathrm{Fal}}(A)$ is bounded above solely in terms of $g$, $K$, and the conductor of $A$. 
%as an analogue of Deligne's Inequality over number fields. 
In particular if \cite[Conj.~3.4]{HindrySzpiro} were true, then the archimedean contribution to $h_{\mathrm{Fal}}(A)$ is always negligible.% if \cite[]{} were true.

\vskip 0.2em
In this paper, we propose another generalization of Szpiro's Conjecture to $g\ge 2$ (Conjecture~\ref{ConjSzpiroWeakIntro}), which is closer to the original Szpiro Conjecture and does not claim to eliminate the archimedean contribution to $h_{\mathrm{Fal}}(A)$; it is hence weaker than \cite[Conj.~3.4]{HindrySzpiro}. We then show that Conjecture~\ref{ConjSzpiroWeakIntro} implies Conjecture~\ref{ConjUBC} and the Lang--Silverman Conjecture. We also prove some unconditional results, which for $g=1$ is similar to  \cite{Petsche2006}.

\subsubsection{Generalized Szpiro ratio and the boundary height}\label{SubsubsectionGSC}
Assume $A/K$ has semi-stable reduction. For each $v$ in the set of  bad reduction places $\mathfrak{Bad}(A/K)$, write $A\pmod v$ for the reduction of the N\'{e}ron model of $A$ at $v$, and $r_v$ for the toric rank of the identity component of  $A\pmod v$. Let $\Phi_v$ be the group scheme of connected components of $A\pmod v$ defined over the residue field $k_v$, and write $\#\Phi_v$ for its order, \textit{i.e.} $\#\Phi_v = \#\Phi_v(\bar{k}_v)$. 

\vskip 0.2em

Let $N_v$ be as in §\ref{SubsectionGeneralNotation}. Define the {\it generalized Szpiro ratio} to be
\begin{equation}\label{EqGeneralizedSzpiroRatioIntro}
\mathfrak{s}_{A/K}:= %\frac{\sum_{v\in \mathfrak{Bad}(A/K)} (\#\Phi_v)^{1/r_v} \log N_v}{\log \mathfrak{f}_{A/K}} = 
\frac{\sum_{v\in \mathfrak{Bad}(A/K)} (\#\Phi_v)^{1/r_v} \log N_v}{\sum_{v\in \mathfrak{Bad}(A/K)} r_v \log N_v}
\end{equation}
if $\mathfrak{Bad}(A/K)\not=\emptyset$ and to be $1/g$ if $A/K$ has good reduction everywhere. Since $r_v \le g$ for all $v$, we have $\mathfrak{s}_{A/K} \ge 1/g$. When $g=1$, this is precisely the Szpiro ratio of $A/K$. 
%The {\it Szpiro Conjecture} says that $\mathfrak{s}_{A/K}$ is bounded above solely in terms of $K$ when $g=1$.

For general $g$, we also take the following {\it boundary height} into consideration. It is the sum of its non-archimedean and  archimedean parts. The non-archimedean part is
\begin{equation}\label{DefnBoundaryHeightNonArchIntro}
h_{\partial, \text{fin}}(A) := \frac{1}{[K:\Q]}\left(\sum_{v \in \mathfrak{Bad}(A/K)} (\#\Phi_v)^{1/r_v}\log N_v \right).
\end{equation}
Our definition of the archimedean part is inspired by \cite{DavidMinorations-de-, Pazuki2013}. Let $L$ be a polarization on $A$ with induced isogeny $\phi_L\colon A\rightarrow A^\vee$. For any maximal isotropic subgroup $H$ of $\ker(\phi_L)$ with respect to the Weil pairing, $L$ descends to a principal polarization $L_0$ on $A/H$; cf. \cite[§23, pp.214-216]{Mum_AV}. 
At each archimedean place $v \in M_{K,\infty}$, $L_0$ gives rise to a {\it Siegel reduced} matrix $\tau_{v,H}$ such that $(A/H)_v(\C) \cong \C^g/(\Z^g+\tau_{v,H} \Z^g)$. Define
\begin{equation}\label{EqAlphaVIntro}
\alpha_v(A,L) := \max_{H\subseteq \ker(\phi_L)\text{ maximal isotropic}} \mathrm{Tr}(\mathrm{Im}(\tau_{v,H})) 	\ge \frac{\sqrt{3}}{2}g ;
\end{equation}
then $\alpha_v(A,L)$ depends only on the class of $[L]$ in $\mathrm{NS}(A)$. Finally, set
\begin{equation}\label{DefnBoundaryHeightIntro}
h_{\partial}(A) := h_{\partial, \text{fin}}(A) +  \frac{1}{[K:\Q]} \min_{[L]} \left( \sum_{v\in M_{K,\infty}} \alpha_v(A,L) \log N_v \right)		\ge \frac{\sqrt{3}}{2}g
\end{equation}
where $L$ runs over all classes of polarizations on $A$ {\it of minimal degree}.

%Then $h_{\partial }(A_{K'}) = h_{\partial }(A)$ for any finite extension $K'/K$ since $A$ has semi-stable reduction over $K$. This justifies {\it stable} in its name. 

Merging \cite[Conj.~22]{Clark_Xarles_Torsion_pt} and (a weaker version of) \cite[Conj.~1.7]{DavidMinorations-de-}, we propose the following conjecture. It can be easily deduced from \cite[Conj.~3.4]{HindrySzpiro}; see Lemma~\ref{LemmaSzpiro}.
\begin{conj}[Generalized Szpiro Conjecture, weak version]\label{ConjSzpiroWeakIntro}
There exist positive constants $\kappa_1 = \kappa_1(g, K) >0$ and $\kappa_2 = \kappa_2(g, K) >0$ such that
\begin{equation}\label{EqSzpiroConjRatioIntro}
\mathfrak{s}_{A/K}  \le \kappa_1
\end{equation}
and
\begin{equation}\label{EqSzpiroConjHeightIntro}
h_{\mathrm{Fal}}(A) \le  \kappa_2 h_{\partial }(A).
\end{equation}
%if $A$ has semi-stable reduction over $K$.
\end{conj}
When $g=1$, \eqref{EqSzpiroConjHeightIntro} holds true with $\kappa_2 = 6/\pi$ by the Faltings--Silverman Formula, so that Conjecture~\ref{ConjSzpiroWeakIntro} becomes precisely Szpiro's Conjecture. %For general $g$, \eqref{EqSzpiroConjHeightIntro} was proposed by David \cite[Conj.~1.7]{DavidMinorations-de-} without considering $h_{\partial,\mathrm{fin}}(A)$. %By applying Weil restrictions, if Conjecture~\ref{ConjSzpiroWeakIntro} holds true for all $g$ and $K$, then both constants depend only on $g$, $[K:\Q]$ and $\mathrm{Disc}(K/\Q)$.

\subsubsection{Our main results}
Let $K$ be a number field.

\begin{thm}[Main result over number field, unconditional]\label{ThmMainNFUnconditional}
For any $g \in \Z_{\ge 1}$, there exist explicit constants $c_1(g), c'_1(g) > 0$ (see \eqref{EqConstantc1NFUncond} and \eqref{EqConstantc1PrimeNFUncond}) such that: For any principally polarized semi-stable $K$-simple abelian variety $(A,L)/K$ of dimension $g$, we have
\begin{equation}\label{EqTorsionNF}
\#A(K)_{\mathrm{tor}} \le c_1(g) \mathfrak{s}_{A/K}^{2g}[K:\Q]^g\left( \log [K:\Q]+\max\left\{ \frac{h_{\mathrm{Fal}}(A)}{h_{\partial}(A)}, 1 \right\} \right)^{g},
\end{equation}
and, for any $P\in A(K)$ with $\bar{\Z\cdot P}^{\mathrm{Zar}} = A$,
\begin{equation}\label{EqLSNFUncondIntro}
\hat{h}_L(P) \ge \frac{1}{c'_1(g) \max\left\{\frac{h_{\mathrm{Fal}}(A)}{h_{\partial}(A)}, 1\right\}} \cdot \frac{\max\{h_{\mathrm{Fal}}(A),1\}}{\mathfrak{s}_{A/K}^{2g+2}[K:\Q]^{2g+1}\left(\log[K:\Q]+\max\left\{\frac{h_{\mathrm{Fal}}(A)}{h_{\partial}(A)}, 1\right\}\right)^{2g}}.
\end{equation}
\end{thm}
When $g=1$, we have $\frac{h_{\mathrm{Fal}}(A)}{h_{\partial}(A)} \le 6/\pi$ by the Faltings--Silverman Formula, so we recover \cite{Petsche2006} (which follows \cite{Silverman1981, silverman1984lower, HS}) with better dependence on the Szpiro ratio but larger constants. 
For $g \ge 2$, some results towards \eqref{EqLSNFUncondIntro} were obtained by David \cite{DavidMinorations-de-}, based on which  Masser \cite{masser1993large} found a family of abelian varieties satisfying the Lang--Silverman Conjecture. See also Pazuki \cite{Pazuki2010, Pazuki2012, Pazuki2013} for some other cases.

\begin{thm}[Main result over number field, conditional]\label{ThmMainNF}
Assume Conjecture~\ref{ConjSzpiroWeakIntro} holds true for $K$ and all $g \ge 1$. Then  for any $A/K$  of semi-stable reduction, we have:
\begin{enumerate}
\item[(i)] the Uniform Boundedness Conjecture (Conjecture~\ref{ConjUBC}) holds true for $A$;
\item[(ii)] the {\it Lang--Silverman Conjecture} holds true for $A$, \textit{i.e.} there exists a constant $\kappa = \kappa(\dim A,K) > 0$ such that for any polarization $L$ on $A$, we have % and any symmetric ample line bundle $L$ on $A$, we have
\[
\hat{h}_L(P) \ge \kappa \cdot \max\{h_{\mathrm{Fal}}(A),1\}, \quad \text{ for any } P\in A(K)\text{ with }\bar{\Z\cdot P}^{\mathrm{Zar}} = A.
\]
\end{enumerate}
\end{thm}
By applying Weil restriction, $\kappa$ can be shown to depend only on $g$ and $[K:\Q]$ if we can remove the semi-stable assumption (see the beginning of §\ref{SectionGeneral} for an improvement on this by our method). %\footnote{Our discussion in §\ref{} removes this if the bad but potentially places }  
The followings can be easily read off from the proof: For Conjecture~\ref{ConjUBC}, we need  Conjecture~\ref{ConjSzpiroWeakIntro} for $K$ and all $g \in \{1,\ldots,\dim A\}$, and $\min_{[L]}$ in \eqref{DefnBoundaryHeightIntro} can be relaxed to $\max_{[L]}$; for Lang--Silverman, we need Conjecture~\ref{ConjSzpiroWeakIntro} for $K$ and all $g \in \{1,\ldots, 8\dim A\}$, which can be further relaxed if $L$ defines a principal polarization; see §\ref{SubsectionLSNFCond}.  
%Also by our proof, the assumption $A/K$  of semi-stable reduction can be relaxed to 
%; see §\ref{}.%then  \eqref{EqSzpiroConjHeightIntro} can be relaxed with $h_{\partial}(A)$ replaced by $h_{\partial}(A,L)$ defined in \eqref{}. 

\subsection{Strategy of our proof}\label{SubsectionStrategy}
Overall speaking, our method is a {\it Diophantine approximation with perturbed metrics}. 

Classical Diophantine approximation method of Liouville, Thue, Siegel, and Roth proceeds as follows: one constructs an auxiliary polynomial with small integral coefficients, uses {\it zero estimates} to find a point at which a suitable derivative does not vanish, bounds its value there, and compares this bound with the Liouville bound.  Vojta \cite{Vojta:siegelcompact} implemented this method in an Arakelov-geometric setting to give a second proof of the Mordell Conjecture, in which auxiliary polynomials are replaced by {\it small sections} of  Hermitian line bundles on an arithmetic variety. More generally, one can work with {\it adelic line bundles} defined by S.~Zhang \cite{ZhangSmPtAdelicMetric} and Yuan--Zhang \cite{YuanZhang_Quasiproj}, which allow to treat archimedean places and places of bad reduction in a unified way. %In this paper, \cite{ZhangSmPtAdelicMetric} is enough.

The {\it main novelty} of our method, compared to Vojta, is that we construct nice bump functions at different places and use them to perturb the metrics of the adelic line bundle. Working with small sections of this perturbed adelic line bundle, we  improve the bounds obtained in the evaluation step and get stronger results.%One thus gets better evaluation step is then 

Let us explain in more details. Let $L$ be a symmetric ample line bundle on our abelian variety $A$ defined over $K$. It has a canonical adelic extension $\bar{L} = (L,\{\|\cdot\|_v\}_{v\in M_K})$, \textit{i.e.} at each place $v$ of $K$, the analytification $(L\otimes_K K_v)^{\mathrm{an}}$ is endowed with its canonical metric  $\|\cdot\|_v$ and the associated height function $h_{\bar{L}}$ coincides with the N\'eron--Tate height. Also define local invariants $\tilde{I}_v(A)$ to be $(\#\Phi_v)^{-1/r_v}$ (for $v$ a non-archimedean place at which $A$ has bad reduction) and to be $\alpha_v(A,L)$ from \eqref{EqAlphaVIntro} (for $v$ archimedean places). 

For simplicity assume $A$ is simple and let us focus on rational torsion points $A(K)_{\mathrm{tor}}$.

Assume $A/K$ has semi-stable reduction. At each  $v \in M_K$,  we construct a bump function 
\[
f_v \colon (A\otimes_K K_v)^{\mathrm{an}} \longrightarrow \R_{\ge 0},
\]
with $f_v \equiv 0$ at places of good reduction, and perturb the canonical metric to  $\|\cdot\|'_v := N_v^{-\epsilon f_v} \|\cdot\|_v$; here $\epsilon>0$ and $N_v$ is  as in §\ref{SubsectionGeneralNotation}. The required properties of $f_v$ will be explained later. These metrics together yield a new adelic line bundle $\bar{L}(\epsilon \mathbf{f}) := (L,\{\|\cdot\|'_v\}_v)$.

It is with this $\bar{L}(\epsilon \mathbf{f})$ on $A$ that we run the Diophantine approximation process. 

The first step is to find a small section $s$ of $\bar{L}(\epsilon \mathbf{f})^{\otimes m}$ for all $m \gg 1$. For this, we prove the arithmetic bigness of $\bar{L}(\epsilon \mathbf{f})$ using arithmetic Hilbert--Samuel (see \cite[Thm.~5.2.2]{YuanZhang_Quasiproj} for the most general version), for which we need  $\bar{L}(\epsilon \mathbf{f})$ to be nef. Nefness can be checked place by place, and this is where enters {\it a first required property for $f_v$}. At archimedean places $v$,   $f_v$ is smooth and we need the curvature form $c_1(L,\|\cdot\|'_v) = c_1(L,\|\cdot\|_v) + \frac{\sqrt{-1}}{2\pi}\partial\bar{\partial}(\epsilon f_v)$ to be a positive $(1,1)$-form.  
At a non-archimedean place $v$, we want $\|\cdot\|'_v$ to be {\it relatively semipositive} (\textit{i.e.} can be uniformly approximated by nef model metrics), which is already true when $A$ has good reduction at $v$. When $A$ has semi-stable bad reduction at $v$, {\it in our particular case} this can be checked using supercurrents on non-archimedean Berkovich spaces. More precisely, the tropicalization maps $(A\otimes_K K_v)^{\mathrm{an}}$ to its {\it skeleton} $\Sigma_v$ which is a real torus, and   $f_v$ is defined as the pullback of a smooth function on $\Sigma_v$.  Lagerberg \cite{Lagergberg} developed a theory of supercurrents on $\Sigma_v$, which was  generalized to non-archimedean Berkovich spaces by Chambert-Loir--Ducros \cite{CLD}. For our purpose we use the variant of Gubler--K\"unnemann \cite{GuKu_deltaform}. {\it In our specific case}, by the semipositivity criterion  \cite[Thm.~4.10]{GublerStefan_ToricMetricAV} (see also \cite[Prop.~8.3.1]{BGJK}) it suffices to show that the $(1,1)$-superform $c_1(L,\|\cdot\|'_v)$ is positive, for which a simple computation is enough.

Next, we apply  Nakamaye's zero estimate \cite{Nakamaye} to find an $x \in A(K)_{\mathrm{tor}}$, such that  the vanishing order $\ell$ of $s$ at $x$ is bounded in terms of $g$ and $\#A(K)_{\mathrm{tor}}$.  If $A$ is not simple, the bound is more complicated and we will do induction after our Diophantine approximation. 

Then, we evaluate the non-zero element $\mathrm{jet}^{\ell}s(x) \in J :=\mathrm{Sym}^{\ell}((\lie A)^\vee) \otimes L_x^{\otimes m}$. Our computation is inspired by Gaudron--R\'{e}mond \cite[§6]{GR25} (while much simpler, \textit{e.g.} no interpolation is needed). 
In this step, {\it we require two other properties for $f_v$}: it vanishes on a suitable region, and its integral with respect to the standard normalized measure is bounded below linearly by our local invariant $\tilde{I}_v(A)$. Now the canonical adelic extension $\bar{L}$ of $L$ and the canonical adelic structure on $\lie A$ make $J$ into an adelic vector space, whose metrics are denoted by $\{\|\cdot\|_{J,v}\}_{v\in M_K}$. The properties of our $f_v$ allow to bound $-\log \|\mathrm{jet}^{\ell}s(x)\|_{J,v}$ from below linearly by $\tilde{I}_v(A)\log N_v$ for each $v$ (with a slight modification at archimedean places). Compare it with the Liouville inequality for adelic vector spaces and Gaudron's bound on the maximal slope of the adelic cotangent space \cite{Gau_Explicit_Abvar}, we get an upper bound of $\#A(K)_{\mathrm{tor}}$ (roughly) in terms of $\frac{h_{\mathrm{Fal}}(A)}{\sum_v \tilde{I}_v(A)\log N_v}$. 

%It is worth noting that this evaluation step is much simpler at non-archimedean places than archimedean place. Indeed
Here are some further remarks on this evaluation step. 
For $v$  non-archimedean, $f_v$ is required to vanish on the classical points $A(K_v)$ which lie in the inverse image of a lattice in $\Sigma_v$ under the tropicalization, and we have an easy comparison $\|\mathrm{jet}^{\ell}s(x)\|_{J,v}\le \sup_{y\in A(K_v)}\|s(y)\|_v$  of Gaudron--R\'{e}mond \cite[§6.7.1]{GR_min_isog}. Then the desired bound on $-\log \|\mathrm{jet}^{\ell}s(x)\|_{J,v}$  follows since $s$ is a small section for the perturbed metric and by our condition on the integral of $f_v$. 
For $v$  archimedean, we divide $(A\otimes_K K_v)^{\mathrm{an}}$ into several strips and require $f_v$ to vanish on some of them, and the evaluation of $\|\mathrm{jet}^{\ell}s(x)\|_{J,v}$ uses computation with theta series in \cite{Gau_Explicit_Abvar} and Cauchy's estimate in complex analysis.

This finishes our {\it Diophantine approximation with perturbed metrics}. 
To conclude for uniformity, we then use the boundary height and invoke Deligne's Arakelov Inequality (in the function field case) and Conjecture~\ref{ConjSzpiroWeakIntro} (in the number field case) to bound $\frac{h_{\mathrm{Fal}}(A)}{\sum_v \tilde{I}_v(A)\log N_v}$ uniformly from above. When $A/K$ is not simple, an induction is also needed.

Removing the semi-stable assumption requires adding local contributions to the denominator $\sum_v \tilde{I}_v(A)\log N_v$ at each bad but potentially good place $v$. We do not construct bump functions at such $v$. Instead, we exploit the fact that $K$-points lie in a {\it proper} infinitesimal neighborhood after passing to the semi-stable extension $K'$ of $K$. More precisely for a place $v'$ of $K'$ lying over $v$ and $\calN$ (resp. $\calA'$) the N\'eron model of $A\otimes_K K_v$ (resp. of $A\otimes_K K'_{v'}$), the section associated to every $K$-point lies in a proper subscheme of $\calA'$, namely the image of the natural map $\calN\otimes_{\calO_{K_v}} \calO_{K'_{v'}}\to \calA'$. Consider the special fiber $N_0$ of this subscheme and its infinitesimal neighborhoods $N_n$, and modify the metric of $\bar{L}(\epsilon \mathbf{f})$ accordingly at this $v$. We then obtain an estimate of $\|\mathrm{jet}^{\ell}s(x)\|_{J,v}$ at this $v$ for most sections $s \in H^0(A,L^{\otimes m})$. In this process, we need to compare $h^0(N_n, \cL^{\otimes m})$ and $h^0(A,L^{\otimes m})$, where $\cL$ is the extension of $L$ to $\cA'$, via their comparisons with $h^0(N_0,\cL^{\otimes m})$. A subtlety  here is that tame ramification is required at $v'$ 
because we need Edixhoven's \cite{Edixhoven1992} to compare $h^0(A,L^{\otimes m})$ and $h^0(N_0,\cL^{\otimes m})$, so we restrict §\ref{SectionGeneral} to the function field case $K=\C(B)$ for simplicity.

\vskip 0.2em
Finally, the sketch above also allows to count the number of rational points of small heights. Then we can prove the Lang--Silverman Conjecture by considering the multiples of the non-torsion point $P \in A(K)$ in question when $L$ defines a principle polarization on $A$, and by using Zarhin's trick and {\it Zarhin's box} for general $L$. 

\vskip 0.3em
In the end, we point out that our method differs from \cite{HS}, which starts with a local-global decomposition of the N\'{e}ron--Tate height $\hat{h}_L$ and uses Elkies's lower bound crucially at archimedean places. Our proof does not use either one of these ingredients.

\subsection{Organization of the paper}
In §\ref{SectionPrelim}, we collect some notation and basic background knowledge throughout the whole paper. In particular, we define the Faltings height in §\ref{SubsectionFaltingsHeight}  and include a summary on adelic geometry over functions field in §\ref{SubsectionPreliminaryAdelicGeometry}.

The constructions of our bump functions are carried in §\ref{SectionArch}-\ref{SectionBumpNonArch}: at archimedean places in §\ref{SectionArch}, and at non-archimedean places in §\ref{SectionBumpNonArch}. The construction at non-archimedean places requires  knowledge on the tropicalization of an abelian variety over a non-archimedean local field and the associated monodromy pairing, for which we recall in §\ref{SectionNonArchPrep}.

The core of our proof is in  §\ref{SectionFirstBoundStatement}-\ref{SectionLocalGlobal}, and we organize as follows. In §\ref{SectionFirstBoundStatement}, we state a first quantitative bound for the number of rational torsion points as well as its analogue for rational points of small height, in terms of $h_{\mathrm{Fal}}(A)$ and the local invariants $\tilde{I}_v(A)$ (we separate the statements for function fields and for number fields, because the latter requires extra subtlety).  These bounds are then proved in §\ref{SectionProofOfFirstBound}, following the {\it Diophantine approximation with perturbed metrics} strategy discussed in §\ref{SubsectionStrategy}. In §\ref{SectionLocalGlobal}, we start relating $\tilde{I}_v(A)$  to global arithmetic invariants, by comparing the boundary height and $h_{\mathrm{Fal}}(A)$. 

Our main results over function fields are then proved in §\ref{SectionSS}-\ref{SectionProofGeneral}, where Deligne's Arakelov Inequality is used. 
In §\ref{SectionSS}, we prove the main results in the essential case, \textit{i.e.} $A/K$ has semi-stable reduction. 
 In §\ref{SectionGeneral}, we explain how to remove the semi-stable assumption. In §\ref{SectionProofGeneral} we finish the proofs of Theorem~\ref{ThmUBCFF} and Theorem~\ref{ThmLangSilvermanFF}.

In §\ref{SectionProofNF} we turn to number fields. We first discuss our Generalized Szpiro Conjecture. Then we prove our main results Theorem~\ref{ThmMainNFUnconditional} and Theorem~\ref{ThmMainNF},  using the bounds in §\ref{SectionFirstBoundStatement}.

Finally, §\ref{SectionEC} specializes the arguments to elliptic curves and proves Theorem~\ref{ThmBoundECFFCharP}. We also include an Appendix~\ref{SectionAppendix} in which we prove a bound on the maximal slope of the adelic cotangent space of $A/K$ in the function field case, and an Appendix~\ref{SectionZarhin} to prove a line bundle version of Zarhin's trick.

\vskip 0.1em
Readers who are interested only in the function field case can skip all discussions at archimedean places (§\ref{SectionArch}, §\ref{SubsectionDivision}-\ref{SubsectionStatementFirstBoundNF}, §\ref{SubsubsectionJetArch}, §\ref{SubsectionPfThmFirstBound}; §\ref{SectionProofNF}). The proof becomes shorter.

\subsection*{Acknowledgements}
We would like to thank Jit Wu Yap and Xinyi Yuan for numerous discussions. We would like to thank \'Eric Gaudron and Ga\"el R\'emond for explanation on their paper \cite{GR25}, Walter Gubler and Klaus K\"unnemann for discussion on supercurrents,  and Junyi Xie for discussions on Berkovich spaces. We would like to thank Vesselin Dimitrov, Marc Hindry, Chenxin Huang, Nicole Looper, Wenbin Luo, Hector Pasten, Yinchong Song for relevant discussions. We would like to thank Pascal Autissier and Joseph Silverman for comments on an earlier draft.
%KG is partially supported by Beijing Natural Science Foundation 24QY0003.

\subsection*{AI statement}
When writing the paper, the authors have used GPT and Gemini to search for well-known mathematics facts and generate bibtex for bibliography.

\section{Some notations and preliminaries}\label{SectionPrelim}

\subsection{Some notation regarding $K$}\label{SubsectionGeneralNotation}
We will unify some notations in the number field and the function field cases. 
Denote by 
\[
M_K = M_K^0 \sqcup M_{K,\infty}
\]
the set of places of $K$, with $M_K^0$ the subset of non-archimedean places and $M_{K,\infty}$ the subset of archimedean places. When $K = \C(B)$ is the function field of an irreducible smooth projective curve $B$, $M_K^0$ is in bijection with points in $B(\C)$ and $M_{K,\infty} = \emptyset$.

When $K$ is a number field, for each place $v\in M_K$ define
\[
N_v := \begin{cases} \#(\cO_K/\mathfrak{p}) & \text{if $v$ corresponds to the maximal ideal $\mathfrak{p}$ of $\cO_K$,} \\
e & \text{if $v$ is a real place,} \\
e^2 & \text{if $v$ is a complex place.} \end{cases}
\]
The {\it residue field} $k_v$ for $v \in M_K^0$ is defined to be $\cO_K/\idp$, where $v$ corresponds to $\idp$.

When $K=\C(B)$ is a function field,  define
\[
N_v := e \qquad \text{for each $v \in M_K$}.
\]
The {\it residue field} $k_v$ for each $v \in M_K$ is $\C$.
%With these notations, the product formula holds for both number fields and function fields.

\subsection{Basic notation regarding abelian varieties}\label{SubsectionConductor}
Throughout the whole paper, $A/K$ is an abelian variety of dimension $g$ and $L$ is a symmetric ample line bundle on $A$ defined over $K$.
\subsubsection{Over number fields}\label{SubsubsectionConductorNF}
Assume $K$ is a number field. Let $\cA$ be the identity component of the N\'{e}ron model of $A$. Write $f \colon \cA \rightarrow \mathrm{Spec}\cO_K$ for the structural morphism.% with zero section $\underline{0}$. % be the zero section. 

For each $v \in M_K^0$, denote by $\cA_v$  the pullback of $f$ along the natural morphism $\mathrm{Spec}(k_v) \rightarrow \mathrm{Spec}\cO_K$. Then $\cA_v$ is an abelian variety of dimension $g$ for all but at most finitely many $v \in M_K^0$. 

The (finite) set of {\it bad reduction places} of $A/K$ is defined to be
\[
\mathfrak{Bad}(A/K) := \{ v \in M_K^0 : \cA_v\text{ is not an abelian variety of dimension }g\}.
\]
We say that $A$ has {\it semi-stable reduction at $v$} if $\cA_v$ is a semi-abelian variety. In this case, denote the toric rank by $r_v \ge 0$. 

Assume  $A/K$ has semi-stable reduction, \textit{i.e.} $A$ has semi-stable reduction at all $v \in M_K^0$. Then $v \in \mathfrak{Bad}(A/K)$ if and only if $r_v \ge 1$.  
The {\it conductor} of $A/K$ is defined to be
\begin{equation}\label{EqConductorNF}
\mathfrak{f}_{A/K} := \prod\nolimits_{v \in M_K^0} N_v^{r_v} = \prod\nolimits_{v \in \mathfrak{Bad}(A/K) } N_v^{r_v}.
\end{equation}
\normalsize
This is compatible with \cite[(3.17)]{HindrySzpiro}.% In general, the definition here is different, but it suffices for our purpose.

\subsubsection{Over function fields}\label{SubsubsectionConductorFF}
Assume $K=\C(B)$ is the function field of an irreducible smooth projective curve $B$. Use the bijection $M_K$ and $B(\C)$.

Let $\cA$ be identity component of the N\'{e}ron model of $A$. Write $f \colon \cA \rightarrow B$ for the structural morphism. Then $\cA_v = f^{-1}(v)$ is an abelian variety of dimension $g$  for all but at most finitely many $v \in B(\C)$.% To unify notation, we still denote by $A\pmod v$ the fiber $f^{-1}(v)$.

The (finite) set of {\it bad reduction places} of $A/K$ is defined to be
\[
\mathfrak{Bad}(A/K) := \{ v\in M_K : \cA_v\text{ is not an abelian variety of dimension }g\}.
\]
We say that $A$ has {\it semi-stable reduction at $v$} if $\cA_v$ is a semi-abelian variety. In this case, denote the toric rank by $r_v \ge 0$. We say that $A/K$ has semi-stable reduction if $\cA/B$ is an abelian scheme, or equivalently $A$ has semi-stable reduction at all $v \in M_K$.

\subsection{Faltings height}\label{SubsectionFaltingsHeight}
We now recall the definition of the {\it stable Faltings height} $h_{\mathrm{Fal}}(A)$.% Use the notations in $\mathsection$\ref{SubsectionConductor}. 

\subsubsection{Over number fields}
Assume $K$ is a number field. Use the notations from $\mathsection$\ref{SubsubsectionConductorNF}. 

We start with the case where $A/K$ has semi-stable reduction. Consider the {\it (determinant) Hodge bundle}, where $\underline{0}$ is the zero section of $f\colon \cA \to \mathrm{Spec}\cO_K$,
\[
\omega_{\cA} := \underline{0}^*\Omega^{\wedge g}_{\cA/\cO_K} = f_*\Omega^{\wedge g}_{\cA/\cO_K}
\]
which is a line bundle on $\mathrm{Spec}\cO_K$. For each $\sigma \colon K \hookrightarrow\C$, endow the metric $\|\cdot\|_{\sigma}$ on $\omega_{\cA} \otimes_{\sigma}\C$ by
\[
\|\alpha\|_{\sigma}^2 = \frac{1}{(2\pi)^g}\left| \int_{A_{\sigma}(\C)} \alpha\wedge \bar{\alpha} \right| \qquad \text{for any }\alpha \in \omega_{\cA} \otimes_{\sigma}\C.
\]
\normalsize
Then we get a Hermitian line bundle $\bar{\omega}_{\cA} := (\omega_{\cA},\{\|\cdot\|_{\sigma}\}_{\sigma})$. The {\it (stable) Faltings height} of $A$ is
\begin{equation}\label{EqFaltingsHeightNF}
h_{\mathrm{Fal}}(A) := \frac{1}{[K:\Q]} \hat{\deg}(\bar{\omega}_{\cA}),
\end{equation}
\normalsize
where $\hat{\deg}(\bar{\omega}_{\cA})$ on the right hand side is the Arakelov degree, \textit{i.e.} equals $\log \#(\omega_{\cA}/s\cO_K) - \sum_{\sigma\colon K \hookrightarrow \C} \log \|s\|_{\sigma}$ for any non-zero $s \in H^0(\cA,\omega_{\cA})$.

For general $A/K$, there exists a finite extension $K'$ of $K$ such that $A_{K'}$ has semi-stable reduction over $K'$. Then we define
\[
h_{\mathrm{Fal}}(A) := h_{\mathrm{Fal}}(A_{K'}).
\]
It is known that this definition does not depend on the choice of $K'$.

\subsubsection{Over function fields}
Assume $K=\C(B)$ is the function field of an irreducible smooth projective curve $B$. Use the notations from $\mathsection$\ref{SubsubsectionConductorFF}. 

Again we start with the case where $A/K$ has semi-stable reduction. 
Consider the {\it (determinant) Hodge bundle}, where $\underline{0}$ is the zero section of $f\colon \cA \to B$,
\[
\omega_{\cA} := \underline{0}^*\Omega^{\wedge g}_{\cA/B} = f_*\Omega^{\wedge g}_{\cA/B}
\]
which is a line bundle on $B$. Define the {\it (stable) Faltings height} of $A$ to be 
\begin{equation}\label{EqFaltingsHeightFF}
h_{\mathrm{Fal}}(A) := \frac{1}{[K:\C(\P^1)]} \deg(\omega_{\cA}).
\end{equation}
\normalsize

For general $A/K$, there exists a finite extension $K'$ of $K$ such that $A_{K'}$ has semi-stable reduction over $K'$. Then we define
\[
h_{\mathrm{Fal}}(A) := h_{\mathrm{Fal}}(A_{K'}).
\]
This definition does not depend on the choice of $K'$.

\subsection{Minkowski reduction}
Consider the lattice $\Z^g$ in the real vector space $\R^g$, and denote by $e_1,\ldots,e_g$ the basis vectors.

A  positive-definite symmetric matrix $B$ is called {\it Minkowski-reduced} if, for the quadratic form $q(x) := x^{\mathrm{t}}Bx$ on $\R^g$, we have $q(e_j) \le q(w)$ for all $j \in \{1,\ldots,g\}$ and $w =(w_1,\ldots,w_g) \in \Z^g$ with $\mathrm{gcd}(w_1,\ldots,w_g)=1$. Here are some basic properties of Minkowski-reduced matrices.

\begin{lemma}\label{LemmaReductionMinkowski}
There exists a number $\gamma_g \in [1,e^{g^2/2})$, depending only on $g$ and called the {\it Minkowski constant}, such that any Minkowski reduced matrix $B=(b_{ij})_{1\le i,j\le g}$ satisfies:
\begin{enumerate}
\item[(i)] $0\le b_{11}\le \ldots \le b_{gg}$,
\item[(ii)] $|b_{ij}| \le \frac{1}{2}|b_{ii}|$ for all $i,j$,
\item[(iii)] $\det B \le \prod_{j=1}^g b_{jj} \le \gamma_g \det B$,% In particular, $\sum_{i=1}^rb_{ii}\geq r(\det B)^{1/r}$$
\item[(iv)] the matrix $B - \frac{1}{e\gamma_g}\diag(b_{11},\ldots,b_{gg})$ is positive definite.
\end{enumerate}
\end{lemma}
\begin{proof}
Parts (i)-(iii) are standard results from reduction theory except the (crude) bound $\gamma_g < e^{g^2/2}$, which itself is an easy computation by a result of van der Waerden (cf. \cite[Thm.~5.16]{Nilsson2023Isospectral}) and Minkowski's Second Theorem. Notice that we can take $\gamma_1=1$.

Let us prove (iv). 
        Set for simplicity $D := \mathrm{diag}(b_{11},\ldots,b_{gg})$. Then $D^{-1/2}BD^{-1/2}$ is a positive-definite symmetric matrix. So it suffices to prove that  the minimal eigenvalue $\lambda_1$ of $D^{-1/2}BD^{-1/2}$ is $>1/(e\gamma_g)$.

    %Now let us bound $\lambda_1$. 
    When $g=1$ then it is clearly $1$ and we can take $\gamma_g = 1$. Assume $g\ge 2$. Use $\lambda_2,\ldots,\lambda_g$ to denote the other eigenvalues of $D^{-1/2}BD^{-1/2}$. Notice that the diagonal entries of $D^{-1/2}BD^{-1/2}$ are $1$. So $\lambda_1 + \sum_{j=2}^g \lambda_j = \mathrm{Tr}(D^{-1/2}BD^{-1/2}) = g$. Hence
        \[
    \det(D^{-1/2}BD^{-1/2}) = \lambda_1 \prod_{j=2}^g\lambda_j \le \lambda_1\left(\frac{\sum_{j=2}^g\lambda_j}{g-1}\right)^{g-1} = \lambda_1\left(\frac{g-\lambda_1}{g-1}\right)^{g-1} = \lambda_1\left( 1+ \frac{1-\lambda_1}{g-1} \right)^{g-1}.
    \]
    But the left hand side equals $(\det B)/\prod_{j=1}^g b_{jj}$, which is $\ge 1/\gamma_g$ by part (iii). The right hand side $< \lambda_1 e^{1-\lambda_1}$ because$(1+\frac{t}{n})^n$ is increasing in $n$ and has limit $e^t$ when $n\rightarrow \infty$. So $\lambda_1e^{1-\lambda_1} \ge 1/\gamma_g$. 
    %The function $x \mapsto x e^{1-x}$ is increasing when $x\in (0,1)$ and its evaluation at $x=1/\gamma_g$ is $\ge 1/\gamma_g$. 
   But $e^{1-\lambda_1} < e$ since $\lambda_1>0$. 
    Hence we have $e\lambda_1 > 1/\gamma_g$. We are done.
\end{proof}

\subsection{Adelic geometry over function fields}\label{SubsectionPreliminaryAdelicGeometry}
In this subsection, we recall the theory of adelic vector spaces and adelic line bundles over function fields $K=\C(B)$. 
%S.~Zhang \cite{ZhangSmPtAdelicMetric} introduced the theory of adelic line bundles over number fields. Let us recall the parallel theory over function fields $K=\C(B)$. 
The discussion applies to arbitrary $k=\bar{k}$. For $v\in B$, let $K_v$ be the local field and $\cO_v$ its (discrete) valuation ring. Take $\vert\varpi_v\vert_v=e^{-1}$ for compatibility with $N_v$ in §\ref{SubsectionGeneralNotation}, where $\varpi_v$ is a uniformizer of $K_v$.
%, where $B/k$ is a smooth projective curve.

%We introduce some basic notations here. Note that our notation for height of a section is the negative of that in \cite{GaudronAdelicSpace}.

\subsubsection{Adelical(ly metrized) vector spaces}
%The theory of adelically metrized vector spaces can be regarded as a generalization or a limiting version of coherent sheaves over $B$.  We will switch between coherent sheaves and the corresponding adelic metrization. 

Let $\cF$ be a vector bundle on $B$, and set $F:=\cF_K$ to be the generic fiber of $\cF$.  Then $F$ is a finite dimensional $K$-vector space. Define the following \emph{model metric} $\|\cdot\|_{v,\mathrm{mod}}$ on $F_v:=F\otimes_K K_v$ to be: for any $f\in F_v$, set
\[
\Vert f\Vert_{v,\mathrm{mod}} :=\inf_{\xi\in K_v}\{\vert \xi\vert_v : f\in \xi\cdot(\mathcal{F}\otimes_{\cO_B}\cO_v)\} \in e^{\Z}.
\]
The unit ball $\{f \in F_v : \|f\|_{v,\mathrm{mod}} \le 1\}$ is an $\cO_v$-lattice, and the model metric is the lattice norm induced by this unit ball. Hence any $\cO_v$-basis $\{f_1,\ldots,f_r\}$ of the unit ball is an orthonormal basis of $F_v$, \textit{i.e.} $\|\sum_{i=1}^r x_i f_i \|_{v,\mathrm{mod}} = \max_{i=1}^r |x_i|_v$ for any $x_1,\ldots,x_r \in K_v$.

%Adelic vector spaces can be regarded as limits of $K$-vector spaces with model metrics.

\begin{defn}\label{DefnAdelicVS}
    Let $E$ be a finite dimensional $K$-vector space. An adelic metrization is a collection $\{\|\cdot\|_v:E_v\to \RR_{\geq 0}\}_{v\in M_K}$ satisfying:
    \begin{itemize}
        %\item The image of $\|\cdot\|_v$ lies in $e^\QQ$.
        \item At each place $v$, the metric is multiplicative for scalar ($|a|_v\|f\|_v=\|af\|_v$ for any $a\in K_v$ and $f\in E_v$) and satisfies  the ultrametric triangular inequality ($\|f_1+f_2\|_v \le \max\{\|f_1\|_v,\|f_2\|_v\}$ for any $f_1,f_2 \in E_v$).
        
        \item (Coherence) There exists a vector bundle $\cE$ on a Zariski open dense subset $U\subseteq B$ such that $\|\cdot\|_v$ is induced by the model metric defined by $\cE$ at each $v \in U$.
    \end{itemize}
    Such a pair $\bar{E}:=(E,\{\|\cdot\|\}_{v\in M_K})$ is called an adelic vector space. If the adelic metric on $\bar E$ coincides with a model metric, we say that $\bar E$ is a model adelic vector space.

\end{defn}

%For every nonzero element $f\in E$, the (absolute) height is $$h_{\bar E}(f):=\frac 1{[K:K_0]}\sum_{v\in M_K}-\log \|f\|_v.$$
%The right hand side is a finite sum by the coherence condition. We use the similar notion of degree, slope as the classical case. See $\mathsection$\ref{subsectionLiouvilleIneq} for an account of these notions. 
Let $\bar{E}:=(E,\{\|\cdot\|\}_{v\in M_K})$  be an adelic space over $K$. 
The set of effective sections is $\hat H^0(\bar E):=\{f\in E: \|f\|_v\leq 1 ,\forall v\}.$ If $\bar E$ is a model adelic vector space, \textit{i.e.} in the coherence condition  we can take $U=B$, then the set of effective sections $\hat H^0(\bar E)$ coincides with the usual global sections $H^0(B,\cE)$ of vector bundle.

Denote by $r := \dim E$. The {\it determinant} $\det \bar{E} := \bigwedge^r \bar{E}$ is defined as follows. For each $v\in M_K$ and the ultrametric $\|\cdot\|_v$, there exist a $K_v$-basis $\{e_{v,1},\ldots,e_{v,r}\}$ of $E_v$ and real numbers $a_{v,1},\ldots,a_{v,r} \in \R$ such that $\left\|\sum_{i=1}^r x_i e_{v,i}\right\|_v  =  \max_{1\le i\le r}   \left\{  |x_i|_v e^{-a_{v,i}}   \right\}$.\footnote{Consider the lattice filtration $\Lambda_{v,t}:=\{x\in E_v:\|x\|_v\leq e^{-t}\}$ with $t\in \RR$. It satisfies $\Lambda_{v,t+1}=\varpi_v\Lambda_{v,t}$. So we are done after choosing a basis adapted to the finitely many jumps of this filtration modulo $\ZZ$.} Then the metric $\|\cdot\|_{\det\bar{E},v}$ is defined by
\begin{equation}\label{EqDetNorm}
\|e_1\wedge \cdots\wedge e_r\|_{\det\bar{E},v}= \prod_{i=1}^r\|e_i\|_v =\prod_{i=1}^r e^{-a_{v,i}} = e^{-\sum_{i=1}^r a_{v,i}}.
\end{equation}

The arithmetic degree of $\bar E$ is defined to be
\begin{equation}\label{EqArithDeg}
\hat \deg(\bar E) := -\sum_{v\in M_K} \log\|s\|_{\det\bar{E},v},\quad\text{ for any non-zero }s\in \det E.
\end{equation}
The definition does not depend on the choice of $s$ by the Product Formula.

The following {\it Minkowski Theorem} over function fields is surely known to experts. We include a proof here since we did not find an explicit reference.
\begin{thm}\label{ThmMinkowskiFF}
Let $\bar E=(E,\{\|\cdot\|_v\}_{v\in M_K})$ be an adelic vector space over $K$ of dimension $r$. Assume there exists a vector bundle $\cE$ on a Zariski open dense subset $U \subseteq B$ such that $\|\cdot\|_v$ is induced by the model metric defined by $\cE$ at each $v \in U$. 
%$v\in M_K$ such that $\|E_v\|_v \not\subseteq e^\ZZ$. 
Then
\[
\hat{\deg}(\bar{E}) > r\!\left( g(B) - 1 + \#(B\setminus U) \right) \Longrightarrow \hat{H}^0(\bar{E}) \not= 0.
\]
\end{thm}
\begin{proof}
Let $v\in M_K$. Let the $K_v$-basis $\{e_{v,1},\ldots,e_{v,r}\}$ of $E_v$ and the real numbers $a_{v,1},\ldots,a_{v,r} \in \R$ be as above \eqref{EqDetNorm}. We may and do take $a_{v,1} = \cdots = a_{v,r} =0$ for all $v \in U$; see above Definition~\ref{DefnAdelicVS}.

Let us modify the metric on $E_v$ at each $v \in B\setminus U$ as follows. For each $v \in B$, set
\[
b_{v,i}:=\lfloor a_{v,i}\rfloor
\]
(notice that $b_{v,i} = a_{v,i} = 0$ when $v \in U$) and define the modified metric by 
\[
\|\sum\nolimits_{i=1}^r x_i e_{v,i}\|^-_v  := \max_{1\le i \le r} \left\{ |x_i|_v e^{-b_{v,i}} \right\}.
\]
Notice that $\|\cdot\|^-_v$ is precisely the lattice norm induced by the $\mathcal O_v$-lattice
\[
\Lambda^-_v :=\bigoplus_{i=1}^r \cO_v\cdot\varpi_v^{-b_{v,i}}e_{v,i}.
\]
We claim that $\Lambda^-_v$ is the unit ball of the original norm $\|\cdot\|_v$. Indeed, we have
\begin{align*}
\|\sum\nolimits_i x_i e_{v,i}\|_v \le 1 & \Longleftrightarrow \mathrm{ord}_v(x_i)+a_{v,i} \ge 0  \text{ for all }i \\
& \Longleftrightarrow \mathrm{ord}_v(x_i)+b_{v,i} \ge 0 \text{ for all }i,\quad\text{ since }\mathrm{ord}_v(x_i) \in \Z\text{ and }b_{v,i}=\lfloor a_{v,i} \rfloor \\
& \Longleftrightarrow \sum\nolimits_i x_i e_{v,i} \in \Lambda^-_v,
\end{align*}
and hence $\{x\in E_v : \|x\|_v \le 1\}=\Lambda^-_v$.

The Beauville--Laszlo Theorem \cite{BeauvilleLaszlo} glues $\cE$ and the lattices $\{\Lambda^-_v\}_{v\not\in U}$ into  a vector bundle $\mathcal E^-$ on $B$ with generic fiber $E$. Indeed for each $v \not\in U$, the $\cO_v$-lattice $\Lambda_v^- \subseteq E_v$ defines a vector bundle on $\mathrm{Spec} \cO_v$, whose restriction to $\mathrm{Spec} K_v$ is $E_v= \Lambda_v^- \otimes_{\cO_v} K_v$ and coincides with the pullback of $\cE$ along $\mathrm{Spec} K_v \rightarrow U$. Then the Beauville--Laszlo Theorem glues $\cE$ and $\Lambda^-_v$. This glueing process terminates within finitely many steps since $\#(B\setminus U)<\infty$. Moreover
\begin{align*}
H^0(B,\cE^-) & = H^0(U,\cE) \times_{\prod_{v\not\in U} E_v} \prod\nolimits_{v\not\in U} \Lambda_v^- \\
&= \{s \in E: \|s\|_v \le 1\text{ for all } v\in M_K \} = \hat{H}^0(\bar{E}).
\end{align*}
We have, by \eqref{EqDetNorm},
\[
\|e_{v,1}\wedge\cdots\wedge e_{v,r}\|_{\det\bar E,v} =e^{-\sum_i a_{v,i}} \quad \text{ and } \quad 
\|e_{v,1}\wedge\cdots\wedge e_{v,r}\|^-_{\det\bar E^-,v} =e^{-\sum_i b_{v,i}}.
\]
It follows that
\[
\hat{\deg}(\bar{E}) - \deg(\cE^-)  = \sum_{v\not\in U}\sum_{i=1}^r (a_{v,i} - b_{v,i}),
\]
which is then $\in  [0, r\#(B\setminus U) )$ because $0\le a_{v,i}-b_{v,i}<1$. Therefore
\[
\hat{\deg}(\bar{E}) > r\!\left( g(B) - 1 + \#(B\setminus U) \right) \Longrightarrow \deg(\cE^-) > r \left( g(B)-1 \right),
\]
which furthermore implies $H^0(B,\cE^-) \not= 0$ by Riemann--Roch on $B$. Hence we are done because $\hat{H}^0(B,\bar{E}) = H^0(B,\cE^-)$.
\end{proof}
\normalsize

\subsubsection{Adelic line bundles}\label{subsubsec:Adeliclb}
Let $X/K$ be a projective variety and let $L$ be a line bundle on $X$ defined over $K$. For each $v \in M_K$, 
 a projective model $(\mathcal{X}_v,\mathcal{L}_v)$ of $(X_{K_v},L_{K_v}^{\otimes n})$ over $\cO_v$ induces a \emph{model metric} $\Vert\cdot\Vert_{\cL_v}$ on $L_v:=L\otimes_K K_v$ as follows: Any $x\in X(K_v)$ extends to
\[
\bar{x} \colon\mathrm{Spec}(\mathcal{O}_{v})\longrightarrow\mathcal{X}.
\]
For $f\in x^*L$, define
\[
\Vert f\Vert_{\cL_v} :=\inf_{\xi\in {K}_v}\{\vert \xi\vert^{1/n}_v : f^n\in \xi \bar{x}^*\mathcal{L}_v\}.
\]
A metric $\Vert\cdot\Vert_v$ on $L_v$ is continuous and bounded if $\log\frac{\Vert\cdot\Vert_v}{\Vert\cdot\Vert_{\cL_v}}$ is continuous and bounded for some model metric $\Vert\cdot\Vert_{\cL_v}$. An \emph{adelic metric} on $L$ is a collection $\{\Vert\cdot\Vert_v\}_{v\in M_K}$ of continuous and bounded metrics $\Vert\cdot\Vert_v$ on $L_v$ for all $v\in B$ satisfying the following property: there exists a projective model $(\cX,\cL)$ of $(X|_U,L|_U^{\otimes n})$ for some $n\ge 1$ and a Zariski open subset $U\subseteq B$ such that $\Vert\cdot\Vert_v$ is the model $\|\cdot\|_{\cL_v}$ for all $v\in U$ -- such a model is said to be {\it relatively semipositive} if $\cL$ is nef on special fibers of $\cX$.

An \emph{adelic line bundle} is a pair $\bar{L}=(L,\{\Vert\cdot\Vert_v\}_{v\in M_K})$ consisting of a line bundle $L$ and an adelic metric $\|\cdot\|:=\{\Vert\cdot\Vert_v\}_{v\in M_K}$ on $L$.

%Any adelic metric $\{\Vert\cdot\Vert_{v}\}$ on $L$ can be obtained as the limit of a sequence of adelic metrics $\{\Vert\cdot\Vert_{n,v}\}(n=1,2,\dots)$, such that $\Vert\cdot\Vert_{n,v}$ is independent of $n$ for $v$ in some open subvariety of $B$, and $\Vert\cdot\Vert_{n,v}/\Vert\cdot\Vert_{v}$ convenges uniformly on $X(\bar{K_v})$ for each $v\in B$. 

To define {\it nefness} of adelic line bundles, we need to following notion. A sequence $\{\|\cdot\|_n : n = 1,2,\ldots\}$ of adelic metrics {\it converges to} an adelic metric $\|\cdot\|$ if there exists a Zariski open subset $U \subseteq B$ such that $\|\cdot\|_{n,v} = \|\cdot\|_v$ for all $n$, and that $\log\frac{\|\cdot\|_{n,v}}{\|\cdot\|_v}$ converges to $0$ uniformly on $X(K)$.  
Finally, an adelic line bundle $\bar{L}=(L,\{\Vert\cdot\Vert_v\}_{v\in M_K})$ is called {\it nef}  if the underlying adelic metric  is the limit of a sequence of adelic metrics induced by relatively semipositive models. 

%Let us look at the particular case where $X=\mathrm{Spec}(K)$. In this case, an adelic line bundle is an 1-dimensional adelic space over $K$ (but not all adelic space arises as adelic line bundles). The {\it degree} of an adelic line bundle $\bar{L}$ is defined to be
%\begin{equation}\label{Eqdeghat}
%\hat{\deg}(\bar{L}):=\sum_{v\in M_K}-\log\Vert s\Vert_v
%\end{equation}
%\normalsize
%for any non-zero section $s$ of $L$. It is well-defined by the product formula. 

Let $\bar{L}$ be an adelic line bundle on $X$. Define the {\it height function} associated with $\bar{L}$ as
\[
h_{\bar{L}} \colon X(\bar{K}) \rightarrow \R, \qquad x \mapsto -\frac{1}{[K(x):\C(\mathbb{P}^1)]}\sum_{v\in M_K}\sum_{z\in \mathrm{Gal}(\bar{K}/K)x \times_K K_v} \log \|s(z)\|_v^{\deg_{K_v}(z)}.
\]
\normalsize
%Notice that when $x \in X(K)$, this simplifies to $h_{\bar{L}}(x) = \frac{1}{[K(x):\C(\mathbb{P}^1)]}\hat{\deg}(x^*\bar{L})$.

For each $s \in  H^0(X,L)$ and $v\in M_K$, define $\|s\|_{v,\sup} := \sup_{x\in X(\bar{K_v})}\|s(x)\|_v$. 
Define
\[
\hat{H}^0(X,\bar{L}) := \{ s \in H^0(X,L) : \|s\|_{v,\sup} \le 1\}.
\]
It is known that $\hat{H}^0(X,\bar{L})$ is a $\C$-vector space. 
If $\bar{L}$ is induced from $\mathcal{L}$ on some projective model $\cX$ of $X$ over $B$, then $\hat{H}^0(X,\bar{L})=H^0(\cX,\mathcal{L})$. 

Finally, the \emph{volume} of $\bar{L}$ is defined to be
\[
\mathrm{vol}(\bar{L})=\lim_{m\to\infty}\frac{(\dim X+1)!}{m^{\dim X+1}}\dim \hat{H}^0(X,m\bar{L}),
\]
where we write $m\bar{L}$ for $\bar{L}^{\otimes m}$. 
This limit always exists. And $\bar{L}$ is said to be {\it big} if  $\mathrm{vol}(\bar{L})>0$. 
For a nef adelic line bundle $\bar{L}$, one can use its top self-intersection (see for example \cite[Prop.~4.1.1]{YuanZhang_Quasiproj}) for definition) to compute its volume.

%, then we say $\bar{L}$ is \emph{big}. %For relatively nef adelic line bundles $\bar{L}_1$ and $\bar{L}_2$, there is Siu's inequality
%$$\mathrm{vol}(\bar{L}_1-\bar{L}_2)\ge\bar{L}_1^{d+1}-(d+1)\bar{L}_1^d\bar{L}_2.$$
%From which, taking $\bar L_2=0$, we have the following (geometric case of) “Arithmetic" Hilbert Samuel Theorem, whose arithmetic case is proved by Gillet-Soul\'e:
\begin{thm}[Arithmetic Hilbert--Samuel]
    For a nef adelic line bundle  $\bar{L}$ on $X$, we have $\mathrm{vol}(\bar{L})=\bar{L}^{\dim X+1}$.
\end{thm}

\subsubsection{Link}\label{SubsubsectionLinkAdelic}
Let $X/K$ be a projective variety and let $\bar{L}$ be an adelic line bundle on $X$ defined over $K$. Then 
\[
\bar{E}_m := (H^0(X,mL),\{ \|\cdot\|_{m\bar{L}, v, \sup}\}_{v \in M_K})
\]
is an adelic vector space over $K$ of dimension $h^0(X,mL)$, and $\hat{H}^0(\bar{E}_m) = \hat{H}^0(X,m\bar{L})$. 
%The arithmetic degree $\hat{\deg}(\bar{E}_m)$ is closely related to the asymptotic behavior of $\dim \hat{H}^0(X,m\bar{L})$. Indeed, we have
In this language, Arithmetic Hilbert--Samuel can be expressed as: For $\bar{L}$ nef, we have
\begin{equation}\label{EqArithmeticHSAdelicSpace}
\bar{L}^{\dim X+1}=\lim_{m\to\infty}\frac{(\dim X+1)!}{m^{\dim X+1}}\hat{\deg}(\bar{E}_m).
\end{equation}

\section{Bump functions at archimedean places}\label{SectionArch}

The goal of this section is to construct the desired bump functions at archimedean places. Take $\sigma \colon K \hookrightarrow \C$. For the curvature from $c_1(L,\|\cdot\|_{\sigma})$, its associated volume form $\mathrm{d}\mu_{\sigma} := \frac{1}{\deg_L(A)}c_1(L,\|\cdot\|_{\sigma})^{\wedge g}$  satisfies 
\[
\int_{A_{\sigma}(\C)} \mathrm{d}\mu_{\sigma} = 1.
\]

We need {\it (multi-)strips} on $A_{\sigma}(\C)$, for which we do the following set up. Take an isogeny
\begin{equation}\label{EqIsogenyArch}
\rho_{\sigma} \colon A_{\sigma} \longrightarrow A_{0,\sigma}
\end{equation}
such that $L_{\sigma}$ descends to a principal polarization $L_{0,\sigma}$ on $A_{0,\sigma}$, \textit{i.e.} $L_{\sigma} \cong \rho_{\sigma}^*L_{0,\sigma}$, and $L_{0,\sigma}$ has a cubist metric $\|\cdot\|_{0,\sigma}$ such that $\|\cdot\|_{\sigma} = \rho_{\sigma}^*\|\cdot\|_{0,\sigma}$.\footnote{The key case of the construction is when $L$ is a principal polarization on $A$. Then we take $\rho_{\sigma} = \mathrm{id}$.} Then $L_0$ gives rise to a {\it Siegel reduced matrix} $\tau_{0,\sigma}$ in the Siegel upper half space such that $A_{0,\sigma}(\C) \cong \C^g/(\Z^g+\tau_{0,\sigma}\Z^g)$; see \cite[Chap.~V, $\mathsection$4, above Lemma~15]{Igusa} for definition and Lemma~\ref{LemmaSiegelReduced} for the properties we need. Then we can identify $\lie A_{0,\sigma}$ with $\C^g$.

We can furthermore identify $\lie A_{\sigma}=\lie A_{0,\sigma} = \C^g$ via $\rho_{\sigma}$. Then $\ker(\exp_{A_{\sigma}})$ is a lattice $\Lambda$ in $\Z^g+\tau_{0,\sigma}\Z^g$. Denote by 
\begin{equation}\label{EqArchUnif}
u_{0,\sigma} \colon \C^g \longrightarrow A_{0,\sigma}(\C) = \C^g/( \Z^g+\tau_{0,\sigma}\Z^g), \qquad u_{\sigma} \colon \C^g \longrightarrow A_{\sigma}(\C) = \C^g/\Lambda
\end{equation}
 the uniformizations (then $u_{0,\sigma} = \rho_{\sigma}\circ u_{\sigma}$), and consider the Betti coordinates
\begin{equation}\label{EqBetti}
\beta_{\sigma} \colon \R^g \times \R^g \longrightarrow \C^g, \qquad (\mathbf{a},\mathbf{b}) \mapsto \mathbf{a} + \tau_{0,\sigma}\mathbf{b}.
\end{equation}

For our purpose, we need to consider all the translates of $[0,\frac{1}{g+2})^g$ chosen as follows. 
Consider the vectors in $[0,1)^g \subseteq \R^g$ whose coordinates are integer multiples of $\frac{1}{2(g+2)}$. There are $N_g := 2^g(g+2)^g$ such vectors, which we call $\mathbf{e}'_1,\ldots,\mathbf{e}'_{N_g}$. We thus have $N_g$ translates $[0, \frac{1}{g+2})^g + \mathbf{e}'_q$ for $q \in \{1,\ldots,N_g\}$, and hence $N_g$ strips 
\begin{equation}\label{EqStripeForVanishing}
\square'_{\sigma,q}:=   (u_{0,\sigma} \circ \beta_{\sigma})\!\left( [0,1]^g \times \Big([0,\frac{1}{g+2} )^g + \mathbf{e}'_q \Big) \right)   \subseteq A_{0,\sigma}(\C)
\end{equation}

for all  $k \in \{1,\ldots,N_g\}$. 
Each $\rho_{\sigma}^{-1}(\square'_{\sigma,q})$ is a multi-stripe in $A_{\sigma}(\C)$.

We can be more precise. There are $h^0(A,L) = \deg(\rho_{\sigma})$ vectors $\mathbf{v}_1,\ldots, \mathbf{v}_{h^0(A,L)}$ in $\C^g$ such that the disjoint union of the translates $\bigsqcup_{l=1}^{h^0(A,L)} \left( \mathbf{v}_l + \beta_{\sigma}([0,1)^{2g}) \right)$ is a fundamental domain for $u_{\sigma}$. Then (denote for simplicity by $I_q:= [0,1]^g \times \big([0,\frac{1}{g+2} )^g + \mathbf{e}'_q \big)$)
\begin{equation}\label{EqStripeForVanishing2}
\rho_{\sigma}^{-1}(\square'_{\sigma,q}) =   u_{\sigma} \!\left( \bigsqcup_{l=1}^{h^0(A,L)} \left(\mathbf{v}_l + \beta_{\sigma}\!(  I_q )\right)\right)   \subseteq A_{\sigma}(\C).
\end{equation}

In practice we need to consider the union of at most $g+1$ such multi-strips. 

\begin{prop}\label{CorBumpFctArch}
Let $\square'_{\sigma}$ be the union of $\le (g+1)$ such $\square'_{\sigma,q}$'s. Then there exists a smooth function 
\[
f_{\sigma} \colon A_{\sigma}(\C) \rightarrow \R_{\ge 0}
\]
satisfying the following properties: 
\begin{enumerate}
\item[(i)] $f_{\sigma}$ vanishes on  $\rho_{\sigma}^{-1}(\square'_{\sigma})$, 
\item[(ii)] $c_1(L, e^{- \epsilon f_{\sigma}}\|\cdot\|_{\sigma}) \ge 0$ for any $\epsilon \in [-1,1]$,
\item[(iii)] For the Minkowski constant $\gamma_g \ge 1$ from Lemma~\ref{LemmaReductionMinkowski}, we have
\begin{equation}\label{EqBumpFctArchInt}
\int_{A_{\sigma}(\C)} f_{\sigma} \mathrm{d}\mu_{\sigma} \ge \frac{\pi}{8e(g+1)^3(g+2)^3(100g+274)\gamma_g}  \mathrm{Tr}(\mathrm{Im}\tau_{0,\sigma}).
\end{equation}
\end{enumerate}
\end{prop}

\begin{rmk}\label{RmkArchDDC}
An equivalent way to state Proposition~\ref{CorBumpFctArch}.(ii) is: We have
\[
-c_1(L,\|\cdot\|_{\sigma}) \le \mathrm{d}\mathrm{d}^{\mathrm{c}}f_{\sigma} \le c_1(L,\|\cdot\|_{\sigma}).%\quad \text{ for all }\epsilon \in [-1,1].
\]
\end{rmk}

\subsection{Preparation and notation}
For simplicity, we will denote by
\[
\tau_{0,\sigma} = X_{\sigma} + \sqrt{-1}Y_{\sigma},
\]
and $Y_{\sigma} = (y_{jj'})_{1\le j,j'\le g}$. 
We will use the following properties for $\tau_{0,\sigma}$ which is Siegel reduced. See \cite[Chap.~V, $\mathsection$4]{Igusa}, and more precisely pp.~194 and pp.~192.
\begin{lemma}\label{LemmaSiegelReduced}
%There exists a constant $\gamma_g \ge 1$, depending only on $g$, such that:
\begin{enumerate}
\item[(i)] All entries of $X_{\sigma}$ have absolute value $\le 1/2$,
\item[(ii)] $Y_{\sigma}$ is a Minkowski-reduced matrix, with $\frac{\sqrt{3}}{2}\le y_{11} \le \ldots \le y_{gg}$.
%\item[(ii)] $\det Y_{\sigma} \le \prod_{j=1}^g y_{jj} \le \gamma_g \det Y_{\sigma}$,
%\item[(iii)] $\frac{\sqrt{3}}{2}\le y_{11} \le y_{22} \le \cdots \le y_{gg}$,
%\item[(iv)] $|y_{jj'}| \le \frac{1}{2}y_{jj}$ for all $1\le j< j'\le g$.
\end{enumerate}
\end{lemma}

For any smooth function $F$ on $\C^g$, the composite $F\circ \rho$ is a smooth function on $\R^{2g}$ for the change of variables $\rho$ defined in \eqref{EqBetti}. Let 
\begin{equation}
\begin{bmatrix} F_{\mathbf{a}\mathbf{a}} & F_{\mathbf{a}\mathbf{b}} \\ F_{\mathbf{b}\mathbf{a}} & F_{\mathbf{b}\mathbf{b}} \end{bmatrix} 
\end{equation}
be the Hessian of $F\circ \rho$.  
Denote by $\mathbf{w}= (w_1,\ldots,w_g)$ the standard coordiantes on $\C^g$. Then \eqref{EqBetti} says that $\mathbf{w} =  ( \mathbf{a}+X_{\sigma}\mathbf{b}) + \sqrt{-1}Y_{\sigma}\mathbf{b}$. 
So one can compute
\begin{align}\label{EqComputationDDbar}
\partial_{\mathbf{w}}\bar{\partial}_{\mathbf{w}} F = & \frac{1}{4} (\mathrm{d}\mathbf{w})^{\mathrm{t}} \left( Y_{\sigma}^{-1} \begin{bmatrix} \tau_{0,\sigma} & -I \end{bmatrix} 
\begin{bmatrix} F_{\mathbf{a}\mathbf{a}} & F_{\mathbf{a}\mathbf{b}} \\ F_{\mathbf{b}\mathbf{a}} & F_{\mathbf{b}\mathbf{b}} \end{bmatrix} 
\begin{bmatrix}  \bar{\tau_{0,\sigma}} \\ -I \end{bmatrix} 
Y_{\sigma}^{-1} \right) \mathrm{d}\bar{\mathbf{w}}. 
\end{align}

\subsection{Auxiliary smooth function on $\R$}
Define the smooth function $\eta \colon \R \rightarrow \R$ by
\[
\eta(t) = \begin{cases} e^{-1/t} &  t > 0 \\ 0 & t\le 0 \end{cases}
\]
and the smooth function $\theta \colon \R \rightarrow \R$ by
\[
\theta(t) = \frac{\eta(t)}{\eta(t)+\eta(1-t)}.
\]
Then $\theta(t) = 0$ for $t\le 0$ and $\theta(t) = 1$ for $t \ge 1$, $|\theta'|_{\infty} \le 5$ and $|\theta''|_{\infty} \le 112$.

Consider the $1$-periodic smooth function $\psi \colon \R \longrightarrow \R$ defined by
\begin{equation}\label{EqAuxFunctionR}
    \psi(s) = \theta\left(\frac{\{s\}-\frac{1}{g+2}}{\delta}\right)
\theta\left(\frac{1-\{s\}}{\delta}\right), \quad\text{ with }\delta:= \frac{1}{4(g+1)(g+2)},
    %\eta\!\left( \! \{s\}-\frac{1}{g+2}\right)\eta(1 - \{s\}) 
\end{equation}
where $\{s\} := s - \lfloor s \rfloor$. Then $\psi$ vanishes precisely on $[0,\frac{1}{g+2}] + \Z$.  On $[\frac{1}{g+2}+\delta, 1-\delta]$,  we have $\psi(s)=1$. Direct computation shows $|\psi|_{\infty} \le 1$, $|\psi'|_{\infty}\le 40(g+1)(g+2)$ and $|\psi''|_{\infty} \le 2 (112+5^2) \cdot 16(g+1)^2(g+2)^2 = 16\cdot 274 (g+1)^2(g+2)^2$. Moreover 
\begin{equation}\label{EqIntegralAux}
\int_0^1 \psi(s)\mathrm{d}s  \ge 1-\frac{1}{g+2}-2\delta\ge \frac{7}{12}\ge \frac{1}{2}.
\end{equation}

\subsection{Proof of Proposition~\ref{CorBumpFctArch}}
Up to renumbering we may assume $\square'_{\sigma} \subseteq  \bigcup_{k=1}^{g+1} \square'_{\sigma,q}$. We will write $\mathbf{e}'_q = (e_q^{(1)},\ldots,e_q^{(g)})$. 
Let
\[
\tilde{F} \colon \R^{2g} = \R^g \times \R^g \longrightarrow \R_{\ge 0}, \qquad (\mathbf{a},\mathbf{b}) \mapsto \frac{4\pi}{ 16e (g+1)^3(g+2)^2(100g+274) \gamma_g}\sum_{j=1}^g y_{jj} \Psi(b_j)
\]
where $\Psi(b_j) = \prod_{q=1}^{g+1}\psi(b_j - e_q^{(j)})$. Here we still denote by $\mathbf{b} = (b_1,\ldots,b_g)$. Then $\tilde{F}$ is smooth, $\Z^{2g}$-periodic, and vanishes on $\bigcup_{q=1}^{g+1}\left( [0,1]^g \times \left([0,\frac{1}{g+2} )^g + \mathbf{e}'_q\right) \right)$. So the composite $\tilde{F} \circ \beta_{\sigma}^{-1} \colon \C^g \rightarrow \R^{2g} \rightarrow \R$ descends to $A_{0,\sigma}(\C)$, \textit{i.e.} there exists a smooth function 
\[
f_{0,\sigma} \colon A_{0,\sigma}(\C) \longrightarrow \R_{\ge 0}
\]
such that $\tilde{F} \circ \beta_{\sigma}^{-1} = f_{0,\sigma}\circ u_{0,\sigma}$ for the uniformization $u_{0,\sigma} \colon \C^g \rightarrow A_{0,\sigma}(\C)$. This $f_{0,\sigma}$ clearly vanishes on $\square'_{\sigma}$.

Set  
\[
f_{\sigma} =  f_{0,\sigma} \circ \rho_{\sigma} \colon A_{\sigma}(\C) \longrightarrow \R_{\ge 0}.
\]
Then $f_{\sigma}$ clearly satisfies (i).

The variable $\mathbf{a}$ does not appear in the definition $\tilde{F}$. So by \eqref{EqComputationDDbar} applied to $\tilde{F} \circ \beta_{\sigma}^{-1}$, we have
\[
\partial_{\mathbf{w}}\bar{\partial}_{\mathbf{w}} \tilde{F} = \frac{1}{4} (\mathrm{d}\mathbf{w})^{\mathrm{t}} Y_{\sigma}^{-1}(\tilde{F}\circ \beta_{\sigma}^{-1})_{\mathbf{b}\mathbf{b}} Y_{\sigma}^{-1}  \mathrm{d}\bar{\mathbf{w}}.
\]
Notice that $(\tilde{F}\circ \beta_{\sigma}^{-1})_{\mathbf{b}\mathbf{b}}$ is the diagonal matrix with diagonal entries $4\pi y_{jj} \Psi''(b_j) / 16e (g+1)^3(g+2)^2(100g+274)  \gamma_g$. Direct computation shows that $|\Psi''|_{\infty} \le (g+1)|\psi''|_{\infty} + g(g+1)|\psi'|_{\infty}^2 \le 16(g+1)^3(g+2)^2(100g+274) $, using the bounds for $|\psi|_{\infty}$ and $|\psi'|_{\infty}$ and $|\psi''|_{\infty}$. So Lemma~\ref{LemmaReductionMinkowski}.(iv) (applied to $B=Y_{\sigma}$) implies that the matrix $Y_{\sigma} + \frac{\epsilon}{4\pi} (\tilde{F} \circ \beta_{\sigma}^{-1})_{\mathbf{b}\mathbf{b}}$ is positive definite for all $\epsilon \in [-1,1]$. Hence 
\begin{align*}
u_{0,\sigma}^*c_1(L_0,e^{-\epsilon f_{0,\sigma}}\|\cdot\|_{0,\sigma}) & = u_{0,\sigma}^*c_1(L_0,\|\cdot\|_{0,\sigma}) + \frac{\sqrt{-1}}{2\pi}\partial_{\mathbf{w}}\bar{\partial}_{\mathbf{w}} (\epsilon f_{0,\sigma} \circ u_{0,\sigma}) \\
& = \frac{\sqrt{-1}}{2}(\mathrm{d}\mathbf{w} )^{\mathrm{t}}\wedge Y_{\sigma}^{-1} \mathrm{d}\bar{\mathbf{w}} +  \frac{\sqrt{-1}}{2\pi}\partial_{\mathbf{w}}\bar{\partial}_{\mathbf{w}} \left(\epsilon  \tilde{F} \circ \beta_{\sigma}^{-1}\right) \ge 0.
\end{align*}
So $c_1(L,e^{-\epsilon f_{\sigma}} \|\cdot\|_{\sigma}) = \rho_{\sigma}^*c_1(L_0,e^{-\epsilon f_{0,\sigma}}\|\cdot\|_{0,\sigma}) \ge 0$. This establishes (ii).

Finally, let us bound $\int_0^1 \Psi(b_j)\mathrm{d}b_j$ from below. We claim that
\begin{equation}\label{EqLowerBoundIntegral}
\int_0^1 \Psi(b_j)\mathrm{d}b_j \ge \frac{1}{2(g+2)}.
\end{equation}
Indeed, notice that $\Psi$ is the product of $(g+1)$ translates of $\psi$ and, on $[0,1)$, $\psi$ equals $1$ on an interval of length $1- \frac{1}{g+2} -\frac{1}{2(g+1)(g+2)}$. Thus $\psi$ is not $1$ on a subset of $[0,1)$ of Lebesgue measure $\le \frac{1}{g+2}+\frac{1}{2(g+1)(g+2)}$. Hence on $[0,1)$, $\Psi$ is not $1$ on a subset of Lebesgue measure $\le (g+1)\left(\frac{1}{g+2}+\frac{1}{2(g+1)(g+2)}\right)$. Thus $\int_0^1 \Psi(b_j)\mathrm{d}b_j  \ge 1-(g+1)\left(\frac{1}{g+2}+\frac{1}{2(g+1)(g+2)}\right) = \frac{1}{2(g+2)}$.

Denote by $\mathrm{d}\mu_{0,\sigma} := \frac{1}{\deg_{L_{0,\sigma}}(A_{0,\sigma})} c_1(L_0,\|\cdot\|_{0,\sigma})^{\wedge g}$. Then $\mathrm{d}\mu_{0,\sigma} = (\rho_{\sigma})_* \mathrm{d}\mu_{\sigma}$. So we can conclude for (iii) by
\begin{flalign*}
\int_{A_{\sigma}(\C)} f_{\sigma} \mathrm{d}\mu_{\sigma} = \int_{A_{0,\sigma}(\C)} f_{0,\sigma} \mathrm{d}\mu_{0,\sigma}  & = \int_{ [0,1]^{2g} } \tilde{F} \mathrm{d}a_1\wedge \cdots \mathrm{d}a_g \wedge \mathrm{d}b_1\wedge \cdots \wedge \mathrm{d}b_g \\
& = \frac{4\pi}{16e(g+1)^3(g+2)^2(100g+274) \gamma_g} \sum_{j=1}^g y_{jj} \int_0^1 \Psi(b_j)\mathrm{d}b_j \\
& \ge \frac{\pi}{8e(g+1)^3(g+2)^3(100g+274) \gamma_g}  \mathrm{Tr}(Y_{\sigma})  \text{ by \eqref{EqLowerBoundIntegral}}. && \qed
\end{flalign*}

\subsection{A technical lemma}
Later on we need to compare wedge powers of $c_1(L,\|\cdot\|_{\sigma})$ and $\mathrm{d}\mathrm{d}^{\mathrm{c}}f_{\sigma}$, which are related by Remark~\ref{RmkArchDDC}. We prove:
\begin{lemma}\label{LemmaComparison11Form}
    Let $\alpha$ and $\beta$ be two $(1,1)$-forms on $\C^g$, with $\alpha$ positive everywhere. Assume $-\alpha\le \beta\le \alpha$. Then for any $k \in \{0,\ldots,g\}$, we have
        \[
    -\alpha^{\wedge g}\le  \alpha^{\wedge (g-k)} \wedge \beta^{\wedge k}\le \alpha^{\wedge g}.
    \]    
\end{lemma}
\begin{proof}
    Since $\alpha$ is positive everywhere, pointwise there exists a basis $\{w_1,\ldots,w_g\}$ of $\C^g$ such that $\alpha = \sum_{j=1}^g \mathrm{d}w_j \wedge \mathrm{d}\bar{w}_j$ and $\beta= \sum_{j=1}^g \lambda_j \mathrm{d}w_j \wedge \mathrm{d}\bar{w}_j$. The assumption $-\alpha\le \beta\le \alpha$ implies that $|\lambda_j| \le 1$ for all $j$. Then $\alpha^{\wedge g} = g!\mathrm{d}w_1 \wedge \mathrm{d}\bar{w}_1\wedge \cdots \wedge \mathrm{d}w_g \wedge \mathrm{d}\bar{w}_g$ and 
       \[
    \alpha^{\wedge (g-k)} \wedge \beta^{\wedge k} = (g-k)!k! \sum_{1\le i_1< \cdots < i_k \le g}(\lambda_{i_1}\cdots \lambda_{i_k}) \mathrm{d}w_1 \wedge \mathrm{d}\bar{w}_1\wedge \cdots \wedge \mathrm{d}w_g \wedge \mathrm{d}\bar{w}_g.
    \]   
    Now we are done since
        \[
    \left|\frac{(g-k)!k! \sum_{1\le i_1< \cdots < i_k \le g}(\lambda_{i_1}\cdots \lambda_{i_k})}{g!}\right| \le \frac{1}{\binom{g}{k} } \sum_{1\le i_1< \cdots < i_k \le g}|\lambda_{i_1}|\cdots |\lambda_{i_k}| \le 1.	
    \qedhere
    \]    
\end{proof}

\section{Revision on tropicalization at non-archimedean places}\label{SectionNonArchPrep}
In this section, we review the theory of uniformization of abelian varieties at  non-archimedean places and the tropicalization. There are other nice reviews, \textit{e.g} \cite[§3]{FosterRabinoffShokriehSoto2018}, \cite[§6-7.1]{deJongShokriehDecomposition}, \textit{etc.}
%We refer to \cite{}.

Take $v \in M_K^0$. For simplicity denote by $F:= K_v$ and by $R$ the ring of integers of $F$. %For any variety over $F$, use $(\cdot)^{\mathrm{an}}$ to denote the Berkovich analytification.

Assume $A$ has semi-stable reduction at $v$. 
By abuse of notation and only in this section, we will use $A$ (resp. use $L$) to denote $A \otimes_K K_v$ (resp. denote $L\otimes_K K_v$). Then $A$ is an abelian variety defined over $F$ of dimension $g$, with polarization given by the ample line bundle $L$. In the whole section, by $(\cdot)^{\mathrm{an}}$ we mean the Berkovich analytification.

The identity component of the N\'eron model $\cA$ of $A$ is a semi-abelian scheme over $\mathrm{Spec} (R)$, whose special fiber has toric rank $r \ge 0$. To simplify the situation, we also assume that {\it the toric part of the special fiber of $\cA$ is split over the residue field $k_v$}. Two extreme cases are: (i) $r=0$, in which case $A$ has good reduction; (ii) $r=g$, in which case the special fiber of $\cA$ is $({\mathbb{G}_{\mathrm{m}}})^{g}$.

 A particular example to keep in mind is the Tate curve $F^{\times}/\mathfrak{q}^{\mathbb{Z}}$, in which case $r=g=1$.

\subsection{Analytic uniformization and Raynaud cross}
The analytic uniformization $p \colon E^{\mathrm{an}} \rightarrow A^{\mathrm{an}}$  fits into the {\it Raynaud cross} 
\begin{equation}\label{EqRaynaudCross}
\begin{split}
\xymatrix{
 & Y \ar[d]_-{\mathfrak{v}} \ar[rd]^-{\varphi} & \\
 T^{\mathrm{an}} \ar[r] & E^{\mathrm{an}} \ar[r]_-{q} \ar[d]^-p & B^{\mathrm{an}} \\
 & A^{\mathrm{an}} &
}
\end{split}
\end{equation}
Let us explain how this diagram is constructed. Before moving on, we point out that in the case where $A$ is a Tate curve $F^{\times}/\mathfrak{q}^{\Z}$, we have $Y = \Z$, $T = E = \mathbb{G}_{\mathrm{m}}$, $B$ is a point, and the map $Y = \Z \rightarrow E^{\mathrm{an}} = F$ is $a \mapsto \mathfrak{q}^a$.

\subsubsection{Horizontal}
The space $E^{\mathrm{an}}$ is the analytification of a semi-abelian variety  $E$ (constructed below) defined over $F$
\begin{equation}\label{EqSemiAbelian}
1 \longrightarrow T \longrightarrow E \longrightarrow B \longrightarrow 1
\end{equation}
whose toric part $T$ has rank $r$ and abelian part $B$ has good reduction. More precisely, the horizontal line of the Raynaud cross \eqref{EqRaynaudCross} is the analytification of \eqref{EqSemiAbelian}. Moreover, $T$ is an $F$-split torus by our assumption.

Here is a brief construction of $E$. The identity component of the N\'eron model $\cA$ of $A$ is a semi-abelian scheme over $S:=\mathrm{Spec} (R)$. Write $\mathfrak{m}$ for the maximal ideal of $R$. Set
\begin{align*}
S_i := \mathrm{Spec}(R/\mathfrak{m}^i), &  \qquad  \cA_i := \cA \times_S S_i \\
\hat{S} := \lim S_i, & \qquad \hat{\cA} := \lim \cA_i.
\end{align*}
Then $\hat{\cA}$ is a formal semi-abelian scheme over $\hat{S}$; it is the formal completion of $\cA$ along its closed fiber. Moreover  $\hat{\cA}$ is algebraizable, \textit{i.e.} there exists a  semi-abelian scheme $\tilde{\cA}$ over $S$ such that $\hat{\cA}$ is the associated formal scheme. Our $E$ is the generic fiber of $\tilde{\cA}$.

To construct the vertical maps of \eqref{EqRaynaudCross}, we also need to apply the discussion above to the dual of $A$, which we denote by $A^{\mathrm{t}}$, to get%. We then get the corresponding semi-abelian variety $E^{\mathrm{t}}$ defined over $F$
\begin{equation}\label{EqSemiAbelianDual}
1 \longrightarrow T^{\mathrm{t}} \longrightarrow E^{\mathrm{t}} \longrightarrow B^{\mathrm{t}} \longrightarrow 1
\end{equation}
where $B^{\mathrm{t}}$ is the dual abelian variety of $B$. The toric part $T^{\mathrm{t}}$ also has rank $r$ and is $F$-split.

\subsubsection{Vertical}
The period lattice $Y := \hom(T^{\mathrm{t}},\mathbb{G}_{\mathrm{m}})$ in \eqref{EqRaynaudCross} is the group of characters of $T^{\mathrm{t}}$, and hence is isomorphic to $\Z^r$.

We first construct the group homomorphism $\varphi\colon Y \rightarrow B(F)$, using \eqref{EqSemiAbelianDual}. For each $u' \in Y = \hom(T^{\mathrm{t}}, \mathbb{G}_{\mathrm{m}})$, we have a pushout diagram in the category of algebraic groups
\begin{equation}\label{EqPushOutDual}
\begin{split}
\xymatrix{
1 \ar[r] & T^{\mathrm{t}} \ar[r] \ar[d]^-{u'} & E^{\mathrm{t}} \ar[r] \ar[d]^-{e_{u'}} & B^{\mathrm{t}} \ar[r] \ar[d]^{=} & 1 \\
1 \ar[r] & \mathbb{G}_{\mathrm{m}} \ar[r] & (E^{\mathrm{t}})_{u'} \ar[r] & B^{\mathrm{t}} \ar[r] & 1.
}
\end{split}
\end{equation}
Now $(E^{\mathrm{t}})_{u'}$ is a semi-abelian variety which is an extension of $B^{\mathrm{t}}$ by $\mathbb{G}_{\mathrm{m}}$, and all such extensions are parametrized by the $F$-points of the dual of $B^{\mathrm{t}}$, \textit{i.e.} by $B(F)$. Hence we obtain a map $\varphi\colon Y \rightarrow B(F)$, which is easily checked to be a group homomorphism.

Next we brielfy explain the construction of $p \colon E^{\mathrm{an}} \rightarrow A^{\mathrm{an}}$. 
Denote by $\eta$ the generic point of $S = \mathrm{Spec}(R)$. Then $\cA_{\eta} = A$. General knowledge of rigid analytic geometry says that there is an open immersion
\[
i_{\cA} \colon \hat{\cA}_{\eta} \longrightarrow (\cA_{\eta})^{\mathrm{an}} = A^{\mathrm{an}},
\]
where $\hat{\cA}_{\eta}$ is the generic fiber of $\hat{\cA}$ viewed as a rigid analytic space (often denoted by $(\hat{\cA})^{\mathrm{rig}}$). But  $E$ is the generic fiber of the algebraization  of $\hat{\cA}$. So there is another open immersion $\hat{\cA}_{\eta} \rightarrow E^{\mathrm{an}}$. By \cite[Thm.~1.2]{BL}, $i_{\cA}$ uniquely extends to a surjective group morphism $p \colon E^{\mathrm{an}} \rightarrow A^{\mathrm{an}}$ whose kernel is a lattice in $E(F)$ of rank $r$ (\textit{i.e.} $\cong \Z^r$); the proof uses local trivializations. In particular, $Y$ and $\ker(p)$ are isomorphic as groups.

The homomorphism $\mathfrak{v}$ is an isomorphism from $Y$ to $\ker(p)$ which lifts $\varphi$. We will explain how to construct this $\mathfrak{v}$ by fixing a trivialization of a certain $\mathbb{G}_{\mathrm{m}}$-biextension above \eqref{EqDefntP}.

\subsubsection{Duality and Poincar\'{e} line bundle}
Later on we shall also use the Raynaud cross associated with the dual abelian variety $A^{\mathrm{t}}$ which is the diagram below on the right. Its horizontal line is the analytification of \eqref{EqSemiAbelianDual}, and $X := \hom(T,\mathbb{G}_{\mathrm{m}})$ is the group of characters of $T$ and hence is isomorphic to $\Z^r$. 
\begin{equation}\label{EqRaynaudCrosses}
\begin{split}
\xymatrix{
 & Y \ar[d]_-{\mathfrak{v}} \ar[rd]^-{\varphi} & & & & X \ar[d]_-{\mathfrak{v}'} \ar[rd]^-{\varphi'}  \\
 T^{\mathrm{an}} \ar[r] & E^{\mathrm{an}} \ar[r]_-{q} \ar[d]^-p & B^{\mathrm{an}} & &  (T^{\mathrm{t}})^{\mathrm{an}} \ar[r] & (E^{\mathrm{t}})^{\mathrm{an}} \ar[r]_-{q'} \ar[d]^-{p'} & (B^{\mathrm{t}})^{\mathrm{an}}  \\
 & A^{\mathrm{an}} & & &   & (A^{\mathrm{t}})^{\mathrm{an}} 
}
\end{split}
\end{equation}
The horizontal lines of the two Raynaud crosses are related by the polarization $\lambda_{A,L} \colon A \rightarrow A^{\mathrm{t}}$ as follows. This polarization induces an isogeny $\cA \rightarrow \cA^{\mathrm{t}}$ between the identity components of their N\'eron models, and hence the following commutative diagram whose vertical maps are isogenies (and are isomorphisms if $L$ is a principal polarization)
\begin{equation}
\begin{split}
\xymatrix{
1 \ar[r] & T \ar[r] \ar[d]^-{\lambda_{T,L}} & E \ar[r] \ar[d]^-{\lambda_{E,L}} & B \ar[r]  \ar[d]^-{\lambda_{B,L}} & 1 \\
1 \ar[r] & T^{\mathrm{t}} \ar[r] & E^{\mathrm{t}} \ar[r] & B^{\mathrm{t}} \ar[r] & 1.
}
\end{split}
\end{equation} 
Notice that $\lambda_{B,L}$ is a polarization on $B$, which is furthermore defined by an ample line bundle $M$ on $B$ such that
\begin{equation}\label{EqIsomLanManPullbackOnE}
q^*M^{\mathrm{an}} = p^*L^{\mathrm{an}}.
\end{equation}
This is because, roughly speaking, the Poincar\'{e} line bundle on $A \times A^{\mathrm{t}}$ extends to the rigidified Poincar\'{e} line bundle on $\cA \times_S \cA^{\mathrm{t}}$ and hence to $E \times E^{\mathrm{t}}$, and furthermore descends to the Poincar\'{e} line bundle on $B \times B^{\mathrm{t}}$. We refer to \cite[Chap.~II, §2]{FC} for more details. The {\it Poincar\'{e} line bundle} $P$ on $B\times B^{\mathrm{t}}$ can be canonically identified with
\begin{equation}\label{EqPoincareMumford}
(1,\lambda_{B,L})^*P = m_B^*M \otimes p_1^*M^{\otimes -1} \otimes p_2^*M^{\otimes -1}.
\end{equation}

Observe that  $\lambda_{T,L}$ induces $\lambda_{T,L}^* \colon Y \rightarrow X$ and fits into the commutative diagram
\begin{equation}\label{EqLambdaNonArch}
\begin{split}
\xymatrix{
Y \ar[r]^-{\lambda_{T,L}^*} \ar[d]_-{\varphi} & X \ar[d]^-{\varphi'} \\
B \ar[r]^-{\lambda_{B,L}} & B^{\mathrm{t}}.
}
\end{split}
\end{equation}
The horizontal maps are isomorphisms if $L$ is a principal polarization.

The following construction is important to defining the tropicalization of $A^{\mathrm{an}}$ later on. 
For each $u \in X$, we  have a  pushout diagram in the category of algebraic groups
\begin{equation}\label{EqPushOut}
\begin{split}
\xymatrix{
1 \ar[r] & T \ar[r] \ar[d]^-{u} & E \ar[r] \ar[d]^-{e_u} & B \ar[r] \ar[d]^{=} & 1 \\
1 \ar[r] & \mathbb{G}_{\mathrm{m}} \ar[r] & E_u \ar[r] & B \ar[r] & 1,
}
\end{split}
\end{equation}
and $E_u$ is a semi-abelian variety which is an extension of $B$ by $\mathbb{G}_{\mathrm{m}}$. Notice that $E_u$ can also be viewed as a $\mathbb{G}_{\mathrm{m}}$-torsor on $B$, whose associated line bundle is $P|_{B\times \{\varphi'(u)\}}$ (\textit{i.e.} $E_u$ is $P|_{B\times \{\varphi'(u)\}}$ with the zero section removed). Moreover, the homomorphism $\varphi \colon X \rightarrow B^{\mathrm{t}}(F)$ is defined using \eqref{EqPushOut}, by sending $u \in X$ to the point of $B^{\mathrm{t}}(F)$ parametrizing $E_u$.

%Under the natural bijection between line bundles and $\mathbb{G}_{\mathrm{m}}$-torsors on $B$,\footnote{The $\mathbb{G}_{\mathrm{m}}$-torsor corresponding to a line bundle $L$ is obtained by removing the zero section.} the corresponding line bundle is $P|_{B\times \{\varphi'(u)\}}$
%Let us write down the group homomorphism $\varphi'$, which is defined similarly to \eqref{EqPushOutDual}, in more details. Recall that  $B^{\mathrm{t}}(F)$ can be naturally identified with $\mathrm{Pic}^0(B)$. Under the natural bijection between line bundles and $\mathbb{G}_{\mathrm{m}}$-torsors on $B$,\footnote{The $\mathbb{G}_{\mathrm{m}}$-torsor corresponding to a line bundle $L$ is obtained by removing the zero section.} $\mathrm{Pic}^0(B)$ correspond to the $\mathbb{G}_{\mathrm{m}}$-torsors whose total space has a structure of commutative algebraic groups. In other words, $B^{\mathrm{t}}(F)$ is the parametrizing space of the semi-abelian varieties which are extensions of $B$ by $\mathbb{G}_{\mathrm{m}}$. 
%and $\varphi'(u) \in B^{\mathrm{t}}(F)$ is the point parametrizing $E_u$.  If we view $E_u$ as a $\mathbb{G}_{\mathrm{m}}$-torsor on $B$, then its associated line bundle on $B$  is  $P|_{B\times \{\varphi'(u)\}}$; in other words, $E_u$ is $P|_{B\times \{\varphi'(u)\}}$ with the zero section removed.

\vskip 0.3em
Finally, let us take a closer look at the construction of the homomorphisms $\mathfrak{v} \colon Y \rightarrow E(F)$ and $\mathfrak{v}' \colon X \rightarrow E^{\mathrm{t}}(F)$, in which we see how the maps $e_u$ from \eqref{EqPushOut} and $e_{u'}$ from \eqref{EqPushOutDual} are related. Denote by $P^{\times}$  the  $\mathbb{G}_{\mathrm{m}}$-torsor on $B\times B^{\mathrm{t}}$ corresponding to $P$. 
Then for any $u \in X$ and $u' \in Y$, we  have
\[
P^{\times}|_{\{\varphi(u')\}\times B^{\mathrm{t}}} = (E^{\mathrm{t}})_{u'} \qquad\text{and} \qquad P^{\times}|_{B\times \{\varphi'(u)\}} = E_u.
\]

For $(\varphi,\varphi') \colon Y \times X \rightarrow B \times B^{\mathrm{t}}$, the pullback $(\varphi,\varphi')^*P^{\times}$ is a $\mathbb{G}_{\mathrm{m}}$-biextension over $Y\times X$, which is trivial. Taking a trivialization of $(\varphi,\varphi')^*P^{\times}$ amounts to taking a bilinear map
\begin{equation}\label{EqDefnt}
t\colon Y \times X \rightarrow P^{\times}
\end{equation}
lifting $(\varphi,\varphi')$, \textit{i.e.} $t(u',u) \in P^{\times}_{\varphi(u'),\varphi'(u)}$ for any $u' \in Y$ and $u\in X$.

\vskip 0.2em
From now on, fix such a $t$ as in \cite[Chap.~II, Thm.~5.1]{FC}.

\vskip 0.2em
To lift $\varphi$ to $\mathfrak{v} \colon Y \rightarrow E(F)$, it suffices to define $e_u \circ \mathfrak{v} \colon Y \rightarrow E_u(F)$ for all $u \in X$. Notice that $P^{\times}_{\varphi(u'),\varphi'(u)} \subseteq P^{\times}|_{B\times \{\varphi'(u)\}} = E_u$. 
Now we define $(e_u \circ \mathfrak{v})(u') := t(u',u)$ for each $u' \in Y$, and this gives rise to the desired lift of $\varphi$ to $\mathfrak{v}$.

The homomorphism $\mathfrak{v}' \colon X \rightarrow E^{\mathrm{t}}(F)$ is defined in a similar way. By construction, %we have 
\begin{equation}\label{EqDefntP}
t(u',u)  = e_u(\mathfrak{v}(u')) = e_{u'}(\mathfrak{v}'(u)) \qquad\text{for all }u'\in Y, u\in X.
\end{equation}

%\[
%t(\{u'\}\times X) \subseteq P|_{\{\varphi(u')\}\times B^{\mathrm{t}}} = (E^{\mathrm{t}})_{u'}
%\]

%\[
%P|_{\{\varphi(u')\}\times B^{\mathrm{t}}} = (E^{\mathrm{t}})_{u'} \qquad\text{and} \qquad P|_{B\times \{\varphi'(u)\}} = E_u.
%\]

%Finally, the maps $e_u$ from \eqref{EqPushOut} and $e_{u'}$ from \eqref{EqPushOutDual} are related as follows. By abuse of notation, we identify a line bundle and the corresponding $\mathbb{G}_{\mathrm{m}}$-torsor. 
%Then for any $u \in X$ and $u' \in Y$, we  have
%\[
%P|_{\{\varphi(u')\}\times B^{\mathrm{t}}} = (E^{\mathrm{t}})_{u'} \qquad\text{and} \qquad P|_{B\times \{\varphi'(u)\}} = E_u.
%\]
%Hence $e_u(u')$ and $e_{u'}(u)$ are points in the fiber $P$ over the point $(\varphi(u'), \varphi'(u)) \in B(F)\times B^{\mathrm{t}}(F)$. It is known that
%\[
%e_u(u') = e_{u'}(u).
%\]
%Hence we can define a bilinear map
%\begin{equation}\label{EqDefntP}
%t \colon Y \times X \longrightarrow P, \qquad (u',u) \mapsto e_u(u') = e_{u'}(u).
%\end{equation}

\subsection{Tropicalization, skeleton, and N\'eron components}\label{SubsectionTropicalization}
Let $X^* := \hom(X,\Z)$ be the dual of $X$, and $X^*_{\R} := X^*\otimes \R$. Let
\[
\langle, \rangle \colon X \times X^*_{\R} \longrightarrow \R
\]
be the natural evaluation pairing.

In the following, we will take $\log$ with respect to the base $N_v$ defined in §\ref{SubsectionGeneralNotation}.%a possibly non-standard base, which we explain now. If $K$ is a function field, then we take the standard base $e$. If $K$ is a number field (recall that we are working at a non-archimedean place $v$ of $K$ corresponding to a maximal ideal $\mathfrak{p}$ of $\cO_K$), then we take the base  to be the norm of $\mathfrak{p}$.

\subsubsection{Tropicalization of $E^{\mathrm{an}}$}
The {\it tropicalization} $\mathrm{trop} \colon T^{\mathrm{an}} \rightarrow X^*_{\R}$ is defined by the rule
\[
\langle u, \mathrm{trop}(z) \rangle = -\log|u(z)|, \qquad\text{for all }u \in X\text{ and }z \in T^{\mathrm{an}}.
\]
Indeed, the function $-\log |\cdot|$ here is the valuation. 
This map extends in a natural way to a surjective homomorphism
\begin{equation}\label{EqTrop}
\mathrm{trop} \colon E^{\mathrm{an}} \longrightarrow X^*_{\R}
\end{equation}
as follows. Let $\|\cdot\|_M$ be the cubist model metric on $M$ (it exists because $B$ has good reduction), and let $\|\cdot\|_P$ be the induced metric on $P$ by \eqref{EqPoincareMumford}. For any $u \in X$, consider the pushout \eqref{EqPushOut} and the associated line bundle $P|_{B\times \{\varphi'(u)\}}$ on $B$. Let $\|\cdot\|_{E_u}$ be the metric on $P|_{B\times \{\varphi'(u)\}}$ induced by $\|\cdot\|_P$. Now define \eqref{EqTrop} by the rule
\[
\langle u, \mathrm{trop}(z) \rangle = -\log\|e_u(z)\|_{E_u}, \qquad\text{for all }u \in X\text{ and }z \in E^{\mathrm{an}}.
\]
This is a desired extension since $e_u(z) = u(z)$ for all  $z \in T^{\mathrm{an}}$.

\subsubsection{Skeleton and tropicalization of $A^{\mathrm{an}}$}
We have a $\Z$-valued non-degenerate bilinear map 
\begin{equation}\label{EqNonDegBilinearPairing}
[\cdot,\cdot] \colon Y \times X \longrightarrow \Z, \qquad (u',u) \mapsto -\log\|t(u',u)\|_P
\end{equation}
where $t$ is  from \eqref{EqDefnt}. 
%Moreover the induced bilinear map $Y \times Y \rightarrow \Z$, $(u'_1, u'_2)\mapsto [u'_1, \lambda_{T,L}^*(u'_2)]$ is positive definite. 
%The non-degenerate bilinear map $[\cdot,\cdot]$
It realizes $Y$ as a subgroup of $X^*$ of finite index, by sending each $u' \in Y$ to $[u', \cdot ] \in X^*$. From now on, we view $Y$ as a subset of $X^*$ under this realization.

By \eqref{EqDefntP}, the inclusion $Y \subseteq X^*_{\R}$ equals the composite $\mathrm{trop}\circ \mathfrak{v} \colon Y \rightarrow E^{\mathrm{an}} \rightarrow X^*_{\R}$. 
Define the {\it skeleton} of $A^{\mathrm{an}}$  to be the real torus of rank $r$
\[
\Sigma := X^*_{\R}/Y.
\]
This gives rise to the following commutative diagram with surjective vertical maps 
\begin{equation}\label{EqSkeletonMap}
\begin{split}
\xymatrix{
1 \ar[r] & Y \ar[r] \ar[d]^-{=} & E^{\mathrm{an}} \ar[r] \ar[d]^-{\mathrm{trop}} & A^{\mathrm{an}} \ar[r] \ar[d]^-{\bar{\mathrm{trop}}} & 1 \\
1 \ar[r] & Y \ar[r] & X^*_{\R} \ar[r] & \Sigma  \ar[r] & 1.
}
\end{split}
\end{equation}
Moreover, $\bar{\mathrm{trop}}$ admits a splitting, \textit{i.e.} there exists an injective morphism $\iota\colon \Sigma \rightarrow A^{\mathrm{an}}$ such that $\bar{\mathrm{trop}}\circ \iota = \mathrm{id}_{\Sigma}$. So  $A^{\mathrm{an}} \rightarrow \Sigma$ can be seen as a contraction of $A^{\mathrm{an}}$ to the subset $\Sigma$.

Since $Y$ is a lattice in $X^*_{\R}$, the  bilinear map $Y \times Y \rightarrow \Z$, $(u'_1, u'_2)\mapsto [u'_1, \lambda_{T,L}^*(u'_2)]$ extends linearly to a bilinear map 
\begin{equation}\label{EqQuadraticFormSkeleton}
b \colon X^*_{\R} \times X^*_{\R} \rightarrow \R.
\end{equation}
The associated quadratic form $q_0 \colon X^*_{\R} \rightarrow \R$ is positive definite since $L$ is ample. 

\subsubsection{N\'eron components}
In practice, we are interested in the classical points $A(F)$, which form a subgroup of $A^{\mathrm{an}}$. The group of components of the special fiber of the N\'eron model of $A$ can be naturally identified with $\Phi := A(F)/\cA(R)$. 
%can be recovered as follows: For the identity component of the N\'eron model $\cA$ of $A$,  $\cA(R)$ is a subgroup of $A(F)$ of finite index, and therefore 
%\[
%\Phi = A(F)/\cA(R).
%\]
It is known that the natural group homomorphism $A(F) \subseteq A^{\mathrm{an}} \xrightarrow{\bar{\mathrm{trop}}} \Sigma = X^*_{\R}/Y$ has image $X^*/Y$ and kernel $\cA(R)$ (see \textit{e.g.} \cite[Lem.~7.1]{deJongShokriehDecomposition}), and hence induces a group isomorphism
\begin{equation}\label{EqNeronComponentsInSkeleton}
\Phi \cong X^*/Y.
\end{equation}
On the other hand, denote by $\det_Y(q_0)$ the determinant of the matrix of $q_0$ under a $\Z$-basis of $Y$; this determinant is independent of the choice of the basis. Then it is known that
\begin{equation}\label{EqMomentNonArch}
    \det\nolimits_Y(q_0) = \#\Phi \cdot [X: \lambda^*_{T,L}(Y)]
\end{equation}
where $\lambda^*_{T,L}$ is the map from \eqref{EqLambdaNonArch} defined be the polarization $L$ on $A$. When $L$ is a principal polarization, we have $[X: \lambda^*_{T,L}(Y)]=1$.

For Tate curve $F^{\times}/\mathfrak{q}^{\Z}$ with the canonical principal polarization, we have $\Phi = \mathrm{ord}_v(\mathfrak{q})$ and the skeleton is a circle $\cong \R/\Z$ of circumference  $\sqrt{\mathrm{ord}_v(\mathfrak{q})}$. Identify $Y$ with $\Z$ and $X^*_{\R}$ with $\R$ (with standard coordinate $x$), then $X^* = \frac{1}{\mathrm{ord}_v(\mathfrak{q})} \Z$ and  $q_0(x) = \mathrm{ord}_v(\mathfrak{q}) \cdot x^2$. 

%the group $\Phi$ is cyclic of order $\mathrm{ord}_v(\mathfrak{q})$. 

%\subsection{Superforms}

%TO DO: Define superforms on $X^*_{\R}$ as Lagerberg.

%Let $b$ be the bilinear pairing on $X^*_{\R}$ defined in \eqref{EqQuadraticFormSkeleton} with associated quadratic form  $q_0$. 
%Let $\{e_i\}_{i=1}^r$ be a $\Z$-basis of $X^*$, which is Minkowski reduced with respect to the quadratic form $q_0$ on $X^*_{\R}$. Use $x_1,\ldots,x_r$ to denote the corresponding coordinates. The following $(1,1)$-superform on $\Sigma$ will be heavily used in the next section:
%\begin{equation}\label{EqSuperformC1}
%\mathrm{d}'\mathrm{d}'' q_0 = \sum_{1\le i,j\le r} b(e_i,e_j) \mathrm{d}'x_i \wedge \mathrm{d}'' x_j.
%\end{equation}
%TO DO: Prove this equality!!!

%Descend to $\Sigma$ because translation invariant.

\section{Bump functions at non-archimedean places}\label{SectionBumpNonArch}

The goal of this section is to construct the desired bump functions at non-archimedean places. Let $v \in M_K^0$ be a non-archimedean place at which $A$ has semi-stable reduction. 
The construction is void if $A$ has good reduction at $v$. 

Here is the setup. Assume the toric part of $A\pmod v$ is split over the residue field $k_v$\footnote{This is always the case when $K=\C(B)$.} until §\ref{SubsectionNonArchNonSplit}. Then we can apply the discussion of §\ref{SectionNonArchPrep} to $A_v := A\otimes_K K_v$ and use the same notation (but we add a subscript $v$ to each object). 
In particular we have a tropicalization
\[
\bar{\mathrm{trop}}_v \colon A_v^{\mathrm{an}} \longrightarrow \Sigma_v := X^*_{v,\R}/Y_v
\]
to the skeleton $\Sigma_v$ which is a real torus of rank $r_v$; see \eqref{EqSkeletonMap}. Moreover, $\bar{\mathrm{trop}}_v(A(K_v))$ is contained in $\Phi_v = X^*_v/Y_v$, which is a finite subgroup of $\Sigma_v$ corresponding to the N\'eron components of $A_v$; see \eqref{EqNeronComponentsInSkeleton}. 

Let $N_v$ be from §\ref{SubsectionGeneralNotation}. Recall that in $\mathsection$\ref{SubsectionTropicalization}, all the $\log$'s are taken with base $N_v$.

Finally, let $\mathrm{d}\mu_v$ be the  canonical volume form on $\Sigma_v$, normalized such that $\int_{\Sigma_v}\mathrm{d}\mu_v = 1$. It is the descent of the standard Lebesgue measure on $X^*_{v,\R}$ divided by $[X_v^*:Y_v] = \#\Phi_v$.

\vskip 0.2em
Here is the main result of this section, which we prove in §\ref{SubsectionConstructionNonArch}-\ref{SubsectionPropNonArchIII}. In §\ref{SubsectionNonArchNonSplit}, we summarize our construction in terms of $f_v = f_{0,v}\circ \bar{\mathrm{trop}}_v \colon A_v^{\mathrm{an}} \longrightarrow \R$ and we also include the case where the toric part of $A\pmod v$ is not necessarily split over $k_v$.%; see §\ref{SubsectionNonArchNonSplit}.
\begin{prop}\label{PropBumpFctNonArchOrig}
%Assume $A$ has split semi-stable reduction. 
There exists a smooth function
\[
f_{0,v} \colon \Sigma_v = X^*_{v,\R}/Y_v \longrightarrow \R
\]
satisfying the following properties:
\begin{enumerate}
\item[(i)] $f_{0,v}|_{\Phi_v} = 0$, and in particular $(f_{0,v}\circ\bar{\mathrm{trop}}_v)(A(K_v)) = 0$;
\item[(ii)] the metric $N_v^{-\epsilon f_{0,v} \circ \bar{\mathrm{trop}}_v}\|\cdot\|_v$ on $L_v^{\mathrm{an}}$ is relatively semipositive for all $\epsilon \in [-1,1]$;
\item[(iii)] for the Minkowski constant $\gamma_{r_v} \ge 1$ from Lemma~\ref{LemmaReductionMinkowski}, we have
\begin{equation}\label{EqIntegralCLMeasureOrig}
\int_{\Sigma_v} f_{0,v} \mathrm{d}\mu_v \ge  \frac{r_v}{64e\gamma_{r_v}}(\#\Phi_v)^{-1/r_v}.
\end{equation}
\end{enumerate}
%Moreover, $f_{0,v}$ is $X^*_v$-periodic, \textit{i.e.}  $f_{0,v}(\bar{x}+\bar{u}) = f_{0,v}(\bar{x})$ for any $\bar{u} \in X^*_v/Y_v$ and $\bar{x} \in \Sigma_v$.
\end{prop}

\begin{rmk}
Let us give an explanation on part (ii). We take $N_v$ as the base of the exponential since all the $\log$'s are taken with this base in $\mathsection$\ref{SubsectionTropicalization}. 
We use S.~Zhang's \cite{ZhangSmPtAdelicMetric}  definition of relatively semipositive metrics on $L_v^{\mathrm{an}}$: it is a uniform limit of nef model metrics. Our proof uses the equivalent description in terms of supercurrents, which will be reviewed in $\mathsection$\ref{SubsectionSuperforms}.\footnote{To our knowledge, the positivity of supercurrents  is not known to imply Zhang's semipositivity in general. In our particular case, they are equivalent  by Gubler--Stadl\"oder \cite[Thm.~4.10]{GublerStefan_ToricMetricAV}.} In this terminology, condition (ii) becomes:
\begin{equation}\label{EqSupFormInequC11}
    c_1(L_v^{\mathrm{an}}, \|\cdot\|_v) + \epsilon \mathrm{d}'\mathrm{d}'' (f_{0,v} \circ \bar{\mathrm{trop}}) \ge 0 ,\quad\text{ for all }\epsilon\in[-1,1].
\end{equation}
\end{rmk}

In the proof of Proposition~\ref{PropBumpFctNonArchOrig}, we shall only work with the place $v$. So in §\ref{SubsectionConstructionNonArch}-\ref{SubsectionPropNonArchIII}, to ease notation we will use the notation from $\mathsection$\ref{SectionNonArchPrep} without adding the subscript $v$.

\subsection{Construction of $f_{0,v}$, proofs of (i) and (iii)}\label{SubsectionConstructionNonArch}
Recall from \eqref{EqQuadraticFormSkeleton} the bilinear form $b$ on $X^*_{\R}$ and the associated quadratic form $q_0$ which is positive-definite. 
An integral basis $e_1,\ldots,e_r$ of $X^*$ is called {\it Minkowski-reduced} with respect to $q_0$, if $q_0(e_j) \le q_0(w)$ for each $j \in \{1,\ldots,r\}$ and any $w = \sum_{j=1}^r w_j e_j \in X^*$ with $\mathrm{gcd}(w_1,\ldots,w_r)=1$. % says that $X^*$ has Minkowski-reduced bases with respect to $q_0$.

General reduction theory guarantees the existence of  a Minkowski-reduced basis $e_1,\ldots,e_r$ of $X^*$. Then the matrix of $q_0$ under this basis is an $r\times r$ Minkowski matrix $B=(b_{ij})_{1\le i,j\le r}$. Let $\gamma_r \ge 1$ be the Minkowski constant from Lemma~\ref{LemmaReductionMinkowski}. 

Now define the function
\[
F_0 \colon X^*_{\R} \longrightarrow \R, \qquad x= \sum_{j=1}^r x_j e_j \longmapsto \frac{1}{e\gamma_r} \sum_{j=1}^r b_{jj}\psi_0(x_j),
\]
where $\psi_0 \colon \R \rightarrow \R$, $t\mapsto \frac{1}{64}\left(1-\cos(2\pi t)\right)$. In particular, $\psi_0$ is smooth and $\Z$-periodic, $\psi_0(0) = 0$ and $|\psi_0''|_{\infty} \le 1$. Hence $F_0$ is $X^*$-periodic, and so descends to a smooth function
\[
f_{0,v} \colon \Sigma_v = X^*_{\R}/Y \longrightarrow \R
\]
 satisfying (i) %
 %and the ``Moreover'' part 
 of Proposition~\ref{PropBumpFctNonArchOrig}. Chasing the diagram \eqref{EqSkeletonMap}~yields
\begin{equation}\label{EqfvFtrop}
    f_{0,v} \circ \bar{\mathrm{trop}} \circ p = F_0\circ \mathrm{trop}
\end{equation}
for the uniformization $p \colon E^{\mathrm{an}} \rightarrow A_v^{\mathrm{an}}$.

Let us turn to (iii). 
 The standard Lebesgue measure on $X^*_{\R}$ is $\mathrm{d}x_1\wedge \cdots \wedge \mathrm{d}x_r$ for the usual differential operator $\mathrm{d}$. By the definition of $\mathrm{d}\mu_v$ (see above Proposition~\ref{PropBumpFctNonArchOrig}) and the definition of $f_{0,v}$, we have
\begin{align*}
\int_{\Sigma_v} f_{0,v} \mathrm{d}\mu_v &= \int_{[0,1]^r} F_0 \mathrm{d}x_1\wedge \cdots \wedge \mathrm{d}x_r  \\
&= \frac{1}{e\gamma_r}\sum_{j=1}^r\int_0^1 b_{jj}\psi_0(x_j)\mathrm dx_j \\
& = \frac{1}{e\gamma_r } \left(\int_0^1 \psi_0(s) \mathrm{d}s\right) \sum_{j=1}^r b_{jj} \ge \frac{1}{64e\gamma_r} \cdot  r \left(\prod_{j=1}^r b_{jj}\right)^{1/r} \quad \text{ by \eqref{EqIntegralAux}}.
\end{align*}
Hence Lemma~\ref{LemmaReductionMinkowski}.(iii) yields
\[
\int_{\Sigma_v} f_{0,v} \mathrm{d}\mu_v \ge \frac{r}{64e\gamma_r}  \left(\det B\right)^{1/r}.
\]
 The matrix $B$ is the matrix of $q_0$ under a $\Z$-basis of $X^*$. 
So $\det B =  \frac{1}{[X^*:Y]^2} \det_Y(q_0) = \frac{\det_Y(q_0)}{(\#\Phi_v)^2}$, and hence $\det B \ge (\#\Phi_v)^{-1}$ by \eqref{EqMomentNonArch}.  This establishes (iii).
 \qed

\subsection{A quick review on supercurrents}\label{SubsectionSuperforms}
Our proof of Proposition~\ref{PropBumpFctNonArchOrig}.(ii) relies on Lagerberg's supercurrents \cite{Lagergberg} on $\R^n$ and its generalization to non-archimedean Berkovich spaces by Chambert-Loir--Ducros \cite{CLD}. For our purpose we use the variant of Gubler--K\"unnemann \cite{GuKu_deltaform}. % laid the foundation of using superforms and supercurrents Berkovich spaces and hence Arakelov Geometry. 
We briefly recall this theory with a focus on our application. %For notation, we will use the ones of \cite{} and \cite{}.

For any $p,q\ge 0$, the space of {\it (p,q)-superforms} on $X^*_{\R}$, denoted by $A^{p,q}(X^*_{\R})$, is the vector space $C^{\infty}(X^*_{\R})\otimes\left(\bigwedge ^p X^*_{\R}  \otimes \bigwedge^q X^*_{\R} \right)$. For any smooth function $x$ on $X^*_{\R}$, we have a  $(1,0)$-superform $\mathrm{d}'x:=x\otimes 1$ and a $(0,1)$-form $\mathrm{d}''x:=1\otimes x$. 

Use $(x_1,\ldots,x_r)$ to denote the coordinate functions under the basis $\{e_1,\ldots,e_r\}$ of $X^*_{\R}$. Then 
any $(p,q)$-superform can be written uniquely as
\begin{equation}\label{EqSuperformBasis}
\sum\nolimits_{1\le i_1<\ldots<i_p\le r,~ 1\le j_1<\ldots<j_q \le r} f_{i_1,\ldots,i_p,j_1,\ldots,j_q} \mathrm{d}'x_{i_1}\wedge \cdots \wedge \mathrm{d}'x_{i_p} \wedge \mathrm{d}''x_{j_1} \wedge \cdots \wedge \mathrm{d}''x_{j_q} .
\end{equation}
with $ f_{i_1,\ldots,i_p,j_1,\ldots,j_q} $ smooth functions on $X^*_{\R}$. 
%a $C^{\infty}(X^*_{\R})$-basis of $A^{p,q}(X^*_{\R})$ is
%\begin{equation}\label{EqSuperformBasis}
%\left\{ \mathrm{d}'x_{i_1}\wedge \cdots \wedge \mathrm{d}'x_{i_p} \wedge \mathrm{d}''x_{j_1} \wedge \cdots \wedge \mathrm{d}''x_{j_q} \right\}_{1\le i_1<\ldots<i_p\le r,~ 1\le j_1<\ldots<j_q \le r}.
%\end{equation}
We have two natural operators
\[
\mathrm{d}' \colon A^{p,q}(X^*_{\R}) \rightarrow A^{p+1,q}(X^*_{\R}), \qquad \mathrm{d}'' \colon A^{p,q}(X^*_{\R}) \rightarrow A^{p,q+1}(X^*_{\R})
\]
which are analogues of $\partial$ and $\bar{\partial}$ in  complex geometry. In particular, we have % for any smooth function $h \colon X^*_{\R} \rightarrow \R$, we have
\begin{equation}\label{EqddfNonArch}
   \mathrm{d}'\mathrm{d}'' h = \sum\nolimits_{1\le i,j\le r} \frac{\partial^2 h}{\partial x_i \partial x_j} \mathrm{d}' x_i \wedge \mathrm{d}'' x_j, \quad \text{ for any smooth function } h \colon X^*_{\R} \to \R.
\end{equation}
We mostly work with the following {\it symmetric} superforms. The automorphism of~algebras
\[
J \colon \bigoplus\nolimits_{p,q} A^{p,q}(X^*_{\R}) \longrightarrow \bigoplus\nolimits_{p,q} A^{p,q}(X^*_{\R}),
\]
defined by $J(\mathrm{d}'x) = \mathrm{d}'' x$ and $J(\mathrm{d}''x) = -\mathrm{d}'x$,  sends $(p,q)$-superforms to $(q,p)$-superforms. A $(p,p)$-superform $\alpha \in A^{p,p}(X^*_{\R})$ is called {\it symmetric} 
%(resp. {\it anti-symmetric}) 
if $J(\alpha) = (-1)^p \alpha$.% (resp. if $J(\alpha) = (-1)^{p+1}\alpha$).

An $(r,r)$-superform $f \mathrm{d}'x_1\wedge\cdots\wedge\mathrm{d}'x_r \wedge \mathrm{d}''x_1\wedge\cdots\wedge\mathrm{d}''x_r$ is called {\it positive} if $f \ge 0$ everywhere on $X^*_{\R}$. A symmetric $(1,1)$-superform $\alpha \in A^{1,1}(X^*_{\R})$ is called {\it positive} if one of the following equivalent conditions holds (see \cite[Lem.~2.2 and Prop.~2.1]{Lagergberg}):
\begin{enumerate}
    \item[(a)] $\alpha\wedge \alpha_1\wedge J(\alpha_1)\wedge \cdots \wedge \alpha_{r-1}\wedge J(\alpha_{r-1})$ is a positive $(r,r)$-superform for all $\alpha_1,\ldots,\alpha_{r-1} \in A^{1,0}(X^*_{\R})$,
    \item[(b)] $(-1)^{(r-1)(r-2)/2} \alpha \wedge \beta \wedge J(\beta)$ is a positive $(r,r)$-superform for every $\beta \in A^{r-1,0}(X^*_{\R})$,
%$(r-1,0)$-superform $\beta$,
    \item[(c)] $\alpha = \sum_{j=1}^r f_j \mathrm{d}'y_j \wedge \mathrm{d}''y_j$ for a basis $\{y_1,\ldots,y_r\}$ of $X^*_{\R}$ and  non-negative  functions~$f_j$.
    \item[(c')] $\alpha = \sum_{1\le i,j\le r} f_{ij} \mathrm{d}'x_i \wedge \mathrm{d}''x_j$ such that the symmetric matrix of smooth functions $(f_{ij})_{1\le i,j\le r}$ is semi-positive definite pointwise.
\end{enumerate}

As in complex geometry, one can also define the space $D_{p,q}(X^*_{\R})$ of  {\it $(p,q)$-supercurrents} on $X^*_{\R}$  to be the dual of $A^{p,q}_{\mathrm{c}}(X^*_{\R})$ -- the space of $(p,q)$-superforms with compact support -- under a suitable topology. Then each $(p,q)$-supercurrent can be uniquely written as \eqref{EqSuperformBasis} with $f_{i_1,\ldots,i_p,j_1,\ldots,j_q}$ a {\it distribution} on $X^*_{\R}$. The differential operators $\mathrm{d}'$, $\mathrm{d}''$ can be extended to supercurrents, and we can also define {\it symmetric} $(p,p)$-supercurrents. 

Write $\langle T, \alpha\rangle$ for the natural pairing of $T \in D_{p,q}(X^*_{\R})$ and $\alpha \in A^{p,q}_{\mathrm{c}}(X^*_{\R})$. 
A symmetric $(1,1)$-supercurrent $T$ is called {\it positive}, if $\langle T, \alpha \rangle \ge 0$ for any positive $\alpha \in A^{1,1}_{\mathrm{c}}(X^*_{\R})$. An $(r,r)$-supercurrent $f \mathrm{d}'x_1\wedge\cdots\wedge\mathrm{d}'x_r \wedge \mathrm{d}''x_1\wedge\cdots\wedge\mathrm{d}''x_r$ is said to be {\it positive} if the distribution $f$ is non-negative everywhere.

For any $\alpha, \beta$ in  $D_{1,1}(X^*_{\R})$ (or $ D_{r,r}(X^*_{\R})$),  use $\alpha \ge \beta$ to denote that $\alpha -\beta$ is positive.% In particular, $\alpha \ge 0$ means that $\alpha$ is positive. 

The discussion above is generalized to $A_v^{\mathrm{an}}$ \cite{CLD}. In our paper we use the variant \cite{GuKu_deltaform}. The polarization $L$ on $A$ defines a {\it positive} $(1,1)$-supercurrent $c_1(L,\|\cdot\|_v)$ on $A_v^{\mathrm{an}}$, and $c_1(L,\|\cdot\|_v)^{\wedge g}$ is a Chambert-Loir measure on $A_v^{\mathrm{an}}$ \cite{Gubler2010NonArchimedean}. In particular,
\begin{equation}\label{EqIntegralCLMeasure1}
\int_{A_v^{\mathrm{an}}} (f_{0,v}\circ \bar{\mathrm{trop}})  \frac{1}{\deg_L(A)}c_1(L_v^{\mathrm{an}},\|\cdot\|_v)^{\wedge g} = \int_{\Sigma_v}f_{0,v} \mathrm{d}\mu_v;
\end{equation}
see for example \cite[Cor.~19.2.8]{CLD}. 
We will rewrite Proposition~\ref{PropBumpFctNonArchOrig}.(iii) using \eqref{EqIntegralCLMeasure1}.
%\begin{equation}\label{EqIntegralCLMeasure1}
%\int_{A_v^{\mathrm{an}}} (f_{0,v}\circ \bar{\mathrm{trop}})  \frac{1}{\deg_L(A)}c_1(L_v^{\mathrm{an}},\|\cdot\|_v)^{\wedge g} \ge  \frac{r_v}{64e\gamma_{r_v}}(\#\Phi_v)^{-1/r_v}.
%\end{equation}

%There are also notions of positivity for symmetric $(p,p)$-forms for every $p$, which are suitable generalizations of the conditions above. However, the generalizations of (a), (b), (c) are in general not equivalent for $2\le p \le r-2$. These generalizations are called weakly positive, positive, strongly positive respectively. 

\subsection{Proof of Proposition~\ref{PropBumpFctNonArchOrig}.(ii)}\label{SubsectionPropNonArchIII}
Write $\mathbf{x} = (x_1,\ldots,x_r)$ for the coordinates on $X^*_{\R}$ under our Minkowski-reduced basis $\{e_1,\ldots,e_r\}$ of $X^*$. Recall $q_0 \colon X^*_{\R} \rightarrow \R$, $\mathbf{x} \mapsto \mathbf{x}^{\mathrm{t}} B \mathbf{x}$.

%The $(1,1)$-supercurrent $c_1(L,\|\cdot\|_v)$ was computed by Gubler--K\"unnemann \cite[Example~8.15]{GuKu_deltaform} as follows. 
We use the following computation of Gubler--K\"unnemann \cite[Eg.~8.15]{GuKu_deltaform}: 
The pullback $p^*\|\cdot\|_v$ under the uniformization  $p \colon E^{\mathrm{an}} \rightarrow A_v^{\mathrm{an}}$ from \eqref{EqRaynaudCross} satisfies
\begin{equation}\label{EqGK}
p^*\|\cdot\|_v = N_v^{-q_0 \circ \mathrm{trop}} \cdot q^*\|\cdot\|_M; % \|\cdot\|_{\mathrm{pl}}
\end{equation}
%for the piecewise linear metric  $\|\cdot\|_{\mathrm{pl}} = 
here $q^*\|\cdot\|_M$ comes for the natural isomorphism $p^*L^{\mathrm{an}} = q^*M^{\mathrm{an}}$ from \eqref{EqIsomLanManPullbackOnE}, with $M$ an ample line bundle on $B$ (which is an abelian variety over $K_v$ with good reduction) and $\|\cdot\|_M$ is the model metric. 
In particular, $c_1(M^{\mathrm{an}},\|\cdot\|_M) \ge 0$.

We wish to invoke the semi-positivity criterion  \cite[Thm.~4.10]{GublerStefan_ToricMetricAV} to the smooth function $f_{\|\cdot\|} := q_0+\epsilon F_0$ on $X^*_{\R}$. We have a direct computation
\begin{align}\label{EqNonArchLong}
\begin{split}
\mathrm{d}'\mathrm{d}''f_{\|\cdot\|}  & = \mathrm{d}'\mathrm{d}'' (\mathbf{x}^{\mathrm{t}} B \mathbf{x}) + \epsilon \mathrm{d}'\mathrm{d}'' F_0 \\
& = (\mathrm{d}'\mathbf{x})^{\mathrm{t}} \wedge B \mathrm{d}''\mathbf{x} + \epsilon \frac{1}{e\gamma_r} \sum\nolimits^r_{j=1}b_{jj}\psi_0''(x_j)\mathrm{d}'x_j \wedge \mathrm{d}''x_j  \quad \text{ by }\eqref{EqddfNonArch} 	 \\
& \ge (\mathrm{d}'\mathbf{x})^{\mathrm{t}} \wedge \left(B - \frac{1}{e\gamma_r}\mathrm{diag}(b_{11},\ldots,b_{rr}) \right) \mathrm{d}''\mathbf{x}		 \qquad \text{ since }|\psi_0''|_\infty\leq 1.
\end{split}
\end{align}
%Therefore $c_1(p^*L^{\mathrm{an}},\|\cdot\|_{\mathrm{sm}}) \ge 0$. So
%\begin{align*}
%p^*c_1(L^{\mathrm{an}},\|\cdot\|_v) & = \mathrm{trop}^* \big( \mathrm{d}'\mathrm{d}'' (\mathbf{x}^{\mathrm{t}} B \mathbf{x})\big) + c_1(p^*L^{\mathrm{an}},\|\cdot\|_{\mathrm{sm}}) 
 %\ge  \mathrm{trop}^* \big(  (\mathrm{d}'\mathbf{x})^{\mathrm{t}} \wedge B \mathrm{d}''\mathbf{x} \big).
%\end{align*}
%Therefore 
%\begin{align*}
%p^*c_1(L^{\mathrm{an}}, N_v^{-\epsilon f_{0,v} \circ \bar{\mathrm{trop}}} \|\cdot\|_v) & = p^*c_1(L^{\mathrm{an}},\|\cdot\|_v) + \mathrm{trop}^*\left(\epsilon \mathrm{d}'\mathrm{d}'' F\right) \qquad\text{ by }\eqref{EqfvFtrop} \\
%& \ge  \mathrm{trop}^*\!\left( (\mathrm{d}'\mathbf{x})^{\mathrm{t}} \wedge B \mathrm{d}''\mathbf{x} + \epsilon \frac{1}{e\gamma_r} \sum^r_{j=1}b_{jj}\psi_0''(x_j)\mathrm{d}'x_j \wedge \mathrm{d}''x_j \right) \\
%& \ge \mathrm{trop}^*\!\left( (\mathrm{d}'\mathbf{x})^{\mathrm{t}} \wedge \left(B+\epsilon \frac{1}{e\gamma_r}\mathrm{diag}(b_{11},\ldots,b_{rr}) \right) \mathrm{d}''\mathbf{x} \right),
%\end{align*}
%where the last inequality holds true because $|\psi_0''|_\infty\leq 1$. 
So $\mathrm{d}'\mathrm{d}''f_{\|\cdot\|} \ge 0$ by Lemma~\ref{LemmaReductionMinkowski}.(iv). Thus $f_{\|\cdot\|}$ is convex for all $\epsilon\in [-1,1]$. 
%So $c_1(L^{\mathrm{an}}, N_v^{-\epsilon f_{0,v} \circ \bar{\mathrm{trop}}} \|\cdot\|_v) \ge 0$ for all $\epsilon \in [-1,1]$ by Lemma~\ref{LemmaReductionMinkowski}.(ii).  
This establishes Proposition~\ref{PropBumpFctNonArchOrig}.(ii) by \cite[Thm.~4.10]{GublerStefan_ToricMetricAV} (or \cite[Prop.~8.3.1]{BGJK}). 
\qed

At this stage we can also prove \eqref{EqSupFormInequC11}, because \eqref{EqfvFtrop} and the discussion above yield $
p^*c_1(L^{\mathrm{an}}, N_v^{-\epsilon f_{0,v} \circ \bar{\mathrm{trop}}} \|\cdot\|_v)  =  \mathrm{trop}^* \mathrm{d}'\mathrm{d}''f_{\|\cdot\|} + q^*c_1(M,\|\cdot\|_M) \ge 0$ for all $\epsilon\in [-1,1]$.

\subsection{General semi-stable reduction}\label{SubsectionNonArchNonSplit}
Assume $A$ semi-stable reduction at $v$. We rewrite Proposition~\ref{PropBumpFctNonArchOrig} without assuming the toric part of $A\pmod v$ to be split over $k_v$.

    Let $F'$ be an unramified finite Galois extension of $F:=K_v$ such that the toric part of $A\pmod v$ is split over the residue field of $F'$. General theory of Berkovich spaces asserts that $A_v^{\an} =(A/F)^{\an}=\frac{(A/F')^{\an}}{\Gal(F'/F)}$. The whole discussion in §\ref{SectionNonArchPrep} applies to $A_{F'}$, and we will write all data with subscript $F'$. For the function $f_{0,F'}$ constructed in Proposition~\ref{PropBumpFctNonArchOrig} for $A_{F'}$, the average of $f_{0,F'}\circ \bar{\mathrm{trop}}_{F'} \colon (A/F')^{\mathrm{an}} \rightarrow \R$ under $\Gal(F'/F)$ descends to a function
\begin{equation}\label{EqBumpFctNonArch}
    f_v \colon A_v^{\mathrm{an}} = (A/F)^{\mathrm{an}} \longrightarrow \R.
\end{equation}
Recall our notation that $\#\Phi_v = \#\Phi_v(\bar{k}_v)$ is the order of the group scheme $\Phi_v$ over $k_v$. So Proposition~\ref{PropBumpFctNonArchOrig} and \eqref{EqSupFormInequC11} with $\epsilon = \pm 1$ imply: (to get \eqref{EqIntegralCLMeasure} we combine \eqref{EqIntegralCLMeasureOrig} with \eqref{EqIntegralCLMeasure1})
\begin{propbis}{PropBumpFctNonArchOrig}\label{PropBumpFctNonArch}
The function $f_v \colon A_v^{\mathrm{an}} \longrightarrow \R$ satisfies:
    \begin{enumerate}
    \item[(i)] $f_v(A(K_v)) = 0$,
    \item[(ii)] the metric $N_v^{-\epsilon f_v}\|\cdot\|_v$ on $L_v^{\mathrm{an}}$ is relatively semipositive for all $\epsilon\in [-1,1]$, and 
    \begin{equation}\label{EqSupFormInequC1}
        -c_1(L_v^{\mathrm{an}},\|\cdot\|_v) \le  \mathrm{d}'\mathrm{d}'' f_v \le c_1(L_v^{\mathrm{an}}, \|\cdot\|_v), %, \quad \text{ for all }\epsilon\in [-1,1],
    \end{equation}
    \item[(iii)] for the Minkowski constant $\gamma_{r_v} \ge 1$ from Lemma~\ref{LemmaReductionMinkowski}, we have
     \begin{equation}\label{EqIntegralCLMeasure}
        \int_{A_v^{\mathrm{an}}} f_v \frac{1}{\deg_L(A)}c_1(L_v^{\mathrm{an}},\|\cdot\|_v)^{\wedge g} \ge \frac{r_v}{64e\gamma_{r_v}} (\#\Phi_v)^{-1/r_v}.
    \end{equation}
\end{enumerate}
\end{propbis} 
We close this section by the following  non-archimedean version of Lemma~\ref{LemmaComparison11Form}, for which we use some facts on  {\it $\delta$-forms} defined in \cite[§4]{GuKu_deltaform}. Both $c_1(L,\|\cdot\|_v)$ and $\mathrm{d}'\mathrm{d}'' f_v$ are $\delta$-forms; see \cite[Eg.~9.17, Prop.~9.11, Defn.~9.12]{GuKu_deltaform}. So we can take their wedge products.
\begin{lemma}\label{LemmaSuperformComparison}
$
c_1(L,\|\cdot\|_v)^{\wedge (g-k)} \wedge (\mathrm{d}'\mathrm{d}'' f_v)^{\wedge k} \le c_1(L,\|\cdot\|_v)^{\wedge g}$ for all $k \in \{0,\ldots,g\}$. %, \quad\text{ for any }k \in \{0,\ldots,g\}.
\end{lemma}
\begin{proof} 
It suffices to check over $F'$ and after pullback by $p \colon E^{\an}\to A_v^{\an}$.  
Now \eqref{EqGK} implies $p^*c_1(L,\|\cdot\|_v)=\mathrm{trop}^*\mathrm{d'd''}(\mathbf{x}^{\mathrm{t}} B \mathbf{x})+q^*c_1(M^{\an},\|\cdot\|_M)$. Set $\alpha := \mathrm{d'd''}(\mathbf{x}^{\mathrm{t}} B \mathbf{x})$ and $\gamma := q^*c_1(M^{\an},\|\cdot\|_M)$. Set $\beta := \mathrm{d'd''}F_0 = p^*(\mathrm{d}'\mathrm{d}'' f_{0,v})$. The computation of \eqref{EqNonArchLong} shows that $-\alpha \le \beta \le \alpha$ on $X^*_{\R}$ as $(1,1)$-superforms. As $\alpha$ is positive everywhere, a verbalized copy of the proof of Lemma~\ref{LemmaComparison11Form} implies  $\beta^{\wedge k}\leq \alpha^{\wedge k}$. So we are done by the following computation (which uses positivity properties of $\delta$-forms  \cite[Prop.~2.9, Lem.~6.9, Prop.~2.13]{GublerKunnemann2019})
\[
(\mathrm{trop}^*\alpha + \gamma)^{\wedge (g-k)} \wedge \mathrm{trop}^*\beta^{\wedge k} \le (\mathrm{trop}^*\alpha + \gamma)^{\wedge (g-k)} \wedge \mathrm{trop}^*\alpha^{\wedge k} \le (\mathrm{trop}^*\alpha + \gamma)^{\wedge g}.		\qedhere% \wedge \mathrm{trop}^*\beta^{\wedge k}
\]
%where the last inequality holds since $\gamma \ge 0$. We are done.
\end{proof}

%We can be more precise: both $c_1(L,\|\cdot\|_v)$ and $\mathrm{d}'\mathrm{d}'' f_v$ are $\delta$-forms in the sense of \cite[§4]{GuKu_deltaform}; see \cite[Eg.~9.17, Prop.~9.11, Defn.~9.12]{GuKu_deltaform}. So we can take wedge products.

\section{Statement of a first bound on the number of  rational torsion points}\label{SectionFirstBoundStatement}
Assume $A/K$ have semi-stable reduction. 
The goal of this section is the state a bound on the number of rational torsion points $A(K)_{\mathrm{tor}} =  \{ x\in A(K): \hat{h}_L(x) = 0\}$ (resp. of the number of small rational points). The proof will be executed in §\ref{SectionProofOfFirstBound}.% In the whole section, we make the following assumption:

\subsection{Statements over function fields}
These bounds are easier to state in the function field case $K= \C(B)$.  {\it Assume furthermore $A/K$ does not have good reduction everywhere.}
\begin{thm}\label{ThmFirstBoundFF}
There exists a proper abelian subvariety $A'$ of $A$ defined over $K$ such that
\begin{equation}\label{EqFirstBoundFF}
 \left( \#\!\left(\frac{A(K)_{\mathrm{tor}}+A'}{A'}\right) \cdot \frac{h^0(A',L)}{h^0(A,L)} \right)^{\frac{1}{g-\dim A'}} \le 512e \cdot g^2(g+2)\gamma_g\frac{[K:\C(\mathbb{P}^1)] \cdot h_{\mathrm{Fal}}(A) }{\sum_{v\in \mathfrak{Bad}(A/K)}  (\#\Phi_v)^{-1/r_v}}.
\end{equation}
\end{thm}
In practice, we will prove the following equivalent version of Theorem~\ref{ThmFirstBoundFF}. 
For each abelian subvariety $A'$ of $A_{\bar{K}}$, define $\mathfrak{N}(A')$ to be the number on the left hand side of \eqref{EqFirstBoundFF}. 
Gaudron--R\'{e}mond \cite[Lem.~6.1]{GR25} proved\footnote{Their proof is executed when $K$ is a number field, but holds verbally for $K= \C(B)$.} 
\[
\min_{A'\text{ proper abelian subvarieties of }A\text{ defined over }K} \mathfrak{N}(A') = \min_{A'\text{ proper abelian subvarieties of }A_{\bar{K}}} \mathfrak{N}(A').
\]
Call this number $\mathfrak{N}$. Then Theorem~\ref{ThmFirstBoundFF} is equivalent to
\begin{equation}\label{EqFirstBoundEqivFF}
\mathfrak{N} \le  512e\cdot g^2(g+2)\gamma_g\frac{[K:\C(\mathbb{P}^1)] \cdot h_{\mathrm{Fal}}(A) }{\sum_{v\in \mathfrak{Bad}(A/K)}  (\#\Phi_v)^{-1/r_v}}.
\end{equation}

\vskip 0.2em

We can also handle small rational points. Indeed, let
\begin{equation}\label{EqGammaPrimeFF}
\Gamma' := \left\{x \in A(K) : \hat{h}_L(x) \le \frac{1}{2024 e\cdot g^2  \gamma_g [K:\C(\mathbb{P}^1)]} \sum_{v \in \mathfrak{Bad}(A/K)} (\#\Phi_v)^{-1/r_v} \right\},
\end{equation}
\begin{thm}\label{ThmFirstBoundPrimeFF}
There exists a proper abelian subvariety $A'$ of $A$ defined over $K$ such that
%\begin{equation}\label{EqFirstBoundPrimeFF}
\[
 \left( \#\!\left(\frac{\Gamma'+A'}{A'}\right) \cdot \frac{h^0(A',L)}{h^0(A,L)} \right)^{\frac{1}{g-\dim A'}} \le 1024 e\cdot g^2(g+2)\gamma_g\frac{[K:\C(\mathbb{P}^1)] \cdot h_{\mathrm{Fal}}(A) }{\sum_{v\in \mathfrak{Bad}(A/K)}  (\#\Phi_v)^{-1/r_v}}.
\]
%\end{equation}
\end{thm}
As for Theorem~\ref{ThmFirstBoundFF}, we have an equivalent version of Theorem~\ref{ThmFirstBoundPrimeFF}: Define
\[
\mathfrak{N}' := \min_{A'\text{ proper abelian subvarieties of }A_{\bar{K}}}  \left( \#\!\left(\frac{\Gamma'+A'}{A'}\right) \cdot \frac{h^0(A',L)}{h^0(A,L)} \right)^{\frac{1}{g-\dim A'}}.
\]
Then in defining $\mathfrak{N}'$, it suffices to let $A'$ run over all proper abelian subvarieties of $A$ defined over $K$ by Gaudron--R\'{e}mond \cite[Lem.~6.1]{GR25}. So Theorem~\ref{ThmFirstBoundPrimeFF} is equivalent to
\begin{equation}\label{EqFirstBoundPrimeEqivFF}
\mathfrak{N}' \le  1024e\cdot g^2(g+2)\gamma_g \frac{[K:\C(\mathbb{P}^1)] \cdot h_{\mathrm{Fal}}(A) }{\sum_{v\in \mathfrak{Bad}(A/K)}  (\#\Phi_v)^{-1/r_v}}.
\end{equation}

\subsection{Division of $A(K)_{\mathrm{tor}}$ in the number field case}\label{SubsectionDivision}
Let $K$ be a number field. 
The statements of the corresponding upper bounds  are then more complicated. We need to decompose $A(K)_{\mathrm{tor}}$ into $(2g)^g(g+2)^g$ subsets and will prove a bound for each one of them. Let us explain this division now. 

Use the notation from §\ref{SectionArch}. 
In particular for each  $\sigma' \colon K \hookrightarrow \C$ and each isogeny $\rho_{\sigma'}\colon A_{\sigma'} \rightarrow A_{0,\sigma'}$ to a principally polarized abelian variety $(A_{0,\sigma'}, L_{0,\sigma'})$ with $\rho_{\sigma'}^*L_{0,\sigma'}\cong L_{\sigma'}$, we have a Siegel reduced matrix $\tau_{0,\sigma'}$ which is the period for $A_{0,\sigma'}$ given by $L_{0,\sigma'}$.

Fix a pair $(\sigma, \rho_\sigma)$ such that $\mathrm{Tr}(\mathrm{Im}\tau_{0,\sigma})$ is maximal among all embeddings $K \hookrightarrow \C$ and all such $A_{\sigma'} \rightarrow A_{0,\sigma'}$; in particular $\alpha_{\sigma}(A,L) \le \mathrm{Tr}(\mathrm{Im}\tau_{0,\sigma})$ by the definition \eqref{EqAlphaVIntro} of $\alpha_{\sigma}(A,L)$. We have the uniformization $u_{0,\sigma} \colon \C^g \rightarrow A_{0,\sigma}(\C)$, and the Betti coordinates $\beta_{\sigma} \colon \R^g \times \R^g \rightarrow \C^g$ sending $(\mathbf{a},\mathbf{b}) \mapsto \mathbf{a}+\tau_{0,\sigma}\mathbf{b}$; see \eqref{EqIsogenyArch}, \eqref{EqArchUnif} and \eqref{EqBetti}.

Consider all the vectors in $[0,1)^g$ whose coordinates are integer multiples of $1/2g(g+2)$. There are $N_g:=(2g)^g(g+2)^g$ such vectors, which we name by $\mathbf{e}_1,\ldots,\mathbf{e}_{N_g}$. Then  $A_{0,\sigma}(\C)$ can be decomposed into the disjoint union of
\begin{equation}\label{EqArchStrip}
\square_{j} := (u_{0,\sigma} \circ \beta_{\sigma})\left([0,1]^g \times \Big([\frac{1}{4g(g+2)}, \frac{3}{4g(g+2)})^g + \mathbf{e}_j \Big) \right) \subseteq A_{0,\sigma}(\C).
\end{equation}
%for $\beta_{\sigma}$ from \eqref{EqBetti} and $u_{\sigma}$ from above \eqref{EqBetti}.

Later on we need to relate $\square_j$ to the wider strips defined in \eqref{EqStripeForVanishing}. For any $i \in \{0,\ldots, g\}$, it is not hard to check that the sum (by convention, the sum of $0$-copies of $\square_j$ is $\{0\}$)
\[
\underbrace{\square_j+\cdots+\square_j}_{i\text{-copies}} := \{x_1+\cdots+x_i : x_1,\ldots,x_i \in \square_j\}
\]
is contained in a subset $\square'_{\sigma,q}$ in $A_{0,\sigma}(\C)$ of the form \eqref{EqStripeForVanishing}. We thus obtain at most $g+1$ strips of $A_{0,\sigma}(\C)$ of the form \eqref{EqStripeForVanishing}. Let their union be $\square'_{\sigma}$.

But in fact, we can do better. It is not hard to show that
\begin{equation}\label{EqBetterArch}
\underbrace{\square_j+\cdots+\square_j}_{i\text{-copies}} \subseteq (u_{0,\sigma} \circ \beta_{\sigma})\left([0,1]^g \times \Big([\frac{1}{4(g+2)}, \frac{3}{4(g+2)})^g + \mathbf{e}'_k \Big) \right) \subsetneq \square'_{\sigma,q}.
\end{equation}
This finer observation will be needed for the archimedean estimates later.

\subsection{Statements of the bounds for number fields}\label{SubsectionStatementFirstBoundNF}
Now we are ready to state the bounds for the number field case. Below let $K$ be a number field. Use the notation $N_v$ from $\mathsection$\ref{SubsectionGeneralNotation}. %Let $K_0$ be $\Q$ in the number field case and be $\C(\mathbb{P}^1)$ in the function field case.
Let 
\[
M'_K:=\mathfrak{Bad}(A/K) \bigcup\{\sigma\},
\]
 where $\sigma \colon K\hookrightarrow \C$ is from $\mathsection$\ref{SubsectionDivision}. 
 Define
\begin{equation}\label{EqDefntildeIA}
\tilde{I}_v(A) := \begin{cases} (\#\Phi_v)^{-1/r_v} & \text{ if }v \in \mathfrak{Bad}(A/K) \\  \mathrm{Tr}(\mathrm{Im}\tau_{0,\sigma}) \ge  \alpha_{\sigma}(A,L) & \text{ if }v=\sigma.\end{cases} 
\end{equation}

\subsubsection{On rational torsion points} 
For each $j \in \{1,\ldots, (2g)^g(g+2)^g\}$, define 
\begin{equation}\label{EqGammaj}
\Gamma_j := \{x \in A(K)_{\mathrm{tor}} : x_{\sigma} \in \rho_{\sigma}^{-1}(\square_{j})\},
%\Gamma_j := \begin{cases} \{x \in A(K)_{\mathrm{tor}} : x_{\sigma} \in \square_{j}\} & \text{ if }K\text{ is a number field} \\ A(K)_{\mathrm{tor}} & \text{ if }K\text{ is a function field}. \end{cases}
\end{equation}
where we use the division and notation from $\mathsection$\ref{SubsectionDivision}. % in the number field case and be $\mathfrak{Bad}(A/K) $ in the function field case.

\begin{thmbis}{ThmFirstBoundFF}\label{ThmFirstBound}
    For any $g \in \Z_{\ge 1}$, there  exists an explicit constant $C_1 = C_1(g)>0$ defined in \eqref{EqFormulaC1} such that:  
For each $j \in \{1,\ldots,(2g)^g(g+2)^g\}$, there exists a proper abelian subvariety $A'$ of $A$, defined over $K$, such that
\begin{equation}\label{EqFirstBound}
 \left( \#\!\left(\frac{\Gamma_j+A'}{A'}\right) \cdot \frac{h^0(A',L)}{h^0(A,L)} \right)^{\frac{1}{g-\dim A'}} \le   C_1 \frac{ \max\{h_{\mathrm{Fal}}(A) ,1\} +\log h^0(A,L) + \log[K:\Q]}{ \frac{1}{[K:\Q]}\sum_{v\in M'_K}   \tilde{I}_v(A) \log N_v }.
\end{equation}
\end{thmbis}

For the convenience of the proof, we will make use of the following notations. 
For each abelian subvariety $A'$ of $A_{\bar{K}}$, define $\mathfrak{N}_j(A')$ to be the number on the left hand side of \eqref{EqFirstBound}. 
%\[
%\mathfrak{N}_j(B):= \left( \#\left(\frac{\Sigma_j+B}{B}\right) \cdot \frac{h^0(B,L)}{h^0(A,L)} \right)^{1/(g - \dim B)}.
%\]
Gaudron--R\'{e}mond \cite[Lem.~6.1]{GR25} proved 
\[
\min_{A'\text{ proper abelian subvarieties of }A\text{ defined over }K} \mathfrak{N}_j(A') = \min_{A'\text{ proper abelian subvarieties of }A_{\bar{K}}} \mathfrak{N}_j(A').
\]
Call this number $\mathfrak{N}_j$. Then the conclusion of Theorem~\ref{ThmFirstBound} is equivalent to: For each $j \in \{1,\ldots,(2g)^g(g+2)^g\}$, we have
\begin{equation}\label{EqFirstBoundEqiv}
\mathfrak{N}_j \le   C_1 \frac{ \max\{h_{\mathrm{Fal}}(A) ,1\} +\log h^0(A,L) + \log[K:\Q]}{ \frac{1}{[K:\Q]}\sum_{v\in M'_K}   \tilde{I}_v(A) \log N_v }.
\end{equation}

\subsubsection{On small rational points}\label{SubsectionSmallRationalPoints}
We can also handle small rational points. Indeed, for each $j \in \{1,\ldots,(2g)^g(g+2)^g\}$, let
\begingroup
\let\small\scriptsize
\begin{equation}\label{EqGammaPrime}
\Gamma'_j := \left\{x \in A(K) :  x_{\sigma} \in \rho_{\sigma}^{-1}(\square_{j}),~ \hat{h}_L(x) \le \frac{\pi}{32 \cdot 16e\cdot g^3 (g+1)^3(g+2)^3 (100g+274)\gamma_g}  \frac{1}{ [K:\Q]}\sum_{v \in M'_K} \tilde{I}_v(A) \log N_v\right\}.
\end{equation}
\endgroup
%and, , let $\Gamma'_j$ be $\{x \in \Gamma' : \}$ in the number field case and be $\Gamma'$ in the function field case.
We will prove that Theorem~\ref{ThmFirstBound} holds true verbally with $\Gamma_j$ replaced by $\Gamma'_j$, up to modifying the constants.

\begin{thmbis}{ThmFirstBoundPrimeFF}\label{ThmFirstBoundPrime}
    For any $g \in \Z_{\ge 1}$, there  exists a constant $C'_1 = C'_1(g)>0$ such that:     
For each $j \in \{1,\ldots,(2g)^g(g+2)^g\}$, there exists a proper abelian subvariety $A'$ of $A$, defined over $K$, such that
\begin{equation}\label{EqFirstBoundPrime}
 \left( \#\!\left(\frac{\Gamma'_j+A'}{A'}\right) \cdot \frac{h^0(A',L)}{h^0(A,L)} \right)^{\frac{1}{g-\dim A'}} \le C'_1 \frac{ \max\{h_{\mathrm{Fal}}(A) ,1\} +\log h^0(A,L) + \log[K:\Q]}{ \frac{1}{[K:\Q]}\sum_{v\in M'_K}   \tilde{I}_v(A) \log N_v }.
\end{equation}
\end{thmbis}
In fact we can take $C'_1(g) = 3 C_1(g)$, with $C_1(g)$ from Theorem~\ref{ThmFirstBound}. 
As for Theorem~\ref{ThmFirstBound},  
for each abelian subvariety $B$ of $A_{\bar{K}}$, define $\mathfrak{N}'_j(A')$ to be the number on the left hand side of \eqref{EqFirstBoundPrime}. 
%\[
%\mathfrak{N}_j(B):= \left( \#\left(\frac{\Sigma_j+B}{B}\right) \cdot \frac{h^0(B,L)}{h^0(A,L)} \right)^{1/(g - \dim B)}.
%\]
Then 
\[
\min_{A'\text{ proper abelian subvarieties of }A\text{ defined over }K} \mathfrak{N}'_j(A') = \min_{A'\text{ proper abelian subvarieties of }A_{\bar{K}}} \mathfrak{N}'_j(A')
\]
by Gaudron--R\'{e}mond \cite[Lem.~6.1]{GR25}. 
Call this number $\mathfrak{N}'_j$. Then the conclusion of Theorem~\ref{ThmFirstBoundPrime} is equivalent to: For each $j \in \{1,\ldots,(2g)^g(g+2)^g\}$, we have
\begin{equation}\label{EqFirstBoundPrimeEqiv}
\mathfrak{N}'_j \le   C'_1 \frac{ \max\{h_{\mathrm{Fal}}(A) ,1\} +\log h^0(A,L) + \log[K:\Q]}{ \frac{1}{[K:\Q]}\sum_{v\in M'_K}   \tilde{I}_v(A) \log N_v }.
\end{equation}

\section{Proofs of the bounds in §\ref{SectionFirstBoundStatement}}\label{SectionProofOfFirstBound}

The goal of this section is to prove the bounds in $\mathsection$\ref{SectionFirstBoundStatement}. 
%Retain the assumptions and the notation from $\mathsection$\ref{SectionFirstBoundStatement}. 
We will focus on rational torsion points (\textit{i.e.} Theorem~\ref{ThmFirstBoundFF} and Theorem~\ref{ThmFirstBound}), and  point out how to modify the proof to handle  rational points of small height (\textit{i.e.} Theorem~\ref{ThmFirstBoundPrimeFF} and Theorem~\ref{ThmFirstBoundPrime}). Through the whole section, we make the following assumption:
\begin{center}
{\tt (Hyp)} $A/K$ has semi-stable reduction, and not everywhere good reduction if $K=\C(B)$.
\end{center}

%Note that in the function field case $K=\CC(B)$, if $A$ has semi-stable reduction, then automatically $A$ has split semi-stable reduction.

A large part of the proof works both in the number field and in the function field case. To ease the writing, we will use some unified notations.  
Let $\bar{L}$ be the canonical adelic extension of $L$. Then $\bar{L}$ is a metrized line bundle on the Berkovich analytification of $A$, and write $\|\cdot\|_v$ for the metric at $v$ for each $v \in M_K$. Moreover, $\hat{h}_L = h_{\bar{L}}$.

When $K$ is a function field $\C(B)$, denote by $M'_K := \mathfrak{Bad}(A/K)$ and by $\tilde{I}_v(A) := (\#\Phi_v)^{-1/r_v}$ for $v \in M'_K$. Also for each $v \in M'_K$, recall the convention $N_v = e$ from $\mathsection$\ref{SubsectionGeneralNotation}.

When $K$ is a number field, recall that $M'_K = \mathfrak{Bad}(A/K) \bigcup\{\sigma\}$ for the embedding $\sigma \colon K \hookrightarrow \C$ chosen in $\mathsection$\ref{SubsectionDivision} and recall the local invariants $\tilde{I}_v(A)$ for each $v \in M'_K$ defined in \eqref{EqDefntildeIA}. For each $v \in M'_K$, the number $N_v$ is defined in $\mathsection$\ref{SubsectionGeneralNotation}.

For each $j \in \{1,\ldots, (2g)^g(g+2)^g\}$, denote by  
\begin{equation}\label{EqGammajUnified}
\Gamma_j := \begin{cases} A(K)_{\mathrm{tor}}	 & \text{if }K\text{ is a function field} \\
\text{the subset \eqref{EqGammaj} of }A(K)_{\mathrm{tor}}	 & \text{if }K\text{ is a number field}.\end{cases}
\end{equation}
This set gives rise to a number $\mathfrak{N}$ defined below Theorem~\ref{ThmFirstBoundFF} (in the function field case) and a number $\mathfrak{N}_j$ defined below Theorem~\ref{ThmFirstBound} (in the number field case). 
To unify notation, we shall use $\mathfrak{N}_j$ to denote this number in both cases. Then proving Theorem~\ref{ThmFirstBoundFF} and Theorem~\ref{ThmFirstBound} is equivalent to proving a suitable upper bound of $\mathfrak{N}_j$; see \eqref{EqFirstBoundEqivFF} and \eqref{EqFirstBoundEqiv}.

When handling rational points of small height (\textit{i.e.} Theorem~\ref{ThmFirstBoundPrimeFF} and Theorem~\ref{ThmFirstBoundPrime}), we use the following unified notations: For each $j \in \{1,\ldots, (2g)^g(g+2)^g\}$, we use (i) $\Gamma'_j$ denotes the set \eqref{EqGammaPrimeFF} when $K$ is a function field and denotes the set \eqref{EqGammaPrime} when $K$ is a number field; (ii) $\mathfrak{N}'_j$ to denote the number $\mathfrak{N}'$ defined above \eqref{EqFirstBoundPrimeEqivFF} when $K$ is a function field and to denote the number $\mathfrak{N}'_j$ defined above \eqref{EqFirstBoundPrimeEqiv} when $K$ is a number field.

\subsection{Perturbed adelic line bundle and existence of small sections}\label{SubsectionSmallSection}
Recall {\tt {(Hyp)}} that $A$ has semi-stable reduction over $K$. Consider the tuple of functions $(f_v)_{v \in M'_K}$, where
\begin{itemize}
    \item For $v = \sigma$, set $f_v$ to be the function $f_{\sigma}$ from Proposition~\ref{CorBumpFctArch}, with $\square'_{\sigma}$ defined above \eqref{EqBetterArch};
    \item For $v \in \mathfrak{Bad}(A/K)$, take the function $f_v$  defined in \eqref{EqBumpFctNonArch}.
\end{itemize}
Next let us normalize the tuple to be
\begin{equation}\label{EqTuplePerturb}
\mathbf{f} :=(f_v - \tilde{J}_v(A))_{v \in M'_K},\qquad \text{ with }\tilde{J}_v(A) := \frac{1}{2}\int_{A_v^{\mathrm{an}}} f_v \cdot \frac{1}{\deg_L(A)}c_1(L,\|\cdot\|_v)^{\wedge g}.
\end{equation} 
Below we shall work with the following {\it twist} of $\bar{L}$. For any $\epsilon \in \R$, we have the adelic line bundle
\[
\bar{L}(\epsilon\mathbf{f})
\]
whose metric at $v \in \mathfrak{Bad}(A/K)$ is $N_v^{-\epsilon (f_v-\tilde{J}_v(A))}\|\cdot\|_v$, at $\sigma$ (and $\bar{\sigma}$) is $e^{-\epsilon(f_{\sigma}-\tilde{J}_{\sigma}(A))}\|\cdot\|_{\sigma}$, and at all the other $v$'s are $\|\cdot\|_v$.

Before moving on, let us point out the following fact. 
By definition \eqref{EqDefntildeIA} of $\tilde{I}_v(A)$, the bounds \eqref{EqBumpFctArchInt} for $v=\sigma$ and \eqref{EqIntegralCLMeasure} for $v\in \mathfrak{Bad}(A/K)$, we have
\begin{equation}\label{EqIvSmall}
    \tilde{J}_v(A) \ge  \frac{\pi}{16e\cdot (g+1)^3(g+2)^3(100g+274)\gamma_g} \tilde{I}_v(A) > 0 \qquad\text{for all }v \in M'_K.
\end{equation}
In the function field case, we have the better bound
\begin{equation}\label{EqIvSmallFF}
    \tilde{J}_v(A) \ge  \frac{g}{128e\gamma_g} \tilde{I}_v(A) > 0 \qquad\text{for all }v \in \mathfrak{Bad}(A/K).
\end{equation}

\begin{prop}\label{PropSmallSection}
Let $\epsilon = 1/(8g)$. Then the adelic line bundle $\bar{L}(\epsilon \mathbf{f})$ is big.

As a consequence, there exists a constant $m_0 > 0$ with the following property: there exists a non-zero small global section $s \in \hat{H}^0(A,m\bar{L}(\epsilon \mathbf{f}))$ for all $m \ge m_0$.
\end{prop}

By ``small'', we mean $\|s\|_{m\bar{L}(\epsilon \mathbf{f}), v, \sup}\le 1$ for all $v\in M_K$. 
The constant $m_0$ depends on many choices. But this will not cause a problem, because eventually we will let $m \rightarrow \infty$.

\begin{proof}
To ease notation, we denote $\mathbf f_0=(f_v)_{v\in M_K'}$, and define the adelic line bundle $\bar L(\epsilon \mathbf f_0)$ similarly. By Proposition~\ref{PropBumpFctNonArch}.(iii) and Corollary~\ref{CorBumpFctArch}.(ii), $\bar{L}(\epsilon \mathbf{f})$ is  nef. So we can apply Arithmetic Hilbert--Samuel to $\bar{L}(\epsilon \mathbf{f})$ and get 
\begin{align}\label{Eq1Volume}
\mathrm{vol}(\bar{L}(\epsilon \mathbf{f})) = \bar{L}( \epsilon\mathbf{f})^{g+1} &= \bar{L}(\epsilon \mathbf f_0)^{g+1} - \sum_{v\in M_K'}\int _{A_v^{\an}} (g+1) \epsilon \tilde   J_v(A) c_1(\bar{L}(\epsilon \mathbf{f}_0))^{\wedge g}    \nonumber \\
&=  \bar{L}(\epsilon \mathbf f_0)^{g+1} - (g+1) \epsilon \sum_{v\in M_K'}  \tilde J_v(A) \int _{A_v^{\an}} c_1(\bar{L}(\epsilon \mathbf{f}_0))^{\wedge g} \\
& = \bar{L}(\epsilon \mathbf f_0)^{g+1} - (g+1) \epsilon  \deg_L(A) \sum_{v\in M_K'}  \tilde J_v(A)  \quad\text{ by Chern's integration formula.}   \nonumber
\end{align}

We can further expand
\[
\bar{L}(\epsilon \mathbf f_0)^{g+1} = \bar{L}^{g+1} + \sum_{k=0}^g \binom{g+1}{k+1}\bar L^{g-k}\cdot\overline \cO(\epsilon \mathbf f_0)^{k+1} = \sum_{k=0}^g \binom{g+1}{k+1}\bar L^{g-k}\cdot\overline \cO(\epsilon \mathbf f_0)^{k+1}.
\]
Next let us compute the sum on the right hand side. Notice that $\binom{g+1}{k+1} = \frac{g+1}{k+1}\binom{g}{k}$. In the computation below, we use the uniform notation $\mathrm{d}\mathrm{d}^{\mathrm{c}}$ for both archimedean and non-archimedean places, which means $\mathrm{d}'\mathrm{d}''$ for non-archimedean places. So
%\Ginlinecomment{Need to change some notation. At non-archimedean $v$ we have $(r_v,r_v)$-superforms.}
\begin{align*}
\bar{L}(\epsilon \mathbf f_0)^{g+1}
& = (g+1)\epsilon \sum_{k=0}^g \sum_{v\in M'_K}  \frac{1}{k+1}\binom{g}{k} \int_{A_v^{\an}}  f_v c_1(L,\|\cdot \|_v)^{\wedge (g-k)}\wedge (\epsilon\mathrm{d}\mathrm{d}^{\mathrm{c}} f_v)^{\wedge k} \\
& = (g+1)\epsilon \sum_{v\in M'_K} \Bigg( \sum_{k=0}^g \binom{g}{k} \int_{A_v^{\an}}  f_v c_1(L,\|\cdot \|_v)^{\wedge (g-k)}\wedge (\epsilon\mathrm{d}\mathrm{d}^{\mathrm{c}} f_v)^{\wedge k}  \\
& \qquad \qquad \qquad -   \sum_{k=1}^g  \frac{k}{k+1}\binom{g}{k} \int_{A_v^{\an}}  f_v c_1(L,\|\cdot \|_v)^{\wedge (g-k)}\wedge (\epsilon\mathrm{d}\mathrm{d}^{\mathrm{c}} f_v)^{\wedge k} \Bigg) \\
& = (g+1)\epsilon \sum_{v\in M'_K} \Bigg( \int_{A_v^{\mathrm{an}}} f_v \Big( c_1(L,\|\cdot\|_v)+\epsilon  \mathrm{d}\mathrm{d}^{\mathrm{c}}f_v\Big)^{\wedge g} \\
& \qquad \qquad \qquad -  \sum_{k=1}^g  \frac{k \epsilon^{k}}{k+1}\binom{g}{k} \int_{A_v^{\an}} f_v c_1(L,\|\cdot \|_v)^{\wedge (g-k)}\wedge (\mathrm{d}\mathrm{d}^{\mathrm{c}} f_v)^{\wedge k} \Bigg).
\end{align*}
Let us estimate the two terms in the last expression.%, for which we will frequently use the comparison of $(1,1)$-(super)forms (standard in the archimedean case, Lemma~\ref{LemmaSuperformComparison} in the non-archimedean case).

For the first term, we use $ \epsilon \mathrm{d}\mathrm{d}^{\mathrm{c}}f_v \ge - \epsilon c_1(L,\|\cdot\|_v)$ by Remark~\ref{RmkArchDDC} when $v$ is archimedean and by \eqref{EqSupFormInequC1} when $v$ is non-archimedean. So $c_1(L,\|\cdot\|_v)+\epsilon  \mathrm{d}\mathrm{d}^{\mathrm{c}}f_v \ge (1-\epsilon)c_1(L,\|\cdot\|_v)\ge 0$. Moreover $f_v \ge 0$. So
\begin{equation}\label{Eq2Volume}
\int_{A_v^{\mathrm{an}}} f_v \Big( c_1(L,\|\cdot\|_v)+\epsilon  \mathrm{d}\mathrm{d}^{\mathrm{c}}f_v\Big)^{\wedge g}  \ge (1-\epsilon)^g \int_{A_v^{\mathrm{an}}} f_v  c_1(L,\|\cdot\|_v)^{\wedge g}  = 2(1-\epsilon)^g \deg_L(A) \tilde{J}_v(A).
\end{equation}

For the second term, we have $c_1(L,\|\cdot \|_v)^{\wedge (g-k)}\wedge (\mathrm{d}\mathrm{d}^{\mathrm{c}} f_v)^{\wedge k} \le  c_1(L,\|\cdot \|_v)^{\wedge g}$ by Lemma~\ref{LemmaSuperformComparison} for $v$ non-archimedean and Lemma~\ref{LemmaComparison11Form} for $v$ archimedean.  
So (since $f_v \ge 0$ and $\frac{k}{k+1} < 1$)
\begin{align}\label{Eq3Volume}
\begin{split}
\sum_{k=1}^g  \frac{k \epsilon^{k}}{k+1}\binom{g}{k} \int_{A_v^{\an}}  f_v c_1(L,\|\cdot \|_v)^{\wedge (g-k)}\wedge (\mathrm{d}\mathrm{d}^{\mathrm{c}} f_v)^{\wedge k}  
%  & \le \sum_{k=1}^g  \frac{k \epsilon^{k}}{k+1}\binom{g}{k} \int_{A_v^{\an}}  f_v c_1(L,\|\cdot \|_v)^{\wedge g} \nonumber \\
 &\le  \sum_{k=1}^g \epsilon^{k} \binom{g}{k} \int_{A_v^{\an}}  f_v c_1(L,\|\cdot \|_v)^{\wedge g} \\
  &= 2  ((1+\epsilon)^g-1) \deg_L(A) \tilde{J}_v(A). 
\end{split}
\end{align}

Taking \eqref{Eq2Volume} - \eqref{Eq3Volume}, we get a lower bound for $\bar{L}(\epsilon \mathbf{f}_0)$. Then we get from \eqref{Eq1Volume}
\begin{equation}\label{EqVolume}
\vol (\bar{L}(\epsilon \mathbf f)) = \bar{L}(\epsilon \mathbf f)^{g+1} \ge (g+1) \epsilon  \deg_L(A) \big( 2(1-\epsilon)^g - 2(1+\epsilon)^g + 1 \big) \sum\nolimits_{v\in M'_K} \tilde{J}_v(A).
\end{equation}
By our choice $\epsilon =\frac{1}{8g}$, we have $ 2(1-\epsilon)^g - 2(1+\epsilon)^g + 1 > 0$. Moreover $M'_K\not=\emptyset$ by assumption (either $K$ is a number field, or $A/K$ does not have good reduction everywhere). So $\sum_{v\in M'_K} \tilde{J}_v(A) > 0$ by \eqref{EqIvSmall}. We are done.%there exists $v \in M'_K$ such that $f_v > 0$. So by Proposition~\ref{PropBumpFctNonArch}.(iv) and Proposition~\ref{CorBumpFctArch}.(iii), we have $\sum_{v\in M'_K} \int_{A_v^{\mathrm{an}}}f_v c_1(L,\|\cdot\|_v)^{\wedge g} > 0$. We are done.
\end{proof}

%An immediate corollary of Proposition~\ref{PropSmallSection} is:  For each $m \gg 1$, there is a non-zero global section $s \in H^0(A, m\bar{L}(\epsilon\mathbf{f}))$, \textit{i.e.} a  global section $s$ of the line bundle $mL$ whose $\sup$-norm is $\le 1$ for the twisted metric $\{e^{-m \epsilon f_v}\|\cdot\|^{\otimes m}_v\}_{v \in M_K}$. 
% \begin{cor}
%     For each $m \gg 1$, there exists a non-zero section $s \in H^0(A,m\bar{L}(\epsilon \mathbf{f}))$ such that
%     \[
%    \frac{1}{m}\sum_{v\in M'_K} \log \|s\|^{\otimes m}_{v,\sup}
%     \]
% \end{cor}
% \begin{proof}
%     Zhang's successive minima says that there exists a global section $s \in H^0(A,mL)$ such that
 %    \[
%      - \frac{1}{m} \sum_{v\in M'_K} \log \sup_{x \in A_v^{\mathrm{an}}} \left(e^{-m\epsilon f_v(x)} \|s(x)\|^{\otimes m}_v \right) \le \frac{\vol(\bar{L}(\epsilon \mathbf{f}))}{(g+1)L^g} + O(\frac{1}{m}).
%     \]
%     Since $f_v \ge 0$, the left hand side is $\ge  - \frac{1}{m} \sum_{v\in M'_K} \log \sup_{x \in A_v^{\mathrm{an}}} \|s(x)\|^{\otimes m}_v$. So
%     \[
%     -  \frac{1}{m}\sum_{v\in M'_K} \log \|s\|^{\otimes m}_{v,\sup} \le \vol(\bar{L}(\epsilon \mathbf{f})) + O(\frac{1}{m}).
%     \]
% \end{proof}
% \Ginlinecomment{This seems to be the bad direction!!! One solution is to change our $f_v$ back to $f_v - \tilde{J}_v(A)$ and do the estimate finer. Looks okay. Will write it.}

\subsection{Choice of an auxiliary point}
We need the following result on zero estimates proved by Nakamaye 
\cite{Nakamaye}. Recall the notations $\Gamma_j$ and $\mathfrak{N}_j$ from the paragraph around \eqref{EqGammajUnified}. For each $i \in \{1,\ldots,g\}$, set $\Gamma_j[i] := \{x_1+\cdots+x_i :x_1,\ldots,x_i \in \Gamma_j$ and set $\Gamma_j[0]:= \{0\}$. Then $\Gamma_j[i] \subseteq A(K)_{\mathrm{tor}}$ because $\Gamma_j \subseteq A(K)_{\mathrm{tor}}$.

\begin{prop}\label{PropVanishingOrder}
Assume $m > \mathfrak{N}_j / g$ and $s \in H^0(A, m\bar{L}(\epsilon \mathbf{f}))$ with $\epsilon = 1/(8g)$. Then there exists a point $x \in \bigcup_{i=0}^g \Gamma_j[i]$ such that the vanishing order $\ell$ of $s$ at $x$ satisfies
\[
\ell \le g\left(\frac{gm}{\mathfrak{N}_j} + 1\right).% \le \frac{2g^2m}{\mathfrak{N}_j}.
\]
\end{prop}
Here, the vanishing order $\ell$ is characterized by: $s \in H^0(A,\cI_x^\ell+mL)$ but $s\not\in H^0(A,\cI_x^{\ell+1}+mL)$, where $\cI_x$ is the ideal sheaf of $\{x\}$. %We have $\ell \ge 1$ since $s$ vanishes on the whole $\Gamma_{\mathbf{j}}$; see Lemma~\ref{LemmaVanishingS}.
\begin{proof}
    Assume that the conclusion does not hold true. Apply the main result of \cite{Nakamaye} to the  set $\Gamma_j$ (so $G=A_{\bar{K}}$, $d=g$, $\Lambda = \Lie A$, and the line bundle $mL$). Then there exists a proper abelian subvariety $B$ of $A_{\bar{K}}$ such that
    \[
    \left(\frac{gm}{\mathfrak{N}_j} + 1\right)^{g- \dim A_0} \#\!\left(\frac{\Sigma+A_0}{A_0}\right) \deg_{mL} A_0 \le \deg_{mL} A.
    \]
    But $\deg_{mL}(A_0)= (\dim A_0)! m^{\dim A_0} h^0(A_0,L)$, $\deg_{mL}A = g! m^g h^0(A,L)$, and $\mathfrak{N}_j(A_0)^{g-\dim A_0} = \#((\Sigma+A_0)/A_0) \cdot h^0(A_0,L)/h^0(A,L)$. So
    \[
      \left(\frac{gm}{\mathfrak{N}_j} + 1\right)^{g- \dim A_0} \left(\frac{\mathfrak{N}_j(A_0)}{m}\right)^{g-\dim A_0}  (\dim A_0)! \le g!.
    \]
    But $\mathfrak{N}_j \le \mathfrak{N}_j(A_0)$ by definition. So $g^{g-\dim A_0} (\dim A_0)! < g! $, which is impossible. Hence we get a contradiction, and we are done.
\end{proof}

\begin{rmk}\label{RmkVanishingOrder}
If we work with the set $\Gamma'_j$ from above $\mathsection$\ref{SubsectionSmallSection}, then the conclusion becomes: we can find a point $x \in \bigcup_{i=0}^g\Gamma'_j[i]$ such that the vanishing order of $s$ at $x$ is $\le g^2m/\mathfrak{N}'_j + g$, with $\Gamma'_j[i] := \{x_1+\cdots+x_i: x_1,\ldots,x_i \in \Gamma'_j\}$.
\end{rmk}

\subsection{Jet estimates}\label{SubsectionJetEstimates}
From now on, fix $\epsilon = 1/(8g)$. For each $m > \max\{m_0, \mathfrak{N}_j/g\}$ (with $m_0$ from Proposition~\ref{PropSmallSection}),  take a non-zero small global section $s \in H^0(A, m\bar{L}(\epsilon \mathbf{f}))$ from Proposition~\ref{PropSmallSection} and then an auxiliary point $x \in \bigcup_{i=0}^g \Gamma_j[i]$ from Proposition~\ref{PropVanishingOrder}. When $K$ is a number field,  the discussion around \eqref{EqBetterArch} implies (see \eqref{EqIsogenyArch} for the definition of $\rho_{\sigma}$)
\[
 \rho_{\sigma}(x_{\sigma}) \in \square'_{\sigma},
\]
and better, $ \rho_{\sigma}(x_{\sigma})$ is in one of the subsets from \eqref{EqBetterArch}; here $x_{\sigma}$ is $x \in A(K) \hookrightarrow A_{\sigma}(\C)$.

Now set
\begin{equation}\label{EqAdelicSpace}
J := \mathrm{Sym}^{\ell}((\lie A)^\vee) \otimes L_x^{\otimes m} = x^*\mathrm{Sym}^{\ell}( (\lie A)^\vee)\otimes L_x^{\otimes m}
\end{equation}
which is a $K$-vector space. We shall endow $J$ with the structure of adelic space, \textit{i.e.} a metric on $J\otimes_K K_v$ at each $v \in M_K$; it suffices to do this for  $\lie A$ and for $L_x$. 

The adelic line bundle $\bar{L}$ makes $L_x = x^*L$ into a natural adelic space $\bar{L}_x$. 
The structure of adelic space $\bar{\lie A}$ on $\lie A$ is given as follows. At $v \in M_{K,\infty}$ corresponding to $\sigma \colon K \hookrightarrow \C$, define $\|z\|_\sigma^2 := H(z,z)$ for any $z \in \Lie A_{\sigma}$ where $H$ is the Riemann form of $L_{\sigma}$. At $v \in M_K^0$, consider the connected N\'eron model $\cA/\cO_K$ of $A$ 
 with $L$ extended to $\cL$ on $\cA$. Then for any $a \in \Lie A$ define $\|a\|_v := \min\{|\lambda|_v : \lambda \in K_v \setminus\{0\},~ a/\lambda \in \lie(\cA/\cO_K)\otimes_{\cO_K}\cO_v\}$.

\vskip 0.2em

The {\it $\ell$-th jet} of $s$ at $x$, denoted by $\mathrm{jet}^\ell s(x)$, is the image of $s$ under the composition
\[
H^0(A,\cI_x^\ell+mL) \rightarrow H^0(\mathrm{Spec}\bar{K}, x^*\!\!\left(\frac{\cI_x^\ell}{\cI_x^{\ell+1}}+mL\right) ) \rightarrow H^0(\mathrm{Spec}\bar{K}, x^*\mathrm{Sym}^\ell(\Omega_A +mL)).
\]
Then $\mathrm{jet}^\ell s(x)$ is an element in $J$, which is non-zero  by the definition of vanishing order.  
Let us estimate $\log \|\mathrm{jet}^{\ell}s(x)\|_v$ in terms of the number $\tilde{J}_v(A)$ (see below \eqref{EqTuplePerturb}) at  $v \in M_K$, where by abuse of notation $\|\cdot\|_v$ denotes the metric on $J \otimes_K K_v$.

\subsubsection{Non-archimedean places}
\begin{lemma}\label{LemmaJetNonArch}
%Assume $v$ is a non-archimedean place of $K$. Then
For any $v \in M_K^0$, we have $\displaystyle
\log \|\mathrm{jet}^{\ell}s(x)\|_v \le - m\epsilon\tilde{J}_v(A) \log N_v$.% for any $v \in M_K^0$. 
\end{lemma}
\begin{proof}
\cite[§6.7.1]{GR_min_isog} implies
\[
\|\mathrm{jet}^{\ell}s(x) \|_v \le \sup\nolimits_{y\in A(K_v)}\|s(y)\|_v.
\]
For the sake of completeness we give a sketch here. 
%We can endow $H^0(A,mL)$  with the structure of adelic vector space over $K$ as in §\ref{SubsubsectionLinkAdelic}. 
Consider the the N\'eron model $\cA$ of $A\otimes_K K_v$, and extend $L\otimes_K K_v$ to a cubical $\Q$-line bundle $\cL$ on $\cA$ defined over $\mathrm{Spec}\cO_{K_v}$; such an extension exists by \cite[§II, 1.2.1]{MB_model}. 
%Consider a {\it Moret--Bailly model} $(\cA,\cL)$ of $(A,L)$ over $\cO_K$. 
Notice that we make a slightly different choice of \cite[§6.1]{GR_min_isog}: we do not take finite extensions of $K$, while we only require $\cA$ and $\cL$ to be defined over  $\cO_{K_v}$ (instead of $\cO_K$) and $\cL$ to be a $\Q$-line bundle. 
Every $y\in A(K_v)$ extends to a section $\bar{y} \in \cA(\cO_{K_v})$ by the N\'eron mapping property. Over the fiber $L_y$, the model metric given by $\cL$ coincides with the canonical metric  $\bar L_y$; see \cite[§III]{MB_model}. 
 %Consider the N\'eron model $\cA/\cO_K$ of $A$ and extend $L$ to $\cL$ on $\cA$. Let $\underline{x}$ be the Zariski closure of $\{x\}$ in $\cA$; then $\underline{x}$ is a section of $\cA/\cO_K$. 
For our $s \in H^0(A,mL) \hookrightarrow H^0(A,mL)\otimes_K K_v$, take $\lambda \in \bar{K_v} \setminus\{0\}$ such that $|\lambda|_v = \sup\nolimits_{y\in A(K_v)}\|s(y)\|_v$. Set $s_0 := s/\lambda$. Then $s_0\in H^0(\cA,m\cL)\otimes \cO_{\bar{K_v}}$ because its valuation at each component of  $\cA_{\cO_{\bar {K}_v}}$ is $\geq 0$ by our choice of $\lambda$. The upshot is  
$\|\mathrm{jet}^{\ell}s(x) \|_v = |\lambda|_v |\mathrm{jet}^{\ell} s_0 (\bar{x})|_v \le |\lambda|_v =  \sup_{y\in A(K_v)}\|s(y)\|_v$.

But  $f_v(A(K_v)) = 0$ by Proposition~\ref{PropBumpFctNonArch}.(i). Since $s$ is a small section of $m\bar{L}(\epsilon \mathbf{f})$, we have $\|s(y)\|_v N_v^{-m\epsilon (f_v(y) - \tilde{J}_v(A))} \le 1$ for any $y \in A(K_v)$. So
\[
\sup\nolimits_{y \in A(K_v)}\|s(y)\|_v \le N_v^{-m\epsilon\tilde{J}_v(A)}.
\]
We are done by combining the two inequalities above.
\end{proof}

\subsubsection{Archimedean places}\label{SubsubsectionJetArch}
If $K$ is a number field, we also need estimates at $v\in M_{K,\infty}$.
\begin{lemma}\label{LemmaJetArch}
Let $\sigma' \colon K\hookrightarrow\C$ be an archimedean place of $K$ which is not $\sigma$ or $\bar{\sigma}$. Then
\[
\log \|\mathrm{jet}^\ell s(x)\|_{\sigma'} \le \frac{g^2m}{2\mathfrak{N}_j}\left(1 + \log^+ \frac{\pi\mathfrak{N}_j}{2g^2} \right) +  o(m).
\]
\end{lemma}
\begin{proof}
We follow the treatment of \cite[Prop.~6.8]{GR25}. 
By \cite[Rmk.~3.2.4]{Gau_Explicit_Abvar}, we have 
    \[
    \|\mathrm{jet}^\ell s(x)\|_{\sigma'} \le \|s\|_{\sigma',L^2} \left(\frac{e \pi m}{\ell}\right)^{\ell/2} e^{o(m)}.
    \]
 The trivial bound $\|s\|_{\sigma',L^2} \le \|s\|_{\sigma',\sup}$ implies $\|s\|_{\sigma',L^2}  \le 1$. 
    The function $x\mapsto (e\pi m/x)^x$ takes its maximum at $x= \pi m$. Hence we are done because $\|s\|_{\sigma', L^2} \le \|s\|_{\sigma', \sup}$, which is $\le 1$ by our choice of $s$. The result follows from Proposition~\ref{PropVanishingOrder}.
\end{proof}

It remains to do the estimates at $\sigma$ and $\bar{\sigma}$. Below we do it for $\sigma$, and the conclusion at $\bar{\sigma}$ is the same. Let $\gamma_g \ge 1$ be the Minkowski constant from Lemma~\ref{LemmaReductionMinkowski}. Set %($\epsilon=1/(8g)$)
\begin{equation}\label{EqHopeForSmallr}
r_{\sigma} := \min\left\{\frac{1}{8e\gamma_g(g+2)}, \frac{2}{\pi}\left(\frac{\epsilon \tilde{J}_{\sigma}(A)}{2}\right)^{1/2} \right\} > 0.
\end{equation}

Use the notation from $\mathsection$\ref{SectionArch}, which we reformulate now. 
For a Siegel reduced matrix in the Siegel upper half space $\tau_{0,\sigma}$, we fine a sub-lattice $\Lambda \subseteq \Z^g + \tau_{0,\sigma} \Z^g = \beta_{\sigma}(\Z^{2g})$ such that $A_{\sigma}(\C) = \C^g/\Lambda$. We then have an isogeny $\rho_{\sigma} \colon A_{\sigma}  \rightarrow A_{0,\sigma} = \C^g/(\Z^g+\tau_{0,\sigma}\Z^g)$. Moreover, there exists a line bundle $L_0$ on $A_{0,\sigma}$ inducing principal polarization, and $(L_{\sigma},\|\cdot\|_{\sigma}) = \rho_{\sigma}^*(L_{0,\sigma},\|\cdot\|_{0,\sigma})$. Thus the Hermitian form associated with $c_1(L_{\sigma},\|\cdot\|_{\sigma})$ is %Riemann form of the line bundle $L_{\sigma}$ on $A_{\sigma}(\C)$ 
%coincides with the Riemann form of  $L_{0,\sigma}$ on $A_{0,\sigma}(\CC)$; it 
% is the positive-definite Hermitian form 
\[
H \colon \C^g\times \C^g \longrightarrow \C, \quad (z,z') \longmapsto \bar{z}^{\mathrm{t}} \mathrm{Im}(\tau_{0,\sigma})^{-1} z'.
\]

For the uniformizations from \eqref{EqArchUnif}, we have
\[
u_{0,\sigma} \colon \C^g\xrightarrow{u_{\sigma}} A_{\sigma}(\C) \xrightarrow{\rho_{\sigma}} A_{0,\sigma}(\C).
\]
We have chosen a fundamental domain $\bigsqcup_{l=1}^{h^0(A,L)} \left( \mathbf{v}_l + \beta_{\sigma}([0,1)^{2g}) \right) \subseteq \C^g$ for $u_{\sigma}$  above \eqref{EqStripeForVanishing2}. For our auxiliary point $x_{\sigma}$ (which is $x \in A(K) \hookrightarrow A_{\sigma}(\C)$), let $z_{\sigma} \in u_{\sigma}^{-1}(x_{\sigma})$ be in this fundamental domain. Then $z_{\sigma} \in \mathbf{v}_l + \beta_{\sigma}([0,1)^{2g})$ for a fixed $l \in \{1,\ldots,h^0(A,L)\}$.

%There exists a unique $z_{0,\sigma} \in u_{0,\sigma}^{-1}(x_{0,\sigma})$ which lies in $\beta_{\sigma}([0,1)^{2g})$. Then $u_{\sigma}(z_{0,\sigma}) \in \rho_{\sigma}^{-1}(x)$. Hence 

Better, recall that $\rho_{\sigma}(x_{\sigma}) \in A_{0,\sigma}(\C)$ is in one of the subsets from \eqref{EqBetterArch} (see above \eqref{EqAdelicSpace}), and hence is in $\square'_{\sigma,q}$ for some $q \in \{1,\ldots, N_g = 2^g(g+2)^g\}$. The definition \eqref{EqStripeForVanishing} of $\square'_{\sigma,q}$ then yields
\[
z_{\sigma}  \in \mathbf{v}_l+ \beta_{\sigma}\!\left([0,1]^g \times \Big([\frac{1}{4(g+2)}, \frac{3}{4(g+2)})^g + \mathbf{e}_q \Big) \right)
\]
By Proposition~\ref{CorBumpFctArch}.(i) and \eqref{EqStripeForVanishing2}, $f_{\sigma} \circ u_{\sigma}$ vanishes on (the wider)
\[
\mathbf{v}_l + \beta_{\sigma} \Big([0,1]^g \times \big([0, \frac{1}{g+2}]^g + \mathbf{e}_q \big) \Big).
\]
Denote by $Y_{\sigma} := \mathrm{Im}(\tau_{0,\sigma})$. 
We claim that 
\begin{equation}\label{EqVanishingBall}
(f_{\sigma}\circ u_{\sigma})
\bigl(z_{\sigma}+Y_{\sigma}^{1/2}\zeta\bigr)=0
\qquad\text{for all }|\zeta|\le r_{\sigma}.
\end{equation}
Indeed, it suffices to show: each of the last $g$ coordinates of $\beta_{\sigma}^{-1}(Y_{\sigma}^{1/2}\zeta) = Y_{\sigma}^{-1/2}\mathrm{Im}(\zeta)$ is $< \frac{1}{4(g+2)}$ if $|\zeta|\le r_{\sigma}$, for the Betti coordinate map $\beta_{\sigma} \colon \R^g\times \R^g \rightarrow \C^g$ sending $(\mathbf{a},\mathbf{b})\mapsto \mathbf{a}+\tau_{0,\sigma}\mathbf{b}$. The smallest eigenvalue $\lambda_{1,\sigma}$ of $Y_{\sigma}$ satisfies $\lambda_{1,\sigma}\ge \frac{\sqrt{3}}{2e\gamma_g}$ (by Lemma~\ref{LemmaReductionMinkowski}.(iv), applied to $B=Y_{\sigma}$, together with Lemma~\ref{LemmaSiegelReduced}.(ii)). So for all $|\zeta|\le r_{\sigma}$,
\[
\left|Y_{\sigma}^{-1/2}\mathrm{Im}(\zeta)\right|_{\infty} \le \frac{|\zeta|}{\sqrt{\lambda_{1,\sigma}}} \le \frac{r_{\sigma}}{\sqrt{\lambda_{1,\sigma}}} < \frac{1}{4(g+2)}
\]
 by our choice of $r_{\sigma}$, as desired.

\begin{prop}\label{PropJetArch}
$\displaystyle
         \log \|\mathrm{jet}^\ell s(x)\|_{\sigma}  \le  - \frac{m \epsilon}{2} \tilde{J}_{\sigma}(A) + \frac{ 10g^4 m }{\mathfrak{N_j}}+ o(m)$.
\end{prop}

\begin{proof}
Classical theory of complex abelian varieties says that the global section $s$ of $mL$ corresponds to a theta function $\vartheta \colon \C^g \rightarrow \C$.  
By \cite[§2.1]{Gau_Explicit_Abvar}, the metric $\|\cdot\|_{\sigma}$ on $mL$ (induced by the cubist metric on $\bar{L}$) satisfies
\begin{equation}\label{EqGaudron1}
\|s(u_{\sigma}(z))\|_{\sigma} = |\vartheta(z)| \exp\left(-\frac{\pi}{2}m H(z,z) \right).
\end{equation}
Next we want to express $\|\mathrm{jet}^{\ell}s(x)\|_{\sigma}^2$, for which we need an $H$-orthonormal basis of $\C^g$. Indeed, for the standard basis $\{e_1,\ldots,e_g\}$ of $\C^g$, we have a natural choice of $H$-orthonormal basis  $\{Y_{\sigma}^{1/2} e_1,\ldots, Y_{\sigma}^{1/2}e_g\}$. 
For a multi-index $\tau=(\tau_1,\ldots,\tau_g) \in \mathbb{N}^g$, denote by $D^{\tau} := \prod_{j=1}^g
\left({\partial/\partial_{Y_{\sigma}^{1/2}e_j}}\right)^{\tau_j}$.  
Then \cite[§2.2]{Gau_Explicit_Abvar} gives
\vskip -0.8em
\begin{equation}\label{EqGaudron2}
\|\mathrm{jet}^{\ell}s(x)\|_{\sigma}^2 = \sum_{\substack{\tau\in\mathbb N^g, ~ |\tau|=\ell}} \frac{\tau!}{\ell!} \left| \frac{1}{\tau!} D^{\tau}\vartheta(z_{\sigma}) \exp\left(-\frac{\pi m}{2}H(z_{\sigma},z_{\sigma})\right) \right|^2.
\end{equation}

For our purpose, we shall use a translated version centered at $z_{\sigma}$ under the $H$-orthonormal basis as follows: define
\[
\vartheta_{z_{\sigma}} \colon \C^g \rightarrow \C, \quad w \mapsto \vartheta(z_{\sigma}+ Y_{\sigma}^{1/2} w) \exp\left( -\pi m H(z_{\sigma}, Y_{\sigma}^{1/2} w) - \frac{\pi m}{2} H(z_{\sigma},z_{\sigma}) \right).
\]
This function is holomorphic in $w$ since $H(z_{\sigma},w)$ is linear in $w$. Using $H(z_{\sigma}+w, z_{\sigma}+w) = H(z_{\sigma}, z_{\sigma}) + H(w,w) + 2 \mathrm{Re} H(z_{\sigma},w)$ and $H(Y_{\sigma}^{1/2}  w, Y_{\sigma}^{1/2} w) = |w|^2$, one gets from \eqref{EqGaudron1}
\begin{equation}\label{EqSGaudron1Variant}
\|s(u_{\sigma}(z + Y_{\sigma}^{1/2}w))\|_{\sigma}  = |\vartheta_{z_{\sigma}}(w)| \exp\left(-\frac{\pi m}{2} |w|^2 \right).
\end{equation}

The section $s$ vanishes at $x$ to order $\ell$. So $ D^{\tau'}\vartheta(z_{\sigma})=0$ for all $|\tau'|<\ell$. Thus 
\[
D^{\tau}\vartheta_{z_{\sigma}}(0) = \exp\left(-\frac{\pi m}{2}H(z_{\sigma},z_{\sigma})\right) D^{\tau}\vartheta(z_{\sigma})
\quad \text{ for }|\tau|=\ell.
\]
So \eqref{EqGaudron2} becomes
\begin{equation}\label{EqGaudronJetTranslated}
\|\mathrm{jet}^{\ell}s(x)\|_{\sigma}^2
=
\sum_{\substack{\tau\in\mathbb N^g,~|\tau|=\ell}} \frac{\tau!}{\ell!} \left| \frac{1}{\tau!}D^{\tau}\vartheta_{z_{\sigma}}(0) \right|^2.
\end{equation}
Cauchy's estimate gives, for $|\tau|=\ell$,
\[
\left| \frac{1}{\tau!} D^{\tau}\vartheta_{z_{\sigma}}(0) \right|  \le \left(\frac{\ell^{\ell}}{\tau^{\tau}} \right)^{1/2} \frac{\sup_{|\zeta| = r_{\sigma}}|\vartheta_{z_{\sigma}}(\zeta)|}{r_{\sigma}^{\ell}},
\]
with the convention $0^0=1$. (Robbins') Stirling's Formula implies $\frac{\tau !}{\ell !} \frac{\ell^{\ell}}{\tau^{\tau}} \le (2\pi \ell)^{g/2}$. 
Hence 
\begin{equation}\label{EqCauchyJetBound}
\|\operatorname{jet}^{\ell}s(x)\|_{\sigma}
\le
\sqrt{ (\ell+1)^g (2\pi \ell)^{g/2} } \frac{\sup_{|\zeta| = r_{\sigma}}|\vartheta_{z_{\sigma}}(\zeta)|}{r_{\sigma}^{\ell}} 
\le (\ell+1)^{3g/4}(2\pi)^{g/2} \frac{\sup_{|\zeta| = r_{\sigma}}|\vartheta_{z_{\sigma}}(\zeta)|}{r_{\sigma}^{\ell}} .
\end{equation}
Let us bound $\sup_{|\zeta| = r_{\sigma}}|\vartheta_{z_{\sigma}}(\zeta)|$. 
Our $s$ is a small section for the perturbed metric. So $|s(y) e^{-m\epsilon (f_{\sigma}(y) - \tilde{J}_{\sigma}(A))}| \le 1$ for all $y\in A_{\sigma}(\C)$. Apply this to $y = u_{\sigma}(z_{\sigma}+Y_{\sigma}^{1/2}\zeta)$ with $|\zeta| = r_{\sigma}$, we get $|s(y)| \le \exp(-m\epsilon \tilde{J}_{\sigma}(A))$ from  \eqref{EqVanishingBall}. Thus  \eqref{EqSGaudron1Variant} implies 
\[
|\vartheta_{z_{\sigma}}(\zeta)| \le \exp\left(-m\epsilon\widetilde J_{\sigma}(A) +\frac{\pi m}{2}r_{\sigma}^2\right) \quad \text{ for all }|\zeta|=r_{\sigma}.
\]
Combining this with \eqref{EqCauchyJetBound} and the upper bound of $\ell$ given by Proposition~\ref{PropVanishingOrder}, we get
\begin{align*}
\log\|\mathrm{jet}^{\ell}s(x)\|_{\sigma} & 
\le  -m\epsilon \tilde{J}_{\sigma}(A) +\left(\frac{g^2m}{\mathfrak N_j}+1\right)\log\frac{1}{r_{\sigma}} +\frac{\pi m}{2}r_{\sigma}^2 +\frac{3g}{4}\log(\frac{g^2m}{\mathfrak N_j}+2)+\log(2\pi)^{g/2} \\
& \le - \frac{m\epsilon}{2} \tilde{J}_{\sigma}(A) + \left(\frac{g^2m}{\mathfrak N_j}+1\right)\log r_{\sigma}^{-1}+ o(m).
\end{align*}
Here the second inequality holds true since $\frac{\pi r_{\sigma}^2}{2}\le \frac{\epsilon \tilde{J}_{\sigma}(A)}{2}$ by choice \eqref{EqHopeForSmallr} of $r_{\sigma}$ and, because  $\mathfrak{N}_j \ge h^0(A,L)^{-1}$ by definition, $\frac{3g}{4}\log(\frac{g^2m}{\mathfrak N_j}+2)+ \log(2\pi)^{g/2} \le \frac{3g}{4}\log( g^2m h^0(A,L) +2)+ \log(2\pi)^{g/2} =o(m)$. 

It remains to handle $\log\frac{1}{r_{\sigma}}$. A lower bound on $\tilde{J}_{\sigma}(A)$ is given by \eqref{EqIvSmall} and the fact that $\tilde{I}_{\sigma}(A) \ge \sqrt{3}/2$. So we get a crude bound (recall $\epsilon = \frac{1}{8g}$)
\[
\log r_{\sigma}^{-1} \le \log\gamma_g + 9\log(g+2) \le 10g^2.
\]
This gives our desired bound.
\end{proof}

\subsection{Liouville inequality for $J$}\label{subsectionLiouvilleIneq}
%Let us recall some more facts on adelic spaces.%, for which we refer to \cite[$\mathsection$4 and 5]{GaudronAdelicSpace}.

Consider now the adelic space $\bar{J}:=(J,\|\cdot\|_v)$ defined over $K$ from \eqref{EqAdelicSpace} and below. Any non-zero $K$-subspace $F \subseteq J$ has a natural structure of adelic subspace $\bar{F}$ of $\bar{J}$ by considering the restriction of $\|\cdot\|_v$ to $F\otimes_K K_v$ at each place $v \in M_K$. The {\it normalized degree} of $\bar{F}$ is defined to be (here $K_0=\Q$ if $K$ is a number field and  $K_0= \C(\mathbb{P}^1)$ if $K=\C(B)$)
\[
\hat{\deg}_{\mathrm{n}}(\bar{F}) := -\frac{1}{[K:K_0]}\sum_{v\in M_K} \log \|s'\|_v \quad \text{ for any non-zero } s' \in \bigwedge\nolimits^{\dim F} F
\] 
and the {\it slope} of $\bar{F}$ is defined to be
\vskip -1em
\[
\hat{\mu}(\bar{F}) := \frac{\hat{\deg}_{\mathrm{n}}(F)}{\dim F}.
\]
See \cite[Lem.~4.4 and 4.5]{GaudronAdelicSpace}. 
Then the {\it maximal slope} of $\bar{J}$ is
\[
\hat{\mu}_{\max}(\bar{J}) := \max\{\hat{\mu}(\bar{F}) : 0\not= \bar{F} \subseteq \bar{J} \}.
\]
Now the non-zero element $\mathrm{jet}^{\ell}s(x) \in J$ defines a line $K\cdot \mathrm{jet}^{\ell}s(x)$ in $J$, and hence an adelic subspace $\bar{K\cdot \mathrm{jet}^{\ell}s(x)}$ of $J$. By definition of the maximal slope, we have
\begin{equation}\label{EqSlope}
\hat{\mu}( \bar{ K\cdot \mathrm{jet}^{\ell}s(x)} ) \le \mu_{\max}(\bar{J}).
\end{equation}
Let us analyze both sides. For the left hand side, by definition we have
\begin{equation}\label{EqSlopePt}
\hat{\mu}( \bar{ K\cdot \mathrm{jet}^{\ell}s(x)} ) = - \frac{1}{[K:K_0] } \sum_{v \in M_K}   \log \|\mathrm{jet}^{\ell}s(x)\|_{v}.
\end{equation}
For the right hand side, we have
\begin{align*}
\hat\mu_{\max}(\bar{J}) &=  m \hat\mu(\bar{L}_x) + \hat\mu_{\max}(\mathrm{Sym}^{\ell}(\bar{\Lie A}^\vee))  \quad\text{by \cite[Prop.~5.7.(1)]{GaudronAdelicSpace}}\\
& =   m h_{\bar{L}}(x) + \hat\mu_{\max}(\mathrm{Sym}^{\ell}(\bar{\Lie A}^\vee))  \quad\text{by direct computation}.
%& \le \ell \left(\frac{g}{2}+1\right) h_{\mathrm{Fal}}(A) + m \hat{h}_L(x) \quad\text{by \cite{}}.
\end{align*}
Now we separate the function field and the number field case.

\noindent\boxed{\text{Case: $K=\C(B)$ is a function field}} By \cite[Thm.~7.1.(1)]{GaudronAdelicSpace}, we then have
\[
\hat\mu_{\max}(\bar{J}) =  m h_{\bar{L}}(x) + \ell \cdot \hat\mu_{\max}(\bar{\Lie A}^\vee).
\]
Therefore by \eqref{EqSlope} and \eqref{EqSlopePt} (and Theorem~\ref{ThmMaxSlopeLieDual}) we have
\begin{equation}\label{EqHeightInequality}
- \frac{1}{[K:\C(\mathbb{P}^1)] } \sum_{v \in M_K}   \log \|\mathrm{jet}^{\ell}s(x)\|_{v} \le  m \hat{h}_L(x) + \ell \left(\frac{g}{2} + 1\right)  h_{\mathrm{Fal}}(A).
\end{equation}
%The last term $h_J(\mathrm{jet}^{\ell}s(x))  =  \frac{1}{[K:K_0] } \sum_{v \in M_K}   \log \|\mathrm{jet}^{\ell}s(x)\|_{v} $. 
An upper bound of $\ell$ is given by Proposition~\ref{PropVanishingOrder}.

\noindent\boxed{\text{Case: $K$ is a number field}} By \cite[Thm.~7.1.(2)]{GaudronAdelicSpace}, we then have
\[
\hat\mu_{\max}(\bar{J}) \le m h_{\bar{L}}(x) + \ell \left( \hat\mu_{\max}(\bar{\Lie A}^\vee) + 2g \log g\right).
\]
Therefore by \eqref{EqSlope}, \eqref{EqSlopePt} and \cite[Prop.~4.6 and below]{Gau_Explicit_Abvar}, we have
\begingroup
\let\small\scriptsize % Temporarily overrides the global hook for this block only
\begin{equation}\label{EqHeightInequalityNF}
- \frac{1}{[K:\Q] } \sum_{v \in M_K} \log \|\mathrm{jet}^{\ell}s(x)\|_{v} \le m \hat{h}_L(x) + \ell \left( (0.6g+1) (h_{\mathrm{Fal}}(A)+ \log h^0(A,L)) + 2g^2\log 10g\right).
\end{equation}
\endgroup

\subsection{Proof of Theorem~\ref{ThmFirstBoundFF}}
Let $K = \C(B)$. Then every place of $K$ is non-archimedean, and hence Lemma~\ref{LemmaJetNonArch} and \eqref{EqIvSmallFF} together imply (with $\epsilon = \frac{1}{8g}$) 
\begin{align}\label{Eq111FF}
\begin{split}
- \frac{1}{[K:\C(\mathbb{P}^1)] } \sum_{v \in M_K}   \log \|\mathrm{jet}^{\ell}s(x)\|_{v} &  \ge \frac{m}{8 \cdot 128 e\gamma_g [K: \C(\mathbb{P}^1)] } \sum_{v \in \mathfrak{Bad}(A/K)}   \tilde{I}_v(A) 	\\
& = \frac{m}{8\cdot 128 e\gamma_g [K: \C(\mathbb{P}^1)] } \sum_{v \in \mathfrak{Bad}(A/K)} \#(\Phi_v)^{-1/r_v}.
\end{split}
\end{align}

Recall our choice that $x \in A(K)_{\mathrm{tor}}$; see above Proposition~\ref{PropVanishingOrder}. 
So $\hat{h}_L(x) = 0$. Applying   \eqref{Eq111FF} and the upper bound on $\ell$ (Proposition~\ref{PropVanishingOrder}) to the  inequality  \eqref{EqHeightInequality}, we get
\[
\frac{m}{8\cdot 128e\gamma_g [K: \C(\mathbb{P}^1)] } \sum_{v \in \mathfrak{Bad}(A/K)} \#(\Phi_v)^{-1/r_v} \le \frac{g(g+2)}{2}  \left(\frac{gm}{\mathfrak{N}_j} + 1\right)  h_{\mathrm{Fal}}(A).
\]
Dividing both sides by $m$ and letting $m \rightarrow \infty$, we get
\[
\frac{1}{8 \cdot 128e\gamma_g [K: \C(\mathbb{P}^1)] } \sum_{v \in \mathfrak{Bad}(A/K)}  \#(\Phi_v)^{-1/r_v}\le \frac{g(g+2)}{2} \frac{g}{\mathfrak{N}_j}  h_{\mathrm{Fal}}(A).
\]
This immediately implies the desired bound \eqref{EqFirstBoundEqivFF}. So we are done for Theorem~\ref{ThmFirstBoundFF}.
\qed

\subsection{Proof of Theorem~\ref{ThmFirstBound}}\label{SubsectionPfThmFirstBound}
Let $K$ be a number field.  By Lemma~\ref{LemmaJetNonArch}, Proposition~\ref{PropJetArch}, and Lemma~\ref{LemmaJetArch}, we have
\begingroup
\let\small\scriptsize % Temporarily overrides the global hook for this block only
\begin{align}\label{EqHeightPt}
\begin{split}
- \frac{1}{[K:\Q] } \sum_{v \in M_K}   \log \|\mathrm{jet}^{\ell}s(x)\|_{v}   \ge \frac{m\epsilon}{2[K:\Q]}& \sum_{v \in M'_K}   \tilde{J}_v(A) \log N_v - \frac{g^2m}{\mathfrak{N}_j}\left(1 + \log^+ \frac{\pi\mathfrak{N}_j}{2g^2} \right)   -\frac{10g^4 m}{\mathfrak{N_j}}- o(m).		
\end{split}
\end{align}
\endgroup
Now $\hat{h}_L(x) = 0$ since $x$ is torsion. So  \eqref{EqHeightPt}, \eqref{EqHeightInequalityNF} and Proposition~\ref{PropVanishingOrder} together imply (to ease notation, we temporarily set $\Theta:= (0.6g+1) (h_{\mathrm{Fal}}(A)+ \log h^0(A,L)) + 2g^2\log 10g$)
\begin{align*}
\frac{m\epsilon}{2[K:\Q]} \sum_{v\in M'_K}   \tilde{J}_v(A) \log N_v \le g\Theta + \frac{g^2m}{\mathfrak{N}_j}\left(\Theta +1 + \log^+\frac{\pi\mathfrak{N}_j}{2g^2} + 10g^2 \right) + o(m).
\end{align*}
Multiply both sides by $\mathfrak{N}_j$, divide both sides by $m$ and let $m \rightarrow \infty$. Recall $\epsilon = \frac{1}{8g}$.  So
\begin{equation}\label{EqEstimateFirstBoundAlmostFinal}
\mathfrak{N}_j \cdot \frac{1}{16g[K:\Q]} \sum_{v\in M'_K}   \tilde{J}_v(A) \log N_v  \le g^2 \left( \Theta + 1 + \log^+\frac{\pi\mathfrak{N}_j}{2g^2} + 10g^2 \right).
\end{equation}
A further coarse estimate (with $\Theta= (0.6g+1) (h_{\mathrm{Fal}}(A)+ \log h^0(A,L)) + 2g^2\log 10g$) yields
\begin{equation}\label{EqEstimateFirstBoundCoarse}
    \mathfrak N_j\le \frac{16 g^3\log^+\mathfrak{N_j}}{\frac 1{[K:\Q]}\sum_{v\in M'_K}   \tilde{J}_v(A) \log N_v}+  128  g^5\log(10g)\frac{\max\{h_\mathrm{Fal}(A),1\}+\log h^0(A,L)}{\frac 1{[K:\Q]}\sum_{v\in M'_K}   \tilde{J}_v(A) \log N_v}.
\end{equation}
To ease notation, denote by
\[
S=\frac{16 \cdot 16e\cdot g^3 (g+1)^3(g+2)^3 (100g+274) \gamma_g}{\pi}\frac{1}{\frac 1{[K:\Q]}\sum_{v\in M'_K}   \tilde{I}_v(A) \log N_v} > 0
\]
and 
\[
T=\frac{128 \cdot 16 e\cdot g^5 (g+1)^3(g+2)^3 (100g+274) \log(10g) \gamma_g}{\pi}\frac{\max\{h_\mathrm{Fal}(A),1\}+\log h^0(A,L)}{\frac 1{[K:\Q]}\sum_{v\in M'_K}   \tilde{I}_v(A) \log N_v}>0.
\]
Applying the comparison of $\tilde{I}_v(A)$ and $\tilde{J}_v(A)$ from \eqref{EqIvSmall} to the bound \eqref{EqEstimateFirstBoundCoarse}, we get
\begin{align}\label{EqEstimateSecondBoundCoarse}
\begin{split}
    \mathfrak N_j\le & S\log^+\mathfrak{N}_j + T.
\end{split}
\end{align}

To conclude it suffices to use following elementary lemma from \cite[Lemma~15]{BakerClayton_GlobalDiscrepancyEC}.

\begin{lemma}\label{LemmaRemoveLog}
   If $\mathfrak N,S,T>0$ with $\mathfrak N\leq  S\log^+\mathfrak N+T$, then $\mathfrak N\leq \max\{1,\frac e {e-1} (S\log S+T)\}$.
\end{lemma}

Notice that $\sum_{v\in M'_K}   \tilde{I}_v(A) \log N_v \ge \tilde{I}_{\sigma}(A) \ge \frac{\sqrt{3}}{2}$; see \eqref{EqDefntildeIA}. So 
\[
\log S  \le \log [K:\Q] + \log \frac{2\cdot 16 \cdot 16 e\cdot g^3 (g+1)^3(g+2)^3 (100g+274) \gamma_g}{\sqrt{3}\pi}.
\]
Hence  our desired bound \eqref{EqFirstBoundEqiv} follows by applying Lemma~\ref{LemmaRemoveLog} to \eqref{EqEstimateSecondBoundCoarse}, with
\begingroup
\let\small\scriptsize % Temporarily overrides the global hook for this block only
\begin{equation}\label{EqFormulaC1}
%\begin{aligned}
C_1(g) = \frac{e}{e-1} C_2(g) \left( 1+\log C_2(g) \right), \qquad C_2(g) = \frac{2048e}{\pi} g^5(g+1)^3(g+2)^3(100g+274) \log(10g)\gamma_g.
%& \frac{2048 e}{(e-1)\pi}\, g^5(g+1)^3(g+2)^3(100g+274) \log(10g)\gamma_g \\
%&\times \left( 1+
%\log\!\left( \frac{2048}{\pi} g^5(g+1)^3(g+2)^3(100g+274) \log(10g)\gamma_g \right) \right).
%\end{aligned}
\end{equation}
\endgroup
This concludes for Theorem~\ref{ThmFirstBound}.
\qed

\subsection{Proof of Theorem~\ref{ThmFirstBoundPrimeFF} and Theorem~\ref{ThmFirstBoundPrime}}
To modify the proofs above to handle rational points of small height, we again fix $\epsilon = 1/(8g)$. For each $m > \max\{m_0, \mathfrak{N}'_j/g\}$ (with $m_0$ from Proposition~\ref{PropSmallSection}),  take a non-zero small global section $s \in H^0(A, m\bar{L}(\epsilon \mathbf{f}))$ from Proposition~\ref{PropSmallSection} and then an auxiliary point $x \in \bigcup_{i=0}^g \Gamma'_j[i]$ from Remark~\ref{RmkVanishingOrder}. Then all arguments and results in $\mathsection$\ref{SubsectionJetEstimates} hold true verbally, with $\mathfrak{N}_j$ replaced by $\mathfrak{N}'_j$. Thus we still have the estimate of the height of the jet as in \eqref{Eq111FF} and \eqref{EqHeightPt}, \textit{i.e.}
\begin{itemize}
\item If $K=\C(B)$ is a function field, then 
\[
- \frac{1}{[K:\C(\mathbb{P}^1)] } \sum_{v \in M_K}   \log \|\mathrm{jet}^{\ell}s(x)\|_{v}  \ge \frac{m}{1024 e\gamma_g [K: \C(\mathbb{P}^1)] } \sum_{v \in \mathfrak{Bad}(A/K)} \#(\Phi_v)^{-1/r_v}.
\]
\item If $K$ is a number field, then
\begingroup
\let\small\scriptsize % Temporarily overrides the global hook for this block only
\begin{align*}
- \frac{1}{[K:\Q] } \sum_{v \in M_K}   \log \|\mathrm{jet}^{\ell}s(x)\|_{v}   \ge \frac{m\epsilon}{2[K:\Q]}& \sum_{v \in M'_K}   \tilde{J}_v(A) \log N_v - \frac{g^2m}{\mathfrak{N}_j}\left(1 + \log^+ \frac{\pi\mathfrak{N}'_j}{2g^2} \right)   -\frac{10g^4 m}{\mathfrak{N_j}}- o(m).		
\end{align*}
\end{itemize}
\endgroup
We still have the inequality \eqref{EqHeightInequality}. But now by definition of $\Gamma'_j$ (which is \eqref{EqGammaPrimeFF} when $K= \C(B)$ and is \eqref{EqGammaPrime} when $K$ is a number field), our bound for $\hat{h}_L(x)$ becomes
\[
\hat{h}_L(x) \le \begin{cases}  g^2 \frac{1}{2048 e\cdot g^2\gamma_g [K:\C(\mathbb{P}^1)]}\sum_{v\in\mathfrak{Bad}(A/K)}  \#(\Phi_v)^{-1/r_v} & \text{if }K=\C(B) \\  g^2 \frac{\pi}{32 \cdot 16e \cdot g^3 (g+1)^3(g+2)^3 (100g+274)\gamma_g[K:\Q]} \sum_{v \in M'_K}\tilde{I}_v(A) \log N_v &\text{if }K\text{ is a number field}. \end{cases}
\]
The vanishing order $\ell$ of $s$ at $x$ is bounded by Proposition~\ref{PropVanishingOrder}, with $\mathfrak{N}_j$ replaced by $\mathfrak{N}'_j$. Hence we have an estimate for each term in  \eqref{EqHeightInequality}.

When $K= \C(B)$ is a function field, applying these estimates to \eqref{EqHeightInequality} we get
\[
\frac{m}{2048e\gamma_g [K: \C(\mathbb{P}^1)] } \sum_{v \in \mathfrak{Bad}(A/K)} \#(\Phi_v)^{-1/r_v} \le \frac{g(g+2)}{2}  \left(\frac{gm}{\mathfrak{N}'_j} + 1\right)  h_{\mathrm{Fal}}(A).
\]
Dividing both sides by $m$ and letting $m \rightarrow \infty$, we obtain the desired bound \eqref{EqFirstBoundPrimeEqivFF}. This concludes for  Theorem~\ref{ThmFirstBoundPrimeFF}.

When $K$ is a number field, we apply these estimates (as well as the comparison of $\tilde{I}_v(A)$ and $\tilde{J}_v(A)$ from \eqref{EqIvSmall}) to \eqref{EqHeightInequality}, divide both sides by $m$ and let $m\rightarrow \infty$. Then we get, similarly to \eqref{EqEstimateFirstBoundAlmostFinal},
\begin{align}\label{EqEstimateFirstBoundAlmostFinalPrime}
\begin{split}
\mathfrak{N}'_j \cdot \frac{1}{32 g[K:\Q]} \sum_{v\in M'_K}   \tilde{J}_v(A) \log N_v  \le g^2 \left( \Theta + 1 + \log^+\frac{\pi\mathfrak{N}_j}{2g^2} + 10g^2 \right).
\end{split}
\end{align}
This yields the desired bound \eqref{EqFirstBoundPrimeEqiv} by Lemma~\ref{LemmaRemoveLog}; the proof is similar to the end of $\mathsection$\ref{SubsectionPfThmFirstBound}. We are done for Theorem~\ref{ThmFirstBoundPrime}. From the proof we see that we can take $C'_1(g) = 3 C_1(g)$.
\qed

\section{Boundary height and Faltings height}\label{SectionLocalGlobal}
Assume $A/K$ has semi-stable reduction.  The definition of {\it boundary height} \eqref{DefnBoundaryHeightNonArchIntro}-\eqref{DefnBoundaryHeightIntro} in the number field case can be extended to the function field case, and is simpler since there are no archimedean places. For $K=\C(B)$, define\begin{equation}\label{DefnBoundaryHeight}
h_{\partial }(A) = h_{\partial,\mathrm{fin}}(A) :=  \frac{1}{[K:K_0]}\left(\sum_{v \in \mathfrak{Bad}(A/K)} (\#\Phi_v)^{1/r_v}\log N_v \right);
\end{equation}
%For every $v\in M_{K,\infty}$, we have $\mathrm{Tr}(\mathrm{Im}\tau_v) \ge \sqrt{3}/2$. So $h_{\partial}(A) \ge \sqrt{3}/2$.
then $h_{\partial }(A_{K'}) = h_{\partial }(A)$ for any finite extension $K'/K$ because $A$ has semi-stable reduction.% over $K$. 

\begin{prop}\label{PropHP}
We have
\[
h_{\partial,\mathrm{fin}}(A) \le \begin{cases} 96 h_{\mathrm{Fal}}(A) & \text{ when }K=\C(B) \\
% c(g) \max\{h_{\mathrm{Fal}}(A),1\} 
96 h_{\mathrm{Fal}}(A) + 8g\log(\pi\sqrt{2}) & \text{ when }K\text{ is a number field}. \end{cases}
\]
If $A$ carries a principal polarization, then both bounds can be further divided by $8$, and 
\[
h_{\partial}(A) \le c(g) \max\{h_{\mathrm{Fal}}(A), 1\}.
\]
\end{prop}
%In the function field case,
Without the explicit constants this is 
 Hindry--Pacheko \cite[Thm.~1.21]{Hindry_Pacheco}. % proved this proposition without computing the explicit constant. %Their proof is by theta functions, so also generalizes to the number field case. 
We hereby give a proof as an application of de Jong--Shokrieh's decomposition of the Faltings height \cite{deJongShokriehDecomposition}.
\begin{proof}
Denote by $K_0 = \C(\mathbb{P}^1)$ when $K=\C(B)$ and $K_0 = \Q$ when $K$ is a number field. 
%By the remark after \eqref{DefnBoundaryHeight}, we may assume $A$ has split semi-stable reduction. 
We first treat the case where $A$ carries a principal polarization $L$. By \cite[Thm.~A]{deJongShokriehDecomposition} (in the number field case) and \cite[(1,12)]{deJongShokriehDecomposition} (in the function field case), we have
\begin{equation}\label{EqDeJongShokrieh}
h_{\mathrm{Fal}}(A) \ge \frac{2}{[K:K_0] } \sum_{v \in M_K} I_v(A,L) \log N_v - g \log(\pi\sqrt{2})
\end{equation}
where $I_v(A,L) \ge 0$ is a local invariant at the place $v$, and the term $- g \log(\pi\sqrt{2})$ can be removed when $K=\C(B)$. We will not need the precise definition of $I_v(A,L)$. 
%To prove the proposition, it suffices to show that $I_v(A,L) \ge \frac{1}{24} (\#\Phi_v)^{1/r_v}$ for $v \in \mathfrak{Bad}(A/K)$ and $I_v(A,L) \ge c_1(g) \mathrm{Tr}(\mathrm{Im}\tau_v) - c_2(g)$ for $v \in M_{K,\infty}$.

For $v \in M_{K,\infty}$, we have $A_v(\CC)=\CC^g/(\Z^g+\tau_v\Z^g)$ for a Siegel reduced matrix $\tau_v$ and $\alpha_v(A,L)=\mathrm{Tr}(\mathrm{Im}\tau_v)$. Now $I_v(A,L) \ge c_1(g) \alpha_v(A,L) - c_2(g)$ by \cite[Prop.~4.1]{Autissier}. 

Let $v \in \mathfrak{Bad}(A/K)$. Let us prove that $I_v(A,L) \ge \frac{1}{24} (\#\Phi_v)^{1/r_v}$. Indeed, $I_v(A,L)$ is half of the tropical moment $I(\Sigma_v)$ by \cite[Thm.~B]{deJongShokriehDecomposition}, where $I(\Sigma_v)$ is defined as follows. The metrized line bundle $(L_v^{\mathrm{an}}, \|\cdot\|_v)$ induces a metric on the skeleton $\Sigma_v$ of $A_v^{\mathrm{an}}$, and hence a metric $\|\cdot\|_{L,v}$ on $X^*_{v,\R}$ via the uniformization $X^*_{v,\R} \rightarrow \Sigma_v = X^*_{v,\R}/Y_v$. We have a subset $\mathrm{Vor}(0) \subseteq X^*_{v,\R}$ from \cite[above~Thm.~B]{deJongShokriehDecomposition} which is a fundamental subset for $X^*_{v,\R} \rightarrow \Sigma_v$, and the tropical moment is defined to be $I(\Sigma_v) := \frac{1}{\vol_{\mathrm{d}\mu}(\mathrm{Vor}(0))}\int_{\mathrm{Vor}(0)} \|\mathbf{x}\|^2_{L,v} \mathrm{d}\mu_v$ for the standard Lebesgue measure $\mathrm{d}\mu_v$ on $X^*_{v,\R}$. Lemma~\ref{LemmaMomentNeron} then gives the desired lower bound of $I_v(A,L)$.

Now the conclusion for principally polarized $A$ follows from \eqref{EqDeJongShokrieh} and the bounds  in the paragraphs above. In general, consider $Z(A):=A^4\times {A^\vee}^4$ which carries a principal polarization by Zarhin's trick. Note that $h_{\partial,\mathrm{fin}}(A)=h_{\partial,\mathrm{fin}}(Z(A))$. Moreover,  $h_{\mathrm{Fal}}(Z(A))=4h_{\mathrm{Fal}}(A)+4h_{\mathrm{Fal}}(A^\vee)=8h_{\mathrm{Fal}}(A)$. So we get the desired inequality for general $A$.
\end{proof}

\begin{lemma}\label{LemmaMomentNeron}
For each $v \in \mathfrak{Bad}(A/K)$, we have
\[
I(\Sigma_v) \ge \frac{r_v}{\pi(r_v+2)} \Gamma\!\left(\frac{r_v}{2}+1\right)^{2/r_v}(\#\Phi_v)^{1/r_v}\ge \frac{1}{12}(\#\Phi_v)^{1/r_v}.
\]
\end{lemma}
\begin{proof}
To ease notation, we omit the subscript $v$ in the proof, and denote by $d := [X:\lambda^*_{T,L}(Y)] \ge 1$; see \eqref{EqLambdaNonArch} for the notation $\lambda^*_{T,L}$.

We have $\Sigma = X^*_{\R}/Y$. Take a $\Z$-basis of $X^*$ and write $\mathbf{x} = (x_1,\ldots,x_r)$ to be the coordinates of $X^*_{\R}$ with respect to this basis. Now $Y$ is a lattice in $X^*_{\R}$ of rank $r$ of covolume $\#\Phi$. Let $B$ be the matrix of $q_0$ under this standard basis. Then $\det(B) = d/\#\Phi$; see \eqref{EqMomentNonArch}. Then the metric $\|\cdot\|_{L,v}$ on $X^*_{\R}$ is: $\|\mathbf{x}\|_{L,v} = \mathbf{x}^{\mathrm{t}}B \mathbf{x}$ for any $\mathbf{x} \in X^*_{\R}$.

Consider the volume form $\mathrm{d}\mathbf{x}:=\mathrm{d}x_1\wedge \cdots \wedge \mathrm{d}x_r$ on $X^*_{\R}$. Then $\vol_{\mathrm{d}\mathbf{x}}(\mathrm{Vor}(0)) = \#\Phi$, so
\[
I(\Sigma) = \frac{1}{\#\Phi} \int_{\mathrm{Vor}(0)} (\mathbf{x}^{\mathrm{t}}B \mathbf{x})  \mathrm{d}\mathbf{x}.
\]

Since $B$ is symmetric and positive definite, there exists a real symmetric positive definite matrix $S$ such that $B = S^2$. Apply the change of coordinates $\mathbf{y} := S\mathbf{x}$. Then $\mathbf{x}^{\mathrm{t}}B \mathbf{x} = |\mathbf{y}|^2$ for the standard metric on $X^*_{\R}$, \textit{i.e.} $|(y_1,\ldots,y_r)|^2 = \sum_{j=1}^r y_j^2$. We have $\mathrm{d}\mathbf{y} =\det(S)\mathrm{d}\mathbf{x} =  d^{1/2} (\#\Phi)^{-1/2}\mathrm{d}\mathbf{x}$. From now on, we will consider the volume form $\mathrm{d}\mathbf{y}$ on $X^*_{\R}$.

Now $\Omega := S\cdot \mathrm{Vor}(0)$ is a Euclidean region such that
\[
\vol_{\mathrm{d}\mathbf{y}}(\Omega) =  \det(S)  \vol_{\mathrm{d}\mathbf{x}}(\mathrm{Vor}(0))=   \sqrt{d \#\Phi},
\]
and
\begin{equation}\label{EqMomentInY}
I(\Sigma) = \frac{1}{\sqrt{d\#\Phi}} \int_{\Omega} |\mathbf{y}|^2 \mathrm{d}\mathbf{y}.
\end{equation}

Take the ball $B(0,R)$ centered at $0$ of radius $R := \Gamma(\frac{r}{2}+1)^{1/r}\pi^{-1/2}(d\#\Phi)^{1/2r}$; then its Euclidean volume  $\vol_{\mathrm{d}\mathbf{y}}(B(0,R))$ equals  $\sqrt{d\#\Phi} = \vol_{\mathrm{d}\mathbf{y}}(\Omega)$. Now
 $|\mathbf{y}|^2 \le R^2$ for all $\mathbf{y} \in B(0,R) \setminus \Omega$ and $|\mathbf{y}|^2 \ge R^2$ for all $\mathbf{y} \in \Omega \setminus B(0,R)$. So
 \[
\int_{\Omega \setminus B(0,R)} |\mathbf{y}|^2 \mathrm{d}\mathbf{y} \ge R^2 \vol_{\mathrm{d}\mathbf{y}}(B(0,R) \setminus \Omega) = R^2 \vol_{\mathrm{d}\mathbf{y}}(\Omega \setminus B(0,R)) \ge \int_{B(0,R) \setminus \Omega} |\mathbf{y}|^2 \mathrm{d}\mathbf{y}.
\]
It follows that 
\small 
\[
\int_{\Omega} |\mathbf{y}|^2 \mathrm{d}\mathbf{y} \ge \int_{B(0,R)} |\mathbf{y}|^2 \mathrm{d}\mathbf{y} = \int_0^R t^2 t^{r-1} 2\pi^{r/2} \Gamma(r/2)^{-1} \mathrm{d}t = 2\pi^{r/2} \Gamma(r/2)^{-1}(r+2)^{-1} R^{r+2}.
\]
This yields the conclusion by \eqref{EqMomentInY} and our choice of $R$ and because $d\ge 1$.
\end{proof}

\section{Proofs of the main results over function fields: semi-stable case}\label{SectionSS}
The goal of this section is to prove our main results over function fields when $A/K$ has semi-stable reduction. More precisely, we prove Theorem~\ref{ThmUBCFF} as well as Theorem~\ref{ThmLangSilvermanFF} (if $L$ defines a principal polarization) assuming $A/K$ has semi-stable reduction. In the whole section, denote by $K_0 := \C(\mathbb{P}^1)$. Then $[K:K_0] = \mathrm{gon}(B)$. Denote by $g(B)$ the genus of $B$. We prove:

\begin{thm}\label{ThmUBCFFOrig}
If $A/K$ has semi-stable reduction and no abelian subvariety of $A$ defined over $K$ has good reduction everywhere, then we have
\begin{equation}\label{EqUBCFFOrig}
\#A(K)_{\mathrm{tor}} \le h^0(A,L) \cdot \left( 24\cdot 512 e\cdot g^4(g+2)\gamma_g\right)^g \left( 2g(B)-1 \right)^{2g}
\end{equation}
for any polarization $L$ on $A$.
\end{thm}
Before moving on, let us point out that Theorem~\ref{ThmUBCFFOrig} immediately implies Theorem~\ref{ThmUBCFF} when $A/K$ has semi-stable reduction by Zarhin's trick and the (short) argument in §\ref{SubsubsectionCaseGoodEverywhere}.

%Before moving on, let us explain how Theorem~\ref{ThmUBCFFOrig} implies Theorem~\ref{ThmUBCFF} when $A/K$ has semi-stable reduction.  By Zarhin's trick, $(A\times A^\vee)^4$ carries a principle polarization. Then we are done by applying Theorem~\ref{ThmUBCFFOrig}  to $(A\times A^\vee)^4$ (which still satisfiies the assumption of this theorem) and noticing  $\#A(K)_{\mathrm{tor}} \le \#(A\times A^\vee)^4(K)_{\mathrm{tor}}$.

\subsection{Local and global Faltings height}
In this subsection, assume that $A$ has semi-stable reduction. 
Assume also that $A$ does {\it not} have good reduction everywhere, \textit{i.e.} $\#\mathfrak{Bad}(A/K) \ge 1$. To ease notation, denote by
\[
\mathfrak{b} := \#\mathfrak{Bad}(A/K).
\]

We prove that the stable Faltings height $h_{\mathrm{Fal}}(A)$ is comparable to the sum of the local invariants $\sum_{v \in \mathfrak{Bad}}(\#\Phi_v)^{-1/r_v}$. The key is to  use is the following {\it Arakelov Inequality} proved by Deligne \cite{Deligne87}, which is the function field version of Szpiro's conjecture. 

\begin{thm}[Deligne]\label{ThmDeligne}
Assume $A/K$ has semi-stable reduction. Then
\[
[K:  K_0 ] h_{\mathrm{Fal}}(A) \le \frac{g}{2} \big(2 g(B)-2 +\mathfrak{b}  \big)
\]
whenever the right hand side is $\ge 0$.
\end{thm}

We use Jensen's Inequality and Theorem~\ref{ThmDeligne} to prove the following proposition:
\begin{prop}\label{PropConsequenceDeligne}
Assume $\mathfrak{b} \ge 1$ (and $\mathfrak{b} \ge 3$ if $g(B) = 0$). Then 
\begin{equation}\label{EqConsequenceDeligne}
\frac{[K:K_0]  h_{\mathrm{Fal}}(A) }{\sum_{v\in \mathfrak{Bad}(A/K)}  (\#\Phi_v)^{-1/r_v}} \le 24 g^2\left(2g(B)-1\right)^2.
\end{equation}
and
\begin{equation}\label{EqSumIFF1}
\frac{ [K:K_0]\max\{h_{\mathrm{Fal}}(A),1\}}{\sum_{v \in \mathfrak{Bad}(A/K)} (\#\Phi_v)^{-1/r_v}} \le 48g^2[K:K_0](2g(B)-1)^2.
\end{equation}
\end{prop}
\begin{proof}
By Jensen's Inequality applied to the function $t\mapsto t^{-1}$, we have
\[
\left(\sum_{v \in \mathfrak{Bad}(A/K)} (\#\Phi_v)^{1/r_v} \right) \left(\sum_{v \in \mathfrak{Bad}(A/K)} (\#\Phi_v)^{-1/r_v}  \right) \ge  \big(\#\mathfrak{Bad}(A/K)\big)^2 =  \mathfrak{b}^2.
\]
Moreover $\sum_{v \in \mathfrak{Bad}(A/K)} (\#\Phi_v)^{1/r_v} \le 96 [K:K_0]h_{\mathrm{Fal}}(A)$ by Proposition~\ref{PropHP}. So the inequality above further implies
\begin{equation}\label{EqEqEqEqFF}
\sum_{v \in \mathfrak{Bad}(A/K)} (\#\Phi_v)^{-1/r_v} \ge \frac{1}{96} \frac{ \mathfrak{b}^2 }{[K:  K_0 ] h_{\mathrm{Fal}}(A)}.
\end{equation}
Thus
\begin{align*}%\label{EqSumIFF1}
\begin{split}
\frac{[K:K_0]h_{\mathrm{Fal}}(A)}{\sum_{v \in \mathfrak{Bad}(A/K)} (\#\Phi_v)^{-1/r_v}}  & \le 96 \left(\frac{[K:K_0]h_{\mathrm{Fal}}(A)}{\mathfrak{b}}\right)^2   \\
& \le 96 \left(\frac{g \left(2g(B)-2+\mathfrak{b}\right)}{2\mathfrak{b}}\right)^2 \quad\text{ by Theorem~\ref{ThmDeligne}} \\
& = 24 g^2 \left( \frac{2g(B)-2}{\mathfrak{b}} + 1\right)^2  \le 24 g^2 \max\{(2g(B)-1)^2, 1\} 
\end{split}
\end{align*}
The last inequality holds true since $\mathfrak{b} \ge 1$ (and $\mathfrak{b} \ge 3$ if $g(B)=0$). This proves \eqref{EqConsequenceDeligne}.

On the other hand, applying Theorem~\ref{ThmDeligne} to  \eqref{EqEqEqEqFF}, we get
\begin{equation}\label{EqSumIFF2}
\frac{1}{\sum_{v \in \mathfrak{Bad}(A/K)} (\#\Phi_v)^{-1/r_v}} \le \frac{96 g (2g(B)-2+\mathfrak{b})}{2 \mathfrak{b}^2} = 48g \left( \frac{2g(B)-2}{\mathfrak{b}^2} + \frac{1}{\mathfrak{b}} \right) \le 48g \max\{2g(B)-1 , 1\},
\end{equation}
where the last inequality follows from $\mathfrak{b} \ge 1$ (and $\mathfrak{b} \ge 3$ if $g(B)=0$). Then we immediately get \eqref{EqSumIFF1} by comparing \eqref{EqConsequenceDeligne} and $[K:K_0]\cdot$\eqref{EqSumIFF2}.
\end{proof}

%\subsection{Proof of the bound of $\#A(K)_{\mathrm{tor}}$ (Theorem~\ref{ThmUBCFF})}
\subsection{Proof of Theorem~\ref{ThmUBCFFOrig}}
By the Poincar\'{e} Irreducibility Theorem, $A$ is isogenous to a product of $K$-simple abelian varieties. Each such $K$-simple abelian varieties (counted with multiplicity) is called a {\it factor} of $A$. By our assumption, we know that no factor of $A$ has good reduction everywhere.

%\subsubsection{Case: no factor of $A$ has good reduction everywhere}\label{SubsubsectionCaseNotGoodEverywhere}
%Assume no factor of $A$ has good reduction everywhere. 

We will proceed by induction on  the number of factors of $A$, which we denote by $n$.

\noindent\boxed{\text{Base step: $n=1$}}
In this case $A$ is simple. Then our assumption becomes: $A$ does not have good reduction everywhere, \textit{i.e.} $\mathfrak{b} = \#\mathfrak{Bad}(A/K) \ge 1$ (and  $\mathfrak{b}  \ge 3$ if $g(B) = 0$). Hence we can apply Theorem~\ref{ThmFirstBoundFF} and get
\begin{equation}\label{EqSimpleFF}
\#A(K)_{\mathrm{tor}} \le  h^0(A,L) \left( 512 e\cdot g^2(g+2)\gamma_g \frac{[K:K_0] h_{\mathrm{Fal}}(A)}{\sum_{v \in \mathfrak{Bad}(A/K)} (\#\Phi_v)^{-1/r_v}}  \right)^g.
\end{equation}
Hence the conclusion follows immediately from \eqref{EqConsequenceDeligne}.

\noindent{\boxed{\text{Induction step}}} For arbitrary $n \ge 2$. 
By Theorem~\ref{ThmFirstBoundFF}, there exists a proper abelian subvariety $A'$ of $A$ defined over $K$ such that 
\[
\#(A/A')(K)_{\mathrm{tor}} \le  \frac{h^0(A,L)}{h^0(A',L)} \left( 512 e\cdot g^2(g+2)\gamma_g \frac{ [K:K_0]h_{\mathrm{Fal}}(A) }{\sum_{v \in \mathfrak{Bad}(A/K)} (\#\Phi_v)^{-1/r_v}}  \right)^{g-\dim A'}.
\]
Hence by \eqref{EqConsequenceDeligne}, we have
\[
\#(A/A')(K)_{\mathrm{tor}} \le  \frac{h^0(A,L)}{h^0(A',L)} \left( 24\cdot 512 e\cdot g^4(g+2)\gamma_g\right)^{g-\dim A'} \left( 2g(B)-1 \right)^{2(g-\dim A')}.
\]

Since $A'\not=A$, the number of factors of $A'$ is $\le n-1$. 
So we can apply  the induction hypothesis to $A'$ and get 
\[
\#A'(K)_{\mathrm{tor}} \le h^0(A',L)  \left( 24\cdot 512 e\cdot g^4(g+2)\gamma_g\right)^{\dim A'}  \left( 2g(B)-1 \right)^{2\dim A'}.
\]
Thus the desired \eqref{EqUBCFFOrig} follows from the product of the two bounds above. 
We are done.
\qed

\subsection{Proof of a weak version of Theorem~\ref{ThmLangSilvermanFF} in the semi-stable case}\label{SubsectionLSFFSemistable}
To present the idea how our bounds allow to treat small rational points, we hereby present a simple proof of Theorem~\ref{ThmLangSilvermanFF} when $A/K$ has semi-stable reduction, {\it provided that} $L$ is a principal polarization on $A$. For general $L$, we need to use {\it Zarhin's box} as in §\ref{SubsectionLSFF}.

Let $A, L$ be as in Theorem~\ref{ThmLangSilvermanFF}, and $A/K$ has semi-stable reduction. 
Let us prove prove:
\begin{equation}\label{EqLowerBoundRationalPtFF}
%\hat{h}_L(P) > \frac{c_0'''(g)}{h^0(A,L)^2} \frac{h_{\mathrm{Fal}}(A)}{ \left(2g(B)-1\right)^{4g+2}}.
\hat{h}_L(P) > \frac{c_0'''(g)}{h^0(A,L)^2} \frac{\max\{h_{\mathrm{Fal}}(A),1\}}{ \mathrm{gon}(B) \cdot \left(2g(B)-1\right)^{4g+2}}.
\end{equation}
for any $P \in A(K)$ such that $\bar{\Z\cdot P}^{\mathrm{Zar}} = A$, with
\[
c_0'''(g) = \left( 48\cdot 2048(24\cdot 1024)^{2g}  e^{2g+1} g^{8g+4}(g+2)^{2g}\gamma_g^{2g+1} \right)^{-1}.
%c_0'''(g) = \left( 48\cdot 96000(24\cdot 48000)^{2g}  g^{8g+4}(g+2)^{2g}\gamma_g^{2g+1} \right)^{-1}.
\]

%By \eqref{EqConsequenceDeligne}, we have
Denote by $K_0 = \C(\mathbb{P}^1)$. Then $\mathrm{gon}(B) = [K:K_0]$. 
By \eqref{EqSumIFF1}, we have
\[
%\frac{1}{2048 g^2 \gamma_g[K:K_0]}\sum_{v \in \mathfrak{Bad}(A/K)} (\#\Phi_v)^{-1/r_v} \ge \frac{1}{24\cdot 2048 g^4 \gamma_g} \cdot \frac{h_{\mathrm{Fal}}(A)}{(2g(B)-1)^2}.
\frac{1}{2048 g^2 \gamma_g[K:K_0]}\sum_{v \in \mathfrak{Bad}(A/K)} (\#\Phi_v)^{-1/r_v} \ge \frac{1}{48\cdot 2048 g^4 \gamma_g} \cdot \frac{\max\{h_{\mathrm{Fal}}(A),1\}}{[K:K_0](2g(B)-1)^2}.
\]
So the set $\Gamma'$ defined in \eqref{EqGammaPrimeFF} contains the following susbset
\[
%\Gamma'' := \left\{x\in A(K) : \hat{h}_L(x) \le \frac{1}{24\cdot 2048 g^4 \gamma_g} \cdot \frac{\max\{h_{\mathrm{Fal}}(A),1\}}{(2g(B)-1)^2} \right\}
\Gamma'' := \left\{x\in A(K) : \hat{h}_L(x) \le \frac{1}{48\cdot 2048e\cdot g^4 \gamma_g} \cdot \frac{\max\{h_{\mathrm{Fal}}(A),1\}}{[K:K_0](2g(B)-1)^2} \right\}.
\]

We wish to invoke Theorem~\ref{ThmFirstBoundPrimeFF}. Set
\[
m := h^0(A,L) \cdot \left( 24 \cdot 1024e\cdot g^4(g+2)\gamma_g  \right)^g \left( 2g(B)-1 \right)^{2g}. %{C_4^{1+\varepsilon}}\right)^g = h^0(A,L) \cdot \left(C'_1 \left(\frac{2}{\min\{C_2,C_3\}}[K:K_0]\right)^{1+\varepsilon} \right)^g.
\]
Then the right hand side of \eqref{EqLowerBoundRationalPtFF} is  %$\frac{1}{m^2 24\cdot 2048 g^4\gamma_g} \frac{h_{\mathrm{Fal}}(A)}{(2g(B)-1)^2}$.
$\frac{1}{m^2 48\cdot 2048e\cdot g^4\gamma_g} \frac{\max\{h_{\mathrm{Fal}}(A),1\}}{[K:K_0](2g(B)-1)^2}$.

Assume \eqref{EqLowerBoundRationalPtFF} is false. Then the rational points 
\[
\pm P, ~ [\pm 2]P,~ \ldots, ~ [\pm m]P
\]
 are in $\Gamma''$, and hence in $\Gamma'$.

For {\it any} proper abelian subvariety $A'$ of $A$ defined over $K$, our assumption $\bar{\mathbb{Z}\cdot P}^{\mathrm{Zar}} = A$ implies that
\[
q(\pm P), q([\pm 2]P), \ldots, q([\pm m]P)
\]
are $2$-by-$2$ distinct for the quotient $q \colon A \rightarrow A/A'$. In particular, $\#\!\left(\frac{\Gamma''+A'}{A'}\right) \ge 2m >m$. So
\[
 \left( \#\!\left(\frac{\Gamma'+A'}{A'}\right) \cdot \frac{h^0(A',L)}{h^0(A,L)} \right)^{\frac{1}{g-\dim A'}} > \left(m \cdot \frac{h^0(A',L)}{h^0(A,L)} \right)^{1/g} \ge \left(\frac{m}{h^0(A,L)}\right)^{1/g},
\]
and hence, by our choice of $m$ and \eqref{EqConsequenceDeligne},
\[
 \left( \#\!\left(\frac{\Gamma'+A'}{A'}\right) \cdot \frac{h^0(A',L)}{h^0(A,L)} \right)^{\frac{1}{g-\dim A'}} >  1024 e\cdot g^2(g+2)\gamma_g \frac{ [K:K_0] h_{\mathrm{Fal}}(A)  }{\sum_{v \in \mathfrak{Bad}(A/K)} (\#\Phi_v)^{-1/r_v}}.
\]
This violates Theorem~\ref{ThmFirstBoundPrimeFF}. So \eqref{EqLowerBoundRationalPtFF} holds true. We are done.
\qed

\section{From semi-stable to general case}\label{SectionGeneral}
For simplicity let $K=\C(B)$ in this section. The treatment also works for  number fields $K$ {\it if} there exists an extension $K'/K$ such that (i) $A_{K'}/K'$ has semi-stable reduction, (ii) $K'/K$ is {\it tamely} ramified at all $v\in M_K^0$ where $A$ has bad but potentially good reduction.\footnote{So the semi-stable assumption in Theorem~\ref{ThmMainNFUnconditional} and Theorem~\ref{ThmMainNF} can be relaxed to this (with different constants), if one combines the discussion in this section and all other treatments over number fields.}

Throughout the whole section, denote by $K_0 := \C(\mathbb{P}^1)$. Then $[K:K_0] = \mathrm{gon}(B)$. 
%for example $A_{K'}$ has semi-stable reduction for some {\it tamely} ramified extension $K'/K$.%the ramification of $K'/K$ to obtain semi-stable reduction is tame; see $\mathsection$ \ref{subsec:SettingupK'}. 

\begin{thm}\label{ThmFirstBoundFFNonSS}
Assume $A/K$ does not have good reduction everywhere. Then there exists a proper abelian subvariety $A'$ of $A$ defined over $K$ such that
\begin{equation}\label{EqFirstBoundFFNonSS}
 \left( \#\!\left(\frac{A(K)_{\mathrm{tor}}+A'}{A'}\right) \cdot \frac{h^0(A',L)}{h^0(A,L)} \right)^{\frac{1}{g-\dim A'}} \le  512\cdot 3^{12g^2+2}g^4(g+2) \gamma_g \left(g(B)+1\right)^2.
% \left( \#\!\left(\frac{A(K)_{\mathrm{tor}}+A'}{A'}\right) \cdot \frac{h^0(A',L)}{h^0(A,L)} \right)^{\frac{1}{g-\dim A'}} \le  512\cdot 3^{12g^2+1}g^4(g+2) \gamma_g \left(g(B)+1\right).
\end{equation}
\end{thm}
%If $A/K$ furthermore has semi-stable reduction, then \eqref{EqFirstBoundFFNonSS} follows immediately from 
In previous sections, we proved \eqref{EqFirstBoundFFNonSS} when $A/K$ furthermore has semi-stable reduction; this is the combination of Theorem~\ref{ThmFirstBoundFF} and Proposition~\ref{PropConsequenceDeligne}. Theorem~\ref{ThmFirstBoundFFNonSS} is the generalization to arbitrary $A/K$ not having good reduction everywhere. Its proof is an adaption of the semi-stable case by adding local contributions at places of bad but potentially good reductions.

Again, define $\mathfrak{N}$ to be the minimum of the left hand side of \eqref{EqFirstBoundFFNonSS} when $A'$ runs over all proper abelian subvarieties of $A$ defined over $K$ (or defined over $\bar{K}$ -- this does not change $\mathfrak{N}$ by Gaudron--R\'emond \cite[Lem.~6.1]{GR25}). Then Theorem~\ref{ThmFirstBoundFFNonSS} is equivalent to:
\begin{equation}\label{EqNCrudeBoundNonSS}
\mathfrak{N}  \le   512\cdot 3^{12g^2+2}g^4(g+2) \gamma_g \left(g(B)+1\right)^2.
\end{equation}

\vskip 0.3em
We can also prove a similar bound for small rational points. 
Define
\begin{equation}\label{EqSmallRationalNonSS}
\Gamma := \left\{ P \in A(K) : \hat{h}_L(P) \le \frac{\max\{h_{\mathrm{Fal}}(A),1\}}{2048\cdot 3^{12g^2+2}g^4 \gamma_g [K:K_0] \left(g(B)+1\right)^2}  \right\}.
\end{equation}

\begin{thm}\label{ThmFirstBoundFFNonSSSmallPoint}
Assume $A/K$ does not have good reduction everywhere. Then there exists a proper abelian subvariety $A'$ of $A$ defined over $K$ such that
\begin{equation}\label{EqFirstBoundFFNonSSSmallPoint}
 \left( \#\!\left(\frac{\Gamma+A'}{A'}\right) \cdot \frac{h^0(A',L)}{h^0(A,L)} \right)^{\frac{1}{g-\dim A'}} \le 1024\cdot 3^{12g^2+2}g^4(g+2) \gamma_g \left(g(B)+1\right)^2.
% \left( \#\!\left(\frac{\Gamma+A'}{A'}\right) \cdot \frac{h^0(A',L)}{h^0(A,L)} \right)^{\frac{1}{g-\dim A'}} < 48000\cdot 3^{12g^2+1}g^4(g+2) \gamma_g [K:K_0] \left(g(B)+1\right).
\end{equation}
\end{thm}

\subsection{Setting up}\label{subsec:SettingupK'}
Let $K' := K(A[3])$. Then $A_{K'}$ has semi-stable reduction over $K'$. We have that $K'/K$ is Galois, and $[K':K] \le \#\mathrm{GL}_{2g}(\Z/3\Z) \le 3^{4g^2}$. Moreover. $K' = \C(B')$ for a smooth projective irreducible curve $B'$ defined over $\C$, whose genus $g(B')\ge 1$.

Most computation in this section is done over $K'$ instead of $K$.

%Take a projective model $\cA' \rightarrow B'$ of $A_{K'}/K$, and let $\cL$ extend $L_{K'}$. 

\subsection{Local discussion}
Let $v \in M_K$ be a place over which $A$ has bad but potentially good reduction. Denote by $F:=K_v$.% Let $e(v|v_0)$ be the ramification index of $v$ over $v_0 := v|_{K_0 }$.% Then $e(v|v_0) \le [K:K_0]$.

Let $v' \in M_{K'}$ be a place lying over $v$; then $v'$ can be seen as a point $B'(\C)$. Let $e(v'|v)$  be the ramification index of $v'$ over $v$. Then $e(v'|v) \le [K':K]\le 3^{4g^2}$.%, and $e(v'|v_0) = e(v'|v)e(v|v_0)$.

Let $F' \subseteq K_{v'}$ be the totally ramified extension of $F$ such that $A_{F'}$ has good reduction. Then $[F':F] = e(v'|v)$. Let $\varpi'$ be the uniformizer of $F'$, and take $\log|\varpi'| = -1$; this is compatible with the convention in $\mathsection$\ref{SubsectionPreliminaryAdelicGeometry} when we work with $K'$. Since $K$ is a function field of characteristic $0$, the extension $F'/F$ is tamely ramified, and hence we can apply results of Edixhoven \cite{Edixhoven1992} -- this is the only place we need to work over $K=\C(B)$.
% and Chai \cite{}.

Let $\mathcal N$ be the N\'eron model of $A$ over $\cO_F$, and let $\cA'$ be the N\'eron model of $A_{F'}$ over $\cO_{F'}$. The universal property of N\'eron models yields a natural morphism
\[
\varphi \colon \mathcal N\otimes_{\cO_F} \cO_{F'} \longrightarrow \cA',
\]
which induces $\varphi_{v'} \colon \mathcal{N}_{v'} \rightarrow \cA'_{v'}$. Denote by 
\begin{equation}\label{EqCodimNeron}
q:= g - \dim \varphi_{v'}(\mathcal{N}_{v'}) = \codim_{\cA'_{v'}}(\varphi_{v'}(\mathcal{N}_{v'})) \ge 1.
\end{equation}
In suitable formal coordinates, the differential map $\mathrm{d}\varphi$ is given by
\[
\mathrm{d}\varphi \colon \lie(\mathcal N)\otimes_{\cO_F}\cO_{F'} \longrightarrow \lie(\cA'), \quad 
(u_1,\ldots,u_{g-q}, y_1,\ldots,y_q) \mapsto (u_1,\ldots,u_{g-q}, \varpi^{\prime b_1}y_1,\ldots,\varpi^{\prime b_q}y_q)
\]
for some positive integers (called the {\it positive elementary divisors} of $\mathrm{d}\varphi$)
\begin{equation}\label{EqPED}
b_1,\ldots,b_q \in \Z_{\ge 1},
\end{equation}
where $u_1,\ldots,u_{g-q}$ is a local coordinate system of $\varphi_{v'}(\mathcal{N}_{v'})$ at the origin and $y_1,\ldots,y_q$ is a local coordinate system of the complementary abelian variety of $\varphi_{v'}(\mathcal{N}_{v'})^{\circ}$ at the origin.

\begin{lemma}\label{LemmaDegree}
$(g-q)! \cdot h^0(\varphi_{v'}(\mathcal{N}_{v'}),\cL) \le 3^q e(v'|v)^{q+1}  g! \cdot h^0(A,L)$.
\end{lemma}
\begin{proof}
This algebro-gemetric result is proved by Looper--Yap \cite[Cor.~7.4]{LooperYap}, and the key ingredient is a result of Edixhoven (reinterpreted as \cite[Lem.~3.2.(2)]{HN}). For the sake of completeness, we include a sketch of the proof here.

The Galois group $\mathrm{Gal}(F'/F) \cong \mu_{e(v'|v)}$, and by \cite[Lem.~3.2.(2)]{HN} we have $\varphi_{v'}(\mathcal{N}_{v'}) = (\cA'_{v'})^{\mu_{e(v'|v)}}$. Let $\cM := \bigotimes_{a \in \mu_{e(v'|v)}}a^*\cL$. Then $\cM$ carries a natural $\mu_{e(v'|v)}$-linearization since the action of $\mu_{e(v'|v)}$ permutes the tensor factors. Since $a^*\cL$ is numerically equivalent to $\cL$ for every $a\in\mu_{e(v'|v)}$, we have $\cM\equiv\cL^{\otimes e(v'|v)}$. The line bundle $\cM^{\otimes 3}$ is very ample and gives a $\mu_{e(v'|v)}$-equivariant projective embedding. Since the fixed locus in the ambient projective space is a union of at most ${e(v'|v)}$ linear subspaces, B\'ezout gives
\vskip -1em
\[
\deg_{\cM^{\otimes 3}}(\varphi_{v'}(\mathcal{N}_{v'}))\le e(v'|v)\deg_{\cM^{\otimes 3}}(\cA'_{v'}).
\]
As $\dim \varphi_{v'}(\mathcal{N}_{v'})=g-q$, this becomes
\[
(3e(v'|v))^{g-q} (g-q)! h^0(\varphi_{v'}(\mathcal{N}_{v'}),\cL)= (3e(v'|v))^{g-q}\deg_{\cL}(\varphi_{v'}(\mathcal{N}_{v'}))\le e(v'|v)(3e(v'|v))^g\deg_{\cL}(\cA'_{v'}).
\]
We are done since $ \deg_{\cL}(\cA'_{v'}) = \deg_L(A) = g!h^0(A,L)$.
\end{proof}

To proceed, we need to study formal neighborhoods of $\varphi_{v'}(\mathcal{N}_{v'})$ in $\cA'$  as follows. 
Denote by $\fm'$ the maximal ideal of $\cO_{F'}$. For any $n \ge 0$, let $\varphi_n \colon \mathcal{N}\otimes_{\cO_F}\cO_{F'}/\fm^{\prime n+1} \to \cA' \otimes_{\cO_{F'}} \cO_{F'}/ \fm^{\prime n+1}$ be induced by $\varphi$ and set
\vskip -1em
\[
N_n := \text{the schematic image of }\mathcal{N}\otimes_{\cO_F}\cO_{F'}/\fm^{\prime n+1}.
\]
Then $N_0 = \varphi_{v'}(\mathcal{N}_{v'})$ and, for any $0\le i \le n$, we have
\[
\hat{\cO}_{N_{i}} \cong \hat{\cO}_{N_0} \hat{\otimes}_{\C} \C[\varpi, y_1,\ldots,y_q]/(\varpi^{\prime j}y_1^{\nu_1} \cdots y_q^{\nu_q} : j+ b_1\nu_1+\ldots + b_q \nu_q \ge i +1).
\]
So the rank of the $\cO_{N_0}$-module $\ker( \cO_{N_{i}} \to \cO_{N_{i-1}})$ is
\[
l_{i} := \# \left\{ (j, \nu_1,\ldots,\nu_q) \in \Z_{\ge 0}^{q+1} : j+ b_1\nu_1+\cdots + b_q\nu_q = i \right\}.
\]
Along the generic point of each component of $N_0$, the generic length of $N_n$ is
\[
\tau_n := \sum\nolimits_{i=0}^n l_i = \# \left\{ (j, \nu_1,\ldots,\nu_q) \in \Z_{\ge 0}^{q+1} : j+ b_1\nu_1+\cdots + b_q\nu_q \le n \right\}.
\]
\begin{lemma}\label{LemmaUniformHS}
%The followings hold true:
\begin{enumerate}
\item[(i)] 
$\displaystyle 
\tau_n = \frac{n^{q+1}}{(q+1)! b_1\cdots b_q} + O(n^q), \quad \text{ as } n \to \infty$.
\item[(ii)] 
$h^0(N_n, m\cL) \le \tau_n \cdot h^0(N_0, m\cL)$.
\end{enumerate}
\end{lemma}
\begin{proof} The generating series of the numbers $l_i$ is
\vskip -1.5em
\[
\sum_{i \ge 0} l_i T^i = \sum_{j,\nu_1,\ldots,\nu_q\ge 0} T^{j+b_1\nu_1+\cdots+b_q\nu_q} = \Big(\sum_{j\ge 0}T^j \Big)\prod_{k=1}^q \Big(\sum_{\nu_i\ge 0}T^{b_k\nu_k} \Big)= \frac{1}{1-T}\prod_{k=1}^q \frac{1}{1-T^{b_k}}.
\]
Hence  the generating series of $\tau_n$ is 
\[
\sum_{n\ge 0}\tau_n T^n = \sum_{n\ge 0}\sum_{i=0}^n l_i T^n = \sum_{i \ge 0}l_i \sum_{n\ge i}T^n = \frac{1}{1-T} \sum_{i \ge 0}l_i T^i  =   \frac{1}{(1-T)^2 \prod_{k=1}^q(1-T^{b_k})}.
\]
Near $T=1$, we have $1-T^{b_k} = b_k(1-T) + O((1-T)^2)$, and hence
\[
\frac{1}{(1-T)^2 \prod_{k=1}^q(1-T^{b_k})} = \frac{1}{b_1\cdots b_q}\frac{1}{(1-T)^{q+2}} + O\left(\frac{1}{(1-T)^{q+1}}\right).
\]
Taking the coefficient of $T^{r-1}$, we get
\[
\tau_n = \frac{1}{b_1\cdots b_q} \binom{n+q+1}{q+1}+O(n^q), \quad \text{ as } n \to \infty.
\]
This establishes (i).

For (ii), consider the (regular) embedding $N_0 \hookrightarrow \cA'_{v'}$ of smooth algebraic groups defined over $\C$. The normal bundle is $\cong (\lie \cA'_{v'} / \lie N_0) \otimes_{\C} \cO_{N_0} \cong \cO_{N_0}^{\oplus q}$. Taking the dual yields
\[
\cI_{N_0}/\cI_{N_0}^2 \cong \cO_{N_0}^{\oplus q}.
\]
Thus all symmetric powers of $\cI_{N_0}/\cI_{N_0}^2$ are also trivial, and therefore
\[
\ker( \cO_{N_{i}} \to \cO_{N_{i-1}})  \cong \bigoplus_{\substack{j,\nu_1,\ldots,\nu_q \ge 0 , ~~ j+b_1\nu_1+\ldots+b_q\nu_q=i}}\cO_{N_0} = \cO_{N_0}^{\oplus l_i}.
\]
This gives a short exact sequence, for all $i \ge 1$,
\[
0 \to \cO_{N_0}^{\oplus l_i} \to \cO_{N_{i}} \to \cO_{N_{i-1}} \to 0.
\]
Tensoring with $\cL^{\otimes m}$ and taking global sections, we get
\[
h^0(N_{i}, \cL^{\otimes m}) \le h^0(N_{i-1}, \cL^{\otimes m}) + l_{i} h^0(N_0, \cL^{\otimes m})
\]
We are done for (ii) by taking the sum of the inequality above over $1\le i\le n$.
\end{proof}

Set
\begin{equation}\label{EqChoiceOfcNonSS}
\lambda := \frac{1}{3e(v'|v)}\left( \binom{g}{q}^{-1} \frac{b_1\cdots b_q}{e(v'|v)} \right)^{1/q} \ge \frac{1}{3g}e(v'|v)^{-1-1/q}.
\end{equation}
%A direct computation using Lemma~\ref{LemmaDegree}, Lemma~\ref{LemmaUniformHS}, and $\log|\varpi'|=1/e(v'|v_0)$ implies:
\begin{cor}\label{CorNonSSBound}
Let $n = \lfloor \lambda m \rfloor$. Then we have
\[
h^0(A,mL) \log |\varpi'|^{-n} - h^0(N_n,m\cL) \ge h^0(A,L)m^{g+1} \frac{\lambda}{2} - O(m^g ).
\]
\end{cor}
\begin{proof}
We have
\begin{align*}
h^0(N_n,m\cL) & \le \left(\frac{n^{q+1}}{(q+1)!b_1\cdots b_q} + O(n^q) \right)  m^{g-q} h^0(N_n, \cL)  \qquad \text{ by Lemma~\ref{LemmaUniformHS}} \\
& \le \left(\frac{n^{q+1}}{(q+1)!b_1\cdots b_q} + O(n^q) \right) m^{g-q} \frac{3^q e(v'|v)^{q+1} g! h^0(A,L)}{(g-q)!} \quad	\text{ by Lemma~\ref{LemmaDegree}} \\
& = 3^q  h^0(A,L)\frac{1}{g+1}\binom{g+1}{q+1} \frac{\lambda^{q+1} e(v'|v)^{q+1} m^{g+1}}{b_1\cdots b_q} + O(m^g) 	\text{ since }n=\lfloor \lambda m \rfloor = \lambda m+O(1) \\
& = \frac{\lambda}{q+1} h^0(A,L) m^{g+1} + O(m^g) \qquad\text{ by choice of $\lambda$ from \eqref{EqChoiceOfcNonSS}.}
\end{align*}
But $h^0(A,mL) = m^g h^0(A,L)$ and $\log|\varpi'| = -1$ and $n = \lambda m +O(1)$, so
\begin{equation}
h^0(A,mL) \log |\varpi'|^{-n} -  h^0(N_n,m\cL) \ge h^0(A,L)m^{g+1} \frac{q}{q+1}\lambda - O(m^g).
\end{equation}
But $1\le q\le g$, so $\frac{q}{q+1}\ge \frac{1}{2}$. We are done.
\end{proof}
%\begin{proof}
% So
%%\begin{align*}
%h^0(A,mL)  \log |\varpi'|^{-r} -  \dim_{\C} H^0(N_r,m\cL) & \ge m^g h^0(A,L) rd - \mu_r \cdot h^0(N_0, \cL^{\otimes m}) \\
%& = m^g h^0(A,L) \left( rd - \frac{r^{q+1}}{(q+1)!b_1\cdots b_q} \frac{h^0(N_0, \cL)}{m^q h^0(A,L)} \right) \\
%& = m^g h^0(A,L) \left( rd - \frac{r^{q+1}}{(q+1)!b_1\cdots b_q} \frac{g!\deg_{\cL}(N_0)}{(g-q)! m^q  \deg_L(A)} \right) \\
%& \ge m^g h^0(A,L) d \left( r - \frac{r^{q+1}}{(q+1)!b_1\cdots b_q} \frac{g!}{(g-q)! m^q } \right)
%\end{align*}
%%We are done.
%\end{proof}

\subsection{Global computation and Proof of Theorem~\ref{ThmFirstBoundFFNonSS}}\label{SubsectionProofOfThmFirstBoundFFNonSS}
Let us proceed to prove Theorem~\ref{ThmFirstBoundFFNonSS} in several steps. The proof is an adaption of $\mathsection$\ref{SectionProofOfFirstBound}.  
Since most computation is done over $K'$, by abuse of notation we let $\bar{L}$ be the canonical adelic extension of $L_{K'} = L\otimes_K K'$. Then 
we adopt the notation $\epsilon = 1/(8g)$ and the construction of the perturbed adelic line bundle $\bar{L}(\epsilon \mathbf{f})$ from $\mathsection$\ref{SubsectionSmallSection} (which is over $A_{K'}$ now).

\subsubsection{Set up of the local invariants}
Set
\[
\mathfrak{B}^0_{A/K'}
:=
\{v'\in M_{K'}: A_{K'}\text{ has good reduction at }v'\text{, but }A\text{ has bad reduction at }v'|_K \}.
\]
%By assumption, $\mathfrak{B}^0_{A/K'}\neq\emptyset$, and $\mathfrak{Bad}(A/K')$ and $\mathfrak{B}^0_{A/K'}$ are disjoint.

Let $v'\in\mathfrak{B}^0_{A/K'}$, and let $v$ be its restriction to $K$. Write $e(v'|v)$ for the ramification index. In particular, $e(v'|v)\leq [K':K]\leq 3^{4g^2}$. Let $q_{v'}\geq 1$ be the integer appearing in \eqref{EqCodimNeron}, and let $\lambda_{v'}$ be the number defined in \eqref{EqChoiceOfcNonSS}, with $q$ replaced by $q_{v'}$ and with the positive integers $b_{v',1},\ldots,b_{v',q_{v'}}$ chosen as in
\eqref{EqPED}. Then
\begin{equation}\label{EqEstimateLambdav}
\lambda_{v'} \ge  \frac{1}{3g}e(v'|v)^{-2} \ge \frac{1}{3^{8g^2+1}g}.
\end{equation}

\subsubsection{Modified norm and a small section over $K'$}

For $m\ge 1$, set
\[
E_m:=H^0(A_{K'},mL).
\]
Then $h^0(A_{K'},mL)=m^g h^0(A,L)$. The adelic line bundle $m\bar{L}(\epsilon \mathbf{f})$ defines an adelic vector space over $K'$
\[
\bar{E}_m := \left(E_m, \{\|\cdot\|_{m\bar{L}(\epsilon\mathbf f),v',\sup}\}_{v'\in M_{K'}}\right).
\]
Moreover, for the N\'eron model $\pi' \colon \cA' \rightarrow B'$ of $A_{K'}/K'$ and the canonical extension $\cL$ of $L_{K'}$, we know that  $\|\cdot\|_{m\bar{L}(\epsilon\mathbf f),v',\sup}$ is induced by the vector bundle $\cE_m := \pi'_*\cL^{\otimes m}|_{U_0}$ on $U_0$, for all $v' \in U_0:= B'\setminus \mathfrak{Bad}(A/K)$, \textit{i.e.} $\cE_m|_{\mathrm{Spec}\cO_{v'}}$ is the unit ball/lattice for  $\|\cdot\|_{m\bar{L}(\epsilon \mathbf{f}), v',\sup} = \|\cdot\|_{m\bar{L}, v',\sup}$.\footnote{Here we use the notation from §\ref{SubsectionPreliminaryAdelicGeometry} and $\cO_{v'} = \hat{\cO_{B',v'}}$ is the discrete valuation ring of $K'_{v'}$.} 

Now we modify the metrics at each $v' \in \mathfrak{B}^0_{A/K'}$. 
For the  infinitesimal neighborhood $N_{v'}:=N_{v',n_{v'}}\subseteq\mathcal A'_{v'}$ introduced in the previous subsection, restriction to $N_{v'}$ gives a homomorphism $
\mathrm{red}_{m,v'}\colon \cE_m|_{\mathrm{Spec}\cO_{v'}}  \longrightarrow H^0\left(N_{v'}, \cL_{v'}^{\otimes m}|_{N_{v'}}\right)$. 
Set
\[
\cS_{m,v'}:=\ker(\mathrm{red}_{m,v'}).
\]
Then $\cS_{m,v'}$ is an $\cO_{v'}$-lattice in $E_m$, and 
\[
\mathrm{length}_{\cO_{v'}} \left(\mathrm{im}(\mathrm{red}_{m,v'})\right) \le h^0(N_{v'},m\cL).
\] 
Denote by $\|\cdot\|_{\cS_{m,v'}}$  the lattice norm induced by $\cS_{m,v'}$. 
We now define a new adelic structure $\bar{E}_m^{\mathrm{new}} := (E_m,\{\|\cdot\|_{\mathrm{new},v'}\}_{v'\in M_{K'}})$ on $E_m$ by
\[
\|s\|_{\mathrm{new},v'} :=
\begin{cases}
\|s\|_{m\bar L(\epsilon\mathbf f),v',\sup},
&
v'\notin\mathfrak{B}^0_{A/K'} \\
|\varpi_{v'}|_{v'}^{n_{v'}-\left\lfloor\frac{\lambda_{v'}}4m\right\rfloor}
\,
\|s\|_{\cS_{m,v'}},
&
v'\in\mathfrak{B}^0_{A/K'}.
\end{cases}
\]
At each $v' \in \mathfrak{B}^0_{A/K'}$, replacing $\cE_m|_{\mathrm{Spec}\cO_{v'}} $ by $\cS_{m,v'}$ decreases $\hat{\deg}$ by $\mathrm{length}\left(\mathrm{im}(\operatorname{red}_{m,v'})\right)$, and multiplying the metric by $|\varpi_{v'}|_{v'}^{
n_{v'}-\lfloor\lambda_{v'}m/4\rfloor}$ decreases $\hat{\deg}$ by $h^0(A,mL) \log|\varpi_{v'}|_{v'}^{n_{v'}-\lfloor\lambda_{v'}m/4\rfloor}$. Therefore we have (recall that $\log|\varpi_{v'}|_{v'} = -1$)
\vskip -1.5em
\begin{align}\label{EqDegEmNew}
\hat{\deg}(\bar{E}_m^{\mathrm{new}})
& = \hat{\deg}(\bar{E}_m) + \sum\nolimits_{v'\in\mathfrak{B}^0_{A/K'}} \left( h^0(A,mL) \log |\varpi_{v'}|_{v'}^{ - n_{v'} + \left\lfloor\frac{\lambda_{v'}}{4}m\right\rfloor} - \mathrm{length} \left(\mathrm{im}(\mathrm{red}_{m,v'})\right) \right)	\nonumber\\
& \ge 
\hat{\deg}(\bar{E}_m) + \sum\nolimits_{v'\in\mathfrak{B}^0_{A/K'}}\left( h^0(A,mL)  \log |\varpi_{v'}|_{v'}^{ - n_{v'} +\left\lfloor\frac{\lambda_{v'}}{4}m\right\rfloor}  -  h^0(N_{v'},m\cL) \right) \\
& \ge \hat{\deg}(\bar{E}_m) + m^{g+1} h^0(A,L) \sum\nolimits_{v'\in \mathfrak{B}^0_{A/K'}} \frac{\lambda_{v'}}{4} - O(m^g) \text{ by Corollary~\ref{CorNonSSBound}.} \nonumber
\end{align}

We wish to apply Theorem~\ref{ThmMinkowskiFF} to $\bar{E}_m^{\mathrm{new}}$. Set
\[
U:=U_0 \setminus \mathfrak{B}^0_{A/K'} = B' \setminus \left( \mathfrak{Bad}(A_{K'}/K')  \sqcup   \mathfrak{B}^0_{A/K'}  \right).
\]
Then for each $m$ and each $v' \in U$, the metric $\|\cdot\|_{\mathrm{new},v'}$ is induced by the vector bundle $\cE_m|_U$ on $U$. Now
\begin{align*}
& \frac{\hat{\deg}(\bar{E}_m^{\mathrm{new}}) - h^0(A,mL) \left(g(B)-1+\#(B\setminus U) \right)}{m^{g+1}h^0(A,L)} \\
\ge & \frac{\bar{L}(\epsilon \mathbf{f})^{g+1}}{(g+1)! h^0(A,L)} + \sum\nolimits_{v'\in\mathfrak{B}^0_{A/K'}} \frac{\lambda_{v'}}{4} +o(1)  \qquad\qquad\qquad\text{ by \eqref{EqDegEmNew} and \eqref{EqArithmeticHSAdelicSpace} }%\vol(\bar L(\epsilon\mathbf{f})) = \lim_{m\to\infty} \frac{\hat{\deg}(\bar{E}_m)}{m^{g+1}/(g+1)!}
\\
\ge & 
\epsilon\left(2(1-\epsilon)^g-2(1+\epsilon)^g+1\right)\sum_{v' \in \mathfrak{Bad}(A_{K'}/K')} \tilde{J}_{v'}(A_{K'}) + \sum\nolimits_{v' \in \mathfrak{B}^0_{A/K'}} \frac{\lambda_{v'}}{4 } + o(1) \quad\text{ by \eqref{EqVolume}}
\end{align*}
as $m \to \infty$. The right hand side $>0$ when $m \gg 1$, because $\mathfrak{Bad}(A_{K'}/K') \bigcup \mathfrak{B}^0_{A/K'} \not=\emptyset$. Thus by Theroem~\ref{ThmMinkowskiFF}, there exists a non-zero global section $s' \in H^0(A_{K'},mL)$ such that
\[
\|s'\|_{\mathrm{new},v'} \le 1\quad \text{ for all }v' \in M_{K'}.
\]

\subsubsection{Small section over $K$ and auxiliary point}
Since we aim to find an abelian subvariety of $A$ defined over $K$, we need to construct a section $s$ defined over $K$ and an auxiliary point $x \in A(K)$. We proceed as follows.

Recall that the extension $K'/K$ is Galois. Set $G:=\mathrm{Gal}(K'/K)$. The local data $q_{v'}$, $d_{v'}$, $d_{0,v}$ and the multiset of elementary divisors $\{b_{v',i}\}$ are constant on $G$-orbits, so $\lambda_{\sigma v'}=\lambda_{v'}$ and $n_{\sigma v'}=n_{v'}$ for all $\sigma \in G$. Thus the modified adelic norms are $G$-equivariant. Set
\[
s:=N_{K'/K}(s')=\prod\nolimits_{\sigma\in G}\sigma(s') \in H^0(A,[K':K]mL) \setminus \{0\}.
\]
Similar to Proposition~\ref{PropVanishingOrder}, we apply Nakamaye's \cite{Nakamaye} to get a point $x \in A(K)_{\mathrm{tor}}$ such that the vanishing order of $s$ at $x$ satisfies
\begin{equation}\label{EqVanishingOrderNonSSK}
 \mathrm{ord}_x(s) \le g\left( \frac{g[K':K]m}{\mathfrak{N}} + 1\right).
\end{equation}

\subsubsection{Jet estimates for $s'$}
From now on we  go back to work over $K'$ for our computation.

Let $\ell$ be the vanishing order of $s'$ at $x \in A(K)$. Then $\ell$ is also the vanishing order of $\sigma(s')$ at $x$ for each $\sigma \in G$. The Leibniz rule then implies
\[
\mathrm{jet}^{[K':K]\ell} s(x) = \bigotimes_{\sigma \in G} \mathrm{jet}^{\ell}(\sigma(s'))(x) \not= 0.
\]
Therefore $[K':K]\ell\le \mathrm{ord}_x(s)$, and hence by \eqref{EqVanishingOrderNonSSK} we have
\begin{equation}\label{EqVanishingOrderNonSS}
\ell \le g\left( \frac{gm}{\mathfrak{N}} + 1\right).
\end{equation}

At $v' \in \mathfrak{Bad}(A_{K'}/K')$, we still have $\log\|\mathrm{jet}^{\ell}s'(x)\|_{v'} \le -m\epsilon \tilde{J}_{v'}(A_{K'})$ by Lemma~\ref{LemmaJetNonArch}. Hence by \eqref{EqIvSmallFF} (and $\epsilon = \frac{1}{8g}$), we have
\begin{equation}\label{EqJetNonSS1}
\log\|\mathrm{jet}^{\ell}s'(x)\|_{v'} \le  -\frac{m}{8\cdot 128e\cdot \gamma_g} (\#\Phi_{v'})^{-1/r_{v'}}.
\end{equation}
At $v' \in \mathfrak{B}^0(A/K')$, again we have $\|\mathrm{jet}^{\ell}s'(x)\|_{v'} \le \sup_{y\in A(K'_{v'})}\|s'(y)\|_{\bar{L},v'}=\sup_{y\in A(K'_{v'})} \|s'(y)\|_{\cL_{v'}}$ as in the proof of Lemma~\ref{LemmaJetNonArch}. Next we shall estimate $\|s'(y)\|_{\cL_{v'}}$ for $y\in A(K_v)$. Indeed, we have

\begin{lemma}
    For every $y\in A(K'_{v'})$ and $0\neq s'\in E_m$, we have $$\|s'(y)\|_{\cL_{v'}}\leq \|s'\|_{\mathrm{new},v'} |\varpi'|_{v'}^{\lfloor\frac {\lambda_{v'}} {4}m\rfloor}.$$
\end{lemma}

An upshot of this lemma and \eqref{EqEstimateLambdav} is then (recall $\log |\varpi'|_{v'} = -1$)%\frac{1}{e(v'|v_0)} \le -\frac{1}{3^{4g^2}[K:K_0]}$)
\begin{equation}\label{EqJetNonSS2}
\log \|\mathrm{jet}^{\ell}s'(x)\|_{v'} \le -\lfloor\frac {\lambda_{v'}} {4}m\rfloor= -\frac{m}{4\cdot 3^{8g^2+1}g} +o(m),  \text{ for any }v \in \mathfrak{B}^0_{A/K'}.
\end{equation}

\begin{proof}
    Let $\xi\in K'_{v'}$ with minimal valuation, such that $\xi^{-1}s'\in \cS_{m,v'}$. Then we have, by definition, $|\xi|_{v'}=\|s'\|_{\cS_m,v'}$.  

    Note that the Zariski closure $\bar y$ of $y$ in $\cA'_{v'}$ is contained in the schematic image of $\varphi\colon  \calN\otimes_{\calO_{K_v}} \calO_{K'_{v'}}\to \calA'_{v'}$, and $\xi^{-1}s'\in \mathcal{S}_{m,v'}$ vanishes on the $n_{v'}$-th infinitesimal neighborhood $N_{v'}$. Hence the restriction of $\xi^{-1}s'$ on $\bar y$ also vanishes on the $n_v$-th infinitesimal neighborhood of the special fiber of $\bar y$. This implies $\|\xi^{-1}s'(y)\|_{\cL_{v'}}\leq |\varpi'|_{v'}^{n_v}$. So
     \[
 \|s'(y)\|_{\cL_{v'}}\leq |\xi|_{v'}|\varpi'|_{v'}^{n_v}=\|s'\|_{\cS_m,v'}|\varpi'|_{v'}^{n_v}= \|s'\|_{\mathrm{new},v'} |\varpi'|_{v'}^{\lfloor\frac {\lambda_{v'}} {4}m\rfloor}
 \]
     as desired.
\end{proof}

The inequality \eqref{EqHeightInequality} is still  true with $K$ replaced by $K'$ and $s$ replaced by $s'$. We have $\hat{h}_L(x) = 0$ since $x\in A(K)_{\mathrm{tor}}$, and so the inequalities \eqref{EqVanishingOrderNonSS}, \eqref{EqJetNonSS1}, and \eqref{EqJetNonSS2} imply
\begingroup
\let\small\scriptsize % Temporarily overrides the global hook for this block only
\begin{align}\label{EqHeightInequalityNonSS}
\begin{split}
\frac{1}{[K':K_0]} \!\! \left(\frac{m}{8\!\cdot\! 128e\gamma_g} \sum_{v' \in \mathfrak{Bad}(A_{K'}/K')} \! \! (\#\Phi_{v'})^{-1/r_{v'}} \right. & \left. + \frac{m\!\cdot\! \#\mathfrak{B}^0_{A/K'}}{4\!\cdot\! 3^{8g^2+1}g} +o(m)\right)   \le \frac{g(g+2)}{2} \! \left(\frac{gm}{\mathfrak{N}}+1\right) h_{\mathrm{Fal}}(A).
\end{split}
\end{align}
\endgroup
Divide both sides by $m$ and let $m\to \infty$. Together with $[K':K_0] = [K':K]  [K:K_0] \le 3^{4g^2}[K:K_0]$ and $e<3$, we get
\begin{equation}\label{EqNFFNonSS}
\mathfrak{N} \le 4\cdot 128\cdot g^2(g+2) 3^{4g^2+1} \gamma_g \frac{[K:K_0] h_{\mathrm{Fal}}(A)}{\sum_{v' \in \mathfrak{Bad}(A_{K'}/K')} (\#\Phi_{v'})^{-1/r_{v'}} + \frac{256e\gamma_g \#\mathfrak{B}^0_{A/K'}}{3^{8g^2+1}g}}.
\end{equation}

\subsubsection{Final conclusion}\label{SubsubsectionFinalConclusionFFNonSS}
To ease notation, denote by 
\[
\mathfrak{b} := \#\mathfrak{Bad}(A_{K'}/K'), \quad \mathfrak{b}_0 := \#\mathfrak{B}^0_{A/K'}.
\]
Then the number of non-semi-stable reduction places of $A/K$ is $\le \mathfrak{b} +  \mathfrak{b}_0$, and hence
\begin{equation}\label{EqGenusBound}
g(B') \le [K':K](g(B)+\mathfrak{b}+ \mathfrak{b}_0).% \le 3^{4g^2}(g(B)+\mathfrak{b}+ \mathfrak{b}_0) .
\end{equation}
Now that $A_{K'}/K'$ has semi-stable reduction, we can apply the Arakelov Inequality (Theorem~\ref{ThmDeligne}) to $A_{K'}/K'$. Combining with \eqref{EqGenusBound}, we get the following crude bound
\begin{equation}\label{EqCrudeFalBoundNonSS}
[K:K_0]h_{\mathrm{Fal}}(A) \le \frac{3g}{2} \left(g(B)+\mathfrak{b}+\mathfrak{b}_0\right).
\end{equation}
To ease notation, denote by (recall our assumption that $\mathfrak{b} + \mathfrak{b}_0 \ge 1$)
\[
 \Delta :=  \sum_{v' \in \mathfrak{Bad}(A_{K'}/K')} (\#\Phi_{v'})^{-1/r_{v'}} +  \frac{256e\gamma_g \mathfrak{b}_0}{3^{8g^2+1}g} >0.% \ge \frac{16000\gamma_g}{g[K': K_0]^3}.
\]
Notice that $A_{K'}/K'$ has semi-stable reduction, and hence we can apply \eqref{EqEqEqEqFF} and get
\begin{equation}\label{EqBoundDelta}
\Delta \ge \sum_{v' \in \mathfrak{Bad}(A_{K'}/K')} (\#\Phi_{v'})^{-1/r_{v'}}\ge \frac{\mathfrak{b}^2}{96\cdot 3^{4g^2}[K:K_0]h_{\mathrm{Fal}}(A)}.% \ge  \frac{\mathfrak{b}^2}{96\cdot 3^{4g^2}[K:K_0]h_{\mathrm{Fal}}(A)}.
\end{equation}
We claim:
\begin{equation}\label{EqBlaBufBuf}
\frac{[K:K_0]h_{\mathrm{Fal}}(A)}{\Delta} \le  3^{8g^2+1}g^2  \left(g(B)+1\right)^2
\end{equation}
and
\begin{equation}\label{EqBlaBufBuf2}
\frac{[K:K_0]\max\{h_{\mathrm{Fal}}(A),1\}}{\Delta} \le 3^{8g^2+1}g^2 [K:\Q] \left(g(B)+1\right)^2.
\end{equation}

\noindent\boxed{\text{Case $\mathfrak{b}_0=0$}} In this case we have $\mathfrak{b} \ge 1$. So taking the inverse of \eqref{EqBoundDelta} and multiplying by $[K:K_0]h_{\mathrm{Fal}}(A)$, we get in this case
\begin{align*}
\frac{[K:K_0]h_{\mathrm{Fal}}(A)}{\Delta} &\le 96\cdot 3^{4g^2} \left( \frac{[K:K_0]h_{\mathrm{Fal}}(A) }{\mathfrak{b}}\right)^2 \\
& \le 96 \cdot 3^{4g^2} \frac{9g^2}{4}\left(\frac{g(B)}{\mathfrak{b}} + 1 \right)^2 \le 24 \cdot 3^{4g^2+2}g^2 \left(g(B)+1\right)^2 \quad\text{ by \eqref{EqCrudeFalBoundNonSS}}.
\end{align*}
This establishes \eqref{EqBlaBufBuf} in this case. For \eqref{EqBlaBufBuf2}, it suffices to notice
 \[
 \frac{1}{\Delta} \le \frac{96 \cdot 3^{4g^2}}{\mathfrak{b}} \frac{[K:K_0]h_{\mathrm{Fal}}(A) }{\mathfrak{b}} \le 96 \cdot 3^{4g^2} \cdot \frac{3g}{2} (g(B)+1) = 48 \cdot 3^{4g^2+1}g \left(g(B)+1\right).
 \]

\noindent\boxed{\text{Case $\mathfrak{b}_0\ge 1$}}
Now \eqref{EqCrudeFalBoundNonSS} implies% (also since $\mathfrak{b}_0 \ge 1$)
\begin{equation}\label{EqNonSSBlablablablabla}
\frac{[K:K_0]h_{\mathrm{Fal}}(A)}{\Delta} \le \frac{3g}{2} \left( \frac{g(B)}{\Delta} + \frac{\mathfrak{b}_0}{\Delta}+ \frac{\mathfrak{b}}{\Delta} \right) \le  \frac{3g}{2}\left( \frac{3^{8g^2+1}g \cdot g(B)}{256e\gamma_g } + \frac{3^{8g^2+1}g}{256e\gamma_g} + \frac{\mathfrak{b}}{\Delta}\right).
\end{equation}
By \eqref{EqBoundDelta}, we get
\[
\frac{3g}{2}\frac{\mathfrak{b}}{\Delta} \le \frac{3g}{2}\sqrt{\frac{96\cdot 3^{4g^2} [K :K_0]h_{\mathrm{Fal}}(A)}{\Delta}} \le \frac{1}{2} \frac{[K:K_0]h_{\mathrm{Fal}}(A)}{\Delta} + \frac{1}{2} \left(\frac{3g}{2}\sqrt{96 \cdot 3^{4g^2}}\right)^2.
\]
Putting this estimate in \eqref{EqNonSSBlablablablabla}, we get 
\begin{equation*}%\label{EqBlaBufBuf}
\frac{[K:K_0]h_{\mathrm{Fal}}(A)}{\Delta} \le \frac{3^{8g^2+2}g^2}{256e\gamma_g}\left( g(B)+1 \right) + 24 g^2 3^{4g^2+2} \le 3^{8g^2+1}g^2  (g(B)+1).
\end{equation*}
This establishes \eqref{EqBlaBufBuf} in this case. For \eqref{EqBlaBufBuf2}, it suffices to notice $\Delta \ge \frac{256e \gamma_g}{3^{8g^2+1}g}$.

\vskip 0.2em
Our desired bound \eqref{EqNCrudeBoundNonSS} now follows immediately from \eqref{EqNFFNonSS} and \eqref{EqBlaBufBuf}. %We are done.
\qed

\subsection{Proof of Theorem~\ref{ThmFirstBoundFFNonSSSmallPoint}}
The proof of Theorem~\ref{ThmFirstBoundFFNonSS} executed in $\mathsection$\ref{SubsectionProofOfThmFirstBoundFFNonSS} can be modified to prove Theorem~\ref{ThmFirstBoundFFNonSSSmallPoint}. Let us explain the modifications.

Define $\mathfrak{N}'$ to be the minimum of the left hand side of \eqref{EqFirstBoundFFNonSSSmallPoint} when $A'$ runs over all proper abelian subvarieties of $A_K$ (or of $A_{\bar{K}}$ -- they are the same by Gaudron--R\'emond \cite[Lem.~6.1]{GR25}). Then Theorem~\ref{ThmFirstBoundFFNonSSSmallPoint} is equivalent to:
\begin{equation}\label{EqNCrudeBoundNonSSSmallPoint}
\mathfrak{N}'  < 8\cdot 128\cdot 3^{12g^2+2}g^4(g+2) \gamma_g \left(g(B)+1\right)^2.
\end{equation}

Run the arguments in $\mathsection$\ref{SubsectionProofOfThmFirstBoundFFNonSS}. 
The first modification is the choice of the auxiliary point $x$ above \eqref{EqVanishingOrderNonSSK}. In the current situation, similar to Remark~\ref{RmkVanishingOrder} our auxiliary point $x$ is chosen to be in $\Gamma[g] := \{x_1+\cdots+x_g : x_i \in \Gamma\}$. Hence by definition \eqref{EqSmallRationalNonSS} of $\Gamma$ we have
\[
\hat{h}_L(x) \le g^2 \frac{\max\{h_{\mathrm{Fal}}(A),1\}}{2048\cdot 3^{12g^2+2}g^4 \gamma_g [K:K_0] \left(g(B)+1\right)^2} \le \frac{1}{2048\gamma_g} \frac{\Delta}{3^{4g^2+1}[K:K_0]}\le \frac{1}{16\cdot 128e\gamma_g} \frac{\Delta}{[K':K_0]},
%\hat{h}_L(x) \le g^2 \frac{\max\{h_{\mathrm{Fal}}(A),1\}}{96000\cdot 3^{8g^2+1}g^4 \gamma_g  \left(g(B)+1\right)} \le \frac{1}{96000\gamma_g} \Delta,
\]
where the second inequality follows from \eqref{EqBlaBufBuf2} and the third inequality follows from $[K':K]\le 3^{4g^2}$ and $2048=16\cdot 128$ and $e<3$.

Then we can proceed before applying the bounds \eqref{EqVanishingOrderNonSS}, \eqref{EqJetNonSS1}, and \eqref{EqJetNonSS2} to the inequality  \eqref{EqHeightInequality}. In the current situation, the upper bound of $\hat{h}_L(x)$ is given above, and hence instead of \eqref{EqHeightInequalityNonSS} we get
\begin{align}\label{EqHeightInequalityNonSSSmallPoint}
\frac{1}{[K':K_0]}\left( \frac{\Delta \cdot m}{16\cdot 128e\gamma_g} +o(m) \right) \le \frac{g(g+2)}{2}\left(\frac{gm}{\mathfrak{N}'}+1\right) h_{\mathrm{Fal}}(A).%\max\{h_{\mathrm{Fal}}(A), 1\}.
\end{align}
Dividing both sides by $m$ and letting $m\to \infty$, we get (noticing that $[K':K_0] = [K':K][K:K_0] \le 3^{4g^2}[K:K_0]$ and $e<3$)
\begin{align*}
\mathfrak{N}' & \le 8\cdot 128 \cdot g^2(g+2) 3^{4g^2+1} \gamma_g \frac{[K:K_0]  h_{\mathrm{Fal}}(A) }{\Delta}.
\end{align*}
Therefore, the desired bound \eqref{EqNCrudeBoundNonSSSmallPoint} follows immediately from \eqref{EqBlaBufBuf}. We are done.
\qed

\section{Proof of the main results over function fields: general case}\label{SectionProofGeneral}

\subsection{Proof of the bound on $\#A(K)_{\mathrm{tor}}$ (Theorem~\ref{ThmUBCFF})}

By the Poincar\'{e} Irreducibility Theorem, $A$ is isogenous to a product of $K$-simple abelian varieties. Each such $K$-simple abelian varieties (counted with multiplicity) is called a {\it factor} of $A$. 

\subsubsection{Case: no factor of $A$ has good reduction everywhere}\label{SubsubsectionCaseNotGoodEverywhere}
Assume no factor of $A$ has good reduction everywhere. Let us prove:
\begin{equation}\label{EqUBCFFOrigNonSS}
\#A(K)_{\mathrm{tor}} \le h^0(A,L)\left( 512\cdot 3^{12g^2+2}g^4(g+2) \gamma_g \right)^g \left( g(B)+1 \right)^{2g}.
\end{equation}
%\[
%\#A(K)_{\mathrm{tor}} \le h^0(A,L) \cdot \left( 24\cdot 24000 g^4(g+2)\gamma_g\right)^g \left( g(B)+1 \right)^{2g}
%\]

Notice that \eqref{EqUBCFF}  follows immediately from \eqref{EqUBCFFOrigNonSS} by  Zarhin's trick (which asserts that $(A\times A^\vee)^4$ carries a principal polarization).

\vskip 0.3em
We will proceed by induction on  the number of factors of $A$, which we denote by $n$. Denote by $K_0 = \C(\mathbb{P}^1)$; then $[K:K_0] = \mathrm{gon}(B)$.

\noindent\boxed{\text{Base step: $n=1$}}
In this case $A$ is simple. Our assumption becomes: $A$ does not have good reduction everywhere. The conclusion then follows immediately from Theorem~\ref{ThmFirstBoundFFNonSS}.%, as well as $g(B)+1 \le 2 \left( g(B)+1\right)^2$.

\noindent{\boxed{\text{Induction step}}} For arbitrary $n \ge 2$. 
By Theorem~\ref{ThmFirstBoundFFNonSS}, there exists a proper abelian subvariety $A'$ of $A$ defined over $K$ such that 
\[
\#(A/A')(K)_{\mathrm{tor}} \le  \frac{h^0(A,L)}{h^0(A',L)} \left( 512\cdot 3^{12g^2+2}g^4(g+2) \gamma_g \right)^{g-\dim A'} \left( g(B)+1 \right)^{2(g-\dim A')}.
\]

Since $A'\not=A$, the number of factors of $A'$ is $\le n-1$. 
So we can apply  the induction hypothesis to $A'$ and get 
\[
\#A'(K)_{\mathrm{tor}} \le h^0(A',L) \left( 512\cdot 3^{12g^2+2}g^4(g+2) \gamma_g \right)^{\dim A'}  \left( g(B)+1 \right)^{2\dim A'}.
\]
Thus the desired bound \eqref{EqUBCFFOrigNonSS} follows from the product of the two bounds above.%, with
%\begin{equation}\label{EqCgNonSS}
%c(g) := \left( 5.2\cdot 10^6 g^4(g+2)\gamma_g 3^{20g^2+2} \right)^g.
%\end{equation}

\subsubsection{Case: $A$ has good reduction everywhere and is traceless}\label{SubsubsectionCaseGoodEverywhere}
Next let us prove: If $A$ has good reduction everywhere, is traceless \textit{i.e.} $(\mathrm{Tr}_{K/\C}A)(\C) =0$, and carries a principle polarization given by a symmetric ample line bundle $L$ defined over $K$, then we have
\begin{equation}\label{EqAimTorsionBdTorGoodRed}
\#A(K)_{\mathrm{tor}} \le 2^{2^{2N^{\prime 2}+4N'\log_2 N'}}, \qquad \text{ where }N'= 4^g(g(B)+3)-1.
\end{equation}
Since $A/K$ has good reduction everywhere, it extends to an abelian scheme $\cA \to B$. The dual abelian scheme $\pi \colon \cA^\vee \to B$ has generic fiber $A^\vee$, and the total space is a projective variety defined over $\C$. Write $\iota \colon A^\vee \to \cA^\vee$ for the natural inclusion. We have the following exact sequence of groups:
\begin{equation}\label{EqExactSequenceLineBundle}
0 \rightarrow \mathrm{Pic}(B) \xrightarrow{\pi^*} \mathrm{Pic}(\cA^\vee) \xrightarrow{\iota^*} \mathrm{Pic}(A^\vee) \rightarrow 0.
\end{equation}
Indeed, $\pi^*$ is clearly injective, $\iota^*$ is surjective as $\cA^\vee$ is regular (so Weil divisors and Cartier divisors are the same), and the exactness at $\mathrm{Pic}(\cA^\vee)$ follows from \cite[Thm.~2.5]{Kleinman_PicardScheme}.

We have $\iota^*\!\left(\mathrm{Pic}^0(\cA^\vee)\right) \subseteq \mathrm{Pic}^0(A^\vee) = A(K)$. Moreover $\mathrm{Pic}^0(\cA^\vee)$ is an abelian variety defined over $\C$, and hence $\iota^*\!\left(\mathrm{Pic}^0(\cA^\vee)\right) \subseteq (\mathrm{Tr}_{K/\C}A)(\C) =0$ by the universal property of the trace. Therefore the exact sequence \eqref{EqExactSequenceLineBundle} induces an exact sequence
\begin{equation}\label{EqExactSequenceNS}
0 \longrightarrow \mathrm{Pic}(B)/\mathrm{Pic}^0(B) \longrightarrow \mathrm{NS}(\cA^\vee) := \mathrm{Pic}(\cA^\vee)/ \mathrm{Pic}^0(\cA^\vee)  \longrightarrow \mathrm{Pic}(A^\vee) \longrightarrow 0.
\end{equation}
But $B$ is a smooth projective irreducible curve defined over $\C$, and hence $ \mathrm{Pic}(B)/\mathrm{Pic}^0(B)= \mathrm{NS}(B) \cong \Z$. Therefore \eqref{EqExactSequenceNS} induces a natural isomorphism $\mathrm{NS}(\cA^\vee)_{\mathrm{tor}} = \mathrm{Pic}(A^\vee)_{\mathrm{tor}}$.

On the other hand, we have that $\mathrm{Pic}(A^\vee)/\mathrm{Pic}^0(A^\vee) \subseteq \mathrm{NS}(A_{0,\bar{K}}^\vee)$ is torsion-free. So the inclusion $\mathrm{Pic}^0(A^\vee) \subseteq \mathrm{Pic}(A^\vee)$ induces $\mathrm{Pic}(A^\vee)_{\mathrm{tor}} = \mathrm{Pic}^0(A^\vee)_{\mathrm{tor}} = A(K)_{\mathrm{tor}}$.

The upshot of the previous two paragraphs is that we have a natural isomorphism 
\[
A(K)_{\mathrm{tor}} = \mathrm{NS}(\cA^\vee)_{\mathrm{tor}} . 
\]
Hence to bound $\#A(K)_{\mathrm{tor}}$, it suffices to bound $\#\mathrm{NS}(\cA^\vee)_{\mathrm{tor}}$. We shall invoke \cite{NeronSeveri}, for which we need to embed $\cA^\vee$ smoothly into a projective space $\P^N$. First the principle polarization on $A$ defined by $L$ gives rise to a symmetric ample line bundle $L'$ on $A^\vee$ such that $h^0(A^\vee, L') = 1$. Then there exists a symmetric relatively ample line bundle $\cL'$ on $\cA^\vee$, whose restriction to the zero section of $\cA^\vee/B$ is trivial, such that $\cL'|_{A^\vee} = 4L_0'$; see \cite[Thm.~XI.1.4]{LNM119}. Moreover the complete linear system of $4L$ embeds $A^\vee$ into  $\P^{4^g-1}$, whose homogeneous ideal is generated by degree $2$ polynomials. 
Next classical algebraic geometry asserts: 
there exists a very ample line bundle $\cM$ on $B$ such that $B$ is embedded into $\P^{g(B)+2}$ as a smooth curve via the global sections of $\cM$, and the homogeneous ideal is generated by degree $2$ polynomials. Then the global sections of $\cL' + \pi^*\cM$ embeds $\cA^\vee$ into $ \P^{4^g-1} \times \P^{g(B)+2}$ as a smooth subvariety. Composing furthermore with the Veronese embedding $\P^{4^g-1} \times \P^{g(B)+2} \subseteq \P^{N'}$ with $N'= 4^g(g(B)+3)-1$, we have embedded $\cA^\vee$ into $\P^{N'}$ as the complete intersection of degree $4$ polynomials. Hence $\mathrm{NS}(\cA^\vee)_{\mathrm{tor}} \le 2^{4^{N^{\prime 2}+2N'\log_2 N'}}$ by \cite[Thm.~4.12]{NeronSeveri}. Thus we get the desired bound \eqref{EqAimTorsionBdTorGoodRed}.

\subsubsection{General case}
Now for an arbitrary $A$, let $A_0$ be the maximal abelian subvariety defined over $K$ which has good reduction everywhere. Then no $K$-factor of $A/A_0$ has good reduction everywhere. The same holds true for $(A/A_0)^4 \times (A/A_0)^{\vee 4}$, which by Zarhin's trick carries a principle polarization. 
Hence the conclusion of the case $\mathsection$\ref{SubsubsectionCaseNotGoodEverywhere} implies that
\[
\#(A/A_0)(K)_{\mathrm{tor}} \le \#\left((A/A_0)^4 \times (A/A_0)^{\vee 4}\right)(K)_{\mathrm{tor}} \le c_0(g) \left(g(B)+1\right)^{16g}
\]
for the explicit number $c_0(g) = \left( 512\cdot 3^{12(8g)^2+2}(8g)^4(8g+2) \gamma_{8g} \right)^{8g} > 0$. 

Next let us turn to $A_0$, which is traceless because $\mathrm{Tr}_{K/\C}(A)=0$ by our assumption. By Zarhin's trick, $(A_0 \times A_0^\vee)^4$ carries a principle polarization and hence satisfies all assumptions of the case $\mathsection$\ref{SubsubsectionCaseGoodEverywhere} (and has dimension $8g$). Hence
\[
\#A_0(K)_{\mathrm{tor}} \le \#(A_0 \times A_0^\vee)^4(K)_{\mathrm{tor}} \le 2^{2^{2N^2+4N\log_2 N}}, \qquad \text{ where }N= 4^{8g}(g(B)+3)-1.
\]
Hence we can conclude by taking the product of the two inequalities above (notice that $2N^2+4N\log_2 N \le 4N^2 \le 4^{16g+1}(g(B)+3)^2$). We are done.
\qed

\subsection{Proof of \eqref{EqLSFF1Intro} in Theorem~\ref{ThmLangSilvermanFF}}\label{SubsectionLSFF}
The  argument  in §\ref{SubsectionLSFFSemistable} allows to prove a bound in flavor of \eqref{EqLSFF1Intro} but with an extra dependence on $h^0(A,L)$. We now upgrade this argument to remove this extra dependence. The key observation is to use a fancier version of Zarhin's trick and use a certain {\it Zarhin's box} to replace the multiple of a fixed point.

Let $A, L$ be as in Theorem~\ref{ThmLangSilvermanFF}. Let $P \in A(K)$ be such that $\bar{\Z\cdot P}^{\mathrm{Zar}} = A$.

\subsubsection{Case: $A/K$ has good reduction everywhere and $A\not=\mathrm{Tr}_{K/k}(A)\otimes_k K$} 
In this case the N\'eron model of $A/K$ is an abelian scheme $\cA \rightarrow B$, and $L$ extends to a symmetric ample line bundle $\cL$ on $B$. The Zariski closure $\bar{P}$ of $P$ in $\cA$ is a section of $\cA \rightarrow B$, and hence $\bar{P}^*\cL$ is a line bundle on the projective curve $B$. In this situation, we have
\[
\hat{h}_L(P) = \frac{1}{[K:\C(\mathbb{P}^1)]}\deg(\bar{P}^*\cL),
\]
where $\deg(\bar{P}^*\cL)$ means the usual degree of a line bundle and hence is in $\Z$. But $\hat{h}_L(P) > 0$ by our assumption $A\not=\mathrm{Tr}_{K/k}(A)\otimes_k K$ and $\bar{\Z\cdot P}^{\mathrm{Zar}} = A$. So $\deg(\bar{P}^*\cL)$ is a positive integer, and thus $\hat{h}_L(P) \ge \frac{1}{[K:\C(\mathbb{P}^1)]}$. Combined with Deligne's Arakelov Inequality (Theorem~\ref{ThmDeligne}) and $\mathrm{gon}(B) = [K:\C(\mathbb{P}^1)]$, we get 
\begin{equation}
\hat{h}_L(P) \ge \frac{\max\{h_{\mathrm{Fal}}(A),1\}}{\max\{\mathrm{gon}(B), g(g(B)-1)\}}.
\end{equation}
This establishes \eqref{EqLSFF1Intro} in this case.

\subsubsection{Case: $A/K$ does not have good reduction everywhere}
Set 
\[
Z(A) := A^4 \times (A^\vee)^4.
\]
By the line bundle version of Zarhin's trick (Lemma~\ref{LemmaZarhinLineBundle}), $Z(A)$ carries a principal polarization induced by a symmetric ample line bundle $M$ on $Z(A)$ defined over $K$ such that  for a suitable $K$-isogeny 
\[
\pi \colon A^8 \longrightarrow Z(A),
\]
$\pi^*M$ induces the same polarization on $A^8$ as  $L^{\boxtimes 8}$. 
Hence
\begin{equation}\label{EqHeightZarhin}
\hat{h}_M\bigl(\pi(x_1,\ldots,x_8)\bigr) =\sum\nolimits_{j=1}^8\hat{h}_L(x_j).
\end{equation}
 
For each $j \in \{1,\ldots,8\}$, define
\[
Q_j:=\pi(0,\ldots,0,P,0,\ldots,0),
\]
where $P$ occurs in the $i$-th coordinate. Then
\begin{equation}\label{EqLSZarDense}
\bar{\Z Q_1+\ldots + \Z Q_8}^{\mathrm{Zar}} = Z(A).
\end{equation}

Denote by $K_0 = \C(\mathbb{P}^1)$; then $[K:K_0] = \mathrm{gon}(B)$. We wish to apply Theorem~\ref{ThmFirstBoundFFNonSSSmallPoint} to the polarized abelian variety $(Z(A),M)$. 
Set
\[
\Gamma_{Z(A)} := \left\{ x \in Z(A)(K) : \hat{h}_M(x) \le \frac{\max\{h_{\mathrm{Fal}}(Z(A)), 1\}}{2048 \cdot 3^{12(8g)^2+2}(8g)^4 \gamma_{8g} [K:K_0] \left( g(B)+1 \right)^2} \right\}.
\]
Then $\Gamma_{Z(A)}$ is contained in the subsets $\Gamma$ defined in \eqref{EqSmallRationalNonSS}, with $A$ replaced by $Z(A)$ and $L$ replaced by $M$.   Since $Z(A)$ does not have good reduction everywhere, Theorem~\ref{ThmFirstBoundFFNonSSSmallPoint} (joint with $h^0(Z(A),M) = 1$) implies: there exists a proper abelian subvariety $A'$ of $A$ defined over $K$ such that
\begin{equation}\label{EqGammaZA}
\left(\#\left(\frac{\Gamma_{Z(A)}+A'}{A'}\right)\right)^{\frac{1}{8g-\dim A'}} \le 1024 \cdot 3^{12(8g)^2+2}(8g)^4(8g+2)\gamma_{8g} \left(g(B)+1\right)^2.
\end{equation}

To ease notation denote by $c_2(g) :=  1024 \cdot 3^{12(8g)^2+2}(8g)^4(8g+2)\gamma_{8g}$. We claim:
\begin{equation}\label{EqConclusionForLSFF}
\hat{h}_L(P) > \frac{\max\{h_{\mathrm{Fal}}(A),1\}}{2 c_2(g)^{2g+1} [K:K_0] \left(g(B)+1\right)^{4g+2}}.%= \frac{1}{2 m^2}\frac{\max\{h_{\mathrm{Fal}}(A),1\}}{c_2(g) [K:K_0] \left( g(B)+1 \right)^2}.
\end{equation}

Assume \eqref{EqConclusionForLSFF} is false. 
Set
\[
m := c_2(g)^g \left(g(B)+1\right)^{2g}.
\]
and consider
\[
\cB_m := \left\{ \sum_{j=1}^8 [a_j]Q_j : 0 \le a_j \le m \right\}.
\]
As \eqref{EqConclusionForLSFF} is false, \eqref{EqHeightZarhin} then implies
\[
\hat{h}_M(x) \le 8 m^2 \hat{h}_L(P) \le \frac{8\max\{h_{\mathrm{Fal}}(A), 1\}}{2 c_2(g) [K:K_0] \left( g(B)+1 \right)^2} \qquad\text{ for any }x \in \cB_m,
\]
and hence $\cB_m$ is contained in the subset $\Gamma_{Z(A)}$ defined above, because $h_{\mathrm{Fal}}(Z(A)) = 8 h_{\mathrm{Fal}}(A)$. Hence \eqref{EqGammaZA} implies
\begin{equation}\label{EqLSBoxUpperBound}
\#\!\left(\frac{\cB_m + A'}{A'} \right)^{\frac{1}{8g-\dim A'}} \le c_2(g) \left(g(B)+1\right)^2 = m^{1/g}.
\end{equation}

On the other hand, we claim:
\begin{equation}\label{EqLSBoxLowerBound}
\#\!\left(\frac{\cB_m + A'}{A'} \right) \ge (m+1)^{\frac{8g-\dim A'}{g}} > m^{\frac{8g-\dim A'}{g}}.
\end{equation}
Indeed, let $q \colon Z(A) \rightarrow Z(A)/A'$ be the quotient map. Let $r$ be the rank of the subgroup $\Xi := \Z q(Q_1) + \ldots + \Z q(Q_8)$ of $(Z(A)/A')(K)$. Then $\#\!\left(\frac{\cB_m + A'}{A'} \right) \ge (m+1)^r$ because there are $r$ independent $q(Q_j)$'s with $j\in \{1,\ldots,8\}$. On the other hand by definition of $Q_j$, each  generator of $\Xi$ is of the form $\varphi(P)$ for some morphism of abelian varieties $\varphi \colon A \rightarrow Z(A)/A'$. So $\dim Z(A)/A' \le gr$ because $\bar{\Xi}^{\mathrm{Zar}} = Z(A)/A'$ by \eqref{EqLSZarDense}, and thus $gr \ge 8g-\dim A'$.  This establishes \eqref{EqLSBoxLowerBound}.

Now \eqref{EqLSBoxUpperBound}  and \eqref{EqLSBoxLowerBound} contradict each other. So \eqref{EqConclusionForLSFF} must hold. This establishes \eqref{EqLSFF1Intro} since $[K:K_0]=\mathrm{gon}(B)$.
\qed

\subsection{Proof of \eqref{EqLSFF2Intro} in Theorem~\ref{ThmLangSilvermanFF}}\label{SubsectionLSFF2}
Denote by $K_0 = \C(\mathbb{P}^1)$. Then $\mathrm{gon}(B) = [K:K_0]$.

Let $A, L$ be as in Theorem~\ref{ThmLangSilvermanFF}. Let $P \in A(K)$ be such that $\bar{\Z\cdot P}^{\mathrm{Zar}} = A$. Since the Faltings height of $\mathrm{Tr}_{K/\C}(A) \otimes K$ is $0$, it suffices to prove the result after taking quotient by $\mathrm{Tr}_{K/\C}(A) \otimes K$. So we may assume $\mathrm{Tr}_{K/\C}(A) = 0$.

Consider the Weil restriction $A_1 := \mathrm{Res}_{K/K_0}(A)$ and the norm line bundle $L_1$ on $A_1$. Then $\mathrm{Tr}_{K_0/\C}(A_1) = 0$ because $\mathrm{Tr}_{K/\C}(A) = 0$. Let $P_1 \in A_1(K_0)$ be the image of $P$ under the canonical identification $A(K) = A_1(K_0)$. Then $\hat{h}_L(P)= \hat{h}_{L_1}(P_1)$.  

Consider the abelian subvariety $A'_1 := \left(\overline{\Z P_1}^{\mathrm{Zar}}\right)^0$ of $A_1$. The image of $P_1$ under $A_1 \rightarrow A_1/A'_1$ is a $K_0$-torsion point, whose order we denote by $m'$. Applying \eqref{EqTorsionFFGon} to $A_1/A'_1$ which is traceless and has dimension $\le g\cdot \mathrm{gon}(B)$, we get
\vskip -1em
\begin{equation}\label{EqLSFF1111111}
m' \le c_0(g \cdot \mathrm{gon}(B)^2) \cdot 2^{2^{2^{32g\cdot \mathrm{gon}(B)^2 +6}}}
\end{equation}
Moreover, $[m']P_1 \in A'_1(K_0)$ and $\Z \cdot [m']P_1$ is Zariski dense in $A'_1$. Now apply \eqref{EqLSFF1Intro} to $(A'_1, L_1|_{A'_1})$ and the point $[m'] P_1$. Then we get
\begin{align}\label{EqLSFF222222}
\begin{split}
m^{\prime \, 2} \hat{h}_L(P)  = m^{\prime \, 2} \hat{h}_{L_1}(P_1) = \hat{h}_{L_1}([m']P_1) &  \ge c'_0(\dim A'_1)\max\{h_{\mathrm{Fal}}(A'_1),1\} \\
 & \ge c'_0(g\cdot \mathrm{gon}(B))\max\{h_{\mathrm{Fal}}(A'_1),1\}.
\end{split}
\end{align}
Next, over $\bar{K_0}$ we have
\[
A_1\otimes_{K_0} \bar{K_0} = \prod\nolimits_{\sigma' \colon K \hookrightarrow \bar{K_0}} A^{\sigma'}.
\]
The projection $A'_1 \rightarrow A^{\sigma'}$ is surjective for each $\sigma'$, because the image of contains the Zariski-dense subset $\Z\cdot [m']P^{\sigma'}$. Hence
\begin{equation}\label{EqLSFF33333333}
h_{\mathrm{Fal}}(A'_1) \ge \frac{1}{[K:K_0]}h_{\mathrm{Fal}}(A).
\end{equation}
Now \eqref{EqLSFF2Intro} follows from \eqref{EqLSFF1111111}, \eqref{EqLSFF222222}, and \eqref{EqLSFF33333333}. We are done.
\qed

\section{Proofs of the main results over number fields}\label{SectionProofNF}

\subsection{Generalized Szpiro Conjecture}
The following conjecture is proposed by Hindry. See \eqref{EqConductorNF} for the definition of $\mathfrak{f}_{A/K}$.
\begin{conj}[Generalized Szpiro Conjecture, strong version, {\cite[Conj.~3.4]{HindrySzpiro}}]\label{ConjSzpiroStrong}
There exist positive constants $\kappa'_2=\kappa'_2(g, K)$ and $\kappa_3=\kappa_3(g, K)$ such that
\begin{equation}\label{EqSzpiroConjStrong}
h_{\mathrm{Fal}}(A) \le \kappa'_2 \log \mathfrak{f}_{A/K} + \kappa_3.
\end{equation}
\end{conj}

We have proposed in $\mathsection$\ref{SubsubsectionGSC} another version of the generalized Szpiro Conjecture, which is closer to the original Szpiro's Conjecture. Recall the non-archimedean boundary height \eqref{DefnBoundaryHeightNonArchIntro}, which in view of the generalized Szpiro ratio $\mathfrak{s}_{A/K} \ge 1/g$ from \eqref{EqGeneralizedSzpiroRatioIntro} is
\begin{equation}\label{EqBoundaryFin}
h_{\partial,\mathrm{fin}}(A) = \frac{1}{[K:\Q]} \mathfrak{s}_{A/K} \log \mathfrak{f}_{A/K} \ge \frac{ \log \mathfrak{f}_{A/K} }{g[K:\Q]}.
\end{equation}

\begin{lemma}\label{LemmaSzpiro}
\cite[Conj.~3.4]{HindrySzpiro} implies Conjecture~\ref{ConjSzpiroWeakIntro}.
\end{lemma}
\begin{proof}
Assume Conjecture~\ref{ConjSzpiroStrong} holds true. 
Then $h_{\mathrm{Fal}}(A) \le \kappa'_2 g[K:\Q] h_{\partial,\mathrm{fin}}(A) + \kappa_3$. But $h_{\partial}(A) \ge h_{\partial,\mathrm{fin}}(A)$ and $h_{\partial}(A) \ge \sqrt{3}g/2$. So  $h_{\mathrm{Fal}}(A) \le  \kappa_2 h_{\partial }(A)$ with $\kappa_2 = \kappa'_2 g[K:\Q]  + \frac{2\kappa_3}{\sqrt{3}g}$. 
This establishes \eqref{EqSzpiroConjHeightIntro}. 
It remains to prove \eqref{EqSzpiroConjRatioIntro}, \textit{i.e.}  $\mathfrak{s}_{A/K}  \le \kappa_1$ for some $\kappa_1=\kappa_1(g,K)$.

If  $\mathfrak{Bad}(A/K) = \emptyset$, then $\mathfrak{s}_{A/K} = 1/g$. So we can take $\kappa_1=1/g$ in this case. Assume $\mathfrak{Bad}(A/K) \not= \emptyset$. Then $\log \mathfrak{f}_{A/K} \ge \log 2$ by its definition \eqref{EqConductorNF}. 
By \eqref{EqBoundaryFin} and  Proposition~\ref{PropHP}, we have $\mathfrak{s}_{A/K}  \log \mathfrak{f}_{A/K}  \le 96[K:\Q] h_{\mathrm{Fal}}(A) + 8g\log(\pi\sqrt{2})$. Bounding $h_{\mathrm{Fal}}(A)$ from above by \eqref{EqSzpiroConjStrong}, we furthermore get
\[
\mathfrak{s}_{A/K} \cdot \log \mathfrak{f}_{A/K} \le 96[K:\Q] \left( \kappa'_2 \log \mathfrak{f}_{A/K} + \kappa_3+8g\log\pi + 8g\log\sqrt{2}\right).
\]
So in this case, we can take $\kappa_1 = 96[K:\Q] \left(\kappa'_2 + \frac{\kappa_3+8g\log\pi}{\log 2} + 4g\right)$ since $\log\mathfrak{f}_{A/K} \ge \log 2$. To sum up, we are done by letting
\[
\kappa_1 := 96 [K:\Q] \left(\kappa'_2 + \frac{\kappa_3}{\log 2} + 8g\log_2 \pi + 4g\right).\qedhere
\]
\end{proof}

\subsection{Bounds in terms of Szpiro ratio}
Let $L$ be a polarization on $A$. 
For each $v \in M_{K,\infty}$, consider $\alpha_v(A,L)$ from \eqref{EqAlphaVIntro}. Define
\begin{equation}\label{EqBoundaryHeightL}
h_{\partial}(A,L) := h_{\partial,\mathrm{fin}}(A) + h_{\partial,\infty}(A,L),  \text{~and~}  h_{\partial,\infty}(A,L) := \frac{1}{[K:\Q]} \sum_{v\in M_{K,\infty}} \alpha_v(A,L)\log N_v.
\end{equation}
Then $h_{\partial,\infty}(A,L) \ge \sqrt{3}/2$ since $\alpha_v(A,L) \ge \sqrt{3}/2$ for each $v \in M_{K,\infty}$.
\begin{prop}\label{PropNFFaltingsOverBoundary}
We have
\begin{align*}
\frac{ \max\{h_{\mathrm{Fal}}(A),1\} + \log h^0(A,L) +\log[K:\Q]}{\frac{1}{[K: \Q ]}\sum_{v \in M'_K} \tilde{I}_v(A)\log N_v} \le  & 
\frac{2}{\sqrt{3}}g^2 \mathfrak{s}_{A/K}^2 [K:\Q]\log[K:\Q] + \frac{\log h^0(A,L)}{h_{\partial}(A,L)}\\
& + g^2 \mathfrak{s}_{A/K}^2 \cdot  \max\left\{\frac{h_{\mathrm{Fal}}(A)}{h_{\partial}(A,L)} , \frac{2}{\sqrt{3}} \right\}  \cdot [K: \Q].
\end{align*}
%Moreover if $h^0(A,L)=1$, then the term $\frac{2(g+e)}{\sqrt{3}e}$ on the right hand side can be replaced by $\frac{2}{\sqrt{3}}$.
\end{prop}
\begin{proof}%[Proof of Proposition~\ref{PropNFFaltingsOverBoundary}]
Use the notation and the choice of $\sigma \colon K \hookrightarrow \C$ from §\ref{SubsectionDivision}-\ref{SubsectionStatementFirstBoundNF}. In particular, $\tilde{I}_{\sigma}(A)\ge h_{\partial,\infty}(A,L)$.

By Jensen's Inequality applied to the function $t\mapsto t^{-1}$, we have
\[
 \left(\sum_{v \in \mathfrak{Bad}(A/K)} (\#\Phi_v)^{1/r_v} \log N_v \right) \left(\sum_{v \in \mathfrak{Bad}(A/K)} (\#\Phi_v)^{-1/r_v} \log N_v \right)  \ge \left(\sum_{v \in \mathfrak{Bad}(A/K)} \log N_v \right)^2.
\]
For the left hand side, we have $\sum_{v \in \mathfrak{Bad}(A/K)} (\#\Phi_v)^{1/r_v} \log N_v    = \mathfrak{s}_{A/K} \log \mathfrak{f}_{A/K}$ by definition of $\mathfrak{s}_{A/K}$. For the right hand side, we have  $\sum_{v \in \mathfrak{Bad}(A/K)} \log N_v \ge \sum_{v \in \mathfrak{Bad}(A/K)} (r_v/g) \log N_v = (\log \mathfrak{f}_{A/K})/g $. 
So the inequality above implies 
\[
\mathfrak{s}_{A/K}  \sum_{v \in \mathfrak{Bad}(A/K)} (\#\Phi_v)^{-1/r_v} \log N_v   \ge \frac{\log \mathfrak{f}_{A/K}}{g^2},
\]
and therefore
\begin{align*}
\sum_{v \in M'_K} \tilde{I}_v(A)\log N_v & = \sum_{v \in \mathfrak{Bad}(A/K)} (\#\Phi_v)^{-1/r_v} \log N_v +  \tilde{I}_{\sigma}(A)  \ge \frac{\frac{1}{g^2}\log \mathfrak{f}_{A/K}  +  \mathfrak{s}_{A/K}  \tilde{I}_{\sigma}(A)  }{\mathfrak{s}_{A/K}}.
\end{align*}
Combing this with 
\[
h_{\partial}(A,L) = h_{\partial,\mathrm{fin}}(A) + h_{\partial,\infty}(A,L) =  \frac{1}{[K:\Q]}\mathfrak{s}_{A/K} \log \mathfrak{f}_{A/K} + h_{\partial,\infty}(A,L) \le \mathfrak{s}_{A/K} \log \mathfrak{f}_{A/K} + \tilde{I}_{\sigma}(A)
\]
and \eqref{EqBoundaryFin},  we get
\begin{align}\label{EqBoundarySumRatio}
\begin{split}
 \frac{h_{\partial}(A,L)}{\sum_{v\in M'_K} \tilde{I}_v(A)\log N_v} 
 & \le \frac{\mathfrak{s}_{A/K}^2 \log \mathfrak{f}_{A/K} +\mathfrak{s}_{A/K}  \tilde{I}_{\sigma}(A)}{\frac{1}{g^2}\log \mathfrak{f}_{A/K}  +  \mathfrak{s}_{A/K}  \tilde{I}_{\sigma}(A) }  \le g^2 \mathfrak{s}_{A/K}^2,
\end{split}
\end{align}
where the last inequality follows from $\mathfrak{s}_{A/K}^2 \ge 1/g^2$. 
It follows, combined with $h_{\partial}(A,L)\ge \sqrt{3}/2$, that 
\begin{align*}%\label{EqBlaBlaBlaNF}
& \frac{\max\{h_{\mathrm{Fal}}(A),1\} +\log h^0(A,L) +\log[K:\Q]}{\frac{1}{[K:\Q]}\sum_{v\in M'_K} \tilde{I}_v(A)\log N_v}  \\
 = & \frac{\max\{h_{\mathrm{Fal}}(A),1\} +\log h^0(A,L) +\log[K:\Q]}{h_{\partial}(A,L)} g^2 \mathfrak{s}_{A/K}^2  [K:\Q] \\
\le &  \left( \max\left\{\frac{h_{\mathrm{Fal}}(A)}{h_{\partial}(A,L)}, \frac{2}{\sqrt{3}}\right\} + \frac{\log h^0(A,L)}{h_{\partial}(A,L)} + \frac{2}{\sqrt{3}}\log[K:\Q]  \right) g^2 \mathfrak{s}_{A/K}^2 [K:\Q].	\qedhere
\end{align*}
\end{proof}

\begin{cor}\label{CorNFFaltingsOverBoundary}
Assume Conjecture~\ref{ConjSzpiroWeakIntro}. Then there exists a polarization  $M_A$ on $A$  of minimal degree such that
\[
\frac{ \max\{h_{\mathrm{Fal}}(A),1\} + \log h^0(A,M_A) +\log[K:\Q]}{\frac{1}{[K: \Q ]}\sum_{v \in M'_K} \tilde{I}_v(A)\log N_v} \le \kappa_3  [K: \Q]  \max\{\log[K:\Q], 1\}.% + \frac{\log h^0(A,L)}{h_{\partial}(A)}
\]
for $\kappa_3 := g^2\kappa_1^2\max\{ \kappa_2 + \frac{2}{\sqrt{3}}, \frac{4}{\sqrt{3}}\} + 2^{17}g^5(\kappa_2+2)\log(14g)$.
%In practice, we can take $M_A$ to be of minimal polarization degree, \textit{i.e.} $h^0(A,M_A) = \min_L h^0(A,L)$, where $L$ runs over all symmetric ample line bundles on $A$ defined over $K$.
\end{cor}
\begin{proof}
Let $M_A$ be a polarization on $A$ of minimal degree. Then $h_{\partial}(A,M_A) \ge h_{\partial}(A)$ by definition. So $h_{\mathrm{Fal}}(A) \le \kappa_2 h_{\partial}(A,M_A)$ by Conjecture~\ref{ConjSzpiroWeakIntro}. Thus Proposition~\ref{PropNFFaltingsOverBoundary} implies
\[
\frac{ \max\{h_{\mathrm{Fal}}(A),1\} + \log h^0(A,L) +\log[K:\Q]}{\frac{1}{[K: \Q ]}\sum_{v \in M'_K} \tilde{I}_v(A)\log N_v} \le \kappa_4  [K: \Q] \max\{1,\log[K:\Q]\} + \frac{\log h^0(A,L)}{h_{\partial}(A,M_A)}
\]
for $\kappa_4 := g^2\kappa_1^2\max\{ \kappa_2 + \frac{2}{\sqrt{3}}, \frac{4}{\sqrt{3}}\}$. Next
by Gaudrons--R\'{e}mond's \cite{GaudronPolarisations-e}, we have
\[
\log h^0(A,M_A) \le 2^{10}g^3 \left( 64g^2\log(14g) + \log [K:\Q] + 2\log \max\{h_{\mathrm{Fal}}(A),1\} \right).
\]
Since $h_{\mathrm{Fal}}(A) \le \kappa_2 h_{\partial}(A,M_A)$ and $h_{\partial}(A,M_A)\ge \sqrt{3}/2$, we are done by the crude bound
\[
\frac{\log h^0(A,M_A)}{h_{\partial}(A)} \le 2^{17}g^5\log(14g) \cdot (\kappa_2+1+\log[K:\Q]).
\qedhere
\]
\end{proof}

\subsection{Proof of Theorem~\ref{ThmMainNFUnconditional}}\label{SubsectionNFUncond}
Let $(A,L)$ be a polarized abelian variety defined over $K$, such that $A$ is $K$-simple and $L$ is a principle polarization. Then $h_{\partial}(A) \le h_{\partial}(A,L)$ by their definitions \eqref{DefnBoundaryHeightIntro} and \eqref{EqBoundaryHeightL}.

By Theorem~\ref{ThmFirstBound}, there exists an explicit positive constant $C_1 = C_1(g)>0$ %from \eqref{EqFormulaC1} 
such that
\[
\#A(K)_{\mathrm{tor}} \le  (2g)^g(g+2)^g \cdot h^0(A,L) \cdot \left(C_1 \frac{ \max\{ h_{\mathrm{Fal}}(A),1\} +\log h^0(A,L) + \log [K:\Q] }{\frac{1}{[K: \Q]}\sum_{v\in M'_K} \tilde{I}_v(A)\log N_v} \right)^g.
\]
Now $h^0(A,L) = 1$. Hence Proposition~\ref{PropNFFaltingsOverBoundary} simplifies to
\begingroup
\let\small\scriptsize % Temporarily overrides the global hook for this block only
\begin{equation}\label{EqECNFBalblablablabla}
\frac{ \max\{ h_{\mathrm{Fal}}(A),1\} +\log h^0(A,L) + \log [K:\Q] }{\frac{1}{[K: \Q]}\sum_{v\in M'_K} \tilde{I}_v(A)\log N_v} \le \frac{2}{\sqrt{3}} \mathfrak{s}_{A/K}^2[K:\Q] \left( \log [K:\Q]+\max\left\{ \frac{h_{\mathrm{Fal}}(A)}{h_{\partial}(A,L)}, 1 \right\} \right).
\end{equation}
\endgroup
Now \eqref{EqTorsionNF} follows immediately from the two bounds above, with (also use $\frac{4}{\sqrt{3}}\le 2.4$)
\begin{equation}\label{EqConstantc1NFUncond}
c_1(g) = (2.4g)^g(g+2)^gC_1(g)^g, \quad \text{ with }C_1(g)\text{ from \eqref{EqFormulaC1}}.
\end{equation}

\vskip 0.2em
Let us turn to non-torsion rational points. To ease notation, denote by 
\begin{equation}\label{EqC5}
C_6(g) := \frac{32 \cdot 16 e g^3(g+1)^3(g+2)^3(100g+274)\gamma_{g}}{\pi}.
\end{equation}
By \eqref{EqBoundarySumRatio}, we know that
\[
\Gamma'' :=\left\{x\in A(K) : \hat{h}_L(x) \le \frac{1}{C_6(g)} \frac{h_{\partial}(A,L)}{g^2 \mathfrak{s}_{A/K}^2 [K:\Q]} \right\}
\]
is contained in $\bigcup_{j=1}^{N_g}\Gamma_j'$ defined in \eqref{EqGammaPrime}, with $N_g = (2g)^g(g+2)^g$. 
 
For the number $C'_1(g) = 3C_1(g)$ from Theorem~\ref{ThmFirstBoundPrime}, set
\[
m := \left( \frac{2}{\sqrt{3}} C_1'(g) \mathfrak{s}_{A/K}^2[K:\Q] \Big(\!\log [K:\Q]+\max\big\{ \frac{h_{\mathrm{Fal}}(A,L)}{h_{\partial}(A)}, 1 \big\} \Big)\right)^g.
\]
Let $P \in A(K)$ be such that $\bar{\Z \cdot P}^{\mathrm{Zar}} = A$. We claim:
\begin{equation}\label{EqLSNFUncond}
\hat{h}_L(P) > \frac{1}{( N_g m)^2} \frac{1}{C_6(g)} \frac{h_{\partial}(A,L)}{g^2 \mathfrak{s}_{A/K}^2 [K:\Q]}.%, \quad \text{ for any }.
\end{equation}
Assume not, then there are $2 N_g m+1$ distinct points $0, \pm P, \ldots, [\pm N_g m]P$ contained in $\Gamma''$. The Pigeonhole Principle then implies that $\Gamma'_j$ contains $ > 2m$ of these points, for some $j\in \{1,\ldots,N_g\}$. 
But $h^0(A,L) = 1$, so Theorem~\ref{ThmFirstBoundPrime} gives an upper bound  $\#\Gamma'_j \le m$  when combined with \eqref{EqECNFBalblablablabla}. We get a contradiction. So \eqref{EqLSNFUncond} must hold true. 

We can simplify \eqref{EqLSNFUncond}. Using the trivial bound $h_{\partial}(A,L) \ge \frac{\sqrt{3}}{2}$ (see below \eqref{EqBoundaryHeightL}), we get $\frac{\max\{h_{\mathrm{Fal}}(A),1\}}{h_{\partial}(A,L)} \le \frac{2}{\sqrt{3}}\max\left\{\frac{h_{\mathrm{Fal}}(A)}{h_{\partial}(A,L)}, 1\right\}$. 
Together with $\frac{2}{\sqrt{3}}\le 1.2$, we can easily obtain \eqref{EqLSNFUncondIntro} from \eqref{EqLSNFUncond} and the choices of $N_g$ and $m$, with the constant
\begin{equation}\label{EqConstantc1PrimeNFUncond}
c'_1(g) =  1.2 (7.2g)^{2g}(g+2)^{2g}C_1(g)^{2g}C_6(g) g^2,  
\end{equation}
for $C_1(g)$ from \eqref{EqFormulaC1} and $C_6(g)$ as in \eqref{EqC5}.
%\vskip 0.3em
%Now let us turn to the case where $A=E$ is an elliptic curve. Then  $h^0(E,L) = 1$ for $L=O(0)$ and $\frac{h_{\mathrm{Fal}}(E)}{h_{\partial}(E)} \le 6/\pi$ by the Faltings--Silverman Formula. Thus we get \eqref{EqTorsionEC} from \eqref{EqTorsionNF} and 
%Moreover, $N_1 = 6$ and $C_1(1) \le 3.8\cdot 10^9$ by \eqref{EqFormulaC1}. Putting these numbers in  (resp. \eqref{EqLSNFUncond}), we get  (resp. \eqref{LSEC}). We are done.
\qed

\subsection{Proof of Theorem~\ref{ThmMainNF}.(i): rational torsion points}

Let $A$ be an abelian variety of dimension $g\ge 1$ defined over the number field $K$.% Denote by $\log^+(\cdot) = \max\{1,\log(\cdot)\}$.

Assume Conjecture~\ref{ConjSzpiroWeakIntro} holds true for $K$ up to dimension $g$. 
Let us prove the following bound by induction on the number of $K$-simple factors of $A$ (call this number $n$): There exists a polarization $M_A$ on $A$ of minimal degree satisfying
\begin{equation}\label{EqUBC}
\#A(K)_{\mathrm{tor}} \le h^0(A,M_A) \cdot C_4(g) \kappa_3^g \cdot \left([K:\Q] \max\{\log[K:\Q], 1\}\right)^g
\end{equation}
for an explicit constant $C_4(g)>0$ depending only on $g$. 
Then Theorem~\ref{ThmMainNF}.(i) follows immediately by applying Zarhin's trick to $A$.

\noindent\boxed{\text{Base step: $n=1$}}
In this case $A$ is $K$-simple. 
Apply Theorem~\ref{ThmFirstBound} to get
%\begin{equation}\label{EqSimpleNF}
\[
\#A(K)_{\mathrm{tor}} \le  (2g)^g(g+2)^g \cdot h^0(A,M_A) \left(  C_1 \frac{ \max\{h_{\mathrm{Fal}}(A),1\} + \log h^0(A,M_A) +\log[K:\Q]}{\frac{1}{[K: \Q ]}\sum_{v \in M'_K} \tilde{I}_v(A)\log N_v}  \right)^g
\]
%\end{equation}
for a constant $C_1 = C_1 (g) > 0$. 
Hence we can conclude by Corollary~\ref{CorNFFaltingsOverBoundary}. 
\vskip 0.2em

\noindent{\boxed{\text{Induction step}}} For  arbitrary $n \ge 2$. 
We still apply Theorem~\ref{ThmFirstBound}, but in a finer way. We have divided $A(K)_{\mathrm{tor}}$ into the union of $N_g := (2g)^g(g+2)^g$ subsets $\Gamma_1,\ldots,\Gamma_{N_g}$. For each $j \in \{1,\ldots,N_g\}$, there exists a proper abelian subvariety $A'_j$ of $A$ defined over $K$ such that 
\[
\#\left( \frac{\Gamma_j+A'_j}{A'_j} \right) \le  \frac{h^0(A,M_A)}{h^0(A'_j,M_A)} \left( C_1 \frac{ \max\{h_{\mathrm{Fal}}(A),1\} + \log h^0(A,M_A) +\log[K:\Q]}{\frac{1}{[K: \Q ]}\sum_{v \in M'_K} \tilde{I}_v(A)\log N_v}  \right)^{g-\dim A'_j}
\]
for an explicit constant $C_1 = C_1 (g) > 0$. So  Corollary~\ref{CorNFFaltingsOverBoundary} implies
\[
\#\left( \frac{\Gamma_j+A'_j}{A'_j} \right) \le  \frac{h^0(A,M_A)}{h^0(A'_j,M_A)} \left(C_1 \kappa_3 [K:\Q]\max\{\log[K:\Q], 1\} \right)^{g-\dim A'_j}.
\]

Since $A'_j \subsetneq A$, the number of factors of $A'_j$ is $\le n-1$. 
So we can apply  the induction hypothesis to $A'_j$ and get 
\[
\#A'_j(K)_{\mathrm{tor}} \le h^0(A'_j,M_{A'_j}) \left(C_4(g-1)\cdot \kappa_3 [K: \Q]\max\{\log[K:\Q], 1\}\right)^{\dim A'_j}.
\]

Taking the product of the two inequalities above, we get
\[
\#\Gamma_j \le h^0(A,M_A)\frac{h^0(A'_j, M_{A'_j})}{h^0(A'_j,M_A)} C_3  \left(\kappa_3 [K:\Q] \max\{\log[K:\Q], 1\} \right)^g
\]
with $C_3 = C_3(g) > 0$ explicit. But $h^0(A'_j,M_{A'_j}) \le h^0(A'_j, M_A|_{A'_j})$ by definition of $M_{A'_j}$, so
\[
\#\Gamma_j \le h^0(A,M_A)C_3  \left(\kappa_3 [K:\Q] \max\{\log[K:\Q], 1\} \right)^g.
\]
Therefore we are done by 
\[
\#A(K)_{\mathrm{tor}} \le \sum_{j=1}^{(2g)^g(g+2)^g} \#\Gamma_j \le h^0(A,M_A) \cdot (2g)^g(g+2)^g C_3  \kappa_3^g \cdot \left([K:\Q] \max\{\log[K:\Q],1\}\right)^g.
\qed
\]

\subsection{Proof of Theorem~\ref{ThmMainNF}.(ii): small rational points}\label{SubsectionLSNFCond}
Let $A, L$ be as in Theorem~\ref{ThmMainNF}.(ii). Let $P \in A(K)$ be such that $\bar{\Z\cdot P}^{\mathrm{Zar}} = A$. 
If $L$ defines a principal polarization, we can apply the arguments as in §\ref{SubsectionNFUncond} and hence relax the requirement \eqref{EqSzpiroConjHeightIntro} to $h_{\mathrm{Fal}}(A) \le \kappa_2 h_{\partial}(A,L)$ (see \eqref{EqBoundaryHeightL} for definition; we trivially have $h_{\partial}(A,L) \ge h_{\partial}(A)$).

For general $L$, set 
\[
Z(A) := A^4 \times (A^\vee)^4.
\]
By the line bundle version of Zarhin's trick (Lemma~\ref{LemmaZarhinLineBundle}), $Z(A)$ carries a principal polarization induced by a symmetric ample line bundle $M$ on $Z(A)$ defined over $K$, such that for a suitable $K$-isogeny 
\[
\pi \colon A^8 \longrightarrow Z(A),
\]
$\pi^*M$ and $L^{\boxtimes 8}$ induce the same polarization on $A^8$. 
Hence
\begin{equation}\label{EqHeightZarhinNF}
\hat{h}_M\bigl(\pi(x_1,\ldots,x_8)\bigr) =\sum\nolimits_{j=1}^8\hat{h}_L(x_j).
\end{equation}

For each $j \in \{1,\ldots,8\}$, define
\[
Q_j:=\pi(0,\ldots,0,P,0,\ldots,0),
\]
where $P$ occurs in the $i$-th coordinate. Then
\begin{equation}\label{EqLSZarDenseNF}
\bar{\Z Q_1+\ldots + \Z Q_8}^{\mathrm{Zar}} = Z(A).
\end{equation}
By  Proposition~\ref{PropNFFaltingsOverBoundary} applied to $Z(A)$ and $M$ and Conjecture~\ref{ConjSzpiroWeakIntro}, we have
\begin{equation}\label{EqNFaaaaaaaaaaa}
\frac{\max\{h_{\mathrm{Fal}}(Z(A)),1\}+ \log h^0(Z(A),M) +\log[K:\Q]}{\frac{1}{[K:\Q]}\sum_{v\in M'_K}\tilde{I}_v(Z(A))\log N_v} \le \kappa_5 [K:\Q]\max\{\log[K:\Q], 1\}
\end{equation}
where $\kappa_5 :=  (8g)^2\kappa_1(8g,K)^2\max\{ \kappa_2(8g,K) + \frac{2}{\sqrt{3}}, \frac{4}{\sqrt{3}}\}$. To ease notation, denote by $C_5(g) := 32 \cdot 16 e(8g)^3(8g+1)^3(8g+2)^3(800g+274)\gamma_{8g} / \pi$. 
So the set
\[
\Gamma'_{Z(A)} := \left\{ x \in Z(A)(K) : \hat{h}_M(x) \le \frac{1}{C_5(g)}\frac{\max\{h_{\mathrm{Fal}}(Z(A)), 1\}}{\kappa_5[K:\Q]\max\{\log[K:\Q], 1\}} \right\}.
\]
is contained in the set $\bigcup_{j=1}^{N_{8g}}\Gamma'_j$ defined in \eqref{EqGammaPrime} but with $(A,L)$ replaced by $(Z(A),M)$. Here $N_{8g} = (2\cdot 8g)^{8g}(8g+2)^{8g}$.

For the number $C'_1>0$ from Theorem~\ref{ThmFirstBoundPrime}, set
\[
m :=  \left(2C'_1(8g) \kappa_5 \right)^g \left( [K:\Q]\max\{\log[K:\Q], 1\} \right)^g.
\]
To conclude for Theorem~\ref{ThmMainNF}.(ii), it suffices to prove:
\begin{equation}\label{EqConclusionForLSNF}
\hat{h}_L(P) > \frac{\max\{h_{\mathrm{Fal}}(A),1\}}{(mN_{8g}^g)^2 C_5(g)\kappa_5 [K:\Q]\max\{\log[K:\Q], 1\}}.%= \frac{1}{2 m^2}\frac{\max\{h_{\mathrm{Fal}}(A),1\}}{c_1(g) [K:K_0] \left( 2g(B)-1 \right)^2}.
\end{equation}

Assume \eqref{EqConclusionForLSNF} is false. Consider
\[
\cB_m := \left\{ \sum_{j=1}^8 [a_j]Q_j : 0 \le a_j \le mN_{8g}^g \right\}.
\]
As \eqref{EqConclusionForLSNF} is false, \eqref{EqHeightZarhinNF} then implies
\[
\hat{h}_M(x) \le  8 (mN_{8g}^g)^2 \hat{h}_L(P) \le  \frac{1}{C_5(g)}\frac{8\max\{h_{\mathrm{Fal}}(A), 1\}}{\kappa_5[K:\Q]\max\{\log[K:\Q], 1\}} \ \qquad\text{ for any }x \in \cB_m,
\]
and hence $\cB_m$ is contained in the subset $\Gamma'_{Z(A)}$ defined above, because $h_{\mathrm{Fal}}(Z(A)) = 8 h_{\mathrm{Fal}}(A)$. Thus $\cB_m \subseteq \bigcup_{j=1}^{N_{8g}}\Gamma'_{Z(A),j}$.

Similar to \eqref{EqLSBoxLowerBound}, we can show, for any proper abelian subvariety $A'$ of $Z(A)$ defined over $K$
\[
\#\!\left(\frac{\cB_m + A'}{A'} \right) \ge (mN_{8g}^g+1)^{\frac{8g-\dim A'}{g}} > m^{\frac{8g-\dim A'}{g}} N_{8g};
\]
the last inequality holds true since $8g-\dim A' \ge 1$. 
The Pigeonhole Principle then implies that there exists $j \in \{1,\ldots,N_{8g}\}$ with
\begin{equation}\label{EqLSBoxLowerBoundNF}
\#\!\left(\frac{\Gamma'_{Z(A),j} + A'}{A'} \right) > m^{\frac{8g-\dim A'}{g}}.
\end{equation}
By our choice of $m$ and since $h^0(Z(A),M)=1$, we get
\[
\left( \#\!\left(\frac{\Gamma'_{Z(A),j} + A'}{A'} \right) \frac{h^0(A',M)}{h^0(Z(A),M)} \right)^{\frac{1}{8g-\dim A'}} > \frac{m^{1/g}}{2} =   C_1'(8g)\kappa_5 [K:\Q]\max\{\log[K:\Q], 1\}.
\]
This contradicts Theorem~\ref{ThmFirstBoundPrime} in view of \eqref{EqNFaaaaaaaaaaa}. So \eqref{EqConclusionForLSNF} must hold. We are done.
\qed

\section{Improvements in the case of elliptic curves}\label{SectionEC}
In the end, let us explain how to improve our results for elliptic curves, \textit{i.e.} $g=1$. There are two improvements: getting better constants, and proving uniformity for function fields of arbitrary characteristic.

Let $k =\bar{k}$ be an algebraically closed field.  Let $B$ be a smooth projective geometrically connected curve defined over $k$. Let $K:=k(B)$ be the function field of $B$. 

Let $E/K$ be a traceless semi-stable elliptic curve, equipped with the canonical principal polarization $L=O(0)$. In this case, $E/K$ does not have everywhere good reduction because the modular map $j \colon B\to \PP^1_k$ is surjective.% (so the preimage of the cusp $j=\infty$ is non empty).

Let $p^{\nu}$ be the inseparable degree of the modular map $j$. When $\mathrm{char}(k)=0$, set $p^{\nu} = 1$.

\subsection{Bump functions}\label{ECBumpFct}
We start by constructing a bump function at each non-archimedean place $v \in \mathfrak{Bad}(E/K)$, which is better than the one from $\mathsection$\ref{PropBumpFctNonArch}.
 %Without loss of generality, we may assume $A$ has split semistable reduction at $v$ up to taking an \'etale base change. 

Continue to use the notations from $\mathsection$\ref{SectionNonArchPrep}. %, and construct a larger bump function than what we constructed in Section~\ref{PropBumpFctNonArch}.  
In this case, $E_v$ is a Tate curve $F^{\times}/\mathfrak{q}^{\Z}$. The horizontal line of the Raynaud cross \eqref{EqRaynaudCross} becomes $\mathbb{G}_{\mathrm{m}} = \mathbb{G}_{\mathrm{m}} \to 0$, and the map $\mathfrak v\colon Y=\Z \to E(F)=F^\times$ maps a generator $1$ of $Y$ to some $\mathfrak{q} \in F^{\times}$ such that $|\mathfrak{q}|_v<1$. Then $n_v :=\ord_v \mathfrak{q}$ is the number of N\'eron components of $A$ at $v$. The skeleton of $A_v^{\mathrm{an}}$ is $\Sigma_v := X^*_{\R}/Y$, while $X^*_{\R} =\R$, $Y=\Z$ and $X^*$ is the lattice $\frac{1}{n_v}\Z$. See the end of $\mathsection$\ref{SectionNonArchPrep}. 

Let $x$ be the standard coordinate on $X^*_{\R}=\R$. Define the following function on $X^*_{\R}$
\[
F_v(x) :=\frac 1 {2n_v}\{n_vx\}(1-\{n_vx\}) \text{  $\ \ \ \ \ $ where \{$\cdot$\} denotes the fraction part.}
\]
Then $F_v$ descends to a function $f_v \colon \Sigma_v=X_{\R}^*/Y\stackrel{x}\simeq \RR/\ZZ\to \RR$ with the following properties:
\begin{enumerate}
\item[(i)] $f_v|_{\Phi_v} = 0$, where $\Phi_v = X^*/Y$;
\item[(ii)] $f_v$ is $\frac{1}{n_v}\Z/\Z$-periodic;
\item[(iii)] $\bar L(f_v)$ is semipositive;
\item[(iv)] $\int_{\Sigma_v} f_v\mathrm d\mu_v= (12n_v)^{-1}$;
\item[(v)] $\bar L(2f_v)$ coincides with the model metric of the admissible $\Q$-extension of $L$ on the the minimal regular (proper) model $\calC/O_F$.
\end{enumerate}
See \cite{FaltingsArithmeticSurface} for the theory of admissible $\Q$-extensions. 
Indeed, (i) and (ii) follow immediately from the construction of $f_v$, and (iv) is a direct computation. For (iii), it suffices to see that  $c_1(\bar L(f_v))=\frac 1 {2n_v} \sum_{u\in \Phi_v} \delta_u+\frac 1 2c_1(\bar L_v)$, which is semipositive. 
For (v), the chern form is $c_1(\bar L(2f_v))=\frac 1 {n_v} \sum_{u\in \Phi_v} \delta_u.$ This coincides with the chern form of the admissible $\Q$-extension of $L$ on the the minimal regular (proper) model $\calC/O_F$. Therefore they are equal by for example the non-archimedean Calabi Theorem \cite[Thm.~1.5]{YuanZhangHodgeIndex}.

\subsection{Application of Arithmetic Hilbert Samuel}

Let $\mathbf f=(f_v)_{v\in \mathfrak{Bad}(E/K)}$ be the collection of bump functions. %From the properties listed above, we are able to compute the volume of $\bar L(\mathbf f)$:

\begin{prop}\label{ArithVolEC}
    $\vol (\bar L(\mathbf f))=\frac 1 {12}\sum_{v\in \mathfrak{Bad}(E/K)}n_v^{-1}$.
\end{prop}

\begin{proof}
    By (iii), $\bar L(\mathbf f)$ is nef. So Arithmetic Hilbert Samuel  implies $\vol (\bar L(\mathbf f)) = \overline{L}(\mathbf f)^2$.
    
Let $\calC/B$ be the minimal regular (proper) model of $E/K$. By property (v), $\bar L(2\mathbf f)$ carries the model metric of the Faltings--Hiljac admissible $\QQ$-extension $\calL$ of $L$ on $\calC/B$. By their theory of admissible $\QQ$-extension, we have $\bar L(2\mathbf f)^2=\calL^2=0$. Thus 
\[
    \overline{L}(\mathbf f)^2 =\frac 12(\overline{L}+\overline{L}(2\mathbf  f))^2
     =\frac 1 4\overline{L}^2+\frac 1 2\overline{L}\cdot \overline{L}(2\mathbf f)+\frac 1 4\overline{L}(2\mathbf f)^2
    =\frac 1 2\overline{L}\cdot \overline{L}(2\mathbf f) = \frac 1 2\overline{L}\cdot (\overline{L} + \overline{\cO}(2\mathbf f)) = \frac{1}{2}\overline{L} \cdot \overline{\cO}(2\mathbf{f}).
\]
Therefore
\[
 \overline{L}(\mathbf f)^2 = \frac{1}{2}\sum_{v\in \mathfrak{Bad}(E/K) }\int_{A_v^{\an}} 2f_v c_1(\overline{L}_v)=\sum_{v\in \mathfrak{Bad}(E/K)}(12n_v)^{-1}.
 \qedhere
\]
\end{proof}

\subsection{A bound for torsion points and small points}

We use slightly different notations from previous sections. For $\varepsilon \in [0,1/24)$, define 
\[
\Gamma_\varepsilon:=\left\{x\in E(K):\hat h_L(x)\leq \frac\varepsilon{[K:k(\mathbb{P}^1)]}\sum_{v\in \mathfrak{Bad}(E/K)} {n_v}^{-1}\right\}
\]
to be the set of small points. Notice that $\Gamma_0 = E(K)_{\mathrm{tor}}$.  In the following we shall prove
\begin{thm}\label{ThmSmallPtECWithInsepDeg}
For the inseparable degree $p^{\nu}$ of the modular map $j \colon B\to \P^1_k$, we have
    \[
    \#\Gamma_\varepsilon\leq\frac {108}{1-24\varepsilon}(p^{\nu})^2\max\{1, 2g(B)-1\}^2.
    \]
 \end{thm}

The proof follows the guideline in $\mathsection$\ref{SectionProofOfFirstBound}.

\subsubsection{Existence of small section} By Zhang's fundamental inequality (see the proof of the Theorem of Successive Minima, \cite[Theorem~5.2]{Zhang_PosLBArithVar}) and the calculation of arithmetic volume Proposition~\ref{ArithVolEC}, the following holds true: For $m\gg 1$, there is a non-zero section $s\in H^0(A,mL)$ such that 
\begin{equation}\label{EqSmallSectionEC}
\frac{1}{m}h_{m\overline{L}(\mathbf f)}(s)\geq \frac{\vol (\bar L(\mathbf f)) }{[K:k(\mathbb{P}^1)] 2\deg L}+o(1)=\frac 1 {[K:k(\mathbb{P}^1)]}\sum_{v\in \mathfrak{Bad}(E/K)}\frac {1} {24n_v}+o(1),
\end{equation}
where 
\vskip -1.4em
\[
h_{m\overline{L}(\mathbf f)}(s):=\frac 1 {[K:k(\PP^1)]}\sum_{v\in M_B} -\log \|s\|_{m\bar L(\mathbf f), v,\sup}.
\]

\subsubsection{Zero estimate.} In this case, we only need to apply the Pigeonhole Principle to get: there is a point $x\in \Gamma_{\varepsilon}$ with vanishing order $\ell=\ord_x s\leq\frac{m}{\#\Gamma_{\varepsilon}}$.

\subsubsection{Liouville inequality.} We still have the  inequality \eqref{EqHeightInequality}, which in the current case becomes
\vskip -1.4em
\begin{equation}\label{EqHeightInequalityEC}
- \frac{1}{m} \frac{1}{[K: k(\mathbb{P}^1)] } \sum_{v \in M_K}   \log \|\mathrm{jet}^{\ell}s(x)\|_{v} \le  \hat{h}_L(x) + \frac{3}{2m} \ell h_{\mathrm{Fal}}(E) \le  \hat{h}_L(x) + \frac{3}{2\#\Gamma_{\varepsilon}} h_{\mathrm{Fal}}(E).
\end{equation}
The left hand side can be estimated as in Lemma~\ref{LemmaJetNonArch}: We have 
 \[
\|\mathrm{jet}^{\ell}s(x) \|_v \le \sup_{y \in E(K_v)}\|s(y)\|_{v}.
\]
But $f_v(E(K_v)) =0$ because $f_v|_{\Phi_v}=0$ by our choice of $\mathbf{f}=(f_v)$. So  
\[
\sup_{y \in E(K_v)}\|s(y)\|_{v} =  \sup_{y \in E(K_v)}\|s(y)\|_{m\overline{L}(\mathbf f), v} \quad\text{ for all }v \in M_K.
\]
Taking $-\log$ and adding all places up, we get
 \[
- \frac{1}{[K: k(\mathbb{P}^1)] } \sum\nolimits_{v \in M_K}   \log \|\mathrm{jet}^{\ell}s(x)\|_{v}  \geq h_{m\overline{L}(\mathbf{f})}(s).
 \]
 Therefore by \eqref{EqSmallSectionEC} and \eqref{EqHeightInequalityEC} and letting $m\to \infty$, we get
  \begin{equation}\label{EqHeightIneqEC}
    \#\Gamma_\varepsilon\leq \frac{36}{1-24\varepsilon}\cdot \frac{h_{\mathrm{Fal}}(E)}{\frac {1}{[K:k(\mathbb{P}^1)]}\sum_{v\in \mathfrak{Bad}(E/K)} {n_v}^{-1}}.
\end{equation}
Notice that $\mathfrak{Bad}(E/K)\not=\emptyset$, so the denominator is $>0$.

\subsubsection{Applying Arakelov inequality}
Since $E/K$ is a traceless elliptic curve, we have $\#\mathfrak{Bad}(E/K)\ge 1$ (and $\#\mathfrak{Bad}(E/K)\ge 3$ if $g(B)=0$).

Recall the Arakelov inequality for family of elliptic curves from \cite[Thm.~3]{GoldfeldSzpiro_BoundTateSha}. %Denote by $\#\mathfrak{Bad}(E/K) := \sum_{v\in \mathfrak{Bad}(E/K)} [k(v):k]$.
\begin{thm}
For the inseparable degree $p^{\nu}$ of the modular map  $j\colon B\to \P^1_k$, we have
    \begin{itemize}  
    \item[(i)]
    $p^{\nu}=1$ if and only if $E/K$ is not a Frobenius pullback. 
    \item[(ii)]The following Arakelov inequality holds 
        \begin{equation}\label{EqArakelovIneqP}
     [K:k(\mathbb{P}^1)]h_{\mathrm{Fal}}(E) \leq \frac {p^{\nu}}{2}(2g(B)-2+\#\mathfrak{Bad}(E/K)).
     \end{equation}
         \end{itemize}
\end{thm}
%We call an elliptic curve $E/K$ satisfying conditions (i) to be \textit{primitive}.

\begin{proof}[Proof of Theorem~\ref{ThmSmallPtECWithInsepDeg}]
Cauchy inequality implies
\[
\left(\sum_{v\in \mathfrak{Bad}(E/K)}{n_v}^{-1}\right) \left(\sum_{v\in \mathfrak{Bad}(E/K)} n_v\right) \ge (\#\mathfrak{Bad}(E/K))^2.
\]
But $\sum_{v\in \mathfrak{Bad}(E/K)} n_v = 12[K:k(\mathbb{P}^1)]h_{\mathrm{Fal}}(E)$ by the Faltings--Silverman Formula. So %, we get
\begin{equation}\label{EqCharpBlabla}
\frac{1}{\sum_{v\in \mathfrak{Bad}(E/K)}{n_v}^{-1}} \le  12 \frac{[K:k(\mathbb{P}^1)]h_{\mathrm{Fal}}(E)}{\left(\#\mathfrak{Bad}(E/K)\right)^2}.
\end{equation}
%The second term on the left hand side can be replaced by the Faltings--Silverman Formula
%%    \[
%   [K:k(\mathbb{P}^1)]h_{\mathrm{Fal}}(E)=\frac 1 {12} \sum_{v\in \mathfrak{Bad}(E/K)} n_v.
%    \]
 % Therefore we have
  \begin{align}\label{EqECSzpiroRatioBound}
 \begin{split}
 \frac {[K:k(\mathbb{P}^1)]h_{\mathrm{Fal}}(E)} {\sum_{v\in \mathfrak{Bad}(E/K)}{n_v}^{-1}} & \le 12\frac{([K:k(\mathbb{P}^1)]h_{\mathrm{Fal}}(E))^2}{(\#\mathfrak{Bad}(E/K))^2} \\
 & \le 12 \left( \frac{p^{\nu}(g(B)-1)}{\#\mathfrak{Bad}(E/K)} +\frac{p^{\nu}}{2} \right)^2 \quad \text{ by \eqref{EqArakelovIneqP}} \\
 & \le 3 (p^{\nu})^2 \left(1+\frac {2g(B)-2}{\#\mathfrak{Bad}(E/K)} \right)^2 \le 3 (p^{\nu})^2 \max\{1,2g(B)-1\}^2.
 \end{split}
 \end{align}
  We are done by \eqref{EqHeightIneqEC}.
\end{proof}
Moreover as an immediate consequence of \eqref{EqCharpBlabla} with \eqref{EqArakelovIneqP}, we have
\begin{equation}\label{EqECSzpiroRatioBoundMax}
\frac{[K:k(\mathbb{P}^1)] }{\sum_{v \in \mathfrak{Bad}(E/K)} n_v^{-1}} \le 6p^{\nu} [K:k(\mathbb{P}^1)] \max\{1,2g(B)-1\}.
\end{equation}

\subsection{Handling inseparable points}\label{SubsectionHandlingInseparablePoints}
The traceless semi-stable elliptic curve $E/K=k(B)$ induces the modular map $j \colon B \rightarrow \mathbb{P}^1_k$, whose inseparable degree is denoted by $p^{\nu}$. Hence there exists an elliptic curve $E_0/K$ such that $E = E_0^{(p^{\nu})}$. Then $E_0$ is also traceless and semi-stable because $F^e \colon E_0\to E_0^{(p^{\nu})}=E$ is a $K$-isogeny. There is a natural identification $E(K)=E_0(K^{\frac 1 {p^{\nu}}})$.

Let $K^\mathrm{insep}:=\bigcup_{n\geq1} K^{\frac 1 {p^n}}$. The following proposition is essentially due to R\"ossler \cite{Rossler_InsepI}; see also Yuan \cite[Prop.~3.6]{Yuan2021Positivity},.%Motivated by an idea of Yuan \cite{Yuan_PosHodgeAbVar}, we bound the degree of all points in $E_0(K^\mathrm{insep})$, and thus obtain a bound for $\#\Gamma_\varepsilon$ without using the inseparable degree $p^{\nu}$. 

\begin{prop}\label{PropHandlingInseparablePoints}
    We have $[K(P):K]\leq 12\max\{1,2g(B)-1\}$ for all  $P \in E_0(K^\mathrm{insep})$.
\end{prop}

\begin{proof}
    Let $\pi \colon \cE_0 \to B$ be the N\'eron model of $E_0/K$. Let $\calP_0$ be the schematic closure of $P$ in $\calE_0$, and $\mathrm{n} \colon \calP \to \calP_0$ be its normalization. Then the structural morphism $\psi\colon\calP\to B$ is a radical quasi-finite morphism (it is not necessarily finite since $\calE_0\to B$ need not be proper). It suffices to bound the degree of the morphism $\psi$.

    As $\calP_0$ is purely inseparable over $B$, we have a canonical isomorphism 
    \[
    \Omega^1_{\calP_0/k} = \Omega^1_{\calP_0/B}.
    \]
So the canonical surjection $(\Omega^1_{\calE_0/B})|_{\calP_0}\to \Omega^1_{\calP_0/B}$ can be written as
  \[
    \tau_0 \colon (\Omega^1_{\calE_0/B})|_{\calP_0}\to \Omega^1_{\calP_0/k}.
  \]
Next let us pull back $\tau_0 $ along the normalization $\mathrm{n} \colon \calP \to \calP_0$. We have $\mathrm{n}^*(\Omega^1_{\calE_0/B})|_{\calP_0} \cong \psi^*\pi_* \Omega_{\calE_0}$ since $\calE_0/B$ is a group scheme. Hence from $\tau_0$, we obtain a non-zero morphism between line bundles on $\calP$
\[
    \tau \colon \psi^*\pi_* \Omega_{\calE_0}\to \Omega^1_{\calP/k}.
    \]

    Let $\calP^{\mathrm{c}}$ be the unique smooth compactification of $\calP$. Then by properness the morphism $\psi$ extends to $\psi^{\mathrm{c}} \colon \calP^{\mathrm{c}}\to B$. Let $E_0$ be the reduced closed subscheme of $B$ on which $\calE_0$ has bad reduction, and let $E:=(\psi^{\mathrm{c}})^{-1}(E_0)$ endowed with the reduced closed scheme structure. Then \cite[Prop.~3.6]{Yuan2021Positivity} asserts that $\tau$ extends to a morphism $\tau^{\mathrm{c}}\colon (\psi^{\mathrm{c}})^*\pi_* \Omega_{\calE_0}\to \Omega^1_{\calP^{\mathrm{c}}/k}(E)$. Taking the degrees of both sides, we get
    \[
    \deg (\psi^{\mathrm{c}})\cdot \deg(\pi_*\Omega_{\calE_0})\leq 2g(B)-2+\deg (E_0)=2g(B)-2+\#\mathfrak{Bad}(E/K).
    \]
    Let $n_{v,0}$ be the size of the N\'eron component group of $E_0$ at the place $v$. The Faltings--Silverman Formula says that 
   \[
    12 \deg(\pi_* \Omega_{\calE_0})=\sum\nolimits_{v\in \mathfrak{Bad}(E/K)}n_{v,0} \ge \#\mathfrak{Bad}(E/K).
    \]
    \vskip -0.5em
    Therefore 
        \[
    [K(P):K]=\deg (\psi^{\mathrm{c}})\leq 12\frac {2g(B)-2+\#\mathfrak{Bad}(E/K)}{\#\mathfrak{Bad}(E/K)}\leq 12\max\{1,2g(B)-1\}.
    \qedhere
    \]
    \end{proof}

\subsection{Proof of Theorem~\ref{ThmBoundECFFCharP}, part (i) and (ii): rational torsion points}
If $\mathrm{char}(k) = 0$, then Theorem~\ref{ThmSmallPtECWithInsepDeg} suffices to conclude that $\#E(K)_{\mathrm{tor}} \le 108 (g(B)-1)^2$. 

Assume $\mathrm{char}(k)=p>0$. Let $E_0/K$ be the elliptic curve constructed at the beginning of $\mathsection$\ref{SubsectionHandlingInseparablePoints}. Then $E(K)\subseteq E_0(K^{\mathrm {insep}})$.

    Fix $p^{\nu_0}$ to be the largest power of $p$ satisfying $p^{\nu_0}\leq 12\max\{1,2g(B)-1\}$. Then Proposition~\ref{PropHandlingInseparablePoints} implies  $E_0(K^{\mathrm {insep}})=E_0^{(p^{\nu_0})}(K)$.
    
   Apply Theorem~\ref{ThmSmallPtECWithInsepDeg} to $E_0^{(p^{\nu_0})}$. Then we get
   \[
   \# E_0^{(p^{\nu_0})}(K)_{\mathrm {tor}}\leq 108p^{2 \nu_0}\max\{1,2g(B)-1\}^2\leq 108 \cdot 144 \max\{1,2g(B)-1\}^4.
   \]
   We are done since $E(K)_{\mathrm{tor}} \subseteq E_0^{(p^{\nu_0})}(K)_{\mathrm {tor}}$.
\qed

%\begin{rmk}
%    It seems that Poonen has proved boundedness in this case. Actually he proved the gonality of modular curve $X(\Gamma)/\bar{\FF}_p$ tends to $\infty$ as $\Gamma\to \{1\}\subset \mathrm{GL}_2(\hat \ZZ)$. (I think his result implies at least uniform boundedness.) His result is not effective though.
%\end{rmk}

\subsection{Proof of Theorem~\ref{ThmBoundECFFCharP}.(iii): Lang-Silverman Conjecture}
Take $\varepsilon=\frac 1 {72}$. Let $P\in E(K)$ be a non-torsion point. Let $N$ be the smallest positive integer with $2N\geq\frac {108}{1-24\varepsilon}p^{2e}\max\{1, 2g(B)-1\}^2$ . Then $[N]P\notin \Gamma_\varepsilon$; otherwise there are $2N+1$ distinct points, namely $0,\pm P,\cdots,\pm[N]P$, contained in $\Gamma_\varepsilon$, contradicting Theorem~\ref{ThmSmallPtECWithInsepDeg}. So
\[
N^2\hat h_L(P)=\hat h_L([N]P)\geq \frac\varepsilon{[K:k(\PP^1)]}\sum_{v\in M'_K} n_v^{-1} \ge \frac {\varepsilon \max\{h_{\mathrm{Fal}}(E),1\}}{6 (p^{\nu})^2 [K:k(\mathbb{P}^1)] \max\{1,2g(B)-1\}^2}
\]
where the last inequality follows from \eqref{EqECSzpiroRatioBound} and \eqref{EqECSzpiroRatioBoundMax}. 
Substituting $\varepsilon=\frac{1}{72}$, we obtain
\begin{equation}\label{EqLSCharp}
\hat{h}_L(P) \ge \frac{\max\{h_{\mathrm{Fal}}(E),1\}}{2^4\cdot 3^{11}(p^{\nu})^6[K:k(\mathbb{P}^1)]\left(2g(B)-1\right)^6}.
\end{equation}

When $\operatorname{char} k=0$, we are done since $\mathrm{gon}(B) = [K:k(\mathbb{P}^1)]$. 
From now on,  assume $\mathrm{char} k=p>0$. Let $E_0/K$ be the elliptic curve constructed at the beginning of $\mathsection$\ref{SubsectionHandlingInseparablePoints}. Then $E(K)\subseteq E_0(K^{\mathrm {insep}})$. Fix $p^{\nu_0}$ to be the largest power of $p$ satisfying $p^{\nu_0}\leq 12\max\{1,2g(B)-1\}$. Then Proposition~\ref{PropHandlingInseparablePoints} implies  $E_0(K^{\mathrm {insep}})=E_0^{(p^{\nu_0})}(K)$. We therefore regard $P\in E(K)$ as a non-torsion point in $E_0^{(p^{\nu_0})}(K)$.
Denote by $\hat h_{0}(\cdot)$ the canonical height associated to the principal polarization on $E_0^{(p^{\nu_0})}$. Applying \eqref{EqLSCharp} to $P\in E_0^{(p^{\nu_0})}(K)$, we get
\begin{flalign*}
\hat h_L(P)=p^{\nu_0-\nu}\hat h_{0}(P)&\geq p^{\nu_0-\nu}\frac{\max\{h_{\mathrm{Fal}}(E_0^{(p^{\nu_0})}),1\}}{16\cdot 3^{11}p^{6\nu_0}[K:k(\mathbb{P}^1)] \left(2g(B)-1\right)^6}\\
    &=(p^{\nu})^{-1}\frac{\max\{p^{\nu_0-\nu}h_{\mathrm{Fal}}(E),1\}}{2^4\cdot 3^{11}p^{5\nu_0}[K:k(\mathbb{P}^1)] \left(2g(B)-1\right)^6}\\
    &=\frac{(p^{\nu})^{-2}}{2^4\cdot 3^{11} [K:k(\mathbb{P}^1)] \left( 2g(B)-1 \right)^6} \max\left\{\frac{h_{\mathrm{Fal}}(E)}{p^{4\nu_0}} , \frac{p^{\nu}}{p^{5\nu_0}}\right\} \\
    &\ge \frac{(p^{\nu})^{-2}}{2^4\cdot 3^{11} [K:k(\mathbb{P}^1)] \left( 2g(B)-1 \right)^6} \frac{\max\{h_{\mathrm{Fal}}(E),1\}}{p^{5\nu_0}} \\
    & \ge \frac{\max\{h_{\mathrm{Fal}}(E),1\}}{2^{14} \cdot 3^{16} (p^{\nu})^2 [K:k(\mathbb{P}^1)]  |2g(B)-1|^{11}}. && \qed
 %   p^{-2e}\cdot\frac{h_{\mathrm{Fal}}(E)}{2^{11}\cdot 3^{15}\left(2g(B)-1\right)^{10}}.
\end{flalign*}

\appendix
\renewcommand{\thesection}{\Alph{section}}
\setcounter{section}{0}

\section{Maximal Slope of (Co)Tangent Space}\label{SectionAppendix}

In this appendix, we give an upper bound on the maximal slope of $\bar{\lie A}^\vee$, where $A/K$ is an abelian variety over a function field $K=\CC(B)$. Our proof is an adaptation of the proof of  \cite{Gau_Explicit_Abvar} in the case of number fields, and also applies to arbitrary $k=\bar{k}$.% We assume $\mathrm{char} k=0$ here for simplicity. Our goal is to prove:
\begin{thm}\label{ThmMaxSlopeLieDual}
$\displaystyle
    \hat{\mu}_{\max}(\bar{\lie A}^\vee) \le \left(\frac{g}{2}+1\right) h_{\mathrm{Fal}}(A)$.
\end{thm}
The same argument also shows $\hat{\mu}_{\max}(\bar{\lie A}) \le h_{\mathrm{Fal}}(A)$. See the remark at the end.

\subsection{The Harder-Narasimhan Polygon}
For every vector bundle $\cE$ on $B$, there is a canonical filtration (called the Harder-Narasimhan filtration, or HN filtration for short) $$0=\cE_0\subseteq \cE_1\subseteq\cdots\subseteq \cE_m=\cE,$$with every subquotient $\cE_i/\cE_{i+1}$  semistable, and $\mu(\cE_1/\cE_0)\geq\mu(\cE_2/\cE_1)\geq\cdots\geq \mu (\cE_m/\cE_{m-1})$. 

The {\it Harder-Narasimhan polygon (HN polygon)} is the piecewise linear function 
\[
P_{\cE} \colon [0,\rank \cE]\to \R
\]
joining the points given by the Harder-Narasimhan filtration: $\{(\rank \cE_i,\deg \cE_i)\}$; 
here we set $\deg 0=0$. In particular, $\mu_{\max}(\cE)=P_{\cE}(1)$.

We list some basic properties of HN polygon.
\begin{lemma}\label{LemmaApp}
The followings hold true:
\begin{enumerate}
    \item[(i)] $P_{\cE^\vee}(x)=P_{\cE}(\rank \cE-x)-\deg \cE$.
    \item[(ii)] $P_{\cE\otimes \cL}(x) = P_{\cE}(x) + x \deg \cL$ for any line bundle $\cL$ on $B$.
    \item[(iii)] Let $\varphi \colon \cE_{K}\to \cE'_{K}$ be an injective $K$-linear map between generic fibers of vector bundles $\cE$ and $\cE'$. Assume that  $\varphi$ extends to a morphism $\cE \to \cE'$. Then
    \[
    P_{\cE}(x)\leq P_{\cE'}(x) \quad \text{ for all }x \in [0,\rank \cE].
    \]
%    In general, we have $P_{\cE}(x)\leq P_{\cE'}(x)+xh(\varphi)$.
    \item[(iv)] Let $\varphi \colon \cE_{K}\to \cE'_{K}$ be a surjective $K$-linear map between generic fibers of vector bundles $\cE$ and $\cE'$. Assume that  $\varphi$ extends to  a $B$-morphism  $\cE \to \cE'$. Then %for any $y\in [0, \rank \cE']$, we have
    \[
    P_{\cE'}(y)\leq P_{\cE}(y+\dim\ker \varphi)-\deg\cE+\deg \cE', \quad \text{for all }  y\in [0, \rank \cE'] .
    \]
\end{enumerate}
\end{lemma}

In (iii), without assuming the extension of $\varphi$ to a $B$-morphism $\cE \to \cE'$, we have $P_{\cE}(x)\leq P_{\cE'}(x)+xh(\varphi)$. Then this recovers the classical slope inequality $\mu_{\max}(\cE)\leq\mu_{\max}(\cE')+h(\varphi)$ by taking $x=1$. Similarly in (iv), without assuming the extension of $\varphi$, we have $P_{\cE'}(y)\leq P_{\cE}(y+\dim\ker \varphi)+(\rank \cE'-y)h(\varphi)-\deg\cE+\deg \cE'$.

\subsection{Evaluation maps}
Fix a torsion point \( a \in A(K) \). Define:
\begin{align*}
\delta_0 &\colon H^0(A, L^{\otimes 2}) \longrightarrow L_a^{\otimes 2}, \quad s \longmapsto s(a), \\
\delta_1 &\colon \ker \delta_0 \longrightarrow \Omega_A \otimes L_a^{\otimes 2}, \quad s \longmapsto \mathrm{jet}^1 s(a).
\end{align*}
%namely,
%\begin{itemize}
%    \item \( \delta_0 \) is the evaluation map of order 0,
%    \item \( \delta_1 \) sends any section $s$ which vanishes at $a$ to the first-order jet of \( s \) at \( a \).% defined on sections that vanishes at \( a \).
%    \item \( \Lie A \) denotes the tangent space at the origin of \(A\).
%    \item \( \mathrm{jet}^1 s(a) \) is the first-order jet of \( s \) at \( a \), taking values in \(\Lie A  \otimes a^* L^{\otimes 2} \).
%\end{itemize}

The map $\delta_0$ is surjective since $L^{\otimes 2}$ is globally generated. The map $\delta_1$ is surjective if and only if the complete linear system $H^0(A,L^{\otimes 2})$ separates tangent vectors at $a$. This holds for generic $a\in A(\bar K)$, by generic smoothness of $A\to \PP(H^0(A,L^{\otimes 2}))$. Moreover, we can take $a$ to be a torsion point, by possibly enlarging $K$ (the integral structure on $\Lie A$ is that on the tensor product, so this do not affect later arguments). 

%\begin{rmk}
%    One can take $L^{\otimes 3}$ to fix the failure of generic smoothness in positive characteristics, so that this proof will hold for $K=k(B)$ with $\operatorname{char} k>0$.% Although as we will see, another important ingredient in our proof, the Multiplicity Estimate, also fails in this setting.
%\end{rmk}

%\begin{defn}
%    Let $A/K$ be an abelian variety, $L$ be an ample symmetric line bundle on $A$, and $a\in A(\bar K)$ be a point. A Moret-Bailly model associated to the triple $(A,L,a)$ is a finite field extension $K'=k(B')/K$, with a triple $(\cA,\cL,\varepsilon_a)$ defined over $B'$, where \begin{itemize}
%        \item $\cA$ is a semiabelian scheme over $B'$ whose generic fiber is $A_{K'}$;
%        \item $\cL$ is a cubical line bundle on $\cA$, whose generic fiber is $L$;
%        \item $\varepsilon_a$ is a section of $\pi':\cA\to B'$ whose geometric generic fiber is $a$.
%    \end{itemize}
%    such that $$\mathcal{K}(\cL^2):=\{a\in \cA:T_a^* \cL^2\simeq \cL^2\}$$ is a finite group scheme over $B'$. 
%\end{defn}

By \cite[§II, 1.2.2]{MB_model}, there exists a {\it Moret-Bailly model} $(\calA,{\calL}_2,\varepsilon_a)$ of $(A,L^{\otimes 2},a)$ defined over a finite field extension $K'$ of $K$, \textit{i.e.} 
\begin{itemize}
        \item $\cA$ is a semi-abelian scheme over $B'$, whose generic fiber is $A_{K'}$;
        \item $\cL$ is a cubical line bundle on $\cA$, whose generic fiber is $L^{\otimes 2}\otimes_K K'$;
        \item $\varepsilon_a$ is a section of $\pi'\colon \cA\to B'$ whose geometric generic fiber is $a$.
\end{itemize}
such that 
\[
\mathcal{K}(\cL_2):=\{a\in \cA(B') :T_a^* \cL_2\simeq \cL_2\}
\]
 is a finite group scheme over $B'$. For a quick survey on this subject, see also \cite[§4.3]{Bost_intrin_ht}.

Endow the sources and targets of $\delta_0,\delta_1$ with the integral structures given by $(\calA,{\calL}_2,\varepsilon_a)$. Then $\pi'_*\calL_2$ is a vector bundle on $B'$; see \cite[Intro., (7.3)]{MB_model}. 
Both $\delta_0$ and $\delta_1$ extend to %maps between vector bundles on $B'$
\begin{align*}
\bm{\delta}_0 &\colon \pi'_*\cL_2 \longrightarrow \varepsilon_a^*\cL_2, \\
\bm{\delta}_1 &\colon \ker \delta_0 \longrightarrow \lie(\cA/B')^\vee \otimes \varepsilon_a^*\cL_2.
\end{align*}

\subsection{Proof of Theorem~\ref{ThmMaxSlopeLieDual}}
Since $\varepsilon_a$ is a torsion section, we have $\deg(\varepsilon_a^*\cL_2) = 0$. So
\[
-\deg (\pi'_*\cL_2)^\vee +\deg (\ker\bm{\delta}_0)^\vee = \deg (\pi'_*\cL_2)-\deg (\ker\bm{\delta}_0) = \deg(\varepsilon_a^*\cL_2) = 0.
\]
Apply Lemma~\ref{LemmaApp}.(iv) for the dual of $\ker(\bm{\delta}_0)\hookrightarrow \pi'_*\cL_2$, we then get
\begin{equation}\label{EqAppA1}
P_{(\ker\bm{\delta}_0)^\vee}(y) \le P_{(\pi'_*\cL_2)^\vee}(y+1) -\deg (\pi'_*\cL_2)^\vee +\deg (\ker\bm{\delta}_0)^\vee = P_{(\pi'_*\cL_2)^\vee}(y+1).
\end{equation}

Now we compute
\begin{align}\label{EqAppA2}
\begin{split}
P_{\lie(\cA/B')}(g-1)& = P_{\lie(\cA/B') \otimes \varepsilon_a^*\cL_2^\vee}(g-1) \quad\text{ by Lemma~\ref{LemmaApp}.(ii)}	 \\
& \leq P_{(\ker\bm{\delta}_0)^\vee}(g-1) \qquad\qquad\text{by Lemma~\ref{LemmaApp}.(iii) applied to $\bm{\delta}_1^\vee$}		\\
& \le P_{(\pi'_*\cL_2)^\vee}(g) \qquad\qquad\qquad\text{by \eqref{EqAppA1} with $y=g-1$.}	
\end{split}	
\end{align}
Notice that $\mathcal{K}(\cL_2)$ is a group of isometric automorphisms of $\pi'_*\cL_2$ which acts irreducibly on $(\pi'_*\cL_2)_{K'} = H^0(A_{K'}, L^{\otimes 2})$. So $\pi'_*\cL_2$ is semi-stable.\footnote{For number fields, this is \cite[Prop.~A.3]{Bost_MWBourbaki}. The proof works verbally for function fields.} So $(\pi'_*\cL_2)^\vee$ is also semistable. So 
\begin{equation}\label{EqAppA3}
P_{(\pi'_*\cL_2)^\vee}(g) = g \cdot \mu((\pi'_*\cL_2)^\vee),
\end{equation}
 and Moret--Bailly's formula \cite[Chap.~VIII, Thm.~1.3, {$FCC(\Spec \ZZ[\frac 1 d],g,d)$}]{MB_model} gives 
\begin{equation}\label{EqAppA4}
\mu((\pi'_*\cL_2)^\vee) = \frac 1 2 [K': \C(\mathbb{P}^1)] h_{\mathrm{Fal}}(A).
\end{equation}
\vskip -0.4em
Therefore we are done by
\begin{flalign*}
 \mu_{\max}(\lie(\cA/B')^\vee) & =P_{\lie(\cA/B')^\vee}(1) \\
 & =P_{\Lie (\cA/B')}(g-1)-\deg \Lie (\cA/B') \qquad\text{ by Lemma~\ref{LemmaApp}.(i)} \\
 & \le \frac{g}{2} [K': \C(\mathbb{P}^1)] h_{\mathrm{Fal}}(A) + [K': \C(\mathbb{P}^1)] h_{\mathrm{Fal}}(A) \text{ by \eqref{EqAppA2}, \eqref{EqAppA3}, \eqref{EqAppA4}.}	&& \qed
\end{flalign*}

\begin{rmk}
We can also prove $\hat{\mu}_{\max}(\bar{\lie A}) \le h_{\mathrm{Fal}}(A)$ as follows:
\[
\begin{array}{llll}
        \mu_{\max}(\lie (\cA/B')) & = P_{\lie(\cA/B')}(1)  \\
        & = P_{\lie(\cA/B') \otimes \varepsilon_a^*\calL_2^\vee}(1) & \text{ by Lemma~\ref{LemmaApp}.(ii)} \\
         & \le P_{(\ker\bm{\delta}_0)^\vee}(1) & \text{ by Lemma~\ref{LemmaApp}.(iii) applied to $\bm{\delta}_1^\vee$} \\
        & \le P_{(\pi'_*\cL_2)^\vee}(2) & \text{ by \eqref{EqAppA1} applied to $y=1$} \\
         & = 2\mu((\pi'_*\cL_2)^\vee)  & \text{ by the semi-stability of $\pi'_*\cL_2$} \\
         & = [K':\C(\mathbb{P}^1)] h_{\mathrm{Fal}}(A) & \text{ by \eqref{EqAppA4}.} \\
\end{array}
\]
\end{rmk}

\section{Line bundle version of  Zarhin's trick}\label{SectionZarhin}
\begin{lemma}\label{LemmaZarhinLineBundle}
Let $A/K$ be an abelian variety and let $L$ be a symmetric ample line bundle on $A$ defined over $K$. Set $Z(A):=A^4\times (A^\vee)^4$. 
Then there exist a symmetric ample line bundle $M$  on $Z(A)$, defined over $K$, such that $h^0(Z(A),M)=1$, and a $K$-isogeny
\[
\pi \colon  A^8\longrightarrow Z(A)
\]
such that $\pi^*M$ and $L^{\boxtimes 8}$ define the same polarization on $A^8$.
\end{lemma}
\begin{proof}
For any line bundle $N$ on an abelian variety, denote by $\phi_N$ the map $t_x^*N \otimes N^{\otimes -1}$.

Denote  by $D:=h^0(A,L)$. Then $\deg\phi_L = D^2$. Let $P_A$ be the rigidified Poincar\'e bundle on $A\times A^\vee$, and let  $q_1,q_2$ be the projections of $A\times A^\vee$ to both factors. By \cite[Thm.~4.1]{BirkenhakeLangeModuli},
\[
L^\dagger:= \left(\det q_{2*}(q_1^*L\otimes P_A) \right)^{\otimes -1}
\]
 is an ample line bundle on $A^\vee$, defined over $K$, and $\phi_{L^\dagger}\circ \phi_L = [D]_A$. Moreover, $L^{\dagger}$ is symmetric since (using $([-1]_A\times 1_A)^*P_A \cong (1_A\times [-1]_{A^\vee})^*P_A \cong P_A^{\otimes -1}$)
\[
[-1]_{A^\vee}^* (q_1^*L\otimes P_A) \cong q_1^*L\otimes P_A^{\otimes -1} \cong  ([-1]_A\times 1_{A^\vee})^*(q_1^*L\otimes P_A) \cong ([-1]_A^*q_1^*L)\otimes P_A \cong q_1^*L\otimes P_A.
\]

The four-square theorem gives $a,b,c,d\in \Z$ such that $a^2+b^2+c^2+d^2 = D-1$. Then %the ``quaternion''
\[
\cI : =\begin{bmatrix}
 a & -b & -c & -d \\
 b & a & d & -c \\
 c & -d & a & b \\
 d & c & -b & a
\end{bmatrix}
\in \mathrm{Mat}_{4\times 4}(\Z) \subseteq \mathrm{End}(A^4)
\]
satisfies $\cI \circ \cI^{\mathrm{t}} = \cI^{\mathrm{t}} \circ \cI = [D-1]$; here $\cI^{\mathrm{t}}$ is also seen as an element in $\mathrm{End}(A^4)$.

Since $\phi_{L^{\boxtimes 4}} = (\phi_L, \phi_L, \phi_L, \phi_L)$ is diagonal, we have
\[
\cI \circ \phi_{L^{\boxtimes 4}} = \phi_{L^{\boxtimes 4}} \circ \cI, \qquad \cI^{\mathrm{t}} \circ \phi_{L^{\boxtimes 4}} = \phi_{L^{\boxtimes 4}} \circ \cI^{\mathrm{t}}.
\]
To ease notation, denote by $X:=A^4$. Then $Z(A) = X\times X^\vee$. Let $p_1, p_2$ be the projections, and let $P_X$ be the normalized Poincar\'{e} bundle on $X\times X^\vee$. Define an endomorphism
\[
\rho = \begin{bmatrix} -\cI^{\mathrm{t}}& 0 \\ 0 & 1 \end{bmatrix} \colon X\times X^\vee\longrightarrow X\times X^\vee,  \qquad (x,y) \mapsto (-\cI^{\mathrm{t}} x, y).
\]

We claim that the line bundle
\[
M:= p_1^*L^{\boxtimes 4} \otimes p_2^*L^{\dagger \boxtimes 4} \otimes\rho^*P_X
\]
is our desired line bundle on $Z(A) = X\times X^\vee$, with
\[
\pi = \begin{bmatrix} \cI & -1 \\ \phi_{L^{\boxtimes 4}} & 0 \end{bmatrix} \colon A^8 = X \times X \longrightarrow Z(A) = X \times X^\vee, \qquad  (x,y) \mapsto (\cI x - y, \phi_{L^{\boxtimes 4}}(x)).
\]

First $M$ and $\pi$ are clearly defined over $K$. Moreover, $M$ is symmetric: the first two factors are symmetric, while $
([-1]_X\times[-1]_{X^\vee})^*P_X\simeq P_X$ and $\rho$ commutes with $[-1]$.

Next, it is easy to compute  that
\[
\ker\pi= \{(x,\cI x): x \in \ker \phi_{L^{\boxtimes 4}}\}, \quad \phi_M = \begin{bmatrix}
\phi_{L^{\boxtimes 4}} & - \cI\\
-\cI^{\mathrm{t}} &\phi_{L^{\dagger\boxtimes 4}}
\end{bmatrix}, \quad \pi^\vee = \begin{bmatrix} \cI^{\mathrm{t}} & \phi_{L^{\boxtimes 4}} \\ -1 & 0 \end{bmatrix}.
\]
So $\pi$ is an isogeny and $\deg \pi = \deg \phi_{L^{\boxtimes 4}} = D^8$ (similarly $\deg \pi^\vee = D^8$),  and
\begin{equation}\label{EqAppB}
\pi^\vee \circ \phi_M \circ \pi = \begin{bmatrix} \phi_{L^{\boxtimes 4}}& 0 \\ 0 & \phi_{L^{\boxtimes 4}} \end{bmatrix} = \phi_{L^{\boxtimes 8}}\quad \text{ using }\phi_{L^\dagger}\circ \phi_L = [D]_A.
\end{equation}
But $\phi_{\pi^*M} = \pi^\vee \circ \phi_M \circ \pi $. So $\pi^*M$ and $L^{\boxtimes8}$ are algebraically equivalent up to an element of $\operatorname{Pic}^0(A^8)$. In particular, $\pi^*M$ is ample. Since $\pi$ is finite and surjective, $M$ is ample.

Finally to prove $h^0(Z(A),M)=1$, it suffices to show $\deg\phi_M=1$. Now \eqref{EqAppB} implies
\[
(\deg\pi^\vee)(\deg\phi_M)(\deg\pi) =\deg\phi_{L^{\boxtimes8}} =D^{16}.
\]
So $\deg\phi_M=1$ since $\deg \pi = \deg \phi_{L^{\boxtimes 4}} = D^8$. We are done.
\end{proof}

\bibliographystyle{alpha}
\bibliography{bibliography}

\end{document}